\documentclass[12pt]{amsart}
\usepackage[left=3cm,tmargin=2cm,bmargin=2.5cm,centering]{geometry}

\usepackage{hyperref}
\usepackage[utf8]{inputenc}
\usepackage[initials]{amsrefs}  
\usepackage{enumerate}
\usepackage{bm,amsmath,amssymb,amsthm}
\usepackage[all]{xy}
\usepackage{stmaryrd}
\usepackage{amsfonts}
\usepackage{dsfont}
\usepackage{comment}
\newtheorem{thm}{Theorem}[section]
\newtheorem{conj}[thm]{Conjecture}
\newtheorem{cor}[thm]{Corollary}
\newtheorem{lem}[thm]{Lemma}
\newtheorem{prop}[thm]{Proposition}
\newtheorem{hypo}[thm]{Hypothesis}
\newtheorem{heur}[thm]{Heuristic}
\theoremstyle{definition}
\newtheorem{defn}[thm]{Definition}
\newtheorem{ex}[thm]{Example}
\theoremstyle{remark}
\newtheorem{rem}[thm]{Remark}

\newtheorem*{claim}{Claim}

\theoremstyle{plain}

\newtheorem*{proposition*}{Proposition}
\newtheorem*{lemma*}{Lemma}
\newtheorem*{corollary*}{Corollary}
\newtheorem{criterion}[thm]{Criterion}
\newtheorem{hypothesis}[thm]{Hypothesis}

\theoremstyle{definition}

\numberwithin{equation}{section}

\renewcommand{\le}{\leqslant}
\renewcommand{\ge}{\geqslant}

\newcommand{\key}[1]{{\sf #1}}

\newcommand{\Lquote}[1]{``#1"}

\newcommand {\pnorm}[1]   {\left\lVert #1 \right\rVert}

\newcommand{\Mdemi}{\frac{1}{2}}

\newcommand{\qtext}[1]{\quad\text{#1}}

\newcommand{\bigp}[1]{\bigl(#1\bigr)}
\newcommand{\Bigp}[1]{\Bigl(#1\Bigr)}
\newcommand{\SB}{\backslash}
\newcommand{\set}[1]{\left\{#1\right\}}

\newcommand{\cusp}{\mathrm{cusp}}
\newcommand{\new}{\mathrm{new}}

\DeclareMathOperator {\GL} {GL}

\DeclareMathOperator {\Res}  {res}

\DeclareMathOperator {\Frob} {Frob}

\DeclareMathOperator {\vol}  {vol}

\DeclareMathOperator {\Gal}  {\TmG\Tma\Tml}

\DeclareMathOperator {\Hom}  {\TmH\Tmo\Tmm}

\DeclareMathOperator {\MSym} {Sym}

\DeclareMathOperator {\MHom} {Hom}

\DeclareMathOperator {\MTr}  {Tr}
\DeclareMathOperator {\Mtr}  {tr}

\DeclareMathOperator {\MRe}  {\Re e}
\DeclareMathOperator {\Mim}  {\Im m}

\newcommand{\BmA}{\mathbb{A}}

\newcommand{\BmC}{\mathbb{C}}

\newcommand{\BmH}{\mathbb{H}}

\newcommand{\BmN}{\mathbb{N}}

\newcommand{\BmQ}{\mathbb{Q}}
\newcommand{\BmR}{\mathbb{R}}

\newcommand{\BmZ}{\mathbb{Z}}
\newcommand{\CmA}{\mathcal{A}}

\newcommand{\CmF}{\mathcal{F}}
\newcommand{\CmG}{\mathcal{G}}
\newcommand{\CmH}{\mathcal{H}}

\newcommand{\CmL}{\mathcal{L}}

\newcommand{\CmS}{\mathcal{S}}

\newcommand{\CmV}{\mathcal{V}}

\newcommand{\FmF}{\mathfrak{F}}
\newcommand{\Fmf}{\mathfrak{f}}
\newcommand{\Fmn}{\mathfrak{n}}
\newcommand{\SmC}{\mathscr{C}}

\newcommand{\FmS}{\mathfrak{S}}

\newcommand{\Q}{\mathbb{Q}}
\newcommand{\Z}{\mathbb{Z}}
\newcommand{\A}{\mathbb{A}}
\newcommand{\B}{\mathbb{B}}
\newcommand{\C}{\mathbb{C}}
\newcommand{\N}{\mathbb{N}}
\newcommand{\abs}[1]{\ensuremath{\left|#1\right|}}
\newcommand{\Mun}{\mathds{1}}

\newcommand{\sumv}{\sum_{v\in\CmV_F}}
\newcommand{\Tpl}{\text{pl}}
\newcommand{\ur}{\mathrm{ur}}
\newcommand{\temp}{\text{temp}}

\newcommand{\gen}{\text{gen}}
\newcommand{\cut}{\text{cut}}
\newcommand{\main}{\text{main}}
\newcommand{\pol}{\text{pol}}
\newcommand{\scusp}{\text{sc}}

\newcommand{\fka}{{\mathfrak a}}

\newcommand{\fkf}{{\mathfrak f}}

\newcommand{\fkn}{{\mathfrak n}}
\newcommand{\fkp}{{\mathfrak p}}

\newcommand{\fkD}{{\mathfrak D}}
\newcommand{\fkG}{{\mathfrak G}}

\newcommand{\mB}{{\mathscr B}}
\newcommand{\mC}{{\mathscr C}}

\newcommand{\mL}{{\mathscr L}}

\newcommand{\mF}{{\mathscr F}}

\newcommand{\mS}{{\mathscr S}}
\newcommand{\mT}{{\mathscr T}}
\newcommand{\mP}{{\mathscr P}}

\newcommand{\mX}{{\mathscr X}}
\newcommand{\mY}{{\mathscr Y}}

\newcommand{\F}{\mathbb{F}}
\newcommand{\G}{\mathbb{G}}
\newcommand{\bF}{\mathbf{F}}
\newcommand{\bG}{\mathbf{G}}
\newcommand{\bN}{\mathbf{N}}
\newcommand{\bB}{\mathbf{B}}

\newcommand{\bM}{\mathbf{M}}
\newcommand{\bT}{\mathbf{T}}

\newcommand{\R}{\mathbb{R}}
\newcommand{\T}{\mathbb{T}}

\newcommand{\Stab}{\mathrm{Stab}}

\newcommand{\semis}{\mathrm{ss}} % ss for semisimple

\newcommand{\pl}{\hat{\mu}^{\mathrm{pl}}}
\newcommand{\plur}{\hat{\mu}^{\mathrm{pl,ur}}}
\newcommand{\plurtemp}{\hat{\mu}^{\mathrm{pl,ur,temp}}}
\newcommand{\ST}{\hat{\mu}^{\mathrm{ST}}}

\newcommand{\Spec}{\mathrm{Spec}\,}

\newcommand{\bs}{\backslash}

\newcommand{\Apt}{{\rm Apt}}
\newcommand{\Sym}{{\rm Sym}}
\newcommand{\ch}{{\rm ch}}
\newcommand{\Fr}{{\rm Fr}}
\newcommand{\ssconj}{{\rm ss-conj}}

\newcommand{\im}{{\rm Im\,}}

\newcommand{\coker}{{\rm coker\,}}

\newcommand{\rank}{{\rm rk}}
\newcommand{\gal}{{\rm Gal}}
\def\Gal{{\rm Gal}}

\newcommand{\aut}{{\rm Aut}}
\newcommand{\ind}{{\rm Ind}}
\newcommand{\nind}{{\rm n\textrm{-}ind}}

\newcommand{\id}{{\rm id}}
\newcommand{\Std}{{\rm Std}}
\newcommand{\sgn}{{\rm sgn}}
\newcommand{\tr}{{\rm tr}\,}
\renewcommand{\Hom}{{\rm Hom}}

\newcommand{\Aut}{{\rm Aut}}
\newcommand{\Out}{{\rm Out}}

\newcommand{\disc}{{\rm disc}}
\newcommand{\spec}{{\rm spec}}
\newcommand{\geom}{{\rm geom}}
\newcommand{\cont}{{\rm cont}}

\newcommand{\Tama}{{\rm Tama}}
\newcommand{\can}{{\rm can}}
\newcommand{\EP}{{\rm EP}}

\newcommand{\Mot}{{\rm Mot}}

\renewcommand{\GL}{{\rm GL}}
\newcommand{\Lie}{{\rm Lie}\,}
\newcommand{\diag}{{\rm diag}}

\newcommand{\level}{\mathrm{lv}}
\newcommand{\weight}{\mathrm{wt}}

\newcommand{\triv}{{\mathbf{1}}}

\newcommand{\cA}{{\mathcal{A}}}
\newcommand{\cAR}{{\mathcal{AR}}}
\newcommand{\cM}{{\mathcal{M}}}

\newcommand{\cT}{{\mathcal{T}}}
\newcommand{\cN}{{\mathcal{N}}}

\newcommand{\cH}{{\mathcal{H}}}
\newcommand{\cS}{{\mathcal{S}}}
\newcommand{\cF}{{\mathcal{F}}}
\newcommand{\cB}{{\mathcal{B}}}

\newcommand{\cO}{\mathcal{O}}
\newcommand{\cV}{\mathcal{V}}

\newcommand{\cha}{{\rm char}\,}

\newcommand{\ad}{{\rm ad}}

\newcommand{\Rep}{{\rm Irr}}

\newcommand{\der}{^{\mathrm{der}}}

\renewcommand{\Res}{\mathrm{Res}}

\newcommand{\supp}{{\rm supp}\,}

\newcommand{\spl}{{\rm spl}}
\newcommand{\Ram}{{\rm Ram}}

\newcommand{\bad}{{\rm bad}}

\def\hat{\widehat}

\def\lg{\langle}
\def\rg{\rangle}
\def\hra{\hookrightarrow}
\def\ra{\rightarrow}

\def\isom{\stackrel{\sim}{\ra}}
\def\ol{\overline}
\def\tilde{\widetilde}

\def\MU{\mbox{\boldmath{$\mu$}}}

\def\benu{\begin{enumerate}}
\def\eenu{\end{enumerate}}

\def\beq{\begin{equation}}
\def\eeq{\end{equation}}

\def\bit{\begin{itemize}}
\def\eit{\end{itemize}}
\usepackage{amscd}
\usepackage{mathrsfs}
\def\CC{{\mathbb C}}
\def\ZZ{{\mathbb Z}}
\newcommand{\ri}{{\mathcal O}}

\newcommand{\Gl}{\mathbf {GL}}
\newcommand{\fg}{{\mathfrak g}}

\newcommand{\Ad}{\operatorname{Ad}}

\newcommand{\sep}{\mathrm{sep}}
\newcommand{\sem}{\mathrm{ss}}

\newcommand{\oi}{{\bf \mathrm{O}}}

\newcommand{\rf}{k}

\newcommand\ldpo{{\mathcal L}_{\ri}}

\newcommand\ord{\mathrm{ord}}
\newcommand\ac{\overline{\mathrm{ac}}}

\newcommand\cC{{\mathcal C}}

\newcommand\mot{\mathrm{mot}}

\begin{document}

\title[]{Sato-Tate theorem for families and low-lying zeros of automorphic $L$-functions}

%\markleft{RIEN}
%--------Information for second author------%
\author[]{Sug Woo Shin}
\address{Department of Mathematics, MIT, Cambridge, MA 02139-4307, USA; School of Mathematics, KIAS, Seoul 130-722,
Republic of Korea}
%\curraddr{...}
\email{prmideal@gmail.com}
%\urladdr{...}
%------Information for first author-------%
\author[]{Nicolas Templier}   %le [] est pour ne pas avoir de running head
\address{Department of Mathematics, Fine Hall, Washington Road, Princeton, NJ 08544-1000.}
\email{templier@math.princeton.edu}
%\curraddr{}
%\thanks will become a 1st page footnote.

%-------- General info -------%
%\dedicatory{...}
\date{\today}
%\translator{...}
%\subjclass[2000]{11F55;11F67;11F70;11F72;11F75;14L15;20G30;22E30;22E35}
%-----------------------------%

%----------Abstract----------------------
\begin{abstract}
We consider certain families of automorphic representations over number fields arising from the principle of functoriality of Langlands. Let $G$ be a reductive group over a number field $F$ which admits discrete series representations at infinity. Let $^{L}G=\widehat G \rtimes \Gal(\bar F/F)$ be the associated $L$-group and $r:{}^L G\to \GL(d,\BmC)$ a continuous homomorphism which is irreducible and does not factor through $\Gal(\bar F/F)$. The families under consideration consist of discrete automorphic representations of $G(\BmA_F)$ of given weight and level and we let either the weight or the level grow to infinity.

We establish a quantitative Plancherel and a quantitative Sato-Tate equidistribution theorem for the Satake parameters of these families. This generalizes earlier results in the subject, notably of Sarnak [Progr. Math. 70 (1987), 321--331.] and Serre [J. Amer. Math. Soc. 10 (1997), no. 1, 75--102.].

As an application we study the distribution of the low-lying zeros of the associated family of $L$-functions $L(s,\pi,r)$, assuming from the principle of functoriality that these $L$-functions are automorphic. We find that the distribution of the $1$-level densities coincides with the distribution of the $1$-level densities of eigenvalues of one of the Unitary, Symplectic and Orthogonal ensembles, in accordance with the Katz-Sarnak heuristics.

We provide a criterion based on the Frobenius--Schur indicator to determine this Symmetry type. If $r$ is not isomorphic to its dual $r^\vee$ then the Symmetry type is Unitary. Otherwise there is a bilinear form on $\BmC^d$ which realizes the isomorphism between $r$ and $r^\vee$. If the bilinear form is symmetric (resp. alternating) then $r$ is real (resp. quaternionic) and the Symmetry type is Symplectic (resp. Orthogonal).
\end{abstract}
%-----------------------------------------

\maketitle

\thispagestyle{empty}

{\small \tableofcontents}

\section{Introduction}\label{sec:intro}

The non-trivial zeros of automorphic $L$-functions are of central significance in modern number theory. Problems on individual zeros, such as the Riemann Hypothesis (GRH), are elusive. There is however a beautiful theory of the statistical distribution of zeros in families. The subject has a long and rich history. A unifying modern viewpoint is that of a comparison with a suitably chosen model of random matrices: the \key{Katz--Sarnak heuristics}. There are both theoretical and numerical evidences for this comparison. Comprehensive results in the function field case~\cite{book:KS} have suggested an analogous picture in the number field case as explained in~\cite{KS:bams}. In a large number of cases, and with high accuracy, the distribution of zeros of automorphic $L$-functions coincide with the distribution of eigenvalues of random matrices. See~\cites{DFK:vanishingunitary,Rubin:computational} for numerical investigations and conjectures and see~\cites{DM06,Guloglu:lowlying,HM07,ILS00,KST:Sp4,RR:low-lying,Rubin01} and the references therein for theoretical results.

The concept of \key{families} is central to modern investigations in number theory. We want to study in the present paper certain families of automorphic representations over number fields in a very general context. The families under consideration are obtained from the discrete spectrum by imposing constraints on the local components at archimedean and non-archimedean places and by applying Langlands global functoriality principle.

Our main result is a Sato--Tate equidistribution theorem for these families (Theorem~\ref{t:intro:error-bound}). An application of this main result we can give some evidence towards the Katz-Sarnak heuristics~\cite{KS:bams} in general and establish a criterion for the random matrix model attached to families, i.e. for the symmetry type.

\subsection{Sato-Tate theorem for families}\label{sec:intro:ST}

  The original Sato-Tate conjecture is about an elliptic curve $E$, assumed to be defined over $\Q$
  for simplicity. The number of points in $E(\F_p)$ for almost all primes $p$ (with good reduction)
  gives rise to an angle $\theta_p$ between $-\pi$ and $\pi$. The conjecture, proved in~\cite{BLGHT11}, asserts that
  if $E$ does not admit complex multiplication then $\{\theta_p\}$ are equidistributed according to
  the measure $\frac{2}{\pi} \sin^2\theta d\theta$.
  In the context of motives a generalization of the Sato-Tate conjecture was formulated by Serre~\cite{Ser94}.

  To speak of the automorphic version of the Sato-Tate conjecture, let $G$ be a connected split reductive group over $\Q$
  with trivial center
  and $\pi$ an automorphic representation of $G(\A)$. Here $G$ is assumed to be split for simplicity
  (however we stress that our results are valid
  without even assuming that $G$ is quasi-split; see \S\ref{s:Sato-Tate} below for details). The triviality of center is not serious
  as it essentially amounts to fixing central character.
  Let $T$ be a maximal split torus of $G$. Denote by $\hat{T}$ its dual torus and $\Omega$ the Weyl group.
  As $\pi=\otimes'_v \pi_v$ is unramified at almost all places $p$, the Satake isomorphism identifies
  $\pi_p$ with a point on $\hat{T}/\Omega$. The automorphic Sato-Tate conjecture should be a prediction
  about the equidistribution of $\pi_p$ on $\hat{T}/\Omega$ with respect to a natural measure
  (supported on a compact subset of $\hat{T}/\Omega$). It seems nontrivial
  to specify this measure in general. The authors do not know how to do it
  without invoking the (conjectural) global $L$-parameter for $\pi$.
  The automorphic Sato-Tate conjecture is known in the limited cases of (the restriction of scalars of)
  $\GL_1$ and $\GL_2$ (\cite{BLGHT11}, \cite{BLGG11}). In an ideal world the conjecture should be closely related to Langlands functoriality.
% follow as a consequence of the Langlands functoriality conjecture.

  In this paper we consider the Sato-Tate conjecture for a \emph{family} of automorphic representations, which
  is easier to state and prove but still very illuminating.
  Our working definition of a family $\{\mathcal{F}_k\}_{k\ge 1}$ is that each $\mathcal{F}_k$ consists of
  all automorphic representations $\pi$ of $G(\A)$ of level $N_k$ with $\pi_\infty$ cohomological of weight $\xi_k$,
  where $N_k\in\Z_{\ge1}$ and $\xi_k$ is an irreducible algebraic representation of $G$, such that either
\begin{enumerate}
  \item (level aspect) $\xi_k$ is fixed, and $N_k\rightarrow \infty$ as $k\rightarrow\infty$ or
  \item (weight aspect) $N_k$ is fixed, and $m(\xi_k)\rightarrow\infty$ as $k\rightarrow\infty$,
\end{enumerate}
   where $m(\xi_k)\in\R_{\ge 0}$ should be thought of as the minimal distance of the highest weight of $\xi_k$ to root hyperplanes. (See \S\ref{sub:st-disc} below for the precise definition.)
    Note that each $\mathcal{F}_k$ has finite cardinality and $|\mathcal{F}_k|\rightarrow \infty$
    as $k\rightarrow\infty$. (For a technical reason $\mathcal{F}_k$ is actually allowed to be a multi-set. Namely
    the same representation can appear multiple times, for instance more than its automorphic multiplicity.)
    In principle we could let $\xi_k$ and $N_k$ vary simultaneously but decided not to do so in the current paper
    in favor of transparency of arguments. For instance families of type (i) and (ii) require somewhat different ingredients of proof in establishing the Sato-Tate theorem for families, and the argument would be easier to understand if we separate them.
    It should be possible to treat the mixed case (where both $N_k$ and $\xi_k$ vary) by combining techniques in the
    two cases (i) and (ii).

    Let $\hat{T}_c$ be the maximal compact subtorus of the complex torus $\hat{T}$.
    The quotient $\hat{T}_c/\Omega$ is equipped with a measure $\ST$, to be called
	the \key{Sato-Tate measure}, coming from the Haar measure on a maximal compact subgroup of $\hat{G}$
    (of which $\hat{T}_c$ is a maximal torus).
    The following is a rough version of our result on the Sato-Tate conjecture for a family.

\begin{thm}\label{Sato-Tate}
  Suppose that $G(\R)$ has discrete series representations.
  Let $\{\mathcal{F}_k\}_{k\ge 1}$ be a family in the level aspect (resp. weight aspect) as above.
  Let $\{p_k\}$ be a strictly increasing sequence of
  primes such that  $N_k$ (resp. $\xi_k$) grows faster than any polynomial in $p_k$
	in the sense that for every $\dfrac{\log p_k}{\log N_k} \to 0$ (resp. $\dfrac{\log p_k}{\log m(\xi_k)} \ra 0$) as $k\ra\infty$.
  Assume that the members of $\mathcal{F}_k$ are unramified at $p_k$ for every $k$.
  Then the Satake parameters $\{\pi_{p_k}: \pi\in \mathcal{F}_k\}_{k\ge 1}$
  are equidistributed with respect to $\ST$.

\end{thm}

  To put things in perspective, we observe that there are three kinds of statistics about the Satake parameters
  of $\{\pi_{p_k}: \pi\in \mathcal{F}_k\}_{k\ge 1}$ depending on how the arguments vary.
\begin{enumerate}
  \item Sato-Tate: $\mathcal{F}_k$ is fixed (and a singleton) and $p_k\rightarrow \infty$.
  \item Sato-Tate for a family: $|\mathcal{F}_k|\rightarrow \infty$ and $p_k\rightarrow \infty$.
  \item Plancherel: $|\mathcal{F}_k|\rightarrow \infty$ and $p_k$ is a fixed prime.
\end{enumerate}
The Sato-Tate conjecture in its original form is about equidistribution in case (i) whereas
our Theorem \ref{Sato-Tate} is concerned with case (ii).
The last item is marked as Plancherel since
the Satake parameters are expected to be equidistributed with respect to the Plancherel measure
(again supported on $\hat{T}_c/\Omega$) in case (iii). This has been shown to be true under the assumption
that $G(\R)$ admits discrete series in \cite{Shi-Plan}.
We derive Theorem \ref{Sato-Tate} from an error estimate (depending on $k$) on the difference between
the Plancherel distribution at $p$ and the actual distribution of the Satake parameters at $p_k$
in $\mathcal{F}_k$. This estimate (see Theorem \ref{t:intro:error-bound} below) refines the main result of \cite{Shi-Plan}
and is far more difficult to prove in that several nontrivial bounds in
harmonic analysis on reductive groups need to be justified.

\subsection{Families of $L$-functions}\label{sec:intro:familyL}
An application of Theorem~\ref{Sato-Tate} is to families of $L$-functions. We are able to verify to some extent the heuristics of Katz--Sarnak~\cite{KS:bams} and determine the symmetry type, see~\S\ref{sec:intro:criterion} below. In this subsection we define the relevant families of $L$-functions and record some of their properties.

Let $r:{}^{L} G \to \GL(d,\BmC)$ be a continuous $L$-homomorphism. We assume the Langlands functoriality principle: for all $\pi \in \CmF_k$ there exists an isobaric automorphic representation $\Pi=r_* \pi$ of $\GL(d,\BmA)$ which is the functorial lift of the automorphic representation $\pi$ of $G(\BmA)$, see~\S\ref{sec:pp:isobaric} for a review of the concept of isobaric representations and \S\ref{s:langlands} for the precise statement of the hypothesis. This hypothesis is only used in Theorem~\ref{th:low-lying}, \S\ref{sec:zeros} and \S\ref{sec:pf}. By the strong multiplicity one theorem $\Pi$ is uniquely determined by all but finitely many of its local factors $\Pi_v=r_* \pi_v$.

To an automorphic representation $\Pi$ on $\GL(d,\BmA)$ we associate its principal $L$-function $L(s,\Pi)$. By definition $L(s,\pi,r)=L(s,\Pi)$. By the theory of Rankin--Selberg integrals or by the integral representations of Godement--Jacquet, $L(s,\Pi)$ has good analytic properties: analytic continuation, functional equation, growth in vertical strips. In particular we know the existence and some properties of its non-trivial zeros, such as the Weyl's law (\S\ref{sec:pp:explicit}).

We denote by $\FmF_k=r_*\CmF_k$ the set of all such $\Pi=r_* \pi$ for $\pi \in \CmF_k$.
Since the strong multiplicity one theorem implies that $\Pi$ is uniquely determined by  its $L$-function $L(s,\Pi)$. We simply refer to $\FmF=r_*\CmF$  as a \key{family of $L$-functions}. %There is no reason to distinguish between families of automorphic representations on $\GL(d,\BmA)$ and families of associated $L$-functions hence

In general there are many ways to construct interesting families of $L$-functions. In a recent manuscript~\cite{Sarn:family}, Sarnak attempts to sort out these constructions into a comprehensive framework and proposes a working definition (see also~\cite{Kowalski:family-survey}).\footnote{Sarnak and the authors gave a more refined and updated framework in \cite{SST} while our paper was under review.}  The families of $L$-functions under consideration in the present paper fit well into that framework. Indeed they are \key{harmonic families} in the sense that their construction involves inputs from local and global harmonic analysis. Other types of families include geometric families constructed as Hasse--Weil $L$-functions of arithmetic varieties and Galois families associated to families of Galois representations.

\subsection{Criterion for the symmetry type}\label{sec:intro:criterion} Katz--Sarnak~\cite{KS:bams} predict that one can associate a \key{symmetry type} to a family of $L$-functions. By definition the symmetry type is the random matrix model which is conjectured to govern the distribution of the zeros. There is a long and rich history for the introduction of this concept.

Hilbert and P\'olya suggested that there might be a spectral interpretation of the zeros of the Riemann zeta function. Nowadays strong evidence for the spectral nature of the zeros of $L$-functions comes from the function field case: zeros are eigenvalues of the Frobenius acting on cohomology. This is exemplified by the equidistribution theorem of Deligne and the results of Katz--Sarnak~\cite{book:KS} on the distribution of the low-lying eigenvalues in geometric families.

In the number field case the first major result towards a spectral interpretation is the pair correlation of high zeros of the Riemann zeta function by Montgomery. Developments then include Odlyzko's extensive numerical study and the determination of the $n$-level correlation by Hejhal and Rudnick--Sarnak~\cite{RS96}. The number field analogue of the Frobenius eigenvalue statistics of~\cite{book:KS} concerns the statistics of low-lying zeros. %There have been results in many cases and we shall review some of them in~\S\ref{sec:intro:ex} below.

More precisely~\cite{KS:bams} predicts that the low-lying zeros of families of $L$-functions are distributed according to a determinantal point process associated to a random matrix ensemble. This will be explained in more details in~\S\ref{sec:intro:rdm} and~\S\ref{sec:intro:lowlying} below. We shall distinguish between the three determinantal point processes associated to the Unitary, Symplectic and Orthogonal ensembles.\footnote{In this paper we do not distinguish in the Orthogonal ensemble between the $O$ $SO(odd)$ and $SO(even)$ Symmetries. We will return to this question in a subsequent work.} Accordingly the symmetry type associated to a family $\FmF$ is defined to be Unitary, Symplectic or Orthogonal (see~\S\ref{sec:intro:lowlying} for typical results).

%Conversely, as we shall see in detail in~\S\ref{sec:intro:lowlying}, the symmetry type is determined from the densities of the low-lying zeros. There are further interpretations of the symmetry type that we shall discuss in~\S\ref{sec:intro:consequences} below.

Before entering into the details of this theory in~\S\ref{sec:intro:rdm} below, we state here our criterion for the symmetry type of the harmonic families $r_*\CmF$ defined above. We recall in~\S\ref{sec:b:fs} the definition of the Frobenius--Schur indicator $s(r)\in \set{-1,0,1}$ associated to an irreducible representation. We shall prove that the symmetry type is determined by $s(r)$. This is summarized in the following which may be viewed as a refinement of the Katz--Sarnak heuristics.

\begin{criterion}\label{criterion} Let $r:{}^LG\to \GL(d,\BmC)$ be a continuous $L$-homomorphism which is irreducible and non-trivial when restricted to $\widehat G$. Consider the family $r_*\CmF$ of automorphic $L$-functions of degree $d$ as above.

(i) If $r$ is not isomorphic to its dual $r^\vee$ then $s(r)=0$ and the symmetry type is Unitary.

(ii) Otherwise there is a bilinear form on $\BmC^d$ which realizes the isomorphism between $r$ and $r^\vee$. By Schur lemma it is unique up to scalar and is either symmetric or alternating. If it is symmetric then $r$ is real, $s(r)=1$ and the symmetry type is Symplectic. If it is alternating then $r$ is quaternionic, $s(r)=-1$ and the symmetry type is Orthogonal.
\end{criterion}

We note that the conditions that $r$ be irreducible and non-trivial when restricted to $\widehat G$ are optimal. If $r$ were trivial when restricted to $\widehat G$ then $L(s,\pi,r)$ would be constant and equal to an Artin $L$-function and the low-lying zeros would correspond to the eventual vanishing of this Artin $L$-function at the central point (which is a different problem). Also the universality exhibited in our criterion may be compared with the GUE universality of the high zeros of~\cite{RS96}.
%Our criterion says that conversely under these conditions on $r$ the Katz--Sarnak heuristics hold to some extent.

If $r$ were reducible then the $L$-functions would factor as a product $L(s,\pi,r_1)L(s,\pi,r_2)$. Suppose that both $r_1$ and $r_2$ are irreducible and non-trivial when restricted to $\widehat G$. If $r_1=r_2$ then clearly the distribution of zeros will be as before but with multiplicity two. If $r_1\not \simeq r_2$ then we expect that the zeros will follow the distribution of the \emph{independent} superposition of the two random matrix ensembles attached to $r_1$ and $r_2$. In other words the zeros of $L(s,\pi,r_1)$ are uncorrelated to the zeros of $L(s,\pi,r_2)$, and one could verify this using the methods of this paper to some extent. In particular we expect no repulsion between the respective sequences of zeros.
% from independent families and their zeros would be distributed according to the superposition of independent point processes.

It would be interesting to study families of automorphic representations over a function field $k=\F_q(X)$ of a curve $X$. To our knowledge the Katz-Sarnak heuristics for such families are not treated in the literature, except in the case of $G=GL(1)$ where harmonic families coincide with the geometric families treated by Katz--Sarnak (e.g. Dirichlet $L$-series with quadratic character are the geometric families of hyperelliptic curves in~\cite{book:KS}*{\S10}). Over function fields our criterion has the following interpretation. We consider families of automorphic representations $\pi$ of $G(\BmA_k)$; For simplicity we suppose that each automorphic representations $\pi$ of $G(\BmA_k)$ in the family $\CmF$ is attached to an irreducible $\ell$-adic representation $\rho: \Gal(k^{sep}/k) \to {}^LG$. Then $r_*\pi$ is attached to the Galois representation $r\circ \rho$, and corresponds to a constructible $\ell$-adic sheaf $F$ of dimension $d$ on the curve $X$. The zeros of the $L$-function $L(s,\pi,r)$ are the eigenvalues of Frobenius on the first cohomology, more precisely the numerator of the $L$-function $L(s,\pi,r)$ is
\[
\det(1- q^{-s}\Fr | H^1(X,F)).
\]
If $s(r)=-1$ (resp. $s(r)=1$) then there is an alternating (resp. symmetric) pairing on the sheaf $F$. The natural pairing on $H^1(X,F)$ induced by the cup product is symmetric (resp. alternating) and invariant by the action of Frobenius. Thus the zeros of $L(s,\pi,r)$ are the eigenvalues of an orthogonal (resp. symplectic) matrix. This is in agreement with the assertion (ii) of our Criterion~\ref{criterion}. We also note the related situation~\cite{Katz:ubiquity}.

Known analogies between $L$-functions and their symmetries over number fields and function fields are discussed in~\cite{KS:bams}*{\S4}. Overall we would like propose Criterion~\ref{criterion} and its analogue for geometric families as an answer to the question mark in the entry 6-A of Table~2 in~\cite{KS:bams}.

\subsection{Automorphic Plancherel density theorem with error bounds}\label{sec:intro:aut-Plan}

  We explain a more precise version of the theorem and method of proof for the Sato-Tate theorem
  for families (\S\ref{sec:intro:ST}). The key is to bound the error terms when
  we approximate the distribution of local components of automorphic representations in a family
  with the Plancherel measure.

  For simplicity of exposition let us assume that $G$ is a split reductive group over $\Q$ with trivial center
  as in \S\ref{sec:intro:ST}.
  A crucial hypothesis is that $G(\R)$ admits an $\R$-anisotropic maximal torus (in which case
  $G(\R)$ admits discrete series representations).
  Let $\mathcal{A}_{\mathrm{disc}}(G)$ denote the set of isomorphism classes of discrete automorphic
  representations of $G(\A)$. We say that
  $\pi\in \mathcal{A}_{\mathrm{disc}}(G)$ has level $N$ and weight $\xi$ if
  $\pi$ has a nonzero fixed vector under the adelic version of the full level $N$ congruence subgroup
  $K(N)\subset G(\A^\infty)$ and if $\pi_\infty\otimes \xi$ has nonzero Lie algebra cohomology.
  In this subsection we make a further simplifying hypothesis that $\xi$ has regular highest weight, in which case
  $\pi_\infty$ as above must be a discrete series representation. (In the main body of this paper, the latter assumption on $\xi$ is necessary only for the results in \S\S\ref{sub:Plan-density}-\ref{sub:general-functions-S_0}, where more general test functions are considered.)

  Define $\mathcal{F}=\mathcal{F}(N,\xi)$ to be the finite multi-set consisting of $\pi\in \mathcal{A}_{\mathrm{disc}}(G)$
  of level $N$ and weight $\xi$, where each such $\pi$ appears in $\mathcal{F}$
  with multiplicity $$a_{\mathcal{F}}(\pi):= \dim (\pi^\infty)^{K(N)}\in \Z_{\ge 0}.$$
  This quantity naturally occurs as the dimension of the $\pi$-isotypical subspace in the cohomology
  of the locally symmetric space for $G$ of level $N$ with coefficient defined by $\xi$.
  The main motivation for allowing $\pi$ to appear $a_{\mathcal{F}}(\pi)$ times is to enable us
  to compute the counting measure below with the trace formula.

  Let $p$ be a prime number. Write $G(\Q_p)^{\wedge}$ for the unitary dual of
  irreducible smooth representations of $G(\Q_p)$. The unramified (resp. unramified and tempered)
  part of $G(\Q_p)^{\wedge}$ is denoted $G(\Q_p)^{\wedge,\mathrm{ur}}$
   (resp. $G(\Q_p)^{\wedge,\mathrm{ur},\mathrm{temp}}$). There is a canonical isomorphism
\begin{equation}
\label{e:intro:spec-satake} G(\Q_p)^{\wedge,\mathrm{ur},\mathrm{temp}}\simeq \hat{T}_c/\Omega.
\end{equation}
  The unramified Hecke algebra of $G(\Q_p)$ will be denoted $\mathcal{H}^{\mathrm{ur}}(G(\Q_p))$.
  There is a map from $\mathcal{H}^{\mathrm{ur}}(G(\Q_p))$ to the space of continuous functions
  on $\hat{T}_c/\Omega$:
  $$\phi\mapsto \hat{\phi}\quad \mbox{determined by}\quad \hat{\phi}(\pi)=\mathrm{tr}\, \pi(\phi),
  ~\forall \pi\in G(\Q_p)^{\wedge,\mathrm{ur},\mathrm{temp}}.$$
  There are two natural measures supported on
  $ \hat{T}_c/\Omega$.
  The Plancherel measure $\hat{\mu}^{\mathrm{pl,\mathrm{ur}}}_p$, dependent on $p$, is
  defined on $G(\Q_p)^{\wedge,\mathrm{ur}}$ and naturally arises
   in local harmonic analysis. The Sato-Tate measure
  $\hat{\mu}^{\mathrm{ST}}$ on $ \hat{T}_c/\Omega$ is independent of $p$ and may be
  extended to $G(\Q_p)^{\wedge,\mathrm{ur}} $ by zero.
   Both $\hat{\mu}^{\mathrm{pl,\mathrm{ur}}}_p$
  and $\hat{\mu}^{\mathrm{ST}}$ assign volume 1 to  $ \hat{T}_c/\Omega$.
  There is yet another measure $\hat{\mu}^{\mathrm{count}}_{\mathcal{F},p}$
  on $G(\Q_p)^{\wedge,\mathrm{ur}}$, which is the averaged counting measure for the $p$-components of members of $\mathcal{F}$.
  Namely
  \begin{equation}
  \label{e:count-meas}\hat{\mu}^{\mathrm{count}}_{\mathcal{F},p}:=
  \frac{1}{|\mathcal{F}|} \sum_{\pi\in \mathcal{F}}  \delta_{\pi_p}
\end{equation}
  where $\delta_{\pi_p}$ denotes the Dirac delta measure supported at $\pi_p$. (Each
	$\pi\in \mathcal{A}_{\mathrm{disc}}(G)$ contributes $a_{\mathcal{F}}(\pi)$ times to the above sum.)
  Our primary goal is to bound the difference between $\hat{\mu}^{\mathrm{pl},\mathrm{ur}}_p$
  and $\hat{\mu}^{\mathrm{count}}_{\mathcal{F},p}$. (Note that our definition
  of $\hat{\mu}^{\mathrm{count}}_{\mathcal{F},p}$ in the main body will be a little different from
  \eqref{e:count-meas} but asymptotically the same, see Remark \ref{r:|F|}.)

  In order to quantify error bounds, we introduce a filtration
  $\{\mathcal{H}^{\mathrm{ur}}(G(\Q_p))^{\le \kappa}\}_{\kappa\in\Z_{\ge0}}$
  on $\mathcal{H}^{\mathrm{ur}}(G(\Q_p))$ as a complex vector space. The filtration is increasing, exhaustive
  and depends on a non-canonical choice.
  Roughly speaking, $\mathcal{H}^{\mathrm{ur}}(G(\Q_p))^{\le \kappa}$ is like the span of
  all monomials of degree $\le \kappa$ when $\mathcal{H}^{\mathrm{ur}}(G(\Q_p))$ is identified with
  (a subalgebra of) a polynomial algebra.
  For each $\xi$, it is possible to assign a positive integer $m(\xi)$ in terms of the highest weight of $\xi$.
  When we say that weight is going to infinity, it means that $m(\xi)$ grows to $\infty$ in the usual sense.

  The main result on error bounds alluded to above is the following.
  (See Theorems \ref{t:level-varies} and \ref{t:weight-varies} for the precise statements and Remarks \ref{r:explicit-const-level} and \ref{r:explicit-const-wt} for an explicit choice of constants.) A uniform bound on orbital integrals, cf. \eqref{e:intro-unif-bound} below, enters the proof of (ii) (but not (i)).% and requires the assumption $p\gg 1$ in (ii) to ensure that $G(\BmQ_p)$.

\begin{thm}\label{t:intro:error-bound}
  Let $\mathcal{F}=\mathcal{F}(N,\xi)$ be as above.
  Consider a prime $p$, an integer $\kappa\ge 1$, and a function $\phi_p\in \mathcal{H}^{\mathrm{ur}}(G(\Q_p))^{\le \kappa}$
  such that $|\phi_p|\le 1$ on $G(\Q_p)$.
\begin{enumerate}
  \item (level aspect) Suppose that $\xi$ remains fixed. There exist constants $A_{\level},B_{\level},C_{\level}>0$ depending only on $G$ such that for any $p$, $\kappa$, $\phi_p$  as above and for any $N$ coprime to $p$,
  $$ \hat{\mu}^{\mathrm{count}}_{\mathcal{F},p}(\hat{\phi}_p)-\hat{\mu}^{\mathrm{pl,\mathrm{ur}}}_p(\hat{\phi}_p)
  = O(p^{A_{\level}+B_{\level}\kappa}N^{-C_{\level}}).$$
    \item (weight aspect) Fix a level $N$. There exist constants
    $A_{\weight},B_{\weight},C_{\weight}>0$ and a lower bound $c>0$ depending only on $G$ such that for any $p\ge c$, $\kappa$, $\phi_p$ as above with $(p,N)=1$ and for any $\xi$,
  $$ \hat{\mu}^{\mathrm{count}}_{\mathcal{F},p}(\hat{\phi}_p)-\hat{\mu}^{\mathrm{pl,\mathrm{ur}}}_p(\hat{\phi}_p)
  = O(p^{A_{\weight}+B_{\weight}\kappa}m(\xi)^{-C_{\weight}}).$$
\end{enumerate}

\end{thm}

  Let $\{ \mathcal{F}_k=\mathcal{F}(N_k,\xi_k)\}_{k\ge1}$ be
  either kind of family in \S\ref{sec:intro:ST}, namely either $N_k\rightarrow \infty$
  and $\xi_k$ is fixed or $N_k$ is fixed and $\xi_k\rightarrow \infty$.
    When applied to $\{ \mathcal{F}_k\}_{k\ge1}$, Theorem \ref{t:intro:error-bound} leads to
  the equidistribution results in the following corollary (cf.
  cases (ii) and (iii) in the paragraph below Theorem \ref{Sato-Tate}).
  Indeed, (i) of the corollary is immediate. Part (ii) is easily derived from the fact
  that $\hat{\mu}^{\mathrm{pl,\mathrm{ur}}}_p$ weakly converges to $\hat{\mu}^{\mathrm{ST}}$
  as $p\rightarrow \infty$. Although the unramified Hecke algebra at $p$ gives rise to only
  regular functions on the complex variety $\hat{T}_c/\Omega$, it is not difficult to extend
  the results to continuous functions on $\hat{T}_c/\Omega$. (See \S\S\ref{sub:Plan-density}-\ref{sub:general-functions-S_0} for details.)

\begin{cor}\label{c:intro:Plan-ST} Keep the notation of Theorem \ref{t:intro:error-bound}.
Let $\hat{\phi}$ be a continuous function on $\hat{T}_c/\Omega$.
  In view of \eqref{e:intro:spec-satake} $\hat{\phi}$ can be extended by zero to a function
  $\hat{\phi}_p$ on $G(\Q_p)^{\wedge,\mathrm{ur}}$ for each prime $p$.
\begin{enumerate}
  \item (Automorphic Plancherel density theorem~\cite{Shi-Plan}) $$\lim_{k\rightarrow \infty} \hat{\mu}^{\mathrm{count}}_{\mathcal{F}_k,p}(\hat{\phi}_p)
  = \hat{\mu}^{\mathrm{pl,\mathrm{ur}}}_p(\hat{\phi}_p).$$
  \item (Sato-Tate theorem for families) Let $\{p_k\}_{k\ge1}$ be a sequence of primes tending to $\infty$.
		Suppose that $\dfrac{\log p_k}{\log N_k}  \rightarrow 0$ (resp. $\dfrac{\log p_k}{\log m(\xi_k)} \rightarrow 0$) as $k\rightarrow \infty$ if $\xi_k$ (resp. $N_k$) remains fixed as $k$ varies. Then
  $$\lim_{k\rightarrow \infty} \hat{\mu}^{\mathrm{count}}_{\mathcal{F}_k,p_k}(\hat{\phi}_{p_k})
  = \hat{\mu}^{\mathrm{ST}}(\hat{\phi}).$$
\end{enumerate}
\end{cor}

  Theorem \ref{t:intro:error-bound} and Corollary \ref{c:intro:Plan-ST} remain valid if any finite number of primes are simultaneously considered in place of $p$ or $p_k$. Moreover (i) of the corollary holds true for more general (and possibly ramified) test functions $\hat \phi_p$ on $G(\Q_p)^{\wedge}$ thanks to Sauvageot's density theorem. It would be interesting to quantify the error bounds in this generality. Finally the above results should be compared with the proposition~4 in~\cite{Serre:pl} and the theorem~1 in \cite{Nag06} for modular forms on $\GL(2)$. We also note~\cite{Sarnak87} for Maass forms (which are not considered in the the present paper).

\subsection{Random matrices}\label{sec:intro:rdm} We provide a brief account of the theory of random matrices. The reader will find more details in \S\ref{sec:zeros:matrix} and extensive treatments in~\cites{book:Mehta,book:KS}.

The Gaussian Unitary Ensemble and Gaussian Orthogonal Ensemble were introduced by Wigner in the study of resonances of heavy nucleus. %The GUE (resp. GOE) ensemble consists of the vector space of $N\times N$ hermitian (resp. orthogonal) matrices $A$ endowed with the Gaussian probability proportional to $e^{-\Mtr(A^2)}dA$.
The Gaussian Symplectic Ensemble was introduced later by Dyson together with his Circular Ensembles. % In all cases one studies the limiting distribution of the eigenvalues of $A$ as $N\to \infty$.
In this paper we are concerned with the ensembles attached to compact Lie groups which are introduced by Katz-Sarnak and occur in the statistics of $L$-functions. (See~\cite{Duen:RMT} for the precise classification of these ensembles attached to different Riemannian symmetric spaces.)

One considers eigenvalues of matrices in compact groups $\CmG(N)$ of large dimension endowed with the Haar probability measure. We have three symmetry types $\CmG=SO(even)$ (resp. $\CmG=U$, $\CmG=USp$); the notation says that for all $N\ge 1$, the groups are $\CmG(N)=SO(2N)$ (resp. $\CmG(N)=U(N)$ and $\CmG(N)=USp(2N)$).

For all matrices $A \in \CmG(N)$ we have an associated sequence of \key{normalized angles}
%(see~\S\ref{sec:zeros:matrix})
\begin{equation}\label{intro:vartheta}
0\le \vartheta_1 \le \vartheta_2 \le \cdots \le \vartheta_N \le N.
\end{equation}
For example in the case $\CmG=U$, the eigenvalues of $A\in U(N)$ are given by $e(\tfrac{\vartheta_j}{N})=e^{2i\pi \vartheta_j/N}$ for $1\le j\le N$.
The normalization is such that the mean spacing of the $(\vartheta_i)$ in~\eqref{intro:vartheta} is about one.

For each $N\ge 1$ these angles $(\vartheta_i)_{1\le i\le N}$ are correlated random variables (a point process). By the Weyl integration formula their joint density is  proportional to
\begin{equation}\label{ginibre}
  \prod_{1\le i<j\le N} \abs{\sin\left( \frac{\pi(\vartheta_i-\vartheta_j)}{N}\right)}^\beta d\vartheta_1 \cdots d\vartheta_N.
\end{equation}
The parameter $\beta$ is a measure of the repulsion between nearby eigenvalues. We have that $\beta=1$ (resp. $\beta=2$, $\beta=4$) for $\CmG=SO(even)$ (resp. $\CmG=U$, $\CmG=USp$).

A fundamental result of Gaudin--Mehta and Dyson, which has been extended to the above ensembles by Katz--Sarnak, is that when $N\to \infty$ the distribution of the angles $(\vartheta_i)_{1\le i\le N}$ converges to a determinantal point process.\footnote{
For other values of $\beta\neq 1,2,4$, the limiting statistics attached to~\eqref{ginibre} has been determined recently by Valk\'o--Vir\'ag in terms of the Brownian carousel.}
The kernel of the limiting point process when $\CmG=U$ is given by the Dyson sine kernel
\begin{equation*}
K(x,y)=\frac{\sin\pi(x-y)}{\pi(x-y)}, \quad x,y\in \BmR_+
\end{equation*}
The kernel for $\CmG=SO(even)$ is $K_+(x,y)=K(x,y)+K(-x,y)$ and the kernel for $\CmG=USp$ is $K_-(x,y)=K(x,y)-K(-x,y)$.

In particular this means that there is a \key{limiting $1$-level density} $W(\CmG)$ for the angles $(\vartheta_i)_{1\le i \le N}$ as $N\to \infty$ (see also Proposition~\ref{prop:KS}). It is given by the following formulas:
\begin{equation}\label{intro:WG}
\begin{aligned}
W(SO(even))(x)&=K_+(x,x)=1+\frac{\sin 2\pi x}{2\pi x},\\
W(U)(x)&=K(x,x)= 1,\\
W(USp)(x)&=K_-(x,x)=1-\frac{\sin 2\pi x}{2\pi x}.
\end{aligned}
\end{equation}

\subsection{Low-lying zeros}\label{sec:intro:lowlying} We can now state more precisely our results on families of $L$-functions. Let $\FmF=r_*\CmF$ be a family of $L$-functions as defined above in~\S\ref{sec:intro:ST}-\ref{sec:intro:familyL}.

For all $\Pi\in \FmF_k$ we denote by $\rho_j(\Pi)$, the zeros of the completed $L$-function $\Lambda(s,\Pi)$, where $j\in \BmZ$. We write $\rho_j(\Pi)=\Mdemi+i\gamma_j(\Pi)$ and therefore $-\frac12 < \mathrm{Re} \gamma_j(\Pi) < \frac12$ for all $j$. By the functional equation $\Lambda(\frac12 + i\gamma,\Pi)=0$ if and only if $\Lambda(\frac12 + i\overline{\gamma},\Pi)=0$. We do not assume the GRH that would further imply $\gamma_j(Pi)\in \R$ for all $j$.

In the case that $\Pi$ is self-dual the zeros occur in complex pairs, namely $L(\frac12+i\gamma,\Pi)=0$ if and only if $\Lambda(\frac12 -i\gamma,\Pi)=0$.

Following Iwaniec--Sarnak we associate an \key{analytic conductor} $C(\FmF_k)\ge 1$ to the family, see~\S\ref{sec:pp:Lfn} and~\S\ref{sec:zeros:cond}.  We assume from now that the family is in the weight aspect, so that for each $k\ge 1$, all  of $\Pi\in \FmF_k$ share the same archimedean factor $\Pi_\infty$ and we can set $C(\FmF_k):=C(\Pi_\infty)$.
(For families in the level aspect we obtain similar results, see Section~\ref{sec:zeros}).
Note that $C(\FmF_k)\to \infty$  and furthermore we shall make the assumption that $\log C(\FmF_k) \asymp \log m(\xi_k)$ as $k\to \infty$.

For a given $\Pi\in \FmF_k$ the number of zeros $\gamma_j(\Pi)$ of bounded height is $\asymp \log C(\FmF_k)$. The \key{low-lying zeros} of $\Lambda(s,\Pi)$ are those within distance $O(1/\log(C(\FmF_k))$ to the central point ; heuristically there are a bounded number of low-lying zeros for a given $\Pi\in \FmF_k$, although this can only be proved on average over the family. For a technical reason related to the fact that the explicit formula counts both the zeros and poles of $\Lambda(s,\Pi)$ (\S\ref{sec:pp:explicit}), we make an hypothesis on the occurrence of poles of $\Lambda(s,\Pi)$ for $\Pi\in \FmF_k$, see~Hypothesis~\ref{hyp:poles}.

%\begin{equation}
%\sum_j \Phi(\frac{\gamma_j(\Pi)}{2\pi} \log C(\FmF_k))
%\end{equation}
%where $\Phi\in \cS(\R)$ is a Schwartz function.

The statistics of low-lying zeros of the family are studied via the functional
\begin{equation} \label{intro:1-level}
D(\FmF_k;\Phi)=\frac{1}{\abs{\FmF_k}}\sum_{\Pi\in\FmF_k} \sum_j \Phi\left(\frac{\gamma_j(\Pi)}{2\pi} \log C(\FmF_k)\right),
\end{equation}
where $\Phi$ is a Paley--Wiener function. This is the $1$-level density for the family $\FmF_k$.
Choosing $\Phi$ as a smooth approximation of the characteristic function of an interval $[a,b]$, the sum~\eqref{intro:1-level} should be thought as
a weighted count of all the zeros  of the family lying in $[a,b]$:
\begin{equation*}
  \frac{2 a \pi}{\log C(\FmF_k)} \le \gamma_j(\Pi) \le \frac{2 b \pi}{\log C(\FmF_k)},
  \quad
(j\in \Z, \Pi \in \cF_k).
\end{equation*}
We want to compare the asymptotic as $k\to \infty$ with the limiting $1$-level density of normalized angles~\eqref{intro:vartheta} of the random matrix ensembles described in~\S\ref{sec:intro:rdm} above.

\begin{thm}\label{th:low-lying} Let $r:{}^LG:\to \GL(d,\BmC)$ be a continuous $L$-homomorphism which is irreducible and non-trivial when restricted to $\widehat G$.
  There exists $\delta>0$ depending on $\FmF$ such that the following holds. Let $\FmF=r_*\CmF$ be a family of $L$-functions in the weight aspect as in~\S\ref{sec:intro:familyL}, assuming the functoriality conjecture as in Hypothesis~\ref{hypo:functorial-lift}. Assume Hypothesis~\ref{hyp:poles} concerning the poles of $\Lambda(s,\Pi)$ for $\Pi\in\FmF_k$. Then for all Paley-Wiener functions $\Phi$ whose Fourier transform $\widehat \Phi$ has support in $(-\delta,\delta)$:

(i) there is a limiting $1$-level density for the low-lying zeros, namely there is a density $W(x)$ such that
\begin{equation*}%\label{eq:th:low-lying}
\lim_{k\to \infty} D(\FmF_k;\Phi) = \int_{-\infty}^{\infty}
\Phi(x)W(x)dx;
\end{equation*}

(ii) the density $W(x)$ is determined by the Frobenius--Schur indicator of the irreducible representation $r$. Precisely,
\begin{equation}\label{W:th:low-lying}
W=\begin{cases}
W(SO(even)),& \text{if $s(r)=-1$,}\\
W(U),& \text{if $s(r)=0$,}\\
W(USp),& \text{if $s(r)=1$.}
\end{cases}
\end{equation}
\end{thm}

The constant $\delta>0$ depends on the family $\FmF$, in other words it depends on the group $G$, the $L$-morphism $r:{}^LG \to \GL(d,\BmC)$ and the limit of the ratio $\dfrac{\log C(\FmF_k)}{\log m(\xi_k)}$. Its numerical value is directly related to the numerical values of the exponents in the error term occurring in Theorem~\ref{t:intro:error-bound}. Although we do not attempt to do so in the present paper, it is interesting to produce a value of $\delta$ that is as large as possible, see~\cite{ILS00} for the case of $\GL(2)$. This would require sharp bounds for orbital integrals as can be seen from the outline below. A specific problem would be to optimize the exponents $a,b,e$ in~\eqref{e:intro-unif-bound}. (In fact we can achieve $e=1$, see \S\ref{sec:intro:outline} below).

Our proofs of Theorems~\ref{t:intro:error-bound} and~\ref{th:low-lying} are effective in the sense that each constant and each exponent in the statements of the estimates could, in principle, be made explicit. Finally we note that~\cite{Cogdell2004} studied following E.~Royer the related question of $L$-values at the edge in the case of symmetric powers of $\GL(2)$ and has noted the relevance of the indicator $s(r)$.

\subsection{Outline of proofs}\label{sec:intro:outline}

A wide range of methods are used in the proof. Among them are the Arthur-Selberg trace formula, the analytic theory of
$L$-functions, representation theory and harmonic analysis on $p$-adic and real groups, and random matrix theory.

The first main result of our paper is Theorem \ref{t:intro:error-bound}, proved in Section~\ref{s:aut-Plan-theorem}.
We already pointed out after stating the theorem that the Sato-Tate equidistribution
for families (Corollary \ref{c:intro:Plan-ST})
is derived from Theorem \ref{t:intro:error-bound}
and the fact that the Plancherel measure tends to the Sato-Tate measure as the residue characteristic
is pushed to $\infty$.

Let us outline the proof of the theorem.
In fact we restrict our attention to part (ii), as (i) is handled by a similar method and only simpler to deal with.
Thus we consider $\cF$ with fixed level and weight $\xi$, where $\xi$ is regarded
as a variable.
Our starting point is to realize that for $\phi_p\in C^\infty_c(G(\Q_p))$,
we may interpret $\hat{\mu}^{\mathrm{count}}_{\cF,p}(\hat{\phi}_p)$ in terms of the spectral side of the trace formula for $G$
evaluated against the function $\phi_p\phi^{\infty,p}\phi_\infty\in C^\infty_c(G(\A))$ for a suitable
$\phi^{\infty,p}$ (depending on $\cF$ and $p$; note that $p$ is allowed to vary) and an Euler-Poincar\'{e} function $\phi_\infty$ at $\infty$ (depending on $\xi$).
Applying the trace formula, which has a simple form thanks to $\phi_\infty$, we get a
geometric expansion for $\hat{\mu}^{\mathrm{count}}_{\cF,p}(\hat{\phi}_p)$:
\beq\label{e:sketch-of-pf-1stThm} \hat{\mu}^{\mathrm{count}}_{\cF,p}(\hat{\phi}_{p})
=\sum_{M\subset G\atop \textrm{cusp. Levi}}
\sum_{\gamma\in M(\Q)/\sim\atop \R\textrm{-ell}}
a'_{M,\gamma}\cdot O^{M(\A^\infty)}_\gamma(\phi^\infty_{M})\frac{\Phi^G_M(\gamma,\xi)}{\dim \xi}.
\eeq
where $a'_{M,\gamma}\in \C$ is a coefficient encoding a certain volume associated with
the connected centralizer of $\gamma$ in $M$
and $\phi^\infty_{M}$ is the constant term of $\phi^\infty$ along (a parabolic subgroup associated with) $M$.
The Plancherel formula identifies the term for $M=G$ and $\gamma=1$ with
$\pl_p(\hat{\phi}_p)$, which basically dominates the right hand side.

The proof of Theorem \ref{t:intro:error-bound}.(ii) boils down to bounding
the other terms on the right hand side of \eqref{e:sketch-of-pf-1stThm}.
Here is a rough explanation of how to analyze each component there.
The first summation is finite and controlled by $G$, so we may as well look at the formula
for each $M$. There are finitely many conjugacy classes in the second summation for which the summand is nonzero.
The number of such conjugacy classes may be bounded by a power of $p$ where the exponent of $p$ depends only on $\kappa$
(measuring the ``complexity'' of $\phi_p$).
The term $a'_{M,\gamma}$, when unraveled, involves a special value of some Artin $L$-function.
We establish a bound on the special value which suffices to deal with $a'_{M,\gamma}$.
The last term $\frac{\Phi^G_M(\gamma,\xi)}{\dim \xi}$ can be estimated by using a character formula
for the stable discrete series character $\Phi^G_M(\gamma,\xi)$ as well as the dimension formula
for $\xi$.
It remains to take care of $O^{M(\A^\infty)}_\gamma(\phi^\infty_{M})$.
This turns out to be the most difficult task since
Theorem \ref{t:intro:error-bound}
asks for a bound that is \emph{uniform as the residue characteristic varies}.

We are led to prove that there exist $a,b,e> 0$, depending only on $G$, such that for almost all $q$,
\beq\label{e:intro-unif-bound}|O^{M(\Q_q)}_\gamma(\phi_q)|\le q^{a+b \kappa} D^M(\gamma)^{-e/2}\eeq
for all semisimple $\gamma$ and all $\phi_q$ with $\phi_q\in \cH^{\ur}(M(\Q_q))^{\le \kappa}$ and $|\phi_q|\le 1$, where $D^M(\cdot)$ denotes the Weyl discriminant.
The justification of \eqref{e:intro-unif-bound} takes up the whole of Section~\ref{s:app:unif-bound}.
 The problem already appears to be deep for the unit elements of unramified Hecke algebras
 in which case one can take $\kappa=0$. (By a different argument based on arithmetic motivic integration, Cluckers, Gordon, and Halupczok establish a stronger uniform bound with $e=1$. This work is presented in Appendix~\ref{s:app:B}.)
 At the (fixed) finite set of primes where wild ramification occurs, the problem comes down to bounding
the orbital integral $|O^{M(\Q_q)}_\gamma(\phi_q)|$ for fixed $q$ and $\phi_q$. It is deduced
from the Shalika germ theory that the orbital integral is bounded by a constant, if normalized
by the Weyl discriminant $D^M(\gamma)^{1/2}$, as $\gamma$ runs over the set of semisimple elements.
See Appendix \ref{s:app:Kottwitz} by Kottwitz for details.

We continue with Theorem~\ref{th:low-lying}. The proof relies heavily on Theorem~\ref{t:intro:error-bound}. The connection between the two statements might not be immediately apparent.

A standard procedure based on the explicit formula (see~Section~\ref{sec:pp}) expresses the sum~\eqref{intro:1-level} over zeros of $L$-function as a sum over prime numbers of Satake parameters. The details are to be found in Section~\ref{sec:pf}, and the result is that $D(\FmF_k,\Phi)$ can be approximated by
\begin{equation} \label{intro:sumprime}
  \sum_{\text{prime } p}
\hat{\mu}^{\mathrm{count}}_{\mathcal{F}_k,p}(\hat{\phi}_p)
\Phi(\frac{ \log p}{\pi \log C(\FmF_k)}).
\end{equation}
Here $\phi_p\in \mathcal{H}^{\mathrm{ur}}(G(\Q_p))^{\le \kappa}$ is suitably chosen such that $\hat \phi_p(\pi_p)$ is a sum of powers of the Satake parameters of $r_*\pi$ (see Sections~\ref{s:Satake-trans} and~\ref{s:Plancherel}). The integer $\kappa$ may be large but it depends only on $r$ so should be considered as fixed. Also the sum is over unramified primes. We have $\log C(\CmF_k) \asymp \log m(\xi_k)$ (see Sections~\ref{s:langlands} and~\ref{sec:zeros}). We deduce that the sum is supported on those primes $p\le m(\xi_k)^{A\delta}$ where $A$ is a suitable constant and $\delta$ is as in Theorem~\ref{th:low-lying}.

We apply Theorem~\ref{t:intro:error-bound} which has two components: the main term and the error term. We begin with the main term which amounts to substituting
$\hat{\mu}^{\mathrm{pl,\mathrm{ur}}}_p(\hat{\phi}_p)$ for $\hat{\mu}^{\mathrm{count}}_{\mathcal{F}_k,p}(\hat{\phi}_p)$ in~\eqref{intro:sumprime}.
Unlike $\hat{\mu}^{\mathrm{count}}_{\mathcal{F}_k,p}$, this term is purely local, thus simpler. Indeed $\hat{\mu}^{\mathrm{pl,\mathrm{ur}}}_p(\hat{\phi}_p)$ can be computed explicitly for low rank groups, see e.g.~\cite{Gro98} for all the relevant properties of the Plancherel measure. However we want to establish Theorem~\ref{th:low-lying} in general so we proceed differently.

Using certain uniform estimates by Kato~\cite{Kat82}, we can approximate $\hat{\mu}^{\mathrm{pl,\mathrm{ur}}}_p(\hat{\phi}_p)$ by a much simpler expression that depends directly on the restriction of $r$ to $\widehat G \rtimes W_{\BmQ_p}$. Then a pleasant computation using the Cebotarev equidistribution theorem, Weyl's unitary trick and the properties of the Frobenius--Schur indicator shows that the sum over primes of this main term contribute $\frac{-s(r)}{2}\Phi(0)$ to~\eqref{intro:sumprime}. This exactly reflects the identities~\eqref{W:th:low-lying} in the statement (ii) of Theorem~\ref{th:low-lying}.

We continue with the error term $O(p^{A_\weight+B_\weight\kappa}m(\xi_k)^{-C_\weight})$ which we need to insert in~\eqref{intro:sumprime}. We can see the reasons why the proof of Theorem~\ref{th:low-lying} requires the full force of Theorem~\ref{t:intro:error-bound} and its error term: the polynomial control by $p^{A_\weight+B_\weight\kappa}$ implies that the sum over primes is at most $m(\xi_k)^{D\delta}$ for some $D>0$; the power saving $m(\xi_k)^{-C_\weight}$ is exactly what is needed to beat $m(\xi_k)^{D\delta}$ when $\delta$ is chosen small enough.

\subsection{Notation}\label{sub:notation}  We distinguish the letter $\CmF$ for families of automorphic representations on general reductive groups and $\FmF=r_*\CmF$ for the families of automorphic representations on $\GL(d)$.

%The notation $f \ll g$ or $f=O(g)$ means that $f\le cg$ for some constant $c$; we write   and $o()$ have their usual meaning.

Let us describe in words the significance of various constants occurring in the main statements. We often use the convention to write multiplicative constants in lowercase letters and constants in the exponents in uppercase or greek letters.
\begin{itemize}
  \item The exponent $\beta$ from Lemma~\ref{l:bound-degree-Satake} is such that for all $\phi\in \cH^{\ur}(\GL_d)$ of degree at most $\kappa$, the pullback $r^*\phi$ is of degree at most $\le \beta \kappa$.
  \item The exponent $b_G$ from Lemma~\ref{l:bounding-phi-on-S_0} controls a bound for the constant term $\abs{\phi_M(1)}$ for all Levi subgroups $M$ and $\phi \in  \cH^{\ur}(G)$ of degree at most $\kappa$.
 % \item The constant $B_5$ in Lemma~\ref{l:bounding1-alpha(gamma)}.
  \item The exponent $0<\theta<\frac12$ is a nontrivial bound towards Ramanujan-Petersson for $\GL(d,\BmA)$.
  \item The integer $i\ge 1$ in Corollary~\ref{c:vanishing-ram-subgroup} is an upper-bound for the ramification of the Galois group $\Gal(E/F)$.
  %\item The constant $C$ in Lemma~\ref{l:dim-and-trace} and~\ref{l:bound-for-st-ds-char}. Decay of Euler-Poincar\'e functions.
%  \item The integer $s \ge 1$ is a bound on the degree of certain extensions of $F$. *to be revised*
  \item The constants $B_\Xi$ and $c_\Xi$ in Lemma~\ref{l:forcing-unipotent} and $A_3,B_3$ in Proposition~\ref{p:bound-number-of-conj} control the number of rational conjugacy classes intersecting a small open compact subgroup.
%  \item The integer $\kappa \ge 1$, and the finite set $S_1$ of unramified nonarchimedean places. Test function $\phi_{S_1}$ in the unramified Hecke algebra $ \cH^{\ur}(G(F_{S_1})^{\le \kappa}$.
  \item The integer $u_G\ge 1$ in Lemma~\ref{l:bounding-conj-in-st-conj} is a uniform upper bound for the number of $G(F_v)$-conjugacy classes in a stable conjugacy class.
  \item The integer $n_G\ge 0$ is the minimum value for the dimension of the unipotent radical of a proper parabolic subgroup of $G$ over $\ol{F}$.
  \item The constant $c>0$ is a bound for the number of connected components $\pi_0(Z(\hat{I}_\gamma)^{\Gamma})$ in Corollary~\ref{c:bounding-pi0-general}.
  \item The exponents $A_\level,B_\level,C_\level>0$ in Theorem~\ref{t:level-varies} (see also Theorem~\ref{t:intro:error-bound}) and $A_\weight,B_\weight,C_\weight>0$ in Theorem~\ref{t:weight-varies}.
  \item For families in the weight aspect, the constant $\eta>0$ which may be chosen arbitrary small enters in the condition~\eqref{dimxik} that the dominant weights attached to $\xi_k$ stay away from the walls.
  \item The exponent $C_{\text{pole}}>0$ in the Hypothesis~\ref{hyp:poles} concerning the density of poles of $L$-functions.
  \item The exponents $0<C_1 < C_2$ control the analytic conductor $C(\FmF_k)$ of the families in the weight aspect (Inequality~\eqref{xik-CFk}) and $0<C_3 < C_4$ in the level aspect (Hypothesis~\ref{hyp:cond}).
  \item The constant $\delta>0$ in Theorem~\ref{th:onelevel} controls the support of the Fourier transform $\widehat \Phi$ of the test function $\Phi$.
  \item The constant $c(f)>0$ depending on the test function $f$ is a uniform upper bound for normalized orbital integrals $D^G(\gamma)^{\frac12} O_\gamma(f)$ (Appendix~\ref{s:app:Kottwitz}).
\end{itemize}
%We adopt the following notation for \Lquote{constants}. We often write multiplicative constants in lowercase and exponents in uppercase or greek letter.

Several constants are attached directly to the group $G$ such as the dimension $d_G=\dim G$, the rank $r_G=\rank G$, the order of the Weyl group $w_G=\abs{\Omega}$, the degree $s_G$ of the smallest extension of $F$ over which $G$ becomes split. Also in Lemma~\ref{l:bounding-phi-on-S_0} the constant $b_G$ gives a bound for the constant terms along Levi subgroups. The constants $a_G,b_G, e_G$ in Theorem~\ref{t:appendeix2} gives a uniform bound for certain orbital integrals.
In general we have made effort to keep light and consistent notation throughout the text.

In Section~\ref{s:bg} we will choose a finite extension $E/F$ which splits maximal tori of subgroups of $G$. The degree $s_G^{\spl}=[E:F]$ will be controlled by $s_G^{\spl}\le s_Gw_G$ (see Lemma~\ref{l:torus-splitting}), while the ramification of $E/F$ will vary.
In Section~\ref{s:Sato-Tate} we consider the finite extension $F_1/F$ such that $\Gal(\ol{F}/F)$ acts on $\hat{G}$ through the faithful action of $\Gal(F_1/F)$.%
%\begin{equation*}
%\Gal(\ol{F}/F)\twoheadrightarrow \Gal(F_1/F)\hra \Out(\hat{G}).
%\end{equation*}
For example if $G$ is a non-split inner form of a split group then $F_1=F$. In Section~\ref{sec:pf} we consider a finite extension $F_2/F_1$ such that the representation $r$ factors through $\hat{G}\rtimes \Gal(F_2/F)$. For a general $G$, there might not be any direct relationship between the extensions $E/F$ and $F_2/F_1/F$.

\subsection{Structure of the paper}

For a quick tour of our main results and the structure of our arguments, one could start reading from Section \ref{s:aut-Plan-theorem} after familiarizing oneself with basic notation, referring to earlier sections for further notation and basic facts as needed.

The first Sections~\ref{s:Satake-trans} and~\ref{s:Plancherel} are concerned with harmonic analysis on reductive groups over local fields, notably the Satake transform, $L$-groups and $L$-morphisms, the properties of the Plancherel measure and the Macdonald formula for the unramified spectrum. We establish bounds for truncated Hecke algebras and for character traces that will play a role in subsequent chapters. In Section~\ref{sec:pp} we recall various analytic properties of automorphic $L$-functions on $\GL(d)$ and notably isobaric sums, bounds towards Ramanujan--Petersson and the so-called explicit formula for the sum of the zeros.
Section~\ref{s:Sato-Tate} introduces the Sato--Tate measure for general groups and Sato--Tate equidistribution for Satake parameters and for families.
The next Section~\ref{s:bg} gathers various background materials on orbital integral, the Gross motive and Tamagawa measure, discrete series characters and Euler--Poincar\'e functions, and Frobenius--Schur indicator. We establish bounds for special values of the Gross motive which will enter in the geometric side of the trace formula.

In Section~\ref{s:app:unif-bound} we establish a uniform bound for orbital integrals of the type~\eqref{e:intro-unif-bound}.
In Section~\ref{s:conj} we establish various bounds on conjugacy classes and level subgroups. How these estimates enter in the trace formula has been detailed in the outline above.

Then we are ready in Section~\ref{s:aut-Plan-theorem} to establish our main result, an automorphic Plancherel theorem for families with error terms and its application to the Sato-Tate theorem for families. The theorem is first proved for test functions on the unitary dual coming from Hecke algebras by orchestrating all the previous results in the trace formula. Then our result is improved to allow more general test functions, either in the input to the Sato-Tate theorem or in the prescribed local condition for the family, by means of Sauvageot's density theorem.

The last three Sections~\ref{s:langlands}, \ref{sec:zeros} and~\ref{sec:pf} concern the application to low-lying zeros. In complete generality we need to rely on Langlands global functoriality and other hypothesis that we state precisely. These unproven assumptions are within reach in the context of endoscopic transfer and we will return to it in subsequent works.

Appendix~\ref{s:app:Kottwitz} by Kottwitz establishes the boundedness of normalized orbital integrals from the theory of Shalika germs.  Appendix~\ref{s:app:B} by Cluckers--Gordon--Halupczok establishes a strong form of~\eqref{e:intro-unif-bound} with $e=1$ by using recent results in arithmetic motivic integration.

\subsection{Acknowledgments} We would like to thank Jim Arthur, Joseph Bernstein, Laurent Clozel, Julia Gordon, Nicholas Katz, Emmanuel Kowalski, Erez Lapid, Philippe Michel, Peter Sarnak, Kannan Soundararajan and Akshay Venkatesh for helpful discussions and comments. We would like to express our gratitude to Robert Kottwitz and Bao Ch\^{a}u Ng\^{o} for helpful discussions regarding Section~\ref{s:app:unif-bound}, especially about the possibility of a geometric approach. We appreciate Brian Conrad for explaining us about the integral models for reductive groups. We thank the referee for a very careful reading.

Most of this work took place during the AY2010-2011 at the Institute for Advanced Study and some of the results have been presented there in March. We thank the audience for their helpful comments and the IAS for providing excellent working conditions. S.W.S. acknowledges support from the National Science Foundation during his stay at the IAS under agreement No. DMS-0635607 and thanks Massachusetts Institute of Technology and Korea Institute for Advanced Study for providing a very amiable work environment. N.T. is supported by a grant \#209849 from the Simons Foundation.

\section{Satake Transforms}\label{s:Satake-trans}

\subsection{$L$-groups and $L$-morphisms}\label{sub:L-groups}

  We are going to recall some definitions and
  facts from \cite[\S1,\S2]{Bor79} and \cite[\S1]{Kot84a}.
  Let $F$ be a local or global field of characteristic 0 with an algebraic closure $\ol{F}$, which we fix.
  Let $W_F$ denote the Weil group of $F$ and set $\Gamma:=\Gal(\ol{F}/F)$.
  Let $H$ and $G$ be connected reductive groups over $F$.
  Let $(\hat{B},\hat{T},\{X_{\alpha}\}_{\alpha\in \Delta^{\vee}})$ be a
  splitting datum fixed by $\Gamma$, from which the $L$-group
  $${}^L G=\hat{G}\rtimes W_F$$ is constructed.
  An $L$-morphism $\eta:{}^L H\ra {}^L G$ is a continuous map commuting with
  the canonical surjections ${}^L H\ra W_F$ and ${}^L G\ra W_F$ such that
  $\eta|_{\hat{H}}$ is a morphism of complex Lie groups.
  A representation of ${}^L G$ is by definition
  a continuous homomorphism ${}^L G\ra \GL(V)$ for some $\C$-vector space $V$ with $\dim V<\infty$
  such that $r|_{\hat{G}}$ is a morphism of complex Lie groups. Clearly giving a representation
  ${}^L G\ra \GL(V)$ is equivalent to giving an $L$-morphism ${}^L G\ra {}^L \GL(V)$.

  Let $f:H\ra G$ be a normal morphism, which means that $f(H)$ is a normal subgroup of $G$.
  Then it gives rise to an $L$-morphism ${}^L G \ra {}^L H$
  as explained in \cite[2.5]{Bor79}. In particular, there is a
  $\Gamma$-equivariant map $Z(\hat{G})\ra Z(\hat{H})$, which is
  canonical (independent of the choice of splittings).
  Thus an exact sequence of connected reductive groups over $F$
  $$1\ra G_1\ra G_2\ra G_3\ra 1$$ gives rise to
  a $\Gamma$-equivariant exact sequence of $\C$-diagonalizable groups
 $$1\ra Z(\hat{G}_3)\ra Z(\hat{G}_2)\ra Z(\hat{G}_1)\ra 1.$$

\subsection{Satake transform}\label{sub:Satake-trans}

  From here throughout this section,
  let $F$ be a finite extension of $\Q_p$ with integer ring $\cO$ and a uniformizer
  $\varpi$. Set $q:=|\cO/\varpi\cO|$. Let $G$ be an \emph{unramified} group over $F$ and
  $B=TU$ be a Borel subgroup decomposed into the maximal torus and the unipotent radical in $B$.
  Let $A$ denote the maximal $F$-split torus in $T$.
  Write $\Phi_F$ (resp. $\Phi$) for the set of all $F$-rational roots (resp. all roots over $\ol{F}$)
  and $\Phi_F^+$ (resp. $\Phi^+$) for the subset of positive roots.
  Choose a smooth reductive model of $G$ over $\cO$
  corresponding to a hyperspecial point on the apartment for $A$.
  Set $K:=G(\cO)$. Denote by $X_*(A)^+$ the subset of $X_*(A)$ meeting the closed Weyl
  chamber determined by $B$, namely
  $\lambda\in X_*(A)^+$ if $\alpha(\lambda)\ge 0$ for all $\alpha\in \Phi_F^+$.
  Denote by $\Omega_F$ (resp. $\Omega$) the $F$-rational Weyl group for $(G,A)$
  (resp. the absolute Weyl group for $(G,T$)), and $\rho_F$ (resp. $\rho$) the half sum
  of all positive roots in $\Phi^+_F$ (resp. $\Phi^+$). A partial order $\le$ is defined on $X_*(A)$ (resp. $X_*(T)$) such that $\mu\le \lambda$ if $\lambda-\mu$ is a linear combination of $F$-rational positive coroots (resp. positive coroots) with nonnegative coefficients. The same order extends to a partial order $\le_{\R}$ on $X_*(A)\otimes_\Z \R$ and $X_*(T)\otimes_\Z \R$ defined analogously.

 Let $F^{\ur}$ denote the maximal unramified extension of $F$. Let $\Fr$ denote the
 geometric Frobenius element of $\Gal(F^{\ur}/F)$.
  Define $W_F^{\ur}$ to be the unramified Weil group, namely the subgroup
  $\Fr^{\Z}$ of  $\Gal(F^{\ur}/F)$.
  Since $\Gal(\ol{F}/F)$ acts on $\hat{G}$ through a finite quotient
  of $\Gal(F^{\ur}/F)$, one can make sense of
   ${}^L G^{\ur}:=\hat{G}\rtimes W_F^{\ur}$.

  Throughout this section we write $G$, $T$, $A$ for $G(F)$, $T(F)$, $A(F)$ if there is no confusion.
  Define $\cH^{\ur}(G):=C^\infty_c(K\bs G / K)$ and $\cH^{\ur}(T):=C^\infty_c(T(F)/T(F)\cap K)$.
  The latter is canonically isomorphic to $\cH^{\ur}(A):=C^\infty_c(A(F)/A(\cO))$ via the inclusion
  $A\hra T$. We can further identify
  $$\cH^{\ur}(T)\simeq \cH^{\ur}(A)\simeq \C[X_*(A)]$$
  where the last $\C$-algebra isomorphism
  matches $\lambda\in X_*(A)$ with $\triv_{\lambda(\varpi)(A\cap K)}\in \cH^{\ur}(A)$.
  Let $\lambda\in X_*(A)$. Write $$\tau^G_\lambda:=\triv_{K\lambda(\varpi)K}\in \cH^{\ur}(G),\quad
  \tau^A_\lambda:=\frac{1}{|\Omega_F|}\sum_{w\in \Omega_F}
  \triv_{w\lambda(\varpi)(A\cap K)}\in \cH^{\ur}(A)^{\Omega_F}.$$
  The sets $\{\tau^G_\lambda\}_{\lambda\in X_*(A)^+}$ and
  $\{\tau^A_\lambda\}_{\lambda\in X_*(A)^+}$ are bases for
  $\cH^{\ur}(G)$ and $\cH^{\ur}(A)^{\Omega_F}$ as $\C$-vector spaces, respectively. Consider the map
  \beq\label{e:Satake-formula}\cH^{\ur}(G)\ra \cH^{\ur}(T),
  \quad f\mapsto \left(t\mapsto \delta_B(t)^{1/2} \int_{U} f(tu)du\right)\eeq
  composed with $\cH^{\ur}(T)\simeq \cH^{\ur}(A)$ above.
  The composite map induces a $\C$-algebra isomorphism \beq\label{e:Satake}\cS^G:\cH^{\ur}(G)\isom
  \cH^{\ur}(A)^{\Omega_F}\eeq called the Satake isomorphism. We often
  write just $\cS$ for $\cS^G$. We note that in general
  $\cS$ does not map $\tau^G_\lambda$ to $\tau^A_{\lambda}$.

  Another useful description of $\cH^{\ur}(G)$ is through representations of ${}^L G^{\ur}$. (The latter notion is defined as in \S\ref{sub:L-groups}.  Write
  $(\hat{G}\rtimes \Fr)_{\ssconj}$ for the set of $\hat{G}$-conjugacy classes of
  semisimple elements in $\hat{G}\rtimes \Fr$.
  Consider the set $$\ch({}^L G^{\ur}):=\{\tr r : (\hat{G}\rtimes \Fr)_{\ssconj}\ra \C\,|\,
  r~\mathrm{is~a~representation~of~} {}^L G^{\ur}\}.$$
  Define $\C[\ch({}^L G^{\ur})]$ to be the $\C$-algebra generated by $\ch({}^L G^{\ur})$
  in the space of functions on $(\hat{G}\rtimes \Fr)_{\ssconj}$.
  For each $\lambda\in X_*(A)^+$ define the quotient
  \beq\label{e:chi_lambda}\chi_\lambda:=\frac{\sum_{w\in \Omega_F} \sgn(w) w(\lambda+\rho_F)}{\sum_{w\in \Omega_F} \sgn(w)w\rho_F},\eeq
  which exists as an element of $\C[X_*(A)]^{\Omega_F}$ and is unique.
(One may view $\chi_\lambda$ as
the analogue in the disconnected case
of the irreducible character of highest weight $\lambda$, cf. proof of Lemma \ref{l:Kostant} below.)
 Then $\{\chi_\lambda\}_{\lambda\in X_*(A)^+}$
  is a basis for $\C[X_*(A)]^{\Omega_F}$ as a $\C$-vector space, cf. \cite[p.465]{Kat82}.
  (Another basis was given by $\tau^A_\lambda$'s above.)
  There is a canonical $\C$-algebra isomorphism
  \beq\label{e:Satake2}  \cT:\C[\ch({}^L G^{\ur})] \isom \cH^{\ur}(A)^{\Omega_F},\eeq
  determined as follows (see \cite[Prop 6.7]{Bor79} for detail):
  for each irreducible $r$, $\tr r|_{\hat{T}}$ is shown to factor through
  $\hat{T}\ra \hat{A}$ (induced by $A\subset T$). Hence $\tr r|_{\hat{T}}$ can be viewed as an element
  of $\C[X^*(\hat{A})]=\C[X_*(A)]$, which can be seen to be invariant under $\Omega_F$. Define $\cT(\tr r)$ to be the latter element.

  Let $r_0$ be an irreducible representation of $\hat{G}$ of highest weight
  $\lambda_0\in X^*(\hat{T})^+=X_*(T)^+$. The group $W_F^{\ur}$ acts on $X^*(\hat{T})^+$.
  Write $\Stab(\lambda_0)\subset W^{\ur}_F$ for the stabilizer subgroup for $\lambda_0$,
which has finite index (since a finite power of $\Fr$ acts trivially on $\hat{G}$ and thus also on $\hat{T}$).
  Put $r:=\ind_{\hat{G}\rtimes \Stab(\lambda_0)}^{^L G^{\ur}} r_0$
  and $\lambda:=\sum_{\sigma\in W^{\ur}_F/\Stab(\lambda_0)} \sigma\lambda_0
  \in X_*(A)^+$.
  Clearly $r$ and $\lambda$ depend only on the $W_F^{\ur}$-orbit of $\lambda_0$.
  Put $i(\lambda_0):=[ W^{\ur}_F:\Stab(\lambda_0)]$.

\begin{lem}\label{l:Kostant}
\benu
\item Suppose that $r$ and $\lambda$ are obtained from $r_0$ and $\lambda_0$ as above. Then
 \beq\label{e:T(tr-r)}\cT(\tr r)=\chi_\lambda.\eeq
\item In general for any irreducible representation $r':{}^L G^{\ur}\ra \GL_d(\C)$
such that $r'(W^{\ur}_F)$ has relatively compact image,
let $r_0$ be any irreducible subrepresentation of $r'|_{\hat{G}}$.
 Let $r$ be obtained from $r_0$ as above.
Then for some $\zeta\in \C^\times$ with $|\zeta|=1$,
$$\tr r'=\zeta\cdot \tr r.$$
\eenu
\end{lem}

\begin{proof}
  Let us prove (i).
  For any $i\ge 1$, let ${}^L G_{i}$ denote the finite $L$-group
  $\hat{G}\rtimes \Gal(F_{i}/F)$ where $F_{i}$ is the degree $i$ unramified
  extension of $F$ in $\ol{F}$.
It is easy to see that $r(\Fr^{i(\lambda_0)})$ is trivial and that
  $r=\ind^{{}^L G_{i(\lambda_0)}}_{\hat{G}} r_0$.
  Then \eqref{e:T(tr-r)} amounts to
Kostant's character formula for a disconnected group (\cite[Thm 7.5]{Kos61}) applied to
${}^L G_{i(\lambda_0)}$.
  As for (ii), let $\lambda_0$ and $\lambda$ be as in the paragraph
 preceding the lemma. Let $j\ge 1$ be such that $G$ becomes split over a degree $j$
unramified extension of $F$. (Recall that $G$ is assumed to be unramified.)
  By twisting $r'$ by a unitary character of $W^{\ur}_F$ one may assume that
  $r'$ factors through ${}^L G_j$.
  Then both $r$ and $r'$ factor through ${}^L G_j$
  and are irreducible constituents of
  $\ind^{{}^L G_{j}}_{\hat{G}} r_0$.
  From this it is easy to deduce that $r'$ is a twist of $r$ by
  a finite character of $W^{\ur}_F$ of order dividing $j$. Assertion (ii) follows.
\end{proof}

  Each $\lambda\in X_*(A)^+$ determines $s_{\lambda,\mu}\in \C$ such that
  \beq\label{e:s-lambda-mu}\cS^{-1}(\chi_\lambda)=
  \sum_{\mu\in X_*(A)^+} s_{\lambda,\mu} \tau^G_{\mu}\eeq
  where only finitely many $ s_{\lambda,\mu}$ are nonzero. In fact Theorem 1.3 of \cite{Kat82} identifies $s_{\lambda,\mu}$ with $K_{\lambda,\mu}(q^{-1})$ defined in (1.2) of that paper, cf. \S4 of \cite{Gro98}. In particular $s_{\lambda,\lambda}\neq 0$ and $s_{\lambda,\mu}\neq 0$ unless $\mu\le \lambda$. The following information will be useful in \S\ref{sub:test-functions}.

\begin{lem}\label{l:bound-on-s}
  Let $\lambda,\mu\in X_*(A)^+$.
  Suppose that $\lambda\star_w \mu:=w(\lambda+\rho_F)- (\mu+\rho_F)$ is nontrivial for all $w\in \Omega_F$.
  For $\kappa\in X_*(A)$ let $p(\kappa)\in \Z_{\ge 0}$ be
  the number of tuples $(c_{\alpha^\vee})_{\alpha^\vee\in (\Phi^{\vee}_F)^+}$
  with $c_{\alpha^\vee}\in \Z_{\ge 0}$ such that $\sum_{\alpha^\vee} c_{\alpha^\vee}\cdot \alpha^\vee = \kappa$. Then
  $$|s_{\lambda,\mu}|\le q^{-1}|\Omega_F| \max_{w\in \Omega_F} p(\lambda\star_w \mu)  .$$
\end{lem}

\begin{proof}
  It is easy to see from the description of $K_{\lambda,\mu}(q^{-1})$ in \cite[(1.2)]{Kat82} that
  $$|K_{\lambda,\mu}(q^{-1})|\le |\Omega_F| \max_{w\in \Omega_F} \hat{\mP}(w(\lambda+\rho_F)-(\mu+\rho_F);q^{-1}).$$
  The definition of $\hat{\mP}$ in \cite[(1.1)]{Kat82} shows that
  $0\le \hat{\mP}(\kappa;q^{-1})\le p(\kappa) q^{-1}$ if $\kappa\neq 0$.
\end{proof}

\subsection{Truncated unramified Hecke algebras}\label{sub:trun-unr-Hecke}

  Set $n:=\dim T$ and $X_*(T)_{\R}:=X_*(T)\otimes_\Z \R$.
  Choose an $\R$-basis $\cB=\{e_1,...,e_{n}\}$ of $X_*(T)_\R$. For each $\lambda\in X_*(T)_\R$,
  written as $\lambda=\sum_{i=1}^{n} a_i(\lambda) e_i$ for unique $a_i(\lambda)\in \R$,
  define $$|\lambda|_{\cB}:=\max_{1\le i\le n} |a_i( \lambda)|,
  \quad \|\lambda\|_{\cB}:=\max_{\omega\in \Omega}( |\omega \lambda|_{\cB}).$$
  When there is no danger of confusion,
  we will simply write $|\cdot|_{\cB}$ or even $|\cdot|$ instead of $|\cdot|_{\cB}$,
  and similarly for $\|\cdot\|_{\cB}$.
  It is clear that $\|\cdot\|_{\cB}$ is $\Omega$-invariant and
  that $|\lambda_1+\lambda_2|_{\cB}\le |\lambda_1|_{\cB}+|\lambda_2|_{\cB}$
  for all $\lambda_1,\lambda_2\in X_*(T)$.
  When $\kappa\in \Z_{\ge 0}$, define \beq\label{e:trun-Hecke}\cH^{\ur}(G)^{\le \kappa,\cB}:=
  \{ \C\mbox{-subspace~of~} \cH^{\ur}(G)
  \mbox{~generated~by~}\tau^G_\lambda,~\lambda\in X_*(A)^+,~\|\lambda\|_{\cB}\le \kappa\}.\eeq
  It is simply written as $\cH^{\ur}(G)^{\le \kappa}$ when the choice of $\cB$ is clear.

%\begin{ex}\label{ex:truncation-GL_d}
%  When $G=\GL_d$, $T$ is the diagonal maximal torus and $\{e_1,...,e_d\}$ is the
%  standard $\Z$-basis of $X_*(T)$, the above definition of $\cH^{\ur}(\GL_d)^{\le \kappa}$
%  coincides with the one in \S\ref{sub:case-of-GL_d}.
%\end{ex}

\begin{lem}\label{l:norm-and-bases}
 Let $\cB$ and $\cB'$ be two $\R$-bases of $X_*(T)_\R$. Then
there exist constants $c_1,c_2,B_1,B_2,B_3,B_4>0$ such that
for all $\lambda\in X_*(T)_\R$,
\benu
\item $c_1 |\lambda|_{\cB'} \le |\lambda|_{\cB}
\le c_2 |\lambda|_{\cB'}$,
\item $B_1 |\lambda|_{\cB} \le \|\lambda\|_{\cB}
\le B_2 |\lambda|_{\cB}$ for all $ \lambda\in X_*(T)_\R$,
\item $B_3 \|\lambda\|_{\cB'} \le \|\lambda\|_{\cB}
\le B_4 \|\lambda\|_{\cB'}$ for all $ \lambda\in X_*(T)_\R$ and
\item $\cH^{\ur}(G)^{\le B_4^{-1}\kappa,\cB'}\subset \cH^{\ur}(G)^{\le \kappa,\cB}\subset \cH^{\ur}(G)^{\le B_3^{-1}\kappa,\cB'}$.
\eenu
\end{lem}

\begin{proof}
  Let us verify (i).
  As the roles of $\cB$ and $\cB'$ can be changed,
  it suffices to prove the existence of $c_2$.
   For this, it suffices to take $c_2=\sup_{|\lambda|_{\cB}\le 1} |\lambda|_{\cB'}$.
   The latter is finite since $|\cdot|_{\cB'}$ is a continuous function
  on the set of $\lambda$ such that $|\lambda|_{\cB}\le 1$, which
  is compact.
  Part (ii) is obtained by
 applying the lemma to the bases $\cB'=\omega \cB$ for all $\omega\in \Omega$.
 Let us check (iii). Let $B_1,B_2>0$ (resp. $B'_1,B'_2>0$)
 be the constants of (ii) for the basis $\cB$ (resp. $\cB'$).
 Then
 $$  c_1B_1(B'_2)^{-1} \|\lambda\|_{\cB'} \le c_1B_1 |\lambda|_{\cB'}  \le B_1 |\lambda|_{\cB} \le \|\lambda\|_{\cB}$$
 and similarly $\|\lambda\|_{\cB}\le c_2B_2(B'_1)^{-1}\|\lambda\|_{\cB'}$.
  Finally (iv) immediately follows from (iii).
\end{proof}

  It is natural to wonder whether the definition of truncation in \eqref{e:trun-Hecke} changes if a different basis $\{\tau^G_\lambda\}$ or $\{\chi_\lambda\}$ is used. We assert that it changes very little in a way that the effect on $\kappa$ is bounded by a $\kappa$-independent constant. To ease the statement define $\cH_i^{\ur}(G)^{\le \kappa,\cB}$ for $i=1$ (resp. $i=2$) to be the $\C$-subspace of $\cH^{\ur}(G)$ generated by $\cS^{-1}(\tau^A_\lambda)$ (resp. $\cS^{-1}(\chi_\lambda)$) for $\lambda\in X_*(A)^+$ with $\|\lambda\|_{\cB}\le \kappa$.

\begin{lem}\label{l:three-truncations}
  There exists a constant $C\ge 1$ such that for every $\kappa\in \Z_{\ge 0}$ and
  for any $i,j\in \{\emptyset, 1,2\}$,
  $$\cH_i^{\ur}(G)^{\le \kappa,\cB}\subset \cH_j^{\ur}(G)^{\le C\kappa,\cB}.$$
\end{lem}

\begin{proof}
  It is enough to prove the lemma for a particular choice of $\cB$ by Lemma \ref{l:norm-and-bases}.
  So we may assume that $\cB$ extends the set of simple coroots in $\Phi^\vee$ by an arbitrary basis of $X_*(Z(G))_{\R}$. Again by Lemma \ref{l:norm-and-bases} the proof will be done if we show that each of the following generates the same $\C$-subspace:
\benu
\item the set of $\tau^G_{\lambda}$ for $\lambda\in X_*(A)^+$ with $|\lambda|_{\cB}\le \kappa$,
\item the set of $\cS^{-1}(\tau^A_\lambda)$ for $\lambda\in X_*(A)^+$ with $|\lambda|_{\cB}\le \kappa$,
\item the set of $\cS^{-1}(\chi_\lambda)$ for $\lambda\in X_*(A)^+$ with $|\lambda|_{\cB}\le \kappa$.
\eenu
  It suffices to show that the matrices representing the change of bases are ``upper triangular'' in the sense that the $(\lambda,\lambda)$ entries are nonzero and $(\lambda,\mu)$ entries are zero unless $\lambda\ge \mu$. (Note that $\lambda\ge \mu$ implies $|\lambda|_{\cB}\ge |\mu|_{\cB}$ by the choice of $\cB$.) We have remarked below \eqref{e:chi_lambda} that $s_{\lambda,\mu}$'s have this property, accounting for (i)$\leftrightarrow$(iii). For (ii)$\leftrightarrow$(iii) the desired property can be seen directly from \eqref{e:chi_lambda} by writing $\chi_\lambda$ in terms of $\tau^A_\mu$'s.

\end{proof}

\subsection{The case of $\GL_d$}\label{sub:case-of-GL_d}

  The case $G=\GL_d$ is considered in this subsection. Let $A=T$ be the diagonal maximal torus and $B$ the group of upper triangular matrices. For $1\le i\le d$,
  take $Y_i\in X_*(A)$ to be $y\mapsto \diag(1,...,1,y,1,...,1)$ with $y$ in the $i$-th place.
  One can naturally identify $X_*(A)\simeq \Z^d$ such that the images of
  $Y_i$ form the standard basis of $\Z^d$.
  Then $\Omega_F$ is isomorphic to $\mS_d$, the symmetric group in $d$ variables acting on $\{Y_1,...,Y_d\}$ via permutation of indices.
  We have the Satake isomorphism
  $$\cS:\cH^{\ur}(\GL_d)\isom \cH^{\ur}(T)^{\Omega_F}\simeq \C[Y_1^{\pm},...,Y_d^{\pm}]^{\mS_d}.$$
  For an alternative description let us introduce standard symmetric polynomials $X_1,...,X_d$ by the equation in a formal $Z$-variable $(Z-Y_1)\cdots (Z-Y_d)=Z^d-X_1Z^{d-1}+\cdots + (-1)^{d} X_d$. Then $$\C[Y_1^{\pm},...,Y_d^{\pm}]^{\mS_d}=\C[X_1,...,X_{d-1},X_d^{\pm}].$$
  Let $\kappa\in \Z_{\ge 0}$. Define $\cH^{\ur}(\GL_d)^{\le \kappa}$, or simply $\cH^{\le \kappa}_d$, to be
  the preimage under $\cS$ of the $\C$-vector space generated by $$\{\sum_{\sigma\in \mS_d} Y_{\sigma(1)}^{a_1}Y_{\sigma(2)}^{a_2}\cdots Y_{\sigma(d)}^{a_d}:  a_1,...,a_d\in [-\kappa,\kappa]\}.$$ The following is standard (cf. \cite{Gro98}).

\begin{lem} Let $r\in \Z_{\ge 1}$. Let $\lambda_r:=(r,0,0,...,0)\in X_*(A)^+$. Then
  $$\cS^{-1}(Y_1^r+\cdots+Y_d^r)=\sum_{\mu\in X_*(A)^+\atop\mu\le \lambda_r} c_{\lambda_r,\mu}\cdot \tau^G_\mu $$
  for $c_{\lambda_r,\mu}\in \C$ with $c_{\lambda_r,\lambda_r}=q^{r(1-d)/2}$,
  where the sum runs over the set of $\mu\in X_*(T)^+$ such that $\mu\le \lambda_r$.
  In particular,
  \begin{eqnarray}
    \cS^{-1}(Y_1+\cdots+Y_d)&=&q^{(1-d)/2}\tau^{G}_{(1,0,...,0)},\nonumber\\
    \cS^{-1}(Y_1^2+\cdots+Y_d^2)&=&q^{1-d}(\tau^{G}_{(2,0,...,0)}+(1-q)\tau^G_{(1,1,0,...,0)}).
    \nonumber
  \end{eqnarray}

%  (Could try to be explicit for all $r\ge 1$. This lemma is about the test function whose trace computes
%  the sum of the $r$-th powers of Satake parameters.)
\end{lem}

\subsection{$L$-morphisms and unramified Hecke algebras}\label{sub:L-mor-unr-Hecke}

  Assume that $H$ and $G$ are unramified groups over $F$.
  Let $\eta:{}^L H\ra {}^L G$ be an unramified $L$-morphism, which means that
  it is inflated from some $L$-morphism ${}^L H^{\ur}\ra {}^L G^{\ur}$
  (the notion of $L$-morphism for the latter is defined as in \S\ref{sub:L-groups}).
  There is a canonically induced
  map $\ch({}^L G^{\ur})\ra \ch({}^L H^{\ur})$.
  Via \eqref{e:Satake} and \eqref{e:Satake2}, the latter map gives rise to a $\C$-algebra map
  $\eta^*: \cH^{\ur}(G)\ra \cH^{\ur}(H)$.

  We apply the above discussion to an unramified representation
  $$r:{}^L G\ra \GL_d(\C).$$ Viewing $r$ as an $L$-morphism ${}^L G\ra {}^L \GL_d$,
  we obtain
%  $r^*:\cH^{\ur}(\GL_d)\ra \cH^{\ur}(G)$ as we now explain.
%  Let $\hat{\T}\subset \GL_d$ be a maximal torus containing $r(\hat{T})$. Then
%  $r|_{\hat{T}}: \hat{T}\ra \hat{\T}$ induces a group homomorphism
%  $X^*(\hat{\T})\ra X^*(\hat{T})$, whose image lies in the $\Gamma$-invariant subgroup.
%  Thereby we obtain $X_*(\T)\ra X_*(T)^{\Gamma}=X_*(A)$, inducing a $\C$-algebra map
%  $\cH^{\ur}(\T)\ra \cH^{\ur}(A)$. Via the Satake isomorphisms for $G$ and $\GL_d$, we finally arrive at
  $$r^*:\cH^{\ur}(\GL_d)\ra \cH^{\ur}(G).$$

\begin{lem}\label{l:bound-degree-Satake}
  Let $\cB$ be an $\R$-basis of $X_*(T)_\R$.
  There exists a constant $\beta>0$ (depending on $\cB$, $d$ and $r$) such that for all $\kappa\in \Z_{\ge 0}$,
  $r^*(\cH^{\ur}(\GL_d)^{\le \kappa})\subset \cH^{\ur}(G)^{\le \beta\kappa,\cB}$ .
\end{lem}

\begin{proof}
  Thanks to Lemma \ref{l:norm-and-bases}, it is enough to deal with a particular choice of
  $\cB$. Choose $\cB$ by extending the set $\Delta^{\vee}$ of simple coroots,
  and write $\cB=\Delta^{\vee}\coprod \cB_0$.
  We begin by proving the following claim: let $\lambda_1,\lambda_2\in X_*(A)^+$ and expand
  the convolution product
  \begin{equation*}
    % \label{e:trun}
  \tau^G_{\lambda_1}*\tau^G_{\lambda_2}=\sum_{\mu} a^{\mu}_{\lambda_1,\lambda_2}
  \tau^G_\mu
  \end{equation*}
  where only $\mu\in X_*(A)^+$ such that $\mu\le_\R \lambda_1+\lambda_2$ contribute (cf. \cite[p.148]{Car79}).
  Only finitely many terms are nonzero.
  Then the claim is that $$|\mu|_{\cB}\le |\lambda_1+\lambda_2|_{\cB}, \quad\mbox{whenever}~a^{\mu}_{\lambda_1,\lambda_2}\neq 0.$$
  To check the claim,   consider $\mu=\sum_{e\in \cB} a_e(\mu) \cdot e $ and $\lambda_1+\lambda_2
  = \sum_{e\in \cB} a_e(\lambda_1+\lambda_2) \cdot e$, where the coefficients are in $\R$.
  The conditions $\mu\le_{\R} \lambda_1+\lambda_2$ and $\mu\in X_*(T)_{\R,+}$
  imply that $a_e(\mu)=a_e(\lambda_1+\lambda_2)$ if $e\in \cB_0$ and
  $0\le a_e(\mu)\le a_e(\lambda_1+\lambda_2)$ if $e\in \Delta^{\vee}$.
  Hence $|\mu|_{\cB}\le |\lambda_1+\lambda_2|_{\cB}$.

  We are ready to prove the lemma. It is explained in Lemma \ref{l:three-truncations} and the remark below it that there exists a constant $\beta_1>0$ which is independent of $\kappa$ such that every $\phi\in \cH^{\ur}(\GL_d)^{\le \kappa}$ can be written as a $\C$-linear combination of
    $$\sum_{\sigma\in \mS_d} Y_{\sigma(1)}^{a_1}Y_{\sigma(2)}^{a_2}\cdots Y_{\sigma(d)}^{a_d},\quad  a_1,...,a_d\in [-\beta_1\kappa,\beta_1\kappa].$$
  Each element above can be rewritten in terms of the symmetric polynomials $X_i$'s of \S\ref{sub:case-of-GL_d}: First, $X_d^{\beta_1\kappa}$ times $\sum_{\sigma\in \mS_d} Y_{\sigma(1)}^{a_1}Y_{\sigma(2)}^{a_2}\cdots Y_{\sigma(d)}^{a_d}$ is a symmetric polynomial of degree $\le 2\beta_1\kappa$, which in turn is a polynomial in $X_1,...,X_d$ of degree $\le  2\beta_1\kappa$. We conclude that every $\phi\in \cH^{\ur}(\GL_d)^{\le \kappa}$ is in the span of monomials
  \beq\label{e:monomials}X_1^{b_1}X_1^{b_2}\cdots X_d^{b_d},\quad b_1,...,b_d\in [-2\beta_1\kappa,2\beta_1\kappa].\eeq
  For each $1\le i\le d$, write $r^*(X_i)$ (resp. $r^*(X_i^{-1})$) as a linear combination of $\tau^G_{\lambda_{i,j}}$ (resp. $\tau^G_{\lambda^-_{i,j}}$) with nonzero coefficients. Define $\beta_0$ to be the maximum among all possible $|\lambda_{i,j}|$ and $|\lambda^-_{i,j}|$. The above claim  $r^*(X_1^{b_1}X_1^{b_2}\cdots X_d^{b_d})$ as in \eqref{e:monomials} is in the $\C$-span of $\tau^G_\mu$ satisfying
    $$|\mu|_{\cB}\le (|b_1|+\cdots + |b_d|)\beta_0\le  2d\beta_0\beta_1\kappa.$$
  So the above span contains $r^*(\phi)$ for $\phi\in \cH^{\ur}(\GL_d)^{\le \kappa}$. By Lemma
  \ref{l:norm-and-bases} there exists a constant $B_2>0$ such that $\|\mu\|_{\cB}\le B_2|\mu|_{\cB}$ for every $\mu\in X_*(T)$. Hence the lemma holds true with $\beta:=2B_2d\beta_0\beta_1$.

\end{proof}

  The map $r$ also induces a functorial transfer for unramified representations
\beq\label{e:r_*} r_*:\Rep^{\ur}(G(F))\ra \Rep^{\ur}(\GL_d(F))\eeq
 uniquely characterized by $\tr r_*(\pi)(\phi)=\tr \pi(r^*\phi)$ for all $\pi\in \Rep^{\ur}(G(F))$ and
$\phi\in \cH^{\ur}(\GL_d(F))$.

\subsection{Partial Satake transform}

  Keep the assumption that $G$ is unramified over $F$.
 Let $P$ be an $F$-rational parabolic subgroup of $G$ with Levi $M$
 and unipotent radical $N$ such that $B=TU$ is contained in $P$.
 Let $\Omega_M$ (resp. $\Omega_{M,F}$) denote the absolute
 (resp. $F$-rational) Weyl group for $(M,T)$.
  A partial Satake transform is defined as (cf. \eqref{e:Satake-formula})
 \[
  %\label{e:parital-Satake}
 \cS^G_M:\cH^{\ur}(G)\ra \cH^{\ur}(M),
  \quad f\mapsto \left(m\mapsto \delta_P(m)^{1/2} \int_{N} f(mn)dn\right)
 \]
  It is well known that $\cS^G=\cS^M\circ \cS^G_M$. More concretely, $\cS^G_M$ is the
 canonical inclusion $\C[X_*(A)]^{\Omega_{M,F}}\hra \C[X_*(A)]^{\Omega_{F}}$ if
 $\cH^{\ur}(M)$ and $\cH^{\ur}(G)$ are identified with the source and the target via $\cS^G$ and $\cS^M$, respectively.
 Since $T$ is a common maximal torus of $M$ and $G$, an $\R$-basis $\cB$ of $X_*(T)_\R$
 determines truncations on $\cH^{\ur}(M)$ and $\cH^{\ur}(G)$.

\begin{lem}
  For any $\kappa\in \Z_{\ge 0}$,
  $\cS^G_M(\cH^{\ur}(G)^{\le \kappa,\cB})\subset \cH^{\ur}(M)^{\le \kappa,\cB}$.
\end{lem}

\begin{proof}
  It is enough to note that $\|\lambda\|_{\cB,M}\le \|\lambda\|_{\cB,G}$ for all $\lambda\in X_*(A)$, which holds
  since the $\Omega_{M}$-orbit of $\lambda$ is contained in the $\Omega$-orbit of $\lambda$.
\end{proof}

\begin{rem}
  Let $\eta:{}^L M\ra {}^L G$ be the embedding of \cite[\S3]{Bor79}, well defined up to $\hat{G}$-conjugacy.
  Then $\cS^G_M$ coincides with $\eta^*:\cH^{\ur}(G)\ra \cH^{\ur}(M)$ of \S\ref{sub:L-mor-unr-Hecke}
\end{rem}

\subsection{Some explicit test functions}\label{sub:test-functions}

  Assume that $r:{}^L G=\hat{G}\rtimes W_F\ra \GL_d(\C)$ is
  an \emph{irreducible} representation arising from
  an unramified $L$-morphism ${}^L G^{\ur} \ra {}^L \GL_d^{\ur}$
  such that $r(W_F)$ is relatively compact.
  For later applications it is useful to study the particular element
  $r^*(Y_1+\cdots+Y_d)$ in $\cH^{\ur}(G)$.

%  Let $\phi=r^*(Y_1+\cdots+Y_d)$. By definition of $r^*$ we may identify $\phi=\tr r$
%  in $\C[\ch({}^L G)]$ (\S\ref{sub:L-mor-unr-Hecke}). Using \eqref{e:Satake2}
%  we can write $$\cT(\tr r)=\sum_{\lambda\in X_*(A)^+} a_{r,\lambda} \chi_\lambda$$
%  with unique $a_{r,\lambda}\in \C$. Write $|r|\in \Z_{\ge 1}$ for the number of
%  nonzero coefficients in the above finite sum.
%

\begin{lem}\label{l:bound-1st-moment}
  Let $\phi=r^*(Y_1+\cdots+Y_d)$. Then
\benu
\item Suppose that $r:{}^L G^{\ur}\ra \GL_d(\C)$ does not factor through
 $W^{\ur}_F$ (or equivalently that $r|_{\hat{G}}$ is not the trivial representation).
 Then $$|\phi(1)|\le
|\Omega_F| \max_{w\in \Omega_F} p(\lambda\star_w 0)\cdot  q^{-1}.$$
\item Suppose that $r|_{\hat{G}}$ is trivial. Then
 $\phi(1)=r(\Fr)$.
\eenu
\end{lem}

%\begin{rem}
%  When $G$ is split and $r|_{\hat{G}}$ has highest weight $\mu\in X_*(A)^+$,
%  \eqref{e:T(tr-r)} tells us that $a_{r,0}=0$ unless $\mu=0$, in which case
%  $r$ is the trivial representation, $\cS^{-1}(\phi)=\triv_{K}$ and $\cS^{-1}(\phi)(1)=1$.
%\end{rem}

\begin{proof}
  Let us do some preparation. By twisting $r$ by an unramified unitary character of $W_F$ (viewed as a character of $^L G$) we may assume that $r=\ind^{^L G_j}_{\hat{G}} r_0$ for some irreducible representation $r_0$ of $\hat{G}$, cf. the proof of Lemma \ref{l:Kostant}.(ii). Let $\lambda_0$ be the highest weight of $r_0$ and define $\lambda\in X_*(A)^+$ as in the paragraph preceding Lemma \ref{l:Kostant}.
  The lemma tells us that $\cS(\phi)=\zeta \chi_\lambda\in \C[X_*(A)]^{\Omega_F}$ with $|\zeta|=1$.

In the case of (ii), $r$ is just an unramified unitary character of $W_F$ (with $d=1$), and it is easily seen that $\chi_\lambda=\tau^A_0$, $\zeta=r(\Fr)$, and so $\phi(1)=r(\Fr)$.
  Let us put ourselves in the case (i) so that $\lambda\neq 0$.
Note that $\phi(1)$ is just the coefficient of $\tau^G_0$
when $\phi=\zeta \cS^{-1}(\chi_\lambda)$
is written with respect to the basis $\{\tau^G_\mu\}$.
 Such a coefficient equals $\zeta \cdot s_{\lambda,0}$ according to \eqref{e:s-lambda-mu}, so
 $|\phi(1)|= | s_{\lambda,0}|$.
Now Lemma \ref{l:bound-on-s} concludes the proof. (Observe
that $\lambda\star_w 0\neq 0$ whenever $0\neq \lambda\in X_*(A)^+$.)
\end{proof}

\subsection{Examples in the split case}
When $G$ is split,
  it is easy to see that $\C[\ch({}^L G^{\ur})]$ is canonically identified with $\C[\ch(\hat{G})]$
  which is generated by finite dimensional characters in the space of functions on $\hat{G}$.
  So we may use $\C[\ch(\hat{G})]$ in place of $\C[\ch({}^L G^{\ur})]$.

\begin{ex}\label{ex:functions-Sp_2n}(When $G=Sp_{2n}$, $n\ge1$)

  Take $r:\hat{G}=SO_{2n+1}(\C)\hra \GL_{2n+1}(\C)$ to be the standard representation.
  Then $$Y_1+\cdots+Y_{2n+1}=\tr(\Std)\in \C[\ch(\GL_{2n+1})]$$
  is mapped to $\tr(r)\in \C[\ch(SO_{2n+1})]$ and
   $$Y^2_1+\cdots+Y^2_{2n+1}=\tr(\Sym^2(\Std)-\wedge^2(\Std))\in \C[\ch(\GL_{2n+1})]$$
  is mapped to $\tr(r)\in \C[\ch(SO_{2n+1})]$.
  Then $\Sym^2(V)$ breaks into $\C$ and an irreducible representation of
  $\hat{G}$ of highest weight $(2,0,...,0)$ in the standard parametrization.
  When $n>1$, $\wedge^2(V)$ is irreducible of highest weight $(1,1,0,...,0)$.
  When $n=1$, $\wedge^2(V)\simeq V^\vee$, i.e. isomorphic to $(\Std)^\vee$. (See
  \cite[\S19.5]{FH91}.) Let us systematically write $\Lambda_\lambda$ for the irreducible representation
  of $SO_{2n+1}$ with highest weight $\lambda$. Then
  \begin{eqnarray}
   r^*(Y_1+\cdots+Y_{2n+1})&=&\tr \Lambda_{(1,0,...,0)},\label{e:Satake-Sp2n}\\
    r^*(Y_1^2+\cdots+Y_{2n+1}^2)
   & =&\tr( \C+\Lambda_{(2,0,...,0)}-\Lambda_{(1,1,0,...,0)}).\nonumber
  \end{eqnarray}
  if $n\ge 2$. If $n=1$, the same is true if $\Lambda_{(1,1,0,...,0)}$ is replaced with
  $\Lambda_{(-1)}$.
  For $i=1,2$, define
  $$\phi^{(i)}:=\cS^{-1}(r^*(Y^i_1+\cdots+Y^i_{2n+1})).$$
  Then one computes
  \begin{eqnarray}
   \phi^{(1)}&=& q^{\frac{1-2n}{2}} \triv_{K \mu_{(1,0,...,0)}(\varpi_v)K} ,\nonumber\\
   \phi^{(2)}
   & =& \triv_K+q^{1-2n} \triv_{K\mu_{(2,0,...,0)}(\varpi_v)K}
   - q^{1-2n}(q-1)\triv_{K\mu_{(1,1,0,...,0)}(\varpi_v)K}.\nonumber
  \end{eqnarray}
  where $\mu_{\lambda}$ is the cocharacter of a maximal torus given by $\lambda$ in the standard parametrization.
  In particular, $\phi^{(1)}(1)=0$ and $\phi^{(2)}(1)=1$.
\end{ex}

\begin{ex}(When $G=SO_{2n}$, $n\ge 2$)

  Take $r:\hat{G}=SO_{2n}(\C)\hra \GL_{2n}(\C)$ to be the standard representation.
 Similarly as before, $\Sym^2(V)$ breaks into $\C$ and an irreducible representation of
  $\hat{G}$ of highest weight $(2,0,...,0)$.
  When $n>1$, $\wedge^2(V)$ is irreducible of highest weight $(1,1,0,...,0)$.
  When $n=1$, $\wedge^2(V)\simeq \C$. (See \cite[\S19.5]{FH91}.) The same
  formulas as \eqref{e:Satake-Sp2n} hold in this case.
  Defining \beq\label{e:Satake-SO2n}\phi^{(i)}:=\cS^{-1}(r^*(Y^i_1+\cdots+Y^i_{2n})),\eeq we
  can compute $\phi^{(1)}$, $\phi^{(2)}$ and see that
   $\phi^{(1)}(1)=0$ and $\phi^{(2)}(1)=1$.
\end{ex}

\begin{ex}(When $G=SO_{2n+1}$)

  Take $r:\hat{G}=Sp_{2n}(\C)\hra \GL_{2n}(\C)$ to be the standard representation.
  Then $$Y_1+\cdots+Y_{2n}=\tr(\Std)\in \C[\ch(\GL_{2n})]$$
  is mapped to $\tr(r\circ \Std)\in \C[\ch(Sp_{2n})]$ and
  Then $$Y^2_1+\cdots+Y^2_{2n}=\tr(\Sym^2(\Std)-\wedge^2(\Std))\in \C[\ch(\GL_{2n})]$$
  is mapped to $\tr(r\circ \Std)\in \C[\ch(Sp_{2n})]$.
    If $n\ge 2$ then $\wedge^2(V)$ breaks into $\C$ and an irreducible representation of
  $\hat{G}$ of highest weight $(1,1,0,...,0)$.
 (See \cite[\S17.3]{FH91}.) We have
 \begin{eqnarray*}
   r^*(Y_1+\cdots+Y_{2n+1})&=&\tr \Lambda_{(1,0,...,0)},\label{e:Satake-SO2n+1}\\
    r^*(Y_1^2+\cdots+Y_{2n+1}^2)
   & =&\tr( \Lambda_{(2,0,...,0)}-\Lambda_{(1,1,0,...,0)}-\C ).\nonumber
  \end{eqnarray*}
  As in Example \ref{ex:functions-Sp_2n}, $\Lambda$ designates a highest weight representation
  (now of $Sp_{2n}$).
  Define $\phi^{(i)}$ as in \eqref{e:Satake-SO2n}.
  By a similar computation as above, $\phi^{(1)}(1)=0$, $\phi^{(2)}(1)=-1$.
\end{ex}

%  As $r|_{\hat{G}}$ may be reducible (which may happen if
%  the finite Galois action on $\hat{G}$ is not trivial), we may write
%  $r|_{\hat{G}}=\oplus_{i=1}^s r_i$ for the decomposition into irreducibles.
%  For each $i$ let $\mu_i\in X^*(\hat{T})^+=X_*(T)^+$ be the highest weight for $r_i$.
%  The multiplicity of the trivial character in
%  $r_i|_{\hat{T}}$ is denoted $m(\triv,r_i)$. Put $m(\triv,r)=\sum_{i=1}^s m(\triv,r_i)$.
%
%\begin{lem} Let $\phi=r^*(Y_1+\cdots+Y_d)$.
%\benu
%\item If $\mu_i\neq \sum_{\alpha^{\vee}\in (\Phi^\vee)^+} c_{\alpha^{\vee}} \alpha^\vee$
%for any collection of $c_{\alpha^{\vee}}\in \Z_{\ge 0}$ then $m(\triv,r_i)=0$.
%\item $\cS^{-1}(\phi)(1)=m(\triv,r)+O(q^{-1})$.
%(Note that $O(q^{-1})$ comes from $$\le \sum_{1\neq \lambda'\in X_*(T)^+}
%|\delta_B(\lambda'(\varpi))^{-1/2} \mu(K\lambda'(\varpi)K\cap NK)|$$
%where $\lambda'$ is $\mu$ minus a nonnegative real combination of $(\Phi^\vee_F)^+$.
%Use the inverse matrix for the Satake transform,
%cf. \cite[p.148]{Car79}, \cite{Rap00} - use his method to obtain a precise bound?
%In fact integral combination is enough according to Rapoport
%A similar consideration appears in Step 2 of Lemma \ref{l:dbl-coset-distance}:
%one needs $s$ for $K'\lambda K'\cap N \subset N_{\ol{x}_0',s}$ there.)
%\eenu
%\end{lem}

\subsection{Bounds for truncated unramified Hecke algebras}

  Let $F$, $G$, $A$, $T$ and $K$ be as in \S\ref{sub:Satake-trans}.
  Throughout this subsection,
  an $\R$-basis $\cB$ of $X_*(T)_{\R}$ will be fixed once and for all.
  Denote by $\rho\in X^*(T)\otimes_\Z \frac{1}{2}\Z$
  half the sum of all $\alpha\in \Phi^+$.

\begin{lem}\label{l:double-coset-volume}
  For any $\mu\in X_*(A)$,
  $[K\mu(\varpi)K:K]\le q^{d_G+r_G+\lg \rho,\mu\rg}$.
\end{lem}

\begin{proof}
  Let $\vol$ denote the volume for the Haar measure on $G(F)$ such that
  $\vol(K)=1$.
  Let $I\subset K$ be an Iwahori subgroup of $G(F)$.
  Then $I=(I\cap U)(I\cap T)(I\cap \ol{U})$.
  We follow the argument of \cite[pp.241-242]{Wal03}, freely using Waldspurger's notation.
  Our $I$, $U$, $\ol{U}$, and $T$ will play the roles of his $H$, $U_0$, $\ol{U}_0$ and $M_0$, respectively.
  For all $m\in \ol{M}_0^+$ (in his notation), it is not hard to verify that
  $c'_{U_0}(m)=c_{\ol{U}_0}(m)=c_{M_0}(m)=1$.
  Then Waldspurger's argument shows
  $$\vol(K\mu(\varpi)K)\le [K:I]^2 \vol(I\mu(\varpi)I)
  \le [K:I]^2 q^{\lg \rho,\mu\rg} \vol(I) = [K:I] q^{\lg \rho,\mu\rg} .$$
  Finally observe that $[K:I]\le |G(\F_q)|\le q^{d_G}(1+\frac{1}{q})^{r_G}\le q^{d_G+r_G}$.
(The middle inequality is easily derived from Steinberg's formula. cf. \cite[(3.1)]{Gro97}.)
\end{proof}

  The following lemma will play a role in studying the level aspect in Section \ref{s:aut-Plan-theorem}.

\begin{lem}\label{l:bounding-phi-on-S_0}
  Let $M$ be an $F$-rational Levi subgroup of $G$.
 There exists a constant $b_G> 0$ (depending only on $G$)
 such that for all $\kappa\in \Z_{>0}$
  and all $\phi\in \cH^{\ur}(G)^{\le \kappa,\cB}$ such that
  $|\phi|\le 1$, we have
  $|\phi_{M}(1)|= O(q^{d_G+r_G+b_G\kappa})$ (the implicit constant being independent of
  $\kappa$ and $\phi$), where we put $\phi_M:=\cS^G_M(\phi)$.
\end{lem}

\begin{proof}
  When $M=G$, the lemma is obvious (with $b_G=0$). Henceforth we assume
  that $M\subsetneq G$. In view of Lemma \ref{l:norm-and-bases},
  it suffices to treat one $\R$-basis $\cB$. Fix a $\Z$-basis $\{e_1,...,e_{\dim A}\}$ of $X_*(A)$,
  and choose any $\cB$ which extends that $\Z$-basis.
   It is possible to write
  $$\phi=\sum_{\|\mu\|\le \kappa} a_\mu\cdot \triv_{K\mu(\varpi)K}$$
  for $|a_\mu|\le 1$. Thus
  $$|\phi_{M}(1)|=\left|\int_{N(F)} \phi(n)dn\right|
  \le  \sum_{\|\mu\|\le \kappa}
    \left| \int_{N(F)}\triv_{K\mu(\varpi)K}(n)dn \right|.$$
  For each $\mu$, $K\mu(\varpi)K$ is partitioned into left $K$-cosets.
  On each coset $\gamma K$, $$\left|\int_{N(F)} \triv_{\gamma K}(n) dn\right|\le \vol(K
  \cap N(F))=1.$$ Hence, together with Lemma \ref{l:double-coset-volume},
  $$|\phi_{M}(1)|\le \sum_{\|\mu\|\le \kappa} [K\mu(\varpi)K:K]\le  \sum_{\|\mu\|\le \kappa}q^{d_G+r_G+\lg \rho,\mu\rg} .$$
  Write $b_{0}$ for the maximum of $|\lg \rho,e_i\rg|$ for $i=1,...,\dim A$.
  Take $b_G:=b_{0}\dim A+2 \dim A$.
  If $\|\mu\|\le \kappa$ then
 $\mu=\sum_{i=1}^{\dim A} a_ie_i$ for $a_i\in\Z$ with $-\kappa\le a_i\le \kappa$.
  Hence the right hand side is bounded by
  $(2\kappa+1)^{\dim A} q^{d_G+r_G+b_{0}\kappa\dim A}\le q^{d_G+r_G+b_G\kappa}$
  since $2\kappa+1\le 2^{2\kappa}\le q^{2\kappa}$.

% Let $\iota_M:{}^L M\ra {}^L G$ be a natural embedding (canonical up to $\hat{G}$-conjugacy).
%  The induced map $\iota^*_M:\cH(G)^{\ur}\ra \cH(M)^{\ur}$ is none other than the partial
%  Satake transform. Set $r_M:=r\circ \iota_M$. Then
%  $r_M^*=\iota_M^* r^*$. We see that
%  $\phi_{S_1,M}=\iota_M^*(\phi_{S_1})\in r_M^*(\cH^{\le \kappa})$.
%  Hence the above argument (with $\phi_{S_1,M}$ and $r_M$ in place of
%  $\phi_{S_1}$ and $r$) shows that $|\phi_{S_1,M}(1)|=O(q_{S_1}^{C_4})$.

\end{proof}

  An elementary matrix computation shows the bound below, which will be used several times.

\begin{lem}\label{l:control-eigenvalue}
  Let $s=\diag(s_1,...,s_m)\in \GL_m(\ol{F}_v)$
  and $u=(u_{ij})_{i,j=1}^{m}\in \GL_m(\ol{F}_v)$.
  Define $v_{\min}(u):=\min_{i,j} v(u_{ij})$ and
similarly $v_{\min}(u^{-1})$.
  Then for any eigenvalue $\lambda$ of $su\in \GL_m(F_v)$,
  $$v(\lambda)\in [v_{\min}(u)+\min_{i} v(s_i),-v_{\min}(u^{-1})+\max_{i} v(s_i)].$$
\end{lem}

\begin{rem}
 The lemma will be typically applied when $u\in \GL_m(\ol{\cO}_v)$ where
 $\ol{\cO}_v$ is the integer ring of $\ol{F}_v$. In this case
 $v_{\min}(u)=v_{\min}(u^{-1})=0$.
\end{rem}

\begin{proof}
 Let $V$ be the underlying $\ol{F}_v$-vector space with
 standard basis $\{e_1,...,e_m\}$.
  Let $\cB_j=\{\vec{i}=(i_1,...,i_j)|1\le i_1<\cdots<i_j\le m\}$.
   Then $\wedge^jV$ has a basis $\{e_{i_1}\wedge \cdots \wedge e_{i_j}\}_{\vec{i}\in \cB_j}$.
  We claim that $$v(\tr(su|\wedge^j V))\ge  j\cdot \min_{i} v(s_i).$$
  Let us verify this.
  Let $(u_{\vec{i},\vec{i}'})_{\vec{i},\vec{i}'\in \cB_j}$ denote the matrix entries for the $u$-action
  on $\wedge^j V$ with respect to the above basis.
  Observe that $v(u_{\vec{i},\vec{i}'})
  \ge j\cdot v_{\min}(u)$ for all $\vec{i},\vec{i}'\in \cB_j$.
  Then
  $$v(\tr(su|\wedge^j V))=v\left(\sum_{\vec{i}\in \cB_j} s_{i_1}s_{i_2}\cdots s_{i_j}\cdot u_{\vec{i},\vec{i}}\right)$$
  $$\ge \min_{\vec{i}} v(s_{i_1}s_{i_2}\cdots s_{i_j}\cdot u_{\vec{i},\vec{i}})
  \ge j\cdot\min_{i} v(s_i) + \min_{\vec{i}} v(u_{\vec{i},\vec{i}})
  \ge j(\min_{i} v(s_i) +v_{\min}(u)) .$$

  The coefficients of the characteristic polynomial for $su\in \GL_m(F_v)$ are given by
  $\tr(su|\wedge^j V)$ up to sign. The above claim and an elementary argument with the Newton polygon show that
  any root $\lambda$ satisfies $v(\lambda)\ge v_{\min}(u)+ \min_{i} v(s_i)$.
  Finally, applying the argument so far to $s^{-1}$ and $u^{-1}$, we obtain
  the upper bound for $v(\lambda)$.

\end{proof}

  As before, the smooth reductive model for $G$ over $\cO$ such that $G(\cO)=K$ will still be denoted $G$.

\begin{lem}\label{l:conj-image-in-diag}
  Let $\Xi:G\hra \GL_m$ be an embedding of algebraic groups over $\cO$. Then there exists a $GL_m(\cO)$-conjugate of $\Xi$ which maps $A$ (a fixed maximal split torus of $G$) into the diagonal maximal torus of $\GL_m$.
\end{lem}

\begin{proof}
  Note that the maximal $F$-split torus $A$ naturally extend to $A\subset G$ over $\cO$, cf. \cite[\S3.5]{Tit79}. The representation of $A$ on a free $\cO$-module of rank $m$ via $\Xi$ defines a weight decomposition of $\cO^m$ into free $\cO$-modules. Choose any refinement of the decomposition to write $\cO^m=L_1\oplus \cdots \oplus L_m$, as the direct sum of rank 1 free $\cO$-submodules. Let $v_i$ be an $\cO$-generator of $L_i$ for $1\le i\le m$. Conjugating $\Xi$ by the matrix representing the change of basis from $\{v_1,...,v_m\}$ to the standard basis for $\cO^m$, one can achieve that $\Xi(A)$ lies in the diagonal maximal torus.
\end{proof}

  Let $\gamma\in G(F)$ be a semisimple element and choose a maximal torus $T_\gamma$ of $G$ defined over $F$ such that $\gamma\in T_\gamma(F)$. Denote by $\Phi(G,T_\gamma)$ the set of roots for $T_\gamma$ in $G$.

\begin{lem}\label{l:bounding1-alpha(gamma)}
  Suppose that there exists an embedding of algebraic groups $\Xi:G\hra \GL_m$ over $\cO$.
  There exists a constant $B_5>0$ such that
  for every $\kappa\in \Z_{\ge 0}$, every $\mu\in X_*(A)$ satisfying $\|\mu\|\le \kappa$,
  every semisimple $\gamma\in K\mu(\varpi)K$ and every $\alpha\in \Phi_\gamma$ (for any choice of $T_\gamma$ as above), we have
  $-B_5\kappa\le v(\alpha(\gamma))\le B_5\kappa$. In particular,
  $|1-\alpha(\gamma)|\le q^{B_5 \kappa}$. % if $\alpha(\gamma)\neq 1$.
\end{lem}

\begin{rem}
  Later $\Xi$ will be provided by Proposition \ref{p:global-integral-model}.
\end{rem}

\begin{proof}
  We may assume that $\Xi(A)$ is contained in the diagonal torus of $\GL_m$, denoted by $\T$, thanks to Lemma \ref{l:conj-image-in-diag}. Write $T$ for the maximal torus of $G$ which is the centralizer of $A$ so that $\Xi(T)\subset \T$. We have a surjection $X^*(\T)\twoheadrightarrow X^*(T)$ induced by $\Xi$. For each $\alpha$ in the set of roots $\Phi(G,T)$, we fix a lift $\tilde\alpha\in X^*(\T)$ once and for all. Set $c_1:=\max_{\alpha\in \Phi(G,T)} \|\tilde{\alpha}\|_{\GL_m}$.

    Let $c_2:=\max_{\|\mu\|\le 1} \|\Xi\circ \mu\|_{\GL_m}$ where $\mu\in X_*(A)_{\R}$
  runs over elements with $\|\mu\|\le 1$. Then
  for any $\kappa\in\Z_{\ge 0}$, $\|\mu\|\le \kappa$ implies
  $\|\Xi\circ \mu\|_{\GL_m}\le c_2\kappa$. Hence $\Xi(\mu(\varpi))$ is a diagonal matrix in which
  each entry $x$ satisfies $-c_2\kappa \le v(x)\le c_2\kappa$.

  We can write $\gamma=k_1 \mu(\varpi)k_2$ for some $k_1,k_2\in G(\cO)$. Then
  $\Xi(\gamma)=k'_1 \Xi(\mu(\varpi)) k'_2$ for $k'_1,k'_2\in \GL_m(\cO)$, and
  $\Xi(\gamma)$ is conjugate to $\Xi(\mu(\varpi))k'_2(k'_1)^{-1}$. It follows from
  Lemma \ref{l:control-eigenvalue} that
  for every eigenvalue $\lambda$ of $\Xi(\gamma)$, we have
  $-c_2\kappa\le v(\lambda)\le c_2\kappa$.

  Choose any $T_\gamma$ as above. There exists an isomorphism $T\simeq T_\gamma$ over $\ol F$ induced by a conjugation action $t\mapsto g t g^{-1}$ given by some $g\in G(\ol F)$. The isomorphism is well defined only up to the Weyl group action but induces a bijection from $\Phi(G,T)$ onto $\Phi(G,T_\gamma)$. Put $\T_\gamma:=\Xi(g)\T \Xi(g)^{-1}$. The conjugation by $\Xi(g)$ induces an isomorphism $\T\simeq \T_\gamma$ over $\ol F$ and a bijection from $\Phi(\GL_m,\T)$ onto $\Phi(\GL_m,\T_\gamma)$. Let $\alpha_\gamma\in \Phi(G,T_\gamma)$ (resp. $\tilde \alpha_\gamma\in \Phi(\GL_m,\T_\gamma)$) denote the image of $\alpha$ (resp. $\tilde \alpha$) under the bijections. By construction, the composition $T_\gamma\simeq T\stackrel{\Xi}{\ra} \T\simeq \T_\gamma$ coincides with the restriction of $\Xi$ to $T_\gamma$. Hence the induced map $X^*(\T_\gamma)\ra X^*(T_\gamma)$ maps $\tilde\alpha_\gamma$ to $\alpha_\gamma$.

  Using the isomorphisms $\T_\gamma(\ol{F})\simeq \T(\ol{F})\simeq (\ol{F}^\times)^m$,  let $(\lambda_1,...,\lambda_m)\in (\ol{F}^\times)^m$
  be the image of $\Xi(\gamma)$ under the composition isomorphism. We may write $\tilde{\alpha}_\gamma$ as
  a character $(\ol{F}^\times)^m\ra \ol{F}^\times$ given by $(t_1,...,t_m)\mapsto t_1^{a_1}\cdots t_m^{a_m}$
  with $a_1,...,a_m\in \Z$ such that $-c_1\le a_i\le c_1$ for every $1\le i\le m$.
  We have $$\alpha_\gamma(\gamma)=\tilde{\alpha}_\gamma(\Xi(\gamma))=\lambda_1^{a_1}\cdots \lambda_m^{a_m},$$
  so $v(\alpha_\gamma(\gamma))=\sum_{i=1}^m a_i v(\lambda_i)$. Hence $-m c_1c_2\kappa\le v(\alpha_\gamma(\gamma))\le m c_1c_2\kappa$,
proving the first assertion of the lemma. From this the last assertion is obvious.

\end{proof}

\begin{rem}\label{r:Lem2.18-indep}
  Suppose that $F$ runs over the completions of a number field $\bF$ at non-archimedean places $v$, that $G$ over $F$ comes from a fixed reductive group $\bG$ over $\bF$, and that $\Xi$ comes from an embedding $\bG\hra GL_m$ over the integer ring of $\bF$ (at least for every $v$ where $\bG$ is unramified). Then $B_5$ of the lemma can be chosen to be independent of $v$ (and dependent only on the data over $\bF$). This is easy to see from the proof.
\end{rem}

\subsection{A lemma on semisimple elements}\label{sub:easy}

  As before let $F$ be a finite extension of $\Q_p$ with multiplicative norm $|\cdot|:F^\times\ra \R^\times_{>0}$ normalized such that a uniformizer is sent to the inverse of the residue field cardinality, and let $G$ be an unramified group over $F$ with a smooth reductive model over $\cO$. The notation for $T_\gamma$ and $\Phi_\gamma$ for a semisimple $\gamma\in G(F)$ is as in the previous subsection.

\begin{lem}\label{l:alpha(gamma)-is-integral}
  Suppose that a semisimple $\gamma\in G(F)$ is conjugate to an element
  of $G(\cO)$ and that $\alpha(\gamma)\neq 1$ and
  $|1-\alpha(\gamma)|\neq 1$ for $\alpha\in \Phi_\gamma$ (for any choice of $T_\gamma$).
  Then $|1-\alpha(\gamma)|\le q^{-1}$.
\end{lem}

\begin{proof}

  By the assumption we may assume $\gamma\in G(\cO)$.
  Choose a maximal torus $T$ in the centralizer of $\gamma$ in $G$ over $\cO_{F^{\ur}}$, the ring of integers in
  $F^{\ur}$, so that $\gamma\in T(\cO)$.
  Such a $T$ exists since a reductive group scheme admits a maximal torus under \'etale localization (\cite[Cor 3.2.7]{Conrad-reductive}).
  Since $T$ should split over a finite \'etale cover of $\Spec \cO_{F^{\ur}}$, it should be that $T$ splits over $\cO_{F^{\ur}}$ already. Then $\alpha$ defines a character $T(\cO_{F^{\ur}})\ra \cO_{F^{\ur}}^{\times}$ and
  $1-\alpha(\gamma)\in \cO_{F^{\ur}}$. As $1-\alpha(\gamma)$ is not a $\varpi$-adic unit,
  $\varpi$ divides $1-\alpha(\gamma)$. The lemma follows.
\end{proof}

%\begin{cor}\label{c:alpha(gamma)-is-integral}
%  Let $\alpha\in \Phi$ and $\gamma$ be as in Lemma \ref{l:alpha(gamma)-is-integral}.
%  If $\alpha(\gamma)\neq 1$ and
%  $|1-\alpha(\gamma)|\neq 1$ then $|1-\alpha(\gamma)|\le q^{-1/w'_G}$.
%\end{cor}
%
%\begin{proof}
%  Under the assumption it is clear that $|1-\alpha(\gamma)|\le |\varpi'|$
%  where $\varpi'$ is a uniformizer of $F'$ as in the preceding lemma.
%  Note that $|\cdot|$ assigns $q^{-1}$ to a uniformizer of $F$.
%  The ramification index $e(F'/F)$ is at most $[F':F]\le w'_G$, so
%  $|\varpi'|=(q^{-1})^{1/e(F'/F)}\le q^{-1/w'_G}$.
%\end{proof}
%

\section{Plancherel measure on the unramified spectrum}\label{s:Plancherel}

\subsection{Basic setup and notation}\label{sub:notation-Plancherel}
  Let $F$ be a finite extension of $\Q_p$.
  Suppose that $G$ is unramified over $F$.
  Fix a hyperspecial subgroup $K$ of $G$.
  Recall the notation from the start of \S\ref{sub:Satake-trans}.
  In particular $\Omega$ (resp. $\Omega_F$) denotes the Weyl group for $(G_{\ol{F}},T_{\ol{F}})$
  (resp. $(G,A)$). There is a natural $\Gal(\ol{F}/F)$-action on $\Omega$,
  under which $\Omega^{\Gal(\ol{F}/F)}=\Omega_F$. (See \cite[\S6.1]{Bor79}.)
  Since $G$ is unramified, $\Gal(\ol{F}/F)$ factors through a finite unramified Galois group.
  Thus there is a well-defined action of $\Fr$ on $\Omega$, and
  $\Omega^{\Fr}=\Omega_F$.

The unitary dual $G(F)^{\wedge}$ of $G(F)$, or simply $G^\wedge$ if there is no danger of ambiguity, is equipped with Fell topology.
  (This notation should not be confused with the dual group $\hat{G}$).
  Let $G^{\wedge,\ur}$ denote the unramified spectrum in
  $G^{\wedge}$, and
 $G^{\wedge,\ur,\temp}$ its tempered sub-spectrum.
 The Plancherel measure $\pl$ on $G^{\wedge}$ is supported on the tempered spectrum $G^{\wedge,\temp}$.
 The restriction of
 $\pl$ to $G^{\wedge,\ur}$ will be written as $\plur$.
 The latter is supported on $G^{\wedge,\ur,\temp}$.
  Harish-Chandra's Plancherel formula (cf. \cite{Wal03}) tells us that
  $\pl(\hat{\phi})=\phi(1)$ for all $\phi\in \cH(G(F))$. In particular,
  $\plur(\hat{\phi})=\phi(1)$ for all $\phi\in \cH^{\ur}(G(F))$.

  \subsection{The unramified tempered spectrum}\label{s:unramified-spectrum}

 % EXPLAIN GENERAL UNRAMIFIED CASE.
 An unramified $L$-parameter $W^{\ur}_F\ra {}^L G^{\ur}$
  is defined to be an $L$-morphism ${}^L H^{\ur}\ra {}^L G^{\ur}$ (\S\ref{sub:L-mor-unr-Hecke})
  with $H=\{1\}$.
  Two such parameters $\varphi_1$ and $\varphi_2$ are considered equivalent
  if $\varphi_1=g\varphi_2 g^{-1}$ for some $g\in \hat{G}$.
  Consider the following sets:
%  Some of them require explanation.

\benu
\item irreducible unramified representations $\pi$ of $G(F)$ up to isomorphism.
\item group homomorphisms $\chi:T(F)/T(F)\cap K \ra \C^\times$
up to $\Omega_{F}$-action.
\item unramified $L$-parameters $\varphi:W^{\ur}_F\ra {}^L G^{\ur}$
up to equivalence.
\item elements of $(\hat{G}\rtimes \Fr)_{\ssconj}$;
  this set was defined in \S\ref{sub:Satake-trans}.
\item $\Omega^{\Fr}$-orbits in $\hat{T}/(\Fr-\id)\hat{T}$.
\item $\Omega_F$-orbits in $\hat{A}$.
\item $\C$-algebra morphisms $\theta:\cH^{\ur}(G)\ra \C$.

\eenu
  Let us describe canonical maps among them in some directions.
\bit
\item[(i)$\rightarrow$(vii)] Choose any $0\neq v\in \pi^K$.
Define $\theta(\phi)$ by $\theta(\phi)v=\int_{G(F)} \phi(g)\pi(g)vdg$.
\item[(ii)$\rightarrow$(i)]
$\pi$ is the unique unramified subquotient of $\nind^{G(F)}_{B(F)} \chi$.
\item[(ii)$\leftrightarrow$(vi)] Induced by
$\Hom(T(F)/T(F)\cap K , \C^\times)\simeq
\Hom(A(F)/A(F)\cap K , \C^\times)$
\beq\label{e:ii-vi}\simeq \Hom(X_*(A),\C^\times)\simeq\Hom(X^*(\hat{A}),\C^\times)\simeq
 X_*(\hat{A})\otimes_\Z \C^\times\simeq \hat{A}\eeq
 where the second isomorphism is induced by $X_*(A)\ra A(F)$ sending $\mu$ to $\mu(\varpi)$.
\item[(iii)$\rightarrow$(iv)] Take $\varphi(\Fr)$.
\item[(v)$\rightarrow$(iv)] Induced by
the inclusion $t\mapsto t\rtimes \Fr$ from $\hat{T}$ to $ \hat{G}\rtimes\Fr$.

\item[(v)$\rightarrow$(vi)] Induced by the surjection $\hat{T}\twoheadrightarrow \hat{A}$,
which is the dual of $A\hra T$. (Recall $\Omega^{\Fr}=\Omega_F$.)
\item[(vii)$\rightarrow$(vi)]
  Via $\cS:\cH^{\ur}(G)\simeq \C[X^*(\hat{A})]^{\Omega_F}$,
  $\theta$ determines an element of (cf. \eqref{e:ii-vi})
  $$\Omega_F\bs\Hom(X^*(\hat{A}),\C^\times)\simeq \Omega_F\bs \hat{A}.$$
\eit

\begin{lem}\label{l:unr-spec}
  Under the above maps, the sets corresponding to (i)-(vii) are in bijection with each other.
\end{lem}
\begin{proof}
  See \S6, \S7 and \S10.4 of \cite{Bor79}.
\end{proof}

  Let $F'$ be the finite unramified extension of $F$ such that
  $\Gal(\ol{F}/F)$ acts on $\hat{G}$ through the faithful action of $\Gal(F'/F)$.
  Write ${}^L G_{F'/F}:=\hat{G}\rtimes\Gal(F'/F)$.
    Let $\hat{K}$ be a maximal compact subgroup of $\hat{G}$ which is $\Fr$-invariant.
  Denote by $\hat{T}_c$ (resp. $\hat{A}_c$) the maximal compact subtorus
  of $\hat{T}$ (resp. $\hat{A}$).

\begin{lem}\label{l:unr-temp-spec}
  The above bijections restrict to the bijections among the sets consisting
  of the following objects.
\benu
\item[(i)$_t$] irreducible unramified tempered representations $\pi$ of $G(F)$
up to isomorphism.
\item[(ii)$_t$] unitary group homomorphisms $\chi:T(F)/T(F)\cap K \ra U(1)$ up to $\Omega_{F}$-action.
\item[(iii)$_t$] unramified $L$-parameters $\varphi:W^{\ur}_F\ra {}^L G^{\ur}$ with bounded image
up to equivalence.
\item[(iv)$_t$] $\hat{G}$-conjugacy classes in $\hat{K}\rtimes \Fr$ (viewed in ${}^L G_{F'/F}$).
\item[(iv)$'_t$] $\hat{K}$-conjugacy classes in $\hat{K}\rtimes \Fr$ (viewed in $\hat{K}\rtimes\Gal(F'/F)$).
\item[(v)$_t$] $\Omega^{\Fr}$-orbits in $\hat{T}_c/(\Fr-\id)\hat{T}_c$.
\item[(vi)$_t$] $\Omega_F$-orbits in $\hat{A}_c$.
\eenu
(The boundedness in (iii)$_t$ means that the projection of $\im \varphi$ into
${}^L G_{F'/F}$ is contained in a maximal compact subgroup of ${}^L G_{F'/F}$.)
\end{lem}

\begin{proof}
  (i)$_t$$\leftrightarrow$(ii)$_t$ is standard and (iii)$_t$$\leftrightarrow$(iv)$_t$ is obvious.
  Also straightforward is (ii)$_t$$\leftrightarrow$(vi)$_t$ in view of \eqref{e:ii-vi}.

  Let us show that (v)$_t$$\leftrightarrow$(vi)$_t$. Choose a topological isomorphism of complex tori
  $\hat{T}\simeq (\C^\times)^{d}$ with $d=\dim T$.
  Using $\C^\times \simeq U(1)\times \R^\times_{>0}$, we can decompose $\hat{T}=\hat{T}_c\times \hat{T}_{nc}$
  such that $\hat{T}_{nc}$ is carried over to $(\R^\times_{>0})^d$ under the isomorphism.
  The decomposition of $\hat{T}$ is canonical in that it is preserved under
  any automorphism of $\hat{T}$. By the same reasoning, there is a canonical decomposition
  $\hat{A}=\hat{A}_c\times \hat{A}_{nc}$ with $\hat{A}_{nc}\simeq (\R^\times_{>0})^{\dim A}$.
  The canonical surjection $\hat{T}\ra \hat{A}$ carries $\hat{T}_c$ onto $\hat{A}_c$
  and $\hat{T}_{nc}$ onto $\hat{A}_{nc}$. (This reduces to the assertion in the case of $\C^\times$,
  namely that any maps $U(1)\ra \R^\times_{>0}$ and $\R^\times_{>0}\ra U(1)$ induced by
  an algebraic map $\C^\times\ra \C^\times$ of $\C$-tori are trivial. This is easy to check.)
  Therefore the isomorphism $\hat{T}/(\Fr-\id)\hat{T}\ra \hat{A}$ of Lemma \ref{l:unr-temp-spec}
  induces an isomorphism
  $\hat{T}_c/(\Fr-\id)\hat{T}_c\ra \hat{A}_c$ (as well as
  $\hat{T}_{nc}/(\Fr-\id)\hat{T}_{nc}\ra \hat{A}_{nc}$).

  Next we show that (iv)$_t$$\leftrightarrow$(v)$_t$.
  It is clear that $t\mapsto t\rtimes \Fr$ maps (v)$_t$ into (iv)$_t$. Since (v)$_t$ and (iv)$_t$ are
  the subsets of (v) and (iv), which are in bijective correspondence,
  we deduce that (v)$_t$$\ra$(iv)$_t$ is injective. To show surjectivity, pick any
  $k\in \hat{K}$. There exists $t\in \hat{T}$ such that the image of $t$ in (iv) corresponds
  under (iv)$\leftrightarrow$(v) to the $\hat{G}$-conjugacy class of $\hat{k}\rtimes \Fr$.
  It is enough to show that we can choose $t\in \hat{T}_c$.
  Consider the subgroup $\hat{T}_c(t)$ of $$\hat{T}/(\Fr-\id)\hat{T}
  = \hat{T}_c/(\Fr-\id)\hat{T}_c~ \times ~\hat{T}_{nc}/(\Fr-\id)\hat{T}_{nc}$$ generated by $\hat{T}_c/(\Fr-\id)\hat{T}_c$
  and the image of $t$.
  The isomorphism (iv)$\leftrightarrow$(v) maps $\hat{T}_c(t)$ into (v)$_t$ by
  the assumption on $t$. Since (v)$_t$ form a compact set, the group
  $\hat{T}_c(t)$ must be contained in a compact subset of $\hat{T}/(\Fr-\id)\hat{T}$.
  This forces the image of $t$ in $\hat{T}_{nc}/(\Fr-\id)\hat{T}_{nc}$ to be trivial.
  (Indeed, the latter quotient is isomorphic as a topological group
  to a quotient of $\R^{\dim T}$ modulo an $\R$-subspace via the exponential map. So
  any subgroup generated by a nontrivial element is not contained in a compact set.)
  Therefore $t$ can be chosen in $\hat{T}_c$.

    It remains to verify that (iv)$_t$, (iv)$'_t$ and (v)$_t$ are in bijection.
    Clearly (iv)$'_t$$\ra$(iv)$_t$ is onto. As we have just seen that
    (iv)$_t$$\leftrightarrow$(v)$_t$, it suffices to observe that (v)$_t$$\ra$(iv)$'_t$ is onto,
   which is a standard fact (for instance in the context of
    the (twisted) Weyl integration formula for $\hat{K}\rtimes \Fr$).

%    This amounts to proving that the continuous map
%$$f:\hat{K}\times (\hat{T}_c\rtimes \sigma) \ra \hat{K}\rtimes \sigma,
%\qquad (k,t\sigma)\mapsto kt\sigma k^{-1}$$
%  is onto. The image of $f$ is closed as the source is compact. On the other hand,

\end{proof}

\subsection{Plancherel measure on the unramified spectrum}\label{sub:plan-unramified}

 Lemma \ref{l:unr-temp-spec} provides a bijection $G^{\wedge,\ur,\temp}\simeq
 \Omega_F\bs \hat{A}_c$, which is in fact a topological isomorphism.
 The Plancherel measure $\plur$ on $G^{\wedge,\ur}$ is supported on $G^{\wedge,\ur,\temp}$.
 We would like to describe its pullback measure on $\hat{A}_c$,
 to be denoted $\plurtemp_{0}$. Note that $\hat{A}_c$ is topologically
 isomorphic to $\hat{T}_c/(\Fr-\id)\hat{T}_c$. (This is induced by
 the natural surjection $\hat{T}_c \twoheadrightarrow\hat{A}_c$.)
  Fix a measure $d\ol{t}$ on the latter which is a push forward from a Haar measure on $\hat{T}_c$.

\begin{prop}\label{p:unr-plan-meas}

  The measure $\plurtemp_{0}$ pulled back to $\hat{T}_c/(\Fr-\id)\hat{T}_c$ is
  $$\plurtemp_{0}(\ol{t})= C\cdot
  \frac{\det(1-\ad(t\rtimes \Fr)|\Lie(\hat{G})/\Lie(\hat{T}^\Fr))}
  {\det(1-q^{-1}\ad(t\rtimes \Fr)|\Lie(\hat{G})/\Lie(\hat{T}^\Fr))}
   d\ol{t}
  $$ for some constant $C\in \C^\times$,
  depending on the normalization of Haar measures.
  Here $t\in \hat{T}_c$ is any lift of $\ol{t}$. (The
  right hand side is independent
  of the choice of $t$.)
\end{prop}

\begin{proof}
  The formula is due to Macdonald (\cite{Mac71}). For our purpose, it is more convenient to
  follow the formulation as in the conjecture of \cite[p.281]{Sha90} (which also discusses the general conjectural formula
  of the Plancherel measure due to Langlands). By that conjecture (known in the unramified case),
  $$\plurtemp_{0}(\ol{t})= C'\cdot \frac{L(1,\sigma^{-1}(\ol{t}),r)}{L(0,\sigma(\ol{t}),r)}
  \frac{L(1,\sigma(\ol{t}),r)}{L(0,\sigma^{-1}(\ol{t}),r)} d\ol{t}$$
  where $C'\in \C^\times$ is a constant,
  $\sigma(\ol{t}):T(F)\ra \C^\times$ is the character corresponding to $\ol{t}$
  (via (ii)$\leftrightarrow$(v) of Lemma \ref{l:unr-spec}), and $r:{}^L T
  \ra \GL(\Lie({}^L U))$ is the adjoint representation. Here ${}^L U$ is the $L$-group of $U$
  (viewed in ${}^L B$). By unraveling the local $L$-factors, obtain
   \beq\label{e:pf-unr-plan}\plurtemp_{0}(\ol{t})= C'\cdot
  \frac{\det(1-\ad(t\rtimes \Fr)|\Lie(\hat{G})/\Lie(\hat{T}))}
  {\det(1-q^{-1}\ad(t\rtimes \Fr)|\Lie(\hat{G})/\Lie(\hat{T}))}
   d\ol{t}. \eeq
   Finally, observe that $\det(1-q^{-s}\ad(t\rtimes \Fr)|\Lie(\hat{T})/\Lie(\hat{T}^{\Fr}))$
   is independent of $\ol{t}$ (and $t$). Therefore the right hand sides are
   the same up to constant in \eqref{e:pf-unr-plan} and the proposition.
\end{proof}

\begin{rem}
  Note that the choice of a Haar measure on $G(F)$ determines the measure $\plurtemp_{0}$. For example
  if the Haar measure on $G(F)$ assigns volume 1 to $K$ then
  $G^{\wedge,\ur,\temp}$ has total volume 1
  with respect to $\plurtemp_{0}(\ol{t})$ as implied by the Plancherel formula for $\triv_K$.
  Hence the product $C\cdot d\ol{t}$.
\end{rem}

\section{Automorphic L-functions}\label{sec:pp}

According to Langlands conjectures, the most general $L$-functions should be expressible as products of the principal $L$-functions $L(s,\Pi)$ associated to cuspidal automorphic representations $\Pi$ of $\GL(d)$ over number fields (for varying $d$). The analytic properties and functional equation of such $L$-functions were first established by Godement--Jacquet for general $d\ge 1$. This involves the Godement--Jacquet integral representation. The other known methods are the Rankin--Selberg integrals, the doubling method and the Langlands--Shahidi method. The purpose of this section is to recall these analytic properties and to set-up notation. More detailed discussions may be found in~\cites{cong:Lfunc04:cogd,Jacq79:principal,cong:park:mich},~\cite{RS96}*{\S2} and~\cite{book:IK04}*{\S5}.

In this section and some of the later sections we use the following notation.
\bit
\item $F$ is a number field, i.e. a finite extension of $\Q$.
\item $G$ is a connected reductive group over $F$ (not assumed to be quasi-split).
\item $Z=Z(G)$ is the center of $G$.
\item $\cV_F$ (resp. $\cV_F^\infty$) is the set of all (resp. all finite) places of $F$.
\item $S_\infty:=\cV_F\bs \cV_F^\infty$.
\item $A_{G}$ is the maximal $F$-split subtorus in the center of $\Res_{F/\Q} G$, and $A_{G,\infty}:=A_G(\R)^0$.
\eit

\subsection{Automorphic forms}\label{sec:pp:autforms}
  Let $\chi:A_{G,\infty}\ra \C^\times$ be a continuous homomorphism.
  Denote by $L^2_\chi(G(F)\bs G(\A_F))$
  the space of all functions $f$ on $G(\A_F)$ which are
  square-integrable modulo $A_{G,\infty}$ and satisfy
  $f(g\gamma z)=\chi(z)f(\gamma)$ for all $g\in G(F)$, $\gamma \in G(\A_F)$ and
  $z\in A_{G,\infty}$. There is a spectral decomposition into discrete and continuous parts
$$L^2_{\chi}(G(F)\bs G(\A_F))=L^2_{\disc,\chi}\oplus L^2_{\cont,\chi},
\qquad L^2_{\disc,\chi}=\hat{\bigoplus_{\pi}}\, m_{\disc,\chi}(\pi)\cdot \pi$$
where the last sum is a Hilbert direct sum running over the set of
all irreducible representations
of $G(\A_F)$ up to isomorphism.
Write $\cAR_{\disc,\chi}(G)$ for the set of isomorphism classes of
all irreducible representations $\pi$
of $G(\A_F)$ such that $m_{\disc,\chi}(\pi)>0$.
Any $\pi\in \cAR_{\disc,\chi}(G)$ is said to
be a discrete automorphic representation of $G(\A_F)$. If $\chi$ is trivial (in particular if $A_{G,\infty}=\{1\}$) then we write $m_{\disc}$ for $m_{\disc,\chi}$.

The above definitions allow a modest generalization. Let $\mathfrak{X}_G$ be a closed subgroup of $Z(\A_F)$ containing $A_{G,\infty}$ and $\omega:Z(\A_F)\cap \mathfrak{X}_G\bs \mathfrak{X}_G\ra \C^\times$ be a continuous (quasi-)character. Then $L^2_{\omega}$, $L^2_{\disc,\omega}$, $m_{\disc,\omega}$ etc can be defined analogously. In fact the Arthur-Selberg trace formula applies to this setting. (See \cite[Ch 3.1]{Arthur}.)

For the rest of Section~\ref{sec:pp} we are concerned with $G=\GL(d)$. Take $\mathfrak{X}_G=Z(\A_F)$ so that $\omega$ is a quasi-character of $Z(F)\SB Z(\BmA_F)$. Note that $A_{G,\infty}=Z(F_\infty)^\circ$ in this case. We denote by $\CmA_\omega(\GL(d,F))$ the space consisting of automorphic functions on $\GL(d,F)\SB \GL(d,\BmA_F)$ which satisfy $f(zg)=\omega(z)f(g)$ for all $z\in Z(\BmA_F)$ and $g\in \GL(d,\BmA_F)$ (see Borel-Jacquet~\cite{BJ79} for the exact definition and the growth condition). We denote by $\CmA_{\cusp,\omega}(\GL(d,F))$ the subspace of cuspidal functions (i.e. the functions with vanishing period against all nontrivial unipotent subgroups).

An automorphic representation $\Pi$ of $\GL(d,\BmA_F)$ is by definition an irreducible admissible representation of $\GL(d,\BmA_F)$ which is a constituent of the regular representation on $\CmA_\omega(\GL(d,F))$. Then $\omega$ is the central character of $\Pi$. The subspace $\CmA_{\cusp,\omega}(\GL(d,F))$ decomposes discretely and an irreducible component is a cuspidal automorphic representation. The notion of cuspidal automorphic representations is the same if the space of cuspidal functions in $L^2_{\omega}(GL(d,F)\bs GL(d,\A_F))$ is used in the definition in place of $\CmA_{\cusp,\omega}(\GL(d,F))$, cf. \cite[\S4.6]{BJ79}.

When $\omega$ is unitary we can work with the completed space $L_\omega^2(\GL(d,F)\SB \GL(d,\BmA_F))$ of square-integrable functions modulo $Z(\BmA_F)$ and with unitary automorphic representations. Note that a cuspidal automorphic representation is unitary if and only if its central character is unitary. We recall the Langlands decomposition of $L_\omega^2(\GL(d,F)\SB \GL(d,\BmA_F))$ into the cuspidal, residual and continuous spectra. What will be important in the sequel is the notion of isobaric representations which we review in~\S\ref{sec:pp:isobaric}.

In the context of $L$-functions, the functional equation involves the contragredient representation $\widetilde \Pi$. An important fact is that the contragredient of a unitary automorphic representation of $\GL(d,\BmA_F)$ is isomorphic to its complex conjugate.

\subsection{Principal $L$-functions}\label{sec:pp:Lfn}
 Let $\Pi=\otimes_v \Pi_v$ be a cuspidal automorphic representation of $\GL(d,\BmA_F)$ with unitary central character. The principal $L$-function associated to $\Pi$ is denoted
\begin{equation*}
L(s,\Pi)=\prod_{v\in \CmV_F^\infty } L(s,\Pi_v).
\end{equation*}
The Euler product is absolutely convergent when $\MRe s>1$. The completed $L$-function is denoted $\Lambda(s,\Pi)$, the product now running over all places $v\in\CmV_F$.
For each finite place $v\in \CmV_F^\infty$, the inverse of the local $L$-function $L(s,\Pi_v)$ is a Dirichlet polynomial in $q_v^{-s}$ of degree $\le d$. Write
\begin{equation*}
L(s,\Pi_v)=\prod^d_{i=1} (1-\alpha_i(\Pi_v)q_v^{-s})^{-1}.
\end{equation*}
Note that when $\Pi_v$ is unramified, $\alpha_i(\Pi_v)$ is non-zero for all $i$ and corresponds to the eigenvalues of a semisimple conjugacy class in $\GL_d(\BmC)$ associated to $\Pi_v$, but when $\Pi_v$ is ramified the Langlands parameters are more sophisticated and we allow some (or even all of) of the $\alpha_i(\Pi_v)$ to be equal to zero. In this way we have a convenient notation for all local $L$-factors.

For each archimedean $v$, the local $L$-function $L(s,\Pi_v)$ is a product of $d$ Gamma factors
\begin{equation}\label{def:Lv-arch}
L(s,\Pi_v)=\prod^d_{i=1} \Gamma_v(s-\mu_i(\Pi_v)),
\end{equation}
where $\Gamma_\BmR(s):=\pi^{-s/2}\Gamma(s/2)$ and $\Gamma_\BmC(s):=2(2\pi)^{-s}\Gamma(s)$. Note that $\Gamma_{\BmC}(s)=\Gamma_\BmR(s)\Gamma_\BmR(s+1)$ by the doubling formula, so when $v$ is complex, $L(s,\Pi_v)$ may as well be expressed as a product of $2d$ $\Gamma_\BmR$ factors.

The completed $L$-function $\Lambda(s,\Pi):=L(s,\Pi)\prod_{v|\infty}L(s,\Pi_v)$ has the following analytic properties. It has a meromorphic continuation to the complex plane. It is entire except when $d=1$ and $\Pi=\abs{.}^{it}$ for some $t\in \BmR$, in which case $L(s,\Pi)=\zeta_F(s+it)$ is (a shift of) the Dedekind zeta function of the ground field $F$ with simple poles at $s=-it$ and $s=1-it$. It is bounded in vertical strips and satisfies the functional equation
\begin{equation}\label{fneq}
\Lambda(s,\Pi)=\epsilon(s,\Pi) \Lambda(1-s,\widetilde \Pi),
\end{equation}
where $\epsilon(s,\Pi)$ is the epsilon factor and $\widetilde \Pi$ is the contragredient automorphic representation. The epsilon factor has the form
\begin{equation}
\epsilon(s,\Pi)=\epsilon(\Pi) q(\Pi)^{\Mdemi-s}
\end{equation}
for some positive integer $q(\Pi)\in \BmZ_{\ge 1}$ and root number $\epsilon(\Pi)$ of modulus one.

Note that $q(\Pi)=q(\widetilde \Pi)$,
$\epsilon(\widetilde \Pi)=\overline{\epsilon(\Pi)}$ and for all $v\in \CmV_F$, $L(s,\widetilde \Pi_v)=\overline{L(\overline{s},\Pi_v)}$. For instance this follows from the fact~\cite{GK75} that $\widetilde \Pi$ is isomorphic to the complex conjugate $\overline\Pi$ (obtained by taking the complex conjugate of all forms in the vector space associated to the representation $\Pi$).

The conductor $q(\Pi)$ is the product over all finite places $v\in\CmV^\infty_F$ of the conductor $q(\Pi_v)$ of $\Pi_v$. Recall that $q(\Pi_v)$ equals one whenever $\Pi_v$ is unramified. It is convenient to introduce as well the conductor of admissible representations at archimedean places. When $v$ is real we let $C(\Pi_v)=\prod\limits^d_{i=1} (2+\abs{\mu_i(\Pi_v)})$ and when $v$ is complex we let $C(\Pi_v)=\prod\limits^d_{i=1} (2+\abs{\mu_i(\Pi_v)}^2)$. Then we let $C(\Pi)$ be the analytic conductor which is the product of all the local conductors
\begin{equation*}
C(\Pi):=
\prod_{v\mid \infty} C(\Pi_v)
\prod_{v\in\CmV^\infty_F} q(\Pi_v) = C(\Pi_\infty) q(\Pi).
\end{equation*}
Note that $C(\Pi)\ge 2$ always.

There is $0\le \theta <\frac{1}{2}$ such that
\begin{equation}\label{towardsR}
\MRe \mu_i(\Pi_v) \le \theta
,\qtext{resp.}\
\log_{q_v} \abs{\alpha_i(\Pi_v)}\le \theta
\end{equation}
for all archimedean $v$ (resp. finite $v$) and $1\le i\le d$. When $\Pi_v$ is unramified we ask for
\begin{equation}\label{unrtowardsR}
\abs{\MRe \mu_i(\Pi_v)} \le \theta
,\qtext{resp.}\
\abs{\log_{q_v} \abs{\alpha_i(\Pi_v)}}\le \theta.
\end{equation}
The value $\theta=\Mdemi - \frac{1}{d^2+1}$ is admissible by an argument of Serre and Luo--Rudnick--Sarnak based on the analytic properties of the Rankin-Selberg convolution $L(s,\Pi\times \widetilde \Pi)$. Note that for all $v$, the local $L$-functions $L(s,\Pi_v)$ are entire on $\MRe s>\theta$ and this contains the central line $\MRe s=\Mdemi$.

The generalized Ramanujan conjecture asserts that all $\Pi_v$ are tempered (see~\cite{Sarnak:GRC} and the references herein). This is equivalent to having $\theta=0$ in the inequalities~\eqref{towardsR} and~\eqref{unrtowardsR}. In particular we expect that when $\Pi_v$ is unramified, $\abs{\alpha_i(\Pi_v)}=1$.

\subsection{Isobaric sums}\label{sec:pp:isobaric} We need to consider slightly more general $L$-functions associated to non-cuspidal automorphic representations on $\GL(d,\BmA_F)$. These $L$-functions are products of the $L$-functions associated to cuspidal representations and studied in the previous \S\ref{sec:pp:Lfn}. Closely related to this construction it is useful to introduce, following Langlands~\cite{cong:auto77:lang}, the notion of isobaric sums of automorphic representations. The concept of isobaric representations is natural in the context of $L$-functions and the Langlands functoriality conjectures.

Let $\Pi$ be an irreducible automorphic representation of $\GL(d,\BmA_F)$. Then a theorem of Langlands~\cite{BJ79} states that there are integers $r\ge 1$ and $d_1,\cdots,d_r\ge 1$ with $d=d_1+\cdots +d_r$ and cuspidal automorphic representations $\Pi_1,\cdots,\Pi_r$ of $\GL(d_1,\BmA_F),\cdots,\GL(d_r,\BmA_F)$  such that $\Pi$ is a constituent of the induced representation of $\Pi_1\otimes \cdots \otimes \Pi_r$ (from the Levi subgroup $\GL(d_1)\times \cdots \times \GL(d_r)$ of $\GL(d)$). A cuspidal representation is unitary when its central character is unitary. When all of $\Pi_j$ are unitary then $\Pi$ is unitary. But the converse is not true: note that even if $\Pi$ is unitary, the representation $\Pi_j$ need not be unitary in general.

We recall the generalized strong multiplicity one theorem of Jacquet and Shalika~\cite{JSI-II}. Suppose $\Pi$ and $\Pi'$ are irreducible automorphic representations of $\GL(d,\BmA_F)$  such that $\Pi_v$ is isomorphic to $\Pi'_v$ for almost all $v\in\CmV_F$ (we say that $\Pi$ and $\Pi'$ are weakly equivalent) and suppose that $\Pi$ (resp. $\Pi'$) is a constituent of the induced representation of $\Pi_1\otimes\cdots \otimes \Pi_r$ (resp. $\Pi'_1\otimes \cdots \otimes \Pi'_{r'}$). Then $r=r'$ and up to permutation the sets of cuspidal representations $\set{\Pi_j}$ and $\set{\Pi'_j}$ coincide. Note that this generalizes the strong multiplicity one theorem of Piatetski-Shapiro which corresponds to the case where $\Pi$ and $\Pi'$ are cuspidal.

Conversely suppose $\Pi_1,\cdots, \Pi_r$ are cuspidal representations of $\GL(d_1,\BmA_F), \cdots, \GL(d_r,\BmA_F)$. Then from the theory of Eisenstein series there is a unique constituent of the induced representation of $\Pi_1\otimes\cdots \otimes \Pi_r$ whose local components coincide at each place $v\in \CmV_F$ with the Langlands quotient of the local induced representation~\cite{cong:auto77:lang}*{\S2}. It is denoted $\Pi_1\boxplus \cdots \boxplus \Pi_r$ and called an isobaric representation (automorphic representations which are not isobaric are called anomalous). The above results of Langlands and Jacquet--Shalika may now be summarized by saying that an irreducible automorphic representation of $\GL(d,\BmA_F)$ is weakly equivalent to a unique isobaric representation.

We now turn to $L$-functions. The completed $L$-function associated to an isobaric representation $\Pi=\Pi_1\boxplus \cdots\boxplus \Pi_r$ is by definition
\begin{equation*}
\Lambda(s,\Pi) = \prod^{r}_{j=1} \Lambda(s,\Pi_j).
\end{equation*}
All notation from the previous subsection will carry over to $\Lambda(s,\Pi)$. Namely we have the local $L$-factors $L(s,\Pi_v)$, the local Satake parameters $\alpha_i(\Pi_v)$ and $\mu_i(\Pi_v)$, the epsilon factor $\epsilon(s,\Pi)$, the root number $\epsilon(\Pi)$, the local conductors $q(\Pi_v)$, $C(\Pi_v)$ and the analytic conductor $C(\Pi)$. The Euler product converges absolutely for $\MRe s$ large enough.

One important difference concerns the bounds for local Satake parameters. Even if we assume that $\Pi$ has unitary central character the inequalities~\eqref{towardsR} may not hold. We shall therefore require a stronger condition on $\Pi$.

\begin{prop}\label{prop:tempered}
Let $\Pi$ be an isobaric representation of $\GL(d,\BmA_F)$. Assume that the archimedean component $\Pi_\infty$ is tempered. Then the bounds towards Ramanujan are satisfied. Namely there is a positive constant $\theta<\Mdemi$ such that for all $1\le i\le d$ and all archimedean (resp. non-archimedean) places $v$,
\begin{equation}\label{eq:prop:tempered}
\MRe \mu_i(\Pi_v) \le \theta
,\qtext{resp. }\
\log_{q_v} \abs{\alpha_i(\Pi_v)}\le \theta.
\end{equation}
\end{prop}

\begin{proof} Let $\Pi=\Pi_1\boxplus \cdots \boxplus \Pi_r$ be the isobaric decomposition with $\Pi_j$ cuspidal. Then we will show that all $\Pi_j$ have unitary central character, which implies Proposition~\ref{prop:tempered}.

By definition we have that $\Pi_\infty$ is a Langlands quotient of the induced representation of $\Pi_{1\infty} \otimes \cdots \otimes \Pi_{r\infty}$. Since $\Pi_\infty$ is tempered, this implies that all $\Pi_{j\infty}$ are tempered, and in particular have unitary central character. Then the (global) central character of $\Pi_j$ is unitary as well.

\end{proof}

\begin{rem} In analogy with the local case, an isobaric representation $\Pi_1 \boxplus \cdots \boxplus \Pi_r$ where all cuspidal representations $\Pi_j$ have unitary central character is called \Lquote{tempered} in~\cite{cong:auto77:lang}. This terminology is fully justified only under the generalized Ramanujan conjecture for $\GL(d,\BmA_F)$. To avoid confusion we use the adjective \Lquote{tempered} for $\Pi=\otimes_v \Pi_v$ only in the strong sense that the local representations $\Pi_v$ are tempered for all $v\in \CmV_F$.
\end{rem}

\begin{rem} In the proof of Proposition~\ref{prop:tempered} we see the importance of the notion of isobaric representations and Langlands quotients. For instance a discrete series representation of $\GL(2,\BmR)$ is a constituent (but not a Langlands quotient) of an induced representation of a non-tempered character of $\GL(1,\BmR)\times \GL(1,\BmR)$.
\end{rem}

\subsection{An explicit formula}\label{sec:pp:explicit} Let $\Pi$ be a unitary cuspidal representation of $\GL(d,\BmA_F)$. Let $\rho_{j}(\Pi)$ denote the zeros of $\Lambda(s,\Pi)$ counted with multiplicities. These are also the non-trivial zeros of $L(s,\Pi)$. The method of Hadamard and de la Vall\'ee Poussin generalizes from the Riemann zeta function to automorphic $L$-functions, and implies that $0<\MRe \rho_j(\Pi)<1$ for all $j$. There is a polynomial $p(s)$ such that $p(s)\Lambda(s,\Pi)$ is entire and of order $1$ ($p(s)=1$ except when $d=1$ and $\Pi=\abs{.}^{it}$, in which case we choose $p(s)=(s-it)(1-it-s)$).

The Hadamard factorization shows that there are $a=a(\Pi)$ and $b=b(\Pi)$ such that
\begin{equation*}
p(s)\Lambda(s,\Pi)=e^{a+bs}
\prod_j \bigl(1-\frac{s}{\rho_j(\Pi)} \bigr)
e^{s/\rho_j(\Pi)}.
\end{equation*}
The product is absolutely convergent in compact subsets away from the zeros $\rho_j(\Pi)$. The functional equation implies that
\begin{equation*}
\sum_j \MRe(\rho_j(\Pi)^{-1})= -\MRe b(\Pi).
\end{equation*}

The number of zeros of bounded imaginary part is bounded above uniformly:
\begin{equation*}
\abs{\set{j,\ \abs{\Mim \rho_j(\Pi)}\le 1}}
\ll \log C(\Pi).
\end{equation*}
Changing $\Pi$ into $\Pi\otimes \abs{.}^{it}$ we have an analogous uniform estimate for the number of zeros with $\abs{\Mim \rho_j(\Pi)-T}\le 1$ (in particular this is $\ll_\Pi \log T$).

Let $N(T,\Pi)$ be the number of zeros with $\abs{\Mim \rho_j(\Pi)}\le T$. Then the following estimate holds uniformly in $T\ge 1$ (Weyl's law):
\begin{equation*}
  %\label{weyls}
N(T,\Pi) = \frac{T}{\pi} \Bigl(d
\log (\frac{T}{2\pi e})
+ \log C(\Pi)
\Bigr)
+O_\Pi(\log T).
\end{equation*}
The error term could be made uniform in $\Pi$, see~\cite{book:IK04}*{\S5.3} for more details\footnote{One should be aware that Theorem~5.8 in~\cite{book:IK04} does not apply directly to our setting because it is valid under certain further assumptions on $\Pi$ such as $\mu_i(\Pi_v)$ being real for archimedean places $v$.}. The main term can be interpreted as the variation of the argument of $C(\Pi)^{s/2}L(s,\Pi_\infty)$ along certain vertical segments.

We are going to discuss an explicit formula (see~\eqref{eq:prop:explicit} below) expressing a weighted sum over the zeros of $\Lambda(s,\Pi)$ as a contour integral. It is a direct consequence of the functional equation~\eqref{fneq} and Cauchy formula. The explicit formula is traditionally stated using the Dirichlet coefficients of the $L$-function $L(s,\Pi)$. For our purpose it is more convenient to maintain the Euler product factorization.

Definte $\gamma_j(\Pi)$ by $\rho_j(\Pi)=\Mdemi+i\gamma_j(\Pi)$. We know that $\abs{\Mim \gamma_j(\Pi)}<\Mdemi$ and under the GRH, $\gamma_j(\Pi)\in \BmR$ for all $j$.

It is convenient to denote by $\Mdemi + ir_j(\Pi)$ the (eventual) poles of $\Lambda(s,\Pi)$ counted with multiplicity. We have seen that poles only occur when $\Pi=\abs{.}^{it}$ in which case the poles are simple and $\set{r_j(\Pi)}=\set{t+\frac{i}{2},-t-\frac{i}{2}}$.

The above discussion applies with little change to isobaric representations. If we also assume that $\Pi_\infty$ is tempered then we have seen in the proof of Proposition~\ref{prop:tempered} that $\Pi=\Pi_1 \boxplus \cdots \boxplus \Pi_r$ with $\Pi_i$ unitary cuspidal representations of $\GL(d_i,\BmA)$ for all $1\le i\le r$. In particular the bounds towards Ramanujan apply and $\abs{\Mim \gamma_j(\Pi)}<\Mdemi$ for all $j$.

Let $\Phi$ be a Paley--Wiener function whose Fourier transform
\begin{equation}\label{def:fourier}
\widehat \Phi(y):=\int_{-\infty}^{+\infty} \Phi(x) e^{-2\pi i xy} dx
\end{equation}
has compact support. Note that $\Phi$ may be extended to an entire function on $\BmC$.
\begin{prop}\label{prop:explicit} Let $\Pi$ be an isobaric representation of $\GL(d,\BmA)$ satisfying the bounds towards Ramanujan~\eqref{towardsR}. With notation as above and for $\sigma>\Mdemi$, the following identity holds
\begin{equation}\label{eq:prop:explicit}
\begin{aligned}
\sum_j &\Phi(\gamma_j(\Pi))=\sum_j \Phi
(r_j(\Pi))
+ \frac{\log q(\Pi)}{2\pi} \hat \Phi(0)
+\\
&+ \frac{1}{2\pi}\int_{-\infty}^{\infty}
\biggl[
\frac{\Lambda'}{\Lambda}(\Mdemi+\sigma+ir,\Pi)
\Phi(r-i\sigma)
+
\overline{
\frac{\Lambda'}{\Lambda}(\Mdemi+\sigma+ir,\Pi)
}
\Phi(r+i\sigma)
\biggr]
dr.
\end{aligned}
\end{equation}
\end{prop}

There is an important remark about the explicit formula that we will use frequently. Therefore we insert it here before going into the proof. The line of integration in~\eqref{eq:prop:explicit} is away from the zeros and poles because $\sigma>1/2$. In particular the line of integration cannot be moved to $\sigma=0$ directly. But we can do the following which is a natural way to produce the sum over primes. First we replace $\Lambda(s,\Pi)$ by its Euler product which is absolutely convergent in the given region ($\MRe s>1$). Then for each of the term we may move the line of integration to $\sigma=0$ because we have seen that $\frac{L'}{L}(s,\Pi_v)$ has no pole for $\MRe s > \theta$. Thus we have
\begin{equation}\label{explicit:switch}
\int_{-\infty}^{\infty}
\frac{\Lambda'}{\Lambda}(\Mdemi+\sigma+ir,\Pi)
\Phi(r-i\sigma)
dr=
\sum_{v\in \CmV_F}
\int_{-\infty}^{\infty}
\frac{L'}{L}
(\Mdemi+ir,\Pi_v)
\Phi(r)
dr.
\end{equation}
The latter expression is convenient to use in practice. The integral in the right-hand side of~\eqref{explicit:switch} is absolutely convergent because $\Phi$ is rapidly decreasing and the sum over $v\in\CmV_F$ is actually finite since the support of $\widehat \Phi$ is compact. \footnote{Note however that it is never allowed to switch the sum and integration symbols in~\eqref{explicit:switch}. This is because the $L$-function is evaluated at the center of the critical strip in which the Euler product does not converge absolutely.}

\begin{proof} The first step is to work with the Mellin transform rather than the Fourier transform. Namely we set
\begin{equation*}
H(\Mdemi + is)=\Phi\bigl(s\bigr),\quad s\in \BmC.
\end{equation*}
Note that $H$ is an entire function which is rapidly decreasing on vertical strips. This justifies all shifting of contours below.

We form the integral
\begin{equation*}
\int_{(2)} \frac{\Lambda'}{\Lambda}(s,\Pi) H(s) \frac{ds}{2i\pi}.
\end{equation*}
We shift the contour to $\MRe s=-1$ crossing zeros and eventual poles of $\dfrac{\Lambda'}{\Lambda}$ inside the critical strip. The sum over the zeros reads
\begin{equation*}
\sum_j H(\rho_j(\Pi)) = \sum_j \Phi(\gamma_j(\Pi))
\end{equation*}
and the sum over the poles reads
\begin{equation*}
-\sum_j \Phi(r_j(\Pi)).
\end{equation*}

Note that since $\epsilon(s,\Pi)=\epsilon(\Pi) q(\Pi)^{\Mdemi-s}$ we have
\begin{equation*}
\frac{\epsilon'}{\epsilon}(s,\Pi)=-\log q(\Pi),\quad s\in\BmC.
\end{equation*}
We obtain as consequence of the functional equation~\eqref{fneq} that
\begin{equation*}
\begin{aligned}
\int_{(-1)}^{} \frac{\Lambda'}{\Lambda}(s,\Pi) H(s) \frac{ds}{2i\pi}&=
\int_{(2)}^{}  \frac{\Lambda'}{\Lambda}(1-s,\Pi) H(1-s) \frac{ds}{2i\pi}\\
&=-\int_{(2)}^{}  \biggl(\log q(\Pi) + \frac{\Lambda'}{\Lambda}(s,\widetilde \Pi)
\biggr) H(1-s) \frac{ds}{2i\pi}.
\end{aligned}
\end{equation*}

Now we observe that
\begin{equation*}
\int_{(2)} H(s) \frac{ds}{2i\pi} = \frac{1}{2\pi} \hat \Phi(0)
\end{equation*}
and also
\begin{equation*}
\int_{(2)}^{} \frac{\Lambda'}{\Lambda}(s,\Pi) H(s) \frac{ds}{2i\pi}
=
\frac{1}{2\pi}\int_{-\infty}^{\infty}
\Phi(r-\frac{3i}{2})
\frac{\Lambda'}{\Lambda}(2+ir,\Pi) dr.
\end{equation*}
and
\begin{equation*}
\begin{aligned}
\int_{(2)}^{} \frac{\Lambda'}{\Lambda}(s,\widetilde \Pi) H(1-s) \frac{ds}{2i\pi}
&=
\frac{1}{2\pi}\int_{-\infty}^{\infty}
\Phi(r+\frac{3i}{2})
\frac{\Lambda'}{\Lambda}(2-ir,\widetilde \Pi) dr\\
&=
\frac{1}{2\pi}\int_{-\infty}^{\infty}
\Phi(r+\frac{3i}{2})
\overline{
\frac{\Lambda'}{\Lambda}(2+ir,\Pi)
}
dr.
\end{aligned}
\end{equation*}

Since $\Lambda(s,\widetilde \Pi)=\overline{\Lambda(\overline{s},\Pi)}$ this concludes the proof of the proposition by collecting all the terms above. Precisely this yields the formula when $\sigma=3/2$, and then we can make $\sigma>1/2$ arbitrary by shifting the line of integration.
\end{proof}

We conclude this section with a couple of remarks on symmetries. The first observation is that the functional equation implies that if $\rho$ is a zero (resp. pole) of $\Lambda(s,\Pi)$ then so is $1-\bar\rho$ (reflexion across the central line). Thus the set $\set{\gamma_j(\Pi)}$ (resp. $\set{r_j(\Pi)}$) is invariant by the reflexion across the real axis (namely $\gamma$ goes into $\overline{\gamma}$). Note that this is compatible with the GRH which predicts that $\MRe \rho_j(\Pi)=\Mdemi$ and $\gamma_j(\Pi)\in \BmR$.

Assuming $\Phi$ is real-valued the explicit formula is an identity between real numbers. Indeed the Schwartz reflection principle gives $\Phi(s)=\overline{\Phi(\overline s)}$ for all $s\in \BmC$. Because of the above remark the sum over the zeros (resp. poles) in~\eqref{eq:prop:explicit} is a real; the integrand is real-valued as well for all $r\in (-\infty,\infty)$.

The situation when $\Pi$ is self-dual occurs often in practice. The zeros $\gamma_j(\Pi)$ satisfy another symmetry which is the reflexion across the origin. Assuming $\Pi$ is cuspidal and non-trivial there is no pole. The explicit formula~\eqref{eq:prop:explicit} simplifies and may be written
\begin{equation*}
\sum_j
\Phi\left( \gamma_j(\Pi) \right)
=\frac{\log q(\Pi)}{2\pi}\widehat \Phi(0)
+\frac{1}{\pi}
\sum_{v\in \CmV_F}
\int_{-\infty}^{\infty}
\frac{L'}{L}(\Mdemi+ir,\Pi_v)
\Phi(r) dr.
\end{equation*}

\section{Sato-Tate equidistribution}\label{s:Sato-Tate}

  Let $G$ be a connected reductive group over a number field $F$ as in the previous section.
  The choice of a $\Gal(\ol{F}/F)$-invariant
  splitting datum $(\hat{B},\hat{T},\{X_{\alpha}\}_{\alpha\in \Delta^{\vee}})$
  as in \S\ref{sub:L-groups} induces
  a composite map $\Gal(\ol{F}/F)\ra \Out(\hat{G})\hra \Aut(\hat{G})$ with open kernel. Let $F_1$ be the unique finite
  extension of $F$ in $\ol{F}$ such that
  $$\Gal(\ol{F}/F)\twoheadrightarrow \Gal(F_1/F)\hra \Aut(\hat{G}).$$

\subsection{Definition of the Sato-Tate measure}\label{sub:ST-meas}

  Set $\Gamma_1:=\Gal(F_1/F)$.
  Let $\hat{K}$ be a maximal compact subgroup of $\hat{G}$ which is $\Gamma_1$-invariant.
  (It is not hard to see that such a $\hat{K}$ exists, cf. \cite{JZ08}.)
  Set $\hat{T}_{c}:=\hat{T}\cap\hat{K}$. (The subscript $c$ stands for ``compact'' as it was in \S\ref{sub:plan-unramified}.)
  Denote by $\Omega_c$ the Weyl group for $(\hat{K},\hat{T}_c)$.

  Let $\theta\in \Gamma_1$.
  Define $\Omega_{c,\theta}$ to be the subset of $\theta$-invariant elements of $\Omega_c$.
  Consider the topological quotient $\hat{K}^\natural_\theta$ of
  $\hat{K}\rtimes \theta$ by the $\hat{K}$-conjugacy equivalence relation.
  Set $\hat{T}_{c,\theta}:=\hat{T}_c/(\theta-\id)\hat{T}_c$.
  Note that the action of $\Omega_{c,\theta}$ on $\hat{T}_c$ induces an action
  on $\hat{T}_{c,\theta}$.
  The inclusion $\hat{T}_c\hra \hat{K}$ induces a canonical topological isomorphism (cf. Lemma \ref{l:unr-temp-spec})
 \beq\label{e:K-T-isom}\hat{K}^\natural_\theta\simeq \hat{T}_{c,\theta}/\Omega_{c,\theta}.\eeq

%  Then $\hat{T}_c=\hat{T}_c^\theta\cdot \hat{T}_{c,\theta}$ and
%  $\hat{T}_c^\theta\cap \hat{T}_{c,\theta}$ is finite (cf. \cite[\S6.3.(2)]{Bor79}).
%  (*** LATTER may not be needed. ***)

  The Haar measure on $\hat{K}$ (resp. on $\hat{T}_c$) with total volume 1 is
  written as $\mu_{\hat{K}}$ (resp. $\mu_{\hat{T}_c}$).
  Then $\mu_{\hat{K}}$ on $\hat{K}\rtimes\theta$ induces the quotient measure
  $\mu_{\hat{K}^\natural_\theta}$ (so that
  for any continuous function $f^\natural$ on $\hat{K}_\theta^\natural$ and its pullback
  $f$ on $\hat{K}$,
  $\int f^\natural \mu_{\hat{K}^\natural_\theta}=\int f \mu_{\hat{K}}$) thus also a measure $\mu_{\hat{T}_c,\theta}$ on $\hat{T}_{c,\theta}$.

%   The quotient measure $\mu_{\hat{T}_c/\Omega}$
%  is induced by $\mu_{\hat{T}_c}$.

\begin{defn}
  The $\theta$-\key{Sato-Tate measure} $\ST_\theta$
  on $\hat{T}_{c,\theta}/\Omega_{c,\theta}$ is the measure
  transported from $\mu_{\hat{K}^\natural_\theta}$ via \eqref{e:K-T-isom}.
%  such that
%  $$\int_{\hat{K}_\theta^\natural} f\cdot \mu_{\hat{K}_\theta^\natural}
%   = \int_{\hat{T}_{c,\theta}/\Omega_{c,\theta}}
%  f\cdot \ST$$
%  for all continuous functions $f$ on $\hat{K}_\theta^\natural$ (which is restricted to $\hat{T}_c/\Omega_c$
%  on the right hand side).
\end{defn}

\begin{lem}\label{l:ST-measure}
  Let $\ST_{\theta,0}$ denote the measure on $\hat{T}_{c,\theta}$
  pulled back from
  $\ST_\theta$ on $\hat{T}_{c,\theta}/\Omega_{c,\theta}$ (so that $\int
   f \ST_{\theta,0}=
  \int \ol{f} \ST_\theta$ for every continuous $\ol{f}$ on $\hat{T}_{c,\theta}/\Omega_{c,\theta}$
  and its pullback $f$). Then
  $$\ST_{\theta,0}=\frac{1}{|\Omega_{c,\theta}|} D_\theta(t) \mu_{\hat{T}_{c,\theta}},$$
  where $D_\theta(t)={\det(1-\ad(t\rtimes \theta)|\Lie(\hat{K})/\Lie(\hat{T}^\theta_c))}$
  and $t$ signifies a parameter on $\hat{T}_{c,\theta}$.
\end{lem}

\begin{proof}
  The twisted Weyl integration formula tells us that
  for a continuous $f:\hat{K}\ra \C$,
$$\int_{\hat{K}} f(k)\mu_{\hat{K}}
= \frac{1}{|\Omega_{c,\theta}|} \int_{\hat{T}^{\mathrm{reg}}_{c,\theta}}
D_\theta(t) \int_{\hat{K}_{t\theta}\bs \hat{K}} f(x^{-1} t x^{\theta}) \cdot dx dt.$$
  Notice that $\hat{K}_{t\theta}$ is the twisted centralizer group of $t$ in $\hat{K}$
  (or, the centralizer group of $t\theta$ in $\hat{K}$).
  On the right hand side, $ \mu_{\hat{T}_{c,\theta}}$ is used for integration.
  When $f$ is a pullback from $\hat{K}^\natural_{\theta}$, the formula simplifies as
$$\int_{\hat{K}^\natural_{\theta}} f(k) \mu_{\hat{K}^\natural_{\theta}}
= \frac{1}{|\Omega_{c,\theta}|} \int_{\hat{T}^{\mathrm{reg}}_{c,\theta}}
D_\theta(t) f(t)\cdot \mu_{\hat{T}_{c,\theta}}$$
and the left hand side is equal to
$\int_{\hat{T}_{c,\theta}} f(t) \ST_{\theta,0}$ by definition.
\end{proof}

\subsection{Limit of the Plancherel measure versus the Sato-Tate measure}\label{sub:lim-of-Plan}

  Let $\theta,\tau\in \Gamma_1$.
  Then clearly $\Omega_{c,\theta}=\Omega_{c,\tau\theta\tau^{-1}}$,
  $\hat{K}^\natural_\theta \simeq \hat{K}^\natural_{\tau\theta\tau^{-1}}$
 via $k\mapsto \tau(k)$ and $\hat{T}_{c,\theta}\simeq \hat{T}_{c,\tau\theta\tau^{-1}}$ via $t\mapsto \tau(t)$.
 Accordingly $\ST_\theta$ and $\ST_{\theta,0}$ are identified with
 $\ST_{\tau\theta\tau^{-1}}$ and $\ST_{\tau\theta\tau^{-1},0}$, respectively.

 Fix once and for all a set of representatives $\mC(\Gamma_1)$ for conjugacy classes in $\Gamma_1$.
 For $\theta\in \mC(\Gamma_1)$, denote by $[\theta]$ its conjugacy class. For each finite place $v$
 such that $G$ is unramified over $F_v$, the geometric Frobenius $\Fr_v\in \Gal(F_v^{\ur}/F_v)$
 gives a well-defined conjugacy class $[\Fr_v]$ in $\Gamma_1$.
  The set of all finite places $v$ of $F$ where $G$ is unramified is partitioned into
  $$\{\cV_F(\theta)\}_{\theta\in \mC(\Gamma_1)}$$ such that
  $v\in \cV_F(\theta)$ if and only if $[\Fr_v]=[\theta]$.

   For each finite place $v$ of $F$,
  the unitary dual of $G(F_v)$ and its Plancherel measure are written as $G^{\wedge}_v$
  and $\pl_v$.
  Similarly adapt the notation of \S\ref{sub:notation-Plancherel} by
  appending the subscript $v$.
  Now fix $\theta\in \mC(\Gamma_1)$ and suppose that $G$ is unramified at $v$ and that
  $v\in \cV_F(\theta)$. We choose $\ol{F}\hra \ol{F}_v$ such that $\Fr_v$ has image $\theta$ in $\Gamma_1$
  (rather than some other conjugate). This rigidifies
  the identification in the second map below. (If $\Fr_v$ maps to $\tau\theta\tau^{-1}$ then the second map is twisted by $\tau$.)
 \beq\label{e:local-isom-v}G(F_v)^{\wedge,\ur,\temp}\stackrel{\mathrm{canonical}}{\simeq} \hat{T}_{c,\Fr_v}/\Omega_{c,\Fr_v}= \hat{T}_{c,\theta}/\Omega_{c,\theta}.\eeq
  By abuse of notation let $\plurtemp_{v}$ (a measure on $G(F_v)^{\wedge,\ur,\temp}$)
  also denote the transported measure
  on $\hat{T}_{c,\theta}/\Omega_{c,\theta}$.
  Let $C_v$ denote the constant of Proposition \ref{p:unr-plan-meas}, which
  we normalize such that $\plurtemp_{v,0}$ has total volume 1.
  Note that $\ST_{\theta}$ also has total volume 1.

\begin{prop}\label{p:lim-of-Plancherel} Fix any $\theta\in \mC(\Gamma_1)$.
As $v\ra\infty$ in $\cV_F(\theta)$,
  we have weak convergence $\plurtemp_{v}\ra \ST_{\theta}$ as $v\ra \infty$.
\end{prop}

\begin{proof}
  It is enough to show that $\plurtemp_{v,0}\ra \ST_{\theta,0}$ on $\hat{T}_{c,\theta}$ as $v$ tends to $\infty$
   in $\cV_F(\theta)$.
   Consider the measure $\plurtemp_{v,1}:=C_v^{-1}\plurtemp_{v,0}$.
   It is clear from the formula of Proposition \ref{p:unr-plan-meas}
   that $\plurtemp_{v,1}\ra \ST_{\theta}$ as $v\ra\infty$ in $\cV_F(\theta)$.
   In particular, the total volume of $\plurtemp_{v,1}$ tends to 1, hence
   $C_v\ra 1$ as $v\ra \infty$  in $\cV_F(\theta)$.
   We conclude that $\plurtemp_{v,0}\ra \ST_{\theta,0}$ as desired.
%  and $\hat{\mu}_0=\prod_{\alpha\in \Phi} (1-t^{-\alpha}) \cdot \mu_{\hat{T}_c}$.
%  It is not hard to see that the limit of $\hat{\mu}_v$ as $v\ra \infty$ is $\hat{\mu}$.
%  In particular, $C_v=\hat{\mu}_{v,0}(\hat{T}_c)^{-1}$ tends to $C:=\hat{\mu}_0(\hat{T}_c)^{-1}$.
%  Thus the limit of $\plurtemp_{v,0}=C_v\hat{\mu}_{v,0}$ as $v\ra \infty$ is
%  $C\hat{\mu}_0$. The comparison with Lemma \ref{l:ST-measure} reveals that
%  $\ST_0=C\hat{\mu}_0$ since they differ by a constant multiple and have total volume 1.
\end{proof}

\begin{rem}
  The above proposition was already noticed by Sarnak for $G=SL(n)$ in \cite[\S4]{Sarnak87}.
\end{rem}

\subsection{The generalized Sato-Tate problem}\label{sub:gen-ST-conj}

Let $\pi$ be a cuspidal\footnote{If $\pi$ is not cuspidal then the hypothesis is never supposed to be satisfied.} tempered automorphic representation of $G(\A_F)$ satisfying

\medskip\noindent
\textbf{Hypothesis.}
The conjectural global $L$-parameter $\varphi_\pi$ for $\pi$ has Zariski dense image in ${}^L G_{F_1/F}$.
\medskip

  Of course this hypothesis is more philosophical than practical. The global Langlands correspondence between ($L$-packets of) automorphic representations and global $L$-parameters of $G(\A_F)$ is far from established. A fundamental problem here is that global $L$-parameters cannot be defined unless the conjectural global Langlands group is defined.  (Some substitutes have been proposed by Arthur in the case of classical groups.  The basic idea is that a cuspidal automorphic representation of $\GL_n$ can be put in place of an irreducible $n$-dimensional representation of the global Langlands group.) Nevertheless, the above hypothesis can often be replaced with another condition, which should be equivalent but can be stated without reference to conjectural objects.  For instance, when $\pi$ corresponds to a Hilbert modular form of weight$\ge 2$ at all infinite places, one can use the hypothesis that it is not a CM form (i.e. not an automorphic induction from a Hecke character over a CM field).

  Let us state a general form of the Sato-Tate conjecture. Let $q_v$ denote the cardinality of the residue field cardinality at a finite place $v$ of $F$. Define $\cV_F(\theta,\pi)^{\le x}:=\{v\in\cV_F(\theta,\pi): q_v\le x\}$ for $x\in \R_{\ge 1}$.

\begin{conj}\label{c:gen-ST}
% (Generalized Sato-Tate conjecture, cf. \cite{Ser94})
  Assume the above hypothesis. For each $\theta\in \mC(\Gamma_1)$, let $\cV_F(\theta,\pi)$ be the subset of $v\in \cV_F(\theta)$ such that $\pi_v$ is unramified. Then $\{\pi_v\}_{v\in \cV_F(\theta,\pi)}$ are equidistributed according to $\ST_\theta$. More precisely $$\frac{1}{|\cV_F(\theta,\pi)^{\le x}|} \sum_{v\in \cV_F(\theta,\pi)^{\le x}} \delta_{\pi_v} \ra \ST_\theta\quad\mbox{as}\quad x\ra\infty.$$
\end{conj}

  The above conjecture is deemed plausible in that it is essentially a consequence of the Langlands functoriality conjecture at least when $G$ is (an inner form of) a split group. Namely if we knew that the $L$-function $L(s,\pi,\rho)$ for any irreducible representation ${}^L G\ra \GL_d$ were a cuspidal automorphic $L$-function for $\GL_d$ then the desired equidistribution is implied by Theorem 1 of \cite[App A.2]{Ser68l}.
%  (*** NEED TO BE JUSTIFIED in NON-SPLIT CASE. ***)

\begin{rem}
  In general when the above hypothesis is dropped, it is likely that $\pi$ comes from an automorphic representation on a smaller group than $G$.  (If $\varphi_\pi$ factors through an injective $L$-morphism ${}^L H_{F_1/F}\ra {}^L G_{F_1/F}$ then the Langlands functoriality predicts that $\pi$ arises from an automorphic representation of $H(\A_F)$.) Suppose that the Zariski closure of $\im(\varphi_\pi)$ in ${}^L G_{F_1/F}$ is isomorphic to ${}^L H_{F_1/F}$ for some connected reductive group $H$ over $F$.  (In general the Zariski closure may consist of finitely many copies of ${}^L H_{F_1/F}$.) Then $\{\pi_v\}_{v\in \cV_F(\theta,\pi)}$ should be equidistributed according to the Sato-Tate measure belonging to $H$ in order to be consistent with the functoriality conjecture.
\end{rem}

  One can also formulate a version of the conjecture where $v$ runs over the set of \emph{all} finite places where $\pi_v$ are unramified by considering conjugacy classes in ${}^L G_{F_1/F}$ rather than those in $\hat{G}\rtimes \theta$ for a fixed $\theta$. For this let $\hat{K}^\natural$ denote the quotient of $\hat{K}$ by the equivalence relation coming from the conjugation by $\hat{K}\rtimes \Gamma_1$. Since $\hat{K}^\natural$ is isomorphic to a suitable quotient of $\hat{T}_c$, the Haar measure on $\hat{K}$ gives rise to a measure, to be denoted $\ST$, on the quotient of $\hat{T}_c$. Let $\cV_F(\pi)^{\le x}$ (where $x\in \R_{\ge 1}$) denote the set of finite places of $F$ such that $\pi_v$ are unramified and $q_v\le x$. By writing $v\ra\infty$ we mean that $q_v$ tends to infinity.

\begin{conj}\label{c:gen-ST-2}
   Assume the above hypothesis. Then as $x\ra\infty$ the set $\{\pi_v\,:\,v\in \cV_F(\pi)^{\le x}\}$
  is equidistributed according to $\ST_{\theta}$. Namely
  $$\frac{1}{|\cV_F(\pi)^{\le x}|}
  \sum_{v\in \cV_F(\pi)^{\le x}} \delta_{\pi_v}
  \ra \ST\quad\mbox{as}\quad x\ra\infty.$$
\end{conj}

\begin{rem}
  Unlike Conjecture \ref{c:gen-ST} it is unnecessary to choose embeddings $\ol{F}\hra \ol{F}_v$ to rigidify \eqref{e:local-isom-v} since the ambiguity in the rigidification is absorbed in the conjugacy classes in ${}^L G_{F_1/F}$. The formulation of Conjecture \ref{c:gen-ST-2} might be more suitable than the previous one in the motivic setting where we would not want to fix $\ol{F}\hra \ol{F}_v$.
   The interested reader may compare Conjecture \ref{c:gen-ST-2}
  with the motivic Sato-Tate conjecture of \cite[13.5]{Ser94}.
%$$\xymatrix{
%\mL_F \ar[d] \ar[r]^-{\varphi} & {}^L G \ar[r]^-{r}  & \GL(V) \\
%\fkG_{\Mot}  \ar[r] & \fkG_M \ar[ur]_-{\mathrm{real.}} }$$
\end{rem}

  The next subsection will discuss the analogue of Conjecture \ref{c:gen-ST} for automorphic families.
    Conjecture \ref{c:gen-ST-2} will not be considered any more in our paper. It is enough to mention that the analogue of the latter conjecture for families of algebraic varieties makes sense and appears to be interesting.

\subsection{The Sato-Tate conjecture for families}\label{sub:ST-families}

  The Sato-Tate conjecture has been proved for Hilbert modular forms in
  (\cite{BLGHT11}, \cite{BLGG11}). Analogous equidistribution
theorems in the function field setting are due to Deligne and Katz. (See \cite[Thm 9.2.6]{book:KS}
for instance.)
  Despite these fantastic developments, we have little unconditional theoretical evidence for
  the Sato-Tate conjecture for general reductive groups over number fields.
  On the other hand, it has been noticed
  that the analogue of the
  Sato-Tate conjecture for families of automorphic representations is more
  amenable to attack. Indeed there was some success in the case of holomorphic modular forms and Maass forms
  (\cite[Thm 2]{CDF97} \cites{Royer:dimension-rang,ILS00,Serre:pl}).
  The conjecture has the following coarse form, which should be thought of
  as a guiding principle rather than a rigorous conjecture. Compare with some precise results
  in \S\ref{sub:ST-theorem}.

\begin{heur}\label{c:ST-families}
   Let $\{\cF_k\}_{k\ge1}$ be a ``general'' sequence of
   finite families of automorphic representations of $G(\A_F)$ such that
  $|\cF_k|\ra \infty$ as $k\ra \infty$. Then $\{\pi_v\in \cF_k\}$
   are equidistributed according to $\ST_{\theta}$ as $k$ and $v$ tend to infinity
   subject to the conditions that $v\in \cV_F(\theta)$ and that
  all members of $\cF_k$ are unramified at $v$.

\end{heur}

  We are not going to make precise what ``general'' means, but merely remark
  that it should be the analogue of the condition that
  the hypothesis of \S\ref{sub:gen-ST-conj} holds for the ``generic fiber''
  of the family when the family has a geometric meaning (see also~\cite{Sarn:family}). In practice one would verify
  the conjecture for many interesting families while simply ignoring the word ``general''.
  %In some sense any family may be considered general if the conjecture holds for that family.
  Some relation between $k$ and $v$ holds when taking limit: $k$ needs to grow fast enough compared to $v$ (or more precisely $\abs{\cF_k}$ needs to grow fast enough compared to $q_v$).

  It is noteworthy that the unpleasant hypothesis of \S\ref{sub:gen-ST-conj}
  can be avoided for families.
  Also note that the temperedness
  assumption is often unnecessary due to the fact that the Plancherel measure is supported on the
  tempered spectrum. This is an indication that most representations in a family are globally tempered, which we will return to in a subsequent work.

  Later we will
  verify the conjecture for
  many families in \S\ref{sub:ST-theorem}
   as a corollary to the automorphic
  Plancherel theorem proved earlier in \S\ref{s:aut-Plan-theorem}.
  Our families arise as the sets of all automorphic representations with increasing level or weight, possibly with prescribed local conditions
  at finitely many fixed places.

\section{Background materials}
\label{s:bg}

  This section collects background materials in the local and global contexts. Subsections \ref{sub:orb-int-const-term} and \ref{sub:lem-ram} are concerned with $p$-adic groups while \S\ref{sub:st-disc}, \S\ref{sub:EP} and \S\ref{sec:b:fs} are with real and complex Lie groups. The rest is about global reductive groups.

\subsection{Orbital integrals and constant terms}\label{sub:orb-int-const-term}
  We introduce some notation in the $p$-adic context.
\bit
\item $F$ is a finite extension of $\Q_p$ with integer ring $\cO$ and multiplicative valuation~$|\cdot|$.
\item $G$ is a connected reductive group over $F$.
\item $A$ is a maximal $F$-split torus of $G$, and put $M_0:=Z_G(A)$.
\item $K$ is a maximal compact subgroup of $G$
corresponding to a special point in the apartment for $A$.
% which is good (\cite[Thm 5]{HC70}, cf. \cite[4.4]{BT72}).
\item $P=MN$ is a parabolic subgroup of $G$ over $F$, with $M$ and $N$ its Levi subgroup and
unipotent radical, such that $M\supset M_0$.

\item $\gamma\in G(F)$ is a semisimple element. (The case of a non-semisimple element
  is not needed in this paper.)
\item $I_\gamma$ is the neutral component
  of the centralizer of $\gamma$ in $G$. Then $I_\gamma$ is a connected reductive group over $F$.
\item $\mu_G$ (resp. $\mu_{I_\gamma}$) is a Haar measure on $G(F)$ (resp. $I_\gamma(F)$).
\item $\frac{\mu_G}{\mu_{I_\gamma}}$ is
  the quotient measure on $I_\gamma(F)\bs G(F)$ induced by $\mu_G$ and $\mu_{I_\gamma}$.
\item $\phi\in C^\infty_c(G(F))$.
\item $D^G(\gamma):=\prod_{\alpha} |1-\alpha(\gamma)|$ for a semisimple $\gamma\in G(F)$, where $\alpha$ runs over the set of roots of $G$ (with respect to any maximal torus in the connnected centralizer of $\gamma$ in $G$) such that $\alpha(\gamma)\neq 1$. Let $M$ be an $F$-rational Levi subgroup of $G$. For a semisimple $\gamma\in G(F)$, we define $D^G_M(\gamma)$ similarly by further excluding those $\alpha$ in the set of roots of $M$.
\eit
  Define the orbital integral
  $$O^{G(F)}_\gamma(\phi,\mu_G,\mu_{I_\gamma}):=
  \int_{I_\gamma(F)\bs G(F)} \phi(x^{-1}\gamma x) \frac{\mu_G}{\mu_{I_\gamma}}.$$
  When the context is clear, we use $O_\gamma(\phi)$ as a shorthand notation.

  We recall the theory of constant terms (cf. \cite[p.236]{vDi72}).
  Choose Haar measures $\mu_K$, $\mu_M$, $\mu_N$,
  on $K$, $M(F)$, $N(F)$, respectively, such that $\mu_G=\mu_K\mu_M\mu_N$ holds
  with respect to $G(F)=KM(F)N(F)$.
  Define the (normalized) constant term $\phi_M\in C^\infty_c(M(F))$ by
  \beq\label{e:const-term-formula}\phi_{M}(m)=\delta_P^{1/2}(m)\int_{N(F)}\int_{K} \phi(kmnk^{-1})\mu_K \mu_N.\eeq
  Although the definition of $\phi_M$ involves not only $M$ but $P$, the following lemma
  shows that the orbital integrals of $\phi_M$ depend only on $M$ by the density of
  regular semisimple orbital integrals, justifying our notation.
\begin{lem}\label{l:orb-int-const-term}
  For all $(G,M)$-regular semisimple $\gamma\in M(F)$,
  $$O_\gamma(\phi_{M},\mu_{M},\mu_{I_\gamma})=D^G_M(\gamma)^{1/2}
  O_\gamma(\phi,\mu_{G},\mu_{I_\gamma}).$$
\end{lem}
\begin{proof}
  \cite[Lem 9]{vDi72}. (Although the lemma is stated for regular elements $\gamma\in G$,
  it suffices to require $\gamma$ to be $(G,M)$-regular. See Lemma 8 of loc. cit.)
\end{proof}
  It is standard that
  the definition and facts we have recollected above extend to the adelic case. (Use \cite[\S\S7-8]{Kot86},
  for instance). We will skip rewriting the analogous definition in the adelic setting.

  Now we restrict ourselves to the local unramified case.
  Suppose that $G$ is unramified over $F$.
  Let $B\subset P \subset G$ be Borel and parabolic subgroups defined over $F$.
  Write $B=TU$ and $P=MN$ where $T$ and $M$ are Levi subgroups such that $T\subset M$ and $U$ and $N$ are unipotent radicals.

\begin{lem}\label{l:const-term-on-unram-Hecke} Let $\phi\in \cH^{\ur}(G)$.
Then $\cS^G_M(\phi)=\phi_M$, in particular
$\cS^G(\phi)= \cS^M(\cS^G_M \phi)=\phi_T$.
\end{lem}

\begin{proof}
  Straightforward from \eqref{e:Satake-formula} and \eqref{e:const-term-formula}.
\end{proof}

%\subsection{A bound on orbital integrals}

%  Use canonical measure on $G(\A^\infty_F)$ and
%  its centralizers.

%\begin{prop}
%  Fix $\phi^\infty\in C^\infty_c(G(\A_F^\infty))$.
%  There exists a constant $C>0$ such that
%  for every semisimple $\gamma\in G(\A^\infty_F)$,
%  $$O^{G(\A_F^\infty)}_\gamma(\phi^\infty)<
%  C  \prod_{v|\infty} |D^G(\gamma)|_v^{1/2}.$$

%\end{prop}

%\begin{proof}
%  Use appendix.
%\end{proof}

\subsection{Gross's motives}\label{sub:Gross-motive}

  Now let $F$ be a finite extension of $\Q$ (although Gross's theory applies more generally).
  Let $G$ be a connected reductive group over $F$ and consider its  quasi-split inner form $G^*$.
  Let $T^*$ be the centralizer of a maximal $F$-split torus of $G^*$. Denote by $\Omega$ the Weyl group for $(G^*,T^*)$ over $\ol{F}$.
  Set $\Gamma=\Gal(\ol{F}/F)$.
  Gross (\cite{Gro97}) attaches to $G$ an Artin-Tate motive
  $$\Mot_G=\bigoplus_{d\ge 1} \Mot_{G,d}(1-d)$$ with coefficients in $\Q$.
  Here $(1-d)$ denotes the Tate twist. The Artin motive $\Mot_{G,d}$
  (denoted $V_d$ by Gross) may be thought of as a $\Gamma$-representation on a $\Q$-vector space whose dimension is $\dim \Mot_{G,d}$.
  Define $$L(\Mot_G):=  L(0,\Mot_G)$$
  to be the Artin $L$-value of $L(s,\Mot_G)$ at $s=0$.
  We recall some properties of $\Mot_G$ from Gross's article.

\begin{prop}\label{p:Gross-motives}
  \benu
\item $\Mot_{G,d}$ is self-dual for each $d\ge 1$.
\item $\sum_{d\ge 1} \dim \Mot_{G,d} = r_G=\rank G$.
\item $\sum_{d\ge 1} (2d-1)\dim \Mot_{G,d} = \dim G$.
\item $|\Omega|=\prod_{d\ge 1} d^{\dim  \Mot_{G,d}}$.
\item If $T^*$ splits over a finite extension $E$ of $F$ then
the $\Gamma$-action on $\Mot_G$ factors through $\Gal(E/F)$.
\eenu
\end{prop}

  The Artin conductor $\fkf(\Mot_{G,d})$ is defined as follows.
  Let $F'$ be the fixed field of the kernel of the Artin representation
  $\Gal(\ol{F}/F)\ra \GL(V_d)$
  associated to $\Mot_{G,d}$.
  For each finite place $v$ of $F$, let $w$ be any place of $F'$ above $v$.
  Let $\Gamma(v)_i:=\Gal(F'_w/F_v)_i$ ($i\ge 0$) denote the $i$-th ramification subgroups. Set
  \beq\label{e:expo-conductor}
  f(G_v,d):=\sum_{i\in \Z_{\ge 0}} \frac{|\Gamma(v)_i|}{|\Gamma(v)_0|}
  \dim (V_d/V_d^{\Gamma(v)_i}),\eeq
  which is an integer independent of the choice of $w$. Write $\fkp_v$ for the prime ideal of $\cO_F$
  corresponding to $v$. If $v$ is unramified in $E$ then $f(G_v,d)=0$.
  Thus the product makes sense in the following definition.
  $$\fkf(\Mot_{G,d}):=\prod_{v\nmid \infty} \fkp_v^{f(G_v,d)}$$
  Let $E$ be the splitting field of $T^*$ (which is an extension of $F$)
and set $s^{\spl}_G:=[E:F]$.

\begin{lem}\label{l:bounding-conductor}
  For every finite place $v$ of $F$,
  $$f(G_v,d)\le (\dim \Mot_{G,d} )\cdot (s^{\spl}_G(1+e_{F_v/\Q_p} \log_p s^{\spl}_G)-1).$$
\end{lem}

\begin{proof}
  Let $F'$, $w$ and $V_d$ be as in the preceding paragraph. Then $F\subset F'\subset E$.
  Set $s_v:=[F'_w:F_v]$ so that $s_v\le s^{\spl}_G$.
  The case $s_v=1$ is obvious (in which case $f(G_v,d)=0$), so we may assume $s_v\ge 2$.
  From \eqref{e:expo-conductor} and Corollary \ref{c:vanishing-ram-subgroup} below,
$$f(G_v,d)\le \sum_{i\ge 0} \dim (V_d/V_d^{\Gamma(v)_i})\le
(\dim V_d)(s_v(1+e_{F_v/\Q_p}\log_p s_v)-1).$$
\end{proof}

  Recall that $w_G=|\Omega|$ is the cardinality of the absolute Weyl group.
  Let $s_G$ be the degree of the smallest extension of $F$ over which $G$ becomes split.
 The following useful lemma implies in particular that $s^{\spl}_G\le w_Gs_G$.

\begin{lem}(\cite[Lem 2.2]{JKZ})\label{l:torus-splitting}
  For any maximal torus $T$ of $G$ defined over $F$, there exists a finite
Galois extension $E$ of $F$ such that $[E:F]\le w_Gs_G$ and $T$ splits over $E$.
\end{lem}

\subsection{Lemmas on ramification}\label{sub:lem-ram}

  This subsection is meant to provide an ingredient of proof
  (namely Corollary \ref{c:vanishing-ram-subgroup}) for Lemma \ref{l:bounding-conductor}.
%  Although the material here is standard, we have not found a convenience reference
%  where the results are stated in the form we desire.
  Fix a prime $p$.
  Let $E$ and $F$ be finite extensions of $\Q_p$ with uniformizers $\varpi_E$ and
  $\varpi_F$, respectively. Normalize valuations $v_E:E^\times\ra \Z$ and
  $v_F:F^\times \ra \Z$ such that $v_E(\varpi_E)=v_F(\varpi_F)=1$.
  Write $e_{E/F}\in \Z_{\ge1}$ for the ramification index
  and $\fkD_{E/F}$ for the different.
  For a nonzero principal ideal $\fka$ of $\cO_E$,
  we define $v_E(\fka)$ to be $v_E(a)$ for any generator $a$ of $\fka$. This is well defined.

\begin{lem}\label{l:bound-diff-wild}
  Let $E$ be a totally ramified Galois extension of $F$ with $[E:F]=p^n$ for $n\ge 0$.
  Then $$v_E(\fkD_{E/F})\le p^n(1+n\cdot e_{F/\Q_p})-1.$$
\end{lem}

\begin{rem} In fact the inequality is sharp. There are totally ramified extensions $E/F$ for which the above equality holds as shown by \"Ore. See also~\cite{Serre:chebotarev}*{\S1} for similar results.
\end{rem}

\begin{proof}
  The lemma is trivial when $n=0$. Next assume $n=1$ but allow $E/F$ to be a non-Galois extension.
  Let $f(x)=\sum_{i=0}^p a_i x^i\in \cO_F[x]$ (with $a_p=1$ and $v_F(a_i)\ge 1$ for $i<p$)
  be the Eisenstein polynomial having $\varpi_E$ as a root.
  By \cite[III.6, Cor 2]{Ser79}, $v_E(\fkD_{E/F})=v_E(f'(\varpi_E))$. The latter equals
  $$v_E\left(\sum_{i=1}^{p} i a_i \varpi_E^{i-1}\right)=\min_{1\le i\le p}v_E(i a_i \varpi_E^{i-1})
    \le v_E(p \varpi_E^{p-1})=e_{E/\Q_p}+p-1.$$

  This prepares us to tackle the case of arbitrary $n$. Choose a sequence of subextensions
  $E=F_0\supset F_1\supset \cdots \supset F_n=F$ such that $[F_m:F_{m+1}]=p$ (where $F_m/F_{m+1}$ may not be a Galois extension).
  By above, $v_{F_m}(\fkD_{F_m/F_{m+1}})\le e_{F_m/\Q_p}+p-1$ for $0\le m\le n-1$.
  Hence
  $$v_{E}(\fkD_{E/F})=\sum_{m=0}^{n-1} v_E(\fkD_{F_m/F_{m+1}})
  \le \sum_{m=0}^{n-1} p^m(e_{F_m/\Q_p}+p-1)
  = n p^n e_{F/\Q_p}+p^n-1.$$

\end{proof}

\begin{lem}\label{l:bound-diff-gen}
  Let $E$ be a finite Galois extension of $F$. Then
  $$v_E(\fkD_{E/F})\le [E:F](1+e_{F/\Q_p} \log_p [E:F])-1.$$
\end{lem}

\begin{proof}
  Let $E^t$ (resp. $E^{\ur}$) be the maximal tame (resp. unramified)
  extension of $F$ in $E$. Then
  $v_{E^t}(\fkD_{E^t/E^{\ur}})=[E^t:E^{\ur}]-1$ by \cite[III.6, Prop 13]{Ser79}.
  Clearly $v_{E^{ur}}(\fkD_{E^{\ur}/F})=0$.
  Together with Lemma \ref{l:bound-diff-wild}, we obtain
$$v_E(\fkD_{E/F})=v_{E}(\fkD_{E/E^t})+
[E:E^t]v_{E^t}(\fkD_{E^t/E^{\ur}})$$ $$\le
[E:E^t](1+e_{E^t/\Q_p}  \log_p [E:E^t] )-1
+ [E:E^t]([E^t:E^{\ur}]-1)$$
  $$=[E:E^{\ur}](1+e_{F/\Q_p} \log_p [E:E^t])-1
  \le [E:F](1+e_{F/\Q_p} \log_p [E:F])-1. $$

\end{proof}

\begin{cor}\label{c:vanishing-ram-subgroup}
  Let $E$ be a finite Galois extension of $F$. Then
  the $i$-th ramification group
 $\Gal(E/F)_{i}$ is trivial for $i=[E:F](1+e_{F/\Q_p} \log_p [E:F])-1$.
\end{cor}

\begin{proof}
  In the notation of section IV.1 of \cite{Ser79},
  we have $\Gal(E/F)_{m}=1$ by definition if
  $m=\max_{1\neq s\in \Gal(E/F)} i_G(s)$.
  But the proposition 4 in that section implies that $m\le v_E(\fkD_{E/F})$,
  so Lemma \ref{l:bound-diff-gen} finishes the proof.
\end{proof}

\subsection{Stable discrete series characters}\label{sub:st-disc}

  In \S\ref{sub:st-disc} and \S\ref{sub:EP} we specialize to the situation
of real groups.

\bit
\item $G$ is a connected reductive group over $\R$.
\item $A_{G,\infty}=A_G(\R)^0$ where $A_G$ is the maximal split torus
in the center of $G$.
% \item $M$ is a cuspidal Levi subgroup of $G$ over $\R$.
\item $K_\infty$ is a maximal compact subgroup of $G(\R)$
and $K'_\infty:=K_\infty A_{G,\infty}$.
\item $q(G):=\frac{1}{2} \dim_{\R} G(\R)/K'_\infty\in \Z_{\ge 0}$.
\item $T$ is an $\R$-elliptic maximal torus in $G$. (Assume that such a $T$ exists.)
\item $B$ is a Borel subgroup of $G$ over $\C$ containing $T$.
\item $I_\gamma$ denotes the connected centralizer of $\gamma\in G(\R)$.
\item $\Phi^+$ (resp. $\Phi$) is the set of positive (resp. all)
roots of $T$ in $G$ over $\C$.
\item $\Omega$ is the Weyl group for $(G,T)$ over $\C$, and $\Omega_c$ is the compact Weyl group.
\item $\rho:=\frac{1}{2}\sum_{\alpha\in \Phi^+} \alpha$.
\item $\xi$ is an irreducible finite dimensional algebraic representation of $G(\R)$.
\item $\lambda_\xi\in X^*(T)$ is the $B$-dominant highest weight for $\xi$.
\item $m(\xi):=\min_{\alpha\in \Phi^+} \lg \lambda_\xi+\rho,\alpha\rg$. We always have $m(\xi)>0$.
\item $\Pi_\disc(\xi)$ is the set of irreducible discrete series representations
of $G(\R)$ with the same infinitesimal character and the same
central character as $\xi$. (This is an $L$-packet for $G(\R)$.)
\item $D_\infty^G(\gamma):=\prod_{\alpha} |1-\alpha(\gamma)|$ for $\gamma\in T(\R)$, where
$\alpha$ runs over elements of $\Phi$ such that $\alpha(\gamma)\neq 1$.
(If $\gamma$ is in the center of $G(\R)$, $D_\infty^G(\gamma)=1$.)
\eit

  If $M$ is a Levi subgroup of $G$ over $\C$ containing $T$,
  the following are defined in the obvious manner as above:
  $\Phi^+_M$, $\Phi_M$, $\Omega_M$, $\rho_M$, $D_\infty^M$. Define  $\Omega^M:=\{\omega\in \Omega: \omega^{-1}\Phi^+_M\subset \Phi^+\}$, which is a set of representatives for $\Omega/\Omega_M$.
   For each \emph{regular} $\gamma\in T(\R)$,
  let us define (cf. \cite[(4.4)]{Art89})
  $$\Phi^G_M(\gamma,\xi):=(-1)^{q(G)}
  D_\infty^G(\gamma)^{1/2} D_\infty^M(\gamma)^{-1/2}\sum_{\pi\in \Pi_\disc(\xi)}
  \Theta_\pi(\gamma)$$
  where $\Theta_\pi$ is the character function of $\pi$.
  It is known that the function $\Phi^G_M(\gamma,\xi)$
  continuously extends to an $\Omega_M$-invariant function on $T(\R)$,
  thus also to a function on $M(\R)$ which is invariant under $M(\R)$-conjugation and
  supported on elliptic elements
  (\cite[Lem 4.2]{Art89}, cf. \cite[Lem 4.1]{GKM97}).
   When $M=G$, simply $\Phi^G_M(\gamma,\xi)= \tr \xi(\gamma)$.

   We would like to have an upper bound for $|\Phi^G_M(\gamma,\xi)|$
   that we will need in \S\ref{sub:weight-varies}.
   This is a refinement of \cite[Lem 4.8]{Shi-Plan}.

\begin{lem}\label{l:dim-and-trace}
\benu
\item $\dim \xi=\prod_{\alpha\in \Phi^+} \frac{\lg \alpha,\lambda_\xi+\rho\rg}{\lg \alpha,\rho\rg}$.
\item There exists a constant $c>0$ independent of $\xi$ such that for every elliptic $\gamma\in G(\R)$
and $\xi$,
  $$\frac{|\tr\xi(\gamma)|}{\dim \xi}
  \le c \frac{D_\infty^G(\gamma)^{-1/2}}{ m(\xi)^{|\Phi^+|-|\Phi^+_{I_\gamma}|}} .$$

\eenu
\end{lem}

\begin{proof}
Part (i) is the standard Weyl dimension formula. Let us prove (ii).
The formula right above the corollary 1.12 in \cite{CC09} implies that
  $$|\tr\xi(\gamma)|\le D_\infty^G(\gamma)^{-1/2} \times
\sum_{\omega\in \Omega^{I_\gamma}} \left(\prod_{\alpha\in \Phi^+_{I_\gamma}}
\frac{ \lg\omega^{-1}\alpha,\lambda_\xi+\rho\rg}{\lg \alpha,\rho_{I_\gamma}\rg}\right).$$
Note that their $M$ is our $I_\gamma$ and that
$|\alpha(\gamma)|=1$ for all $\alpha\in \Phi$ and all elliptic $\gamma\in G(\R)$.
Hence by (i),
$$\frac{|\tr\xi(\gamma)|}{\dim \xi}
  \le D_\infty^G(\gamma)^{-1/2}\sum_{\omega\in \Omega^{I_\gamma}}
  \frac  {\prod_{\alpha\in \Phi^+} \lg \alpha,\rho\rg} {\prod_{\alpha\in \Phi^+_{I_\gamma}} \lg \alpha,\rho_{I_\gamma}\rg}
  \left(\prod_{\alpha\in \Phi^+\bs \omega^{-1}\Phi^+_{I_\gamma}}
  \lg \lambda_\xi+\rho,\alpha\rg \right)^{-1}$$
$$ \le D_\infty^G(\gamma)^{-1/2}|\Omega^{I_\gamma}|  \frac{\prod_{\alpha\in \Phi^+} \lg \alpha,\rho\rg}{\prod_{\alpha\in \Phi^+_{I_\gamma}} \lg \alpha,\rho_{I_\gamma}\rg}
    m(\xi)^{-(|\Phi^+|-|\Phi^+_{I_\gamma}|)}.$$
\end{proof}

\begin{lem}\label{l:bound-for-st-ds-char}
  Assume that $M$ is a Levi subgroup of $G$ over $\R$ containing an elliptic maximal torus.
There exists a constant $c>0$ independent of $\xi$ such that for every elliptic
$\gamma\in M(\R)$,
 $$\frac{|\Phi^G_M(\gamma,\xi)|}{\dim \xi}
 \le c \frac{D_\infty^M(\gamma)^{-1/2}}{ m(\xi)^{|\Phi^+|-|\Phi^+_{I^M_\gamma}|}} .$$

\end{lem}

\begin{proof}
  As the case $M=G$ is already proved by Lemma \ref{l:dim-and-trace}.(ii), we assume that $M\subsetneq G$.
  Fix an elliptic maximal torus $T\subset M$. Since every elliptic element has a conjugate
  in $T(\R)$ and both sides of the inequality
  are conjugate-invariant,
  it is enough to verify the lemma for $\gamma\in T(\R)$.
  In this proof we borrow some notation and facts
  from \cite[pp.494-498]{GKM97} as well as \cite[pp.272-274]{Art89}.
  For the purpose of proving Lemma \ref{l:bound-for-st-ds-char}, we may restrict to
  $\gamma\in \Gamma^+$, corresponding to a closed chamber for the root system of $T(\R)$ in $G(\R)$.
  (See page 497 of \cite{GKM97} for the precise definition.)
  The proof of \cite[Lem 4.1]{GKM97} shows that
    $$\Phi^G_M(\gamma,\xi)=\sum_{\omega\in \Omega^{M}} c(\omega,\xi)\cdot \tr \xi^M_\omega(\gamma)$$
  where $\xi^M_\omega$ is the irreducible representation of $M(\R)$ of highest weight
  $\omega(\xi+\rho)-\rho_M$. We claim that there is a constant $c_1>0$  independent of $\xi$
  such that $$|c(\omega,\xi)|\le c_1$$ for all $\omega$ and $\xi$.
  The coefficients $  c(\omega,\xi)$ can be computed by rewriting the right hand
  side of \cite[(4.8)]{Art89}
  as a linear combination of $\tr \xi^M_\omega(\gamma)$ using the Weyl character formula.
  In order to verify the claim, it suffices to point out that $\ol{c}(Q^+_{ys\lambda},R^+_H)$
  in Arthur's (4.8) takes values in a finite set which is independent of $\xi$ (or
  $\tau$ in Arthur's notation). This is obvious: as $Q^+_{ys\lambda}\subset
  \Phi^\vee$ and $R^+_H\subset \Phi$, there are finitely many possibilities for
  $Q^+_{ys\lambda}$ and $R^+_H$.

  Now by Lemma \ref{l:dim-and-trace}.(i),
    $$\frac{\dim \xi^M_\omega}{\dim \xi}= \frac{\prod_{\alpha\in \Phi^+} \lg \alpha,\rho \rg}
   {\prod_{\alpha\in \Phi_M^+} \lg \alpha,\rho_M \rg} \prod_{\alpha\in
   \Phi^+\bs \Phi^+_M} \lg \alpha,\lambda_\xi+\rho\rg^{-1}
   \le c_2  m(\xi)^{-(|\Phi^+|-|\Phi^+_M|)}$$
   with $c_2=(\prod_{\alpha\in \Phi^+} \lg \alpha,\rho \rg)
   (\prod_{\alpha\in \Phi_M^+} \lg \alpha,\rho_M \rg)^{-1}>0$.
    According to Lemma \ref{l:dim-and-trace}.(ii),
    there exists a constant $c_3>0$ such that
    $$\frac{|\tr\xi^M_\omega(\gamma)|}{\dim \xi^M_\omega}
  \le c_3 \frac{D_\infty^M(\gamma)^{-1/2}}{ m(\xi)^{|\Phi^+_M|-|\Phi^+_{I^M_\gamma}|}} .$$
  To conclude the proof, multiply the last two formulas.

\end{proof}

\subsection{Euler-Poincar\'{e} functions}\label{sub:EP}

  We continue to use the notation of \S\ref{sub:st-disc}.
  Let $\ol{\mu}^{\EP}_\infty$ denote the Euler-Poincar\'{e} measure on $G(\R)/A_{G,\infty}$
  (so that its induced measure on the compact inner form has volume 1).
There exists a unique Haar measure $\mu^{\EP}_\infty$ on $G(\R)$
which is compatible with $\ol{\mu}^{\EP}_\infty$ and the standard Haar measure on $A_{G,\infty}$.
Write $\omega_{\xi}$ for the central character of $\xi$ on $A_{G,\infty}$.
  Let $\Pi(\omega_{\xi}^{-1})$ denote the set of irreducible admissible representations of $G(\R)$
  whose central characters on $A_{G,\infty}$ are $\omega_{\xi}^{-1}$.
  For $\pi \in \Pi(\omega_{\xi}^{-1})$,
  define
  $$\chi_{\EP}(\pi\otimes \xi):=\sum_{i\ge 0} (-1)^i \dim H^i(\Lie G(\R),K'_\infty,
  \pi \otimes \xi).$$
  Clozel and Delorme (\cite{CD90}) constructed a bi-$K_\infty$-finite
  function $\phi_\xi\in C^\infty(G(\R))$ which transforms under $A_{G,\infty}$ by $\omega_{\xi}$ and is compactly supported modulo $A_{G,\infty}$,
  such that
  $$\forall \pi \in \Pi(\omega_{\xi}^{-1}),\quad \tr \pi(\phi_\xi,\mu^{\EP}_\infty)=\chi_{\EP}(\pi\otimes\xi).$$
  The following are well-known:
\bit
\item $\chi_{\EP}(\pi\otimes\xi)=0$ unless $\pi\in \Pi(\omega_{\xi}^{-1})$ has the same infinitesimal character as $\xi^\vee$.
\item If the highest weight of $\xi$ is regular then $\chi_{\EP}(\pi\otimes \xi)\neq 0$
if and only if $\pi\in \Pi_{\disc}(\xi^\vee)$.
\item If $\pi \in \Pi(\omega_{\xi}^{-1})$ is a discrete series and $\chi_{\EP}(\pi\otimes \xi)\neq 0$
then $\pi\in \Pi_{\disc}(\xi^\vee)$ and $\chi_{\EP}(\pi\otimes \xi)=(-1)^{q(G)}$.
More precisely, $\dim H^i(\Lie G(\R),K'_\infty,
  \pi \otimes \xi)$ equals 1 if $i=q(G)$ and 0 if not.

\eit

\subsection{Canonical measures and Tamagawa measures}\label{sub:can-measure}

  We return to the global setting so that $F$ and
 $G$ are as in \S\ref{sub:Gross-motive}. Let $G_\infty:=(\Res_{F/\Q}G)\times_\Q \R$, to which
 the contents of \S\ref{sub:st-disc} and \S\ref{sub:EP} apply. In particular we have a measure $\mu^{\EP}_\infty$ on $G_\infty(\R)$.
  For each finite place $v$ of $F$,
  define $\mu^{\can}_v:=\Lambda(\Mot_{G_{v}}^\vee(1))\cdot |\omega_{G_v}|$
  in the notation of \cite{Gro97} where
 $|\omega_{G_v}|$ is the ``canonical'' Haar measure on $G(F_v)$ as in \S11 of that article.
  When $G$ is unramified over $F_v$, the measure
  $\mu^{\can}_v$ assigns volume 1 to a hyperspecial subgroup of $G(F_v)$.
  In particular, $$\mu^{\can,\EP}:=\prod_{v\nmid \infty} \mu^{\can}_v \times \mu^{\EP}_\infty$$ is a well-defined measure
  on $G(\A_F)$.
 % The Lebesgue measure on $A_{G,\infty}$ and $\prod_{v|\infty}\mu^{\can}_v$
 % induce a quotient measure $\ol{\mu}^{\can}_\infty$ on $G(F_\infty)/A_{G,\infty}$.

  Let $\ol{\mu}^{\Tama}$ denote the Tamagawa measure
  on $G(F)\bs G(\A_F)/A_{G,\infty}$, so that its volume is the Tamagawa number (cf. \cite[p.629]{Kot88})
\beq \label{e:Tamagawa}\tau(G):=
  \ol{\mu}^{\Tama}(G(F)\bs G(\A_F)/A_{G,\infty})=|\pi_0(Z(\hat{G})^{\Gal(\ol{F}/F)})\cdot
  |\ker^1(F,Z(\hat{G}))|^{-1}.\eeq
 The Tamagawa measure $\mu^{\Tama}$ on $G(\A_F)$ of \cite{Gro97} is compatible with
 $\ol{\mu}^{\Tama}$ if $G(F)$ and $A_{G,\infty}$ are equipped with
  the point-counting measure and the Lebesgue measure, respectively.
    The ratio of two Haar measures on $G(\A_F)$ is computed as:

\begin{prop}(\cite[(10.5)]{Gro97}
  $$\frac{\mu^{\can,\EP}}{\mu^{\Tama}}
  = \frac{ L(\Mot_G)\cdot |\Omega|/|\Omega_c|}{e(G_\infty) 2^{\rank_{\R} G_\infty}   }.$$
\end{prop}

  The following notion will be useful in that the Levi subgroups contributing to the trace formula
  in \S\ref{s:aut-Plan-theorem}
  turn out to be the cuspidal ones.

\begin{defn}
  We say that $G$ is cuspidal if $G_0:=\Res_{F/\Q}G$ satisfies the condition that
  $A_{G_0}\times _\Q \R$ is the maximal split torus in the center of $G_0\times _\Q \R$.
\end{defn}

  Assume that $G$ is cuspidal, so that $G(\R)/A_{G,\infty}$ contains a maximal
  $\R$-torus which is anisotropic. % Its Euler-Poincar\'{e} measure is
%  denoted by $\ol{\mu}^{\EP}_\infty$, which
%  is nonzero (\cite[p.136]{Ser71}, cf. \cite[Lem 7.3]{Gro97}).

%\begin{prop}\label{p:can-EP-at-infty}
%  $$\frac{\ol{\mu}^{\can}_\infty}{\ol{\mu}^{\EP}_\infty}= \frac{(2\pi)^{-\dim A_G}}{|\Omega/\Omega_c|}
%  \prod_{d\ge 1} \left( \frac{(2\pi)^d}{(d-1)!} \right)^{\dim \Mot_{G,d}}.$$
%\end{prop}

%\begin{proof} When $G$ has compact center ($A_G=1$), this follows from
%the proof of the proposition 7.6 in \cite{Gro97} (see (7.4) and (7.7) there).
%In general we reduce to the case of $G/A_G$, which has compact center,
%by noting that $\Mot_{G,d}$ is $\Mot_{G/A_G,d}$ if $d>1$ and
%$\Q^{\dim A_G}+ \Mot_{G/A_G,d}$ if $d=1$.
%\end{proof}

\begin{cor}\label{c:canonical-measure} $$\frac{\ol{\mu}^{\can,\EP}(G(F)\bs G(\A_F)/A_{G,\infty})}{
  \ol{\mu}^{\EP}_\infty (\ol{G}(F_\infty)/A_{G,\infty})}=
 \frac{\tau(G) \cdot  L(\Mot_G)  \cdot |\Omega|/|\Omega_c|}{e(G_\infty) 2^{[F:\Q] r_G} }.
   $$
\end{cor}

\begin{proof}
  It suffices to remark that
  the Euler-Poincar\'{e} measure on a compact Lie group has total volume 1,
  hence  $\ol{\mu}^{\EP}_\infty (\ol{G}(F_\infty)/A_{G,\infty})=1$.
\end{proof}

%\begin{rem}
%  Later the right hand side of Proposition \ref{p:can-EP-at-infty} appears in the trace formula
%  with $Z_M(\gamma)^0$ in place of $G$. It would be useful later that the ratio
%  depends only on $Z_M(\gamma)^0(\R)$, which is anisotropic modulo $A_{M,\infty}$.
%\end{rem}

\subsection{Bounds for Artin $L$-functions} For later use we estimate the $L$-value $L(\Mot_G)$ in Corollary \ref{c:canonical-measure}.

\begin{prop}\label{prop:Brauer}
Let $s \ge 1$ and $E$ be a Galois extension of $F$ of degree $[E:F]\le s$.	
\begin{enumerate}[(i)]
\item  For all $\epsilon>0$ there exists a constant $c=c(\epsilon,s,F)>0$ which depends only on $\epsilon$, $s$ and $F$ such that the following holds: For all non-trivial irreducible representations $\rho$ of $\Gal(E/F)$,
\begin{equation*}
  cd_E^{-\epsilon}\le L(1,\rho) \le c d_E^\epsilon.
\end{equation*}
 \item The same inequalities hold for the residue $\mathrm{Res}_{s=1} \zeta_E(s)$ of the Dedekind zeta function of $E$.
 \item There is a constant $A_1=A_1(s,F)>0$ which depends only on $s$ and $F$ such that for all \textsf{faithful} irreducible representation $\rho$ of $\Gal(E/F)$,
  \begin{equation*}
	d_{E/F}^{A_1} \le \N_{F/\Q}(\Fmf_\rho) \le d_{E/F}^{1/\dim(\rho)},
  \end{equation*}
  where $d_{E/F}=\N_{F/\Q}(\fkD_{E/F})$ is the relative discriminant of $E/F$; recall that $d_E= d_F^{[E:F]} d_{E/F}$.
\end{enumerate}
\end{prop}

\begin{proof}
  The assertion (ii) is Brauer--Siegel theorem~\cite{Brauer47}*{Theorem~2}. We also note the implication (i) $\Rightarrow$ (ii) which follows from the formula
  \begin{equation} \label{Brauer:product}
	\zeta_E(s)= \prod_{\rho} L(s,\rho)^{\dim \rho}.
  \end{equation}
where $\rho$ ranges over all irreducible representations of $\Gal(E/F)$.

  The proof of assertion (i) is reduced to the $1$-dimensional case by Brauer induction as in~\cite{Brauer47}. In this reduction one uses the fact that if $E'/F'$ is a subextension of $E/F$ then the absolute discriminant $d_{E'}$ of $E'$ divides the absolute discriminant $d_{E}$ of $E$. Also we may assume that $E'/F'$ is cyclic. For a character $\chi$ of $\Gal(E'/F')$ we have the convexity bound $L(1,\chi)\le c d^\epsilon_{E'}$ (Landau). The lower bound for $L(1,\chi)$ follows from (ii) and the product formula~\eqref{Brauer:product}.

  In the assertion (iii) the right inequality follows from the discriminant-conductor formula which implies that $\Fmf^{\dim (\rho)}_\rho\mid \fkD_{E/F}$. The left inequality follows from local considerations. Let $v$ be a finite place of $F$ dividing $\fkD_{E/F}$; since $\rho$ is faithful, its restriction to the inertia group above $v$ is non-trivial and therefore $v$ divides $\Fmf_\rho$. Since $v(\fkD_{E/F})$ is bounded above by a constant $A_1(s,F)$ depending only on $[E:F]\le s$ and $F$ by Lemma~\ref{l:bound-diff-gen}, we have $ v(\fkD_{E/F}) \le  A_1 v(\Fmf_\rho)$ which concludes the proof.
\end{proof}

\begin{cor}\label{c:bound-on-L(1)}
   For all integers $R, D,s\in \Z_{\ge 1}$, and $\epsilon>0$ there is a constant $c_1=c_1(\epsilon,R,D,s,F)>0$ (depending on $R$, $D$, $s$, $F$ and $\epsilon$) with the following property
\benu
\item For any $G$ such that $r_G\le R$, $\dim G\le D$,
  $Z(G)$ is $F$-anisotropic, and $G$ splits over a Galois extension of $F$ of degree $\le s$,
$$|L(\Mot_G)|\le c_1 \prod_{d=1 }^{\lfloor \frac{d_G+1}{2}\rfloor}
  \N_{F/\Q}(\fkf(\Mot_{G,d}))^{d-\Mdemi+\epsilon}.$$
\item  There is a constant $A_{20}=A_{20}(R,D,s,F)$ such that for any $G$ as in (i),
  $$|L(\Mot_G)|\le c_1 \prod_{v\in \Ram(G)}
  q_v^{A_{20}}.$$
\eenu
The choice $A_{20}= \frac{(D+1)Rs}{2} \max\limits_{\text{prime $p$}} (1+e_{F_v/\Q_p}\log s)$ is admissible.
\end{cor}

\begin{proof} The functional equation for $\Mot_G$ reads
  $$L(\Mot_G)=L(\Mot_G^\vee(1))\epsilon(\Mot_G)\cdot \frac{L_\infty(\Mot_G^\vee(1))}{L_\infty(\Mot_G)}$$
where   $ \epsilon(\Mot_G)=|\Delta_F|^{d_G/2}\prod_{d\ge 1} \N_{F/\Q}(\fkf(\Mot_{G,d}))^{d-\frac{1}{2}}$.

The (possibly reducible) Artin representation for $\Mot_{G,d}$
  factors through $\Gal(E/F)$ with $[E:F]\le s$ by the assumption.
  Let $A_1=A_1(s,F)$ be as in (iii) of Proposition \ref{prop:Brauer}.
For all $\epsilon>0$, (i) of Proposition~\ref{prop:Brauer} implies that there is a constant $c=c(\epsilon,s,F)>1$
  depending only on $s$ and $F$ such that
  $$|L(\Mot_G^\vee(1))|\le \prod_{d\ge 1} \left( c \N_{F/\Q}(\fkf(\Mot_{G,d}))^{A_1\epsilon}\right)^{\dim \Mot_{G,d}}
  \le c^{r_G}\prod_{d\ge 1}  \N_{F/\Q}(\fkf(\Mot_{G,d}))^{\epsilon A_1 r_G}.$$
  Formula (7.7) of \cite{Gro97}, the first equality below, leads to the following bound since only
$1\le d \le \lfloor \frac{d_G+1}{2}\rfloor$ can contribute in view of Proposition \ref{p:Gross-motives}.(iii).
\begin{equation*}
\left|\frac{L_\infty(\Mot^\vee_G(1))}{L_\infty(\Mot_G)}\right| = 2^{-[F:\Q]r_G}
\prod_{d\ge 1} \left( \frac{(d-1)!}{(2\pi)^d}\right)^{\dim \Mot_{G,d}}
\le 2^{-[F:\Q]r_G} \left( \lfloor \frac{d_G-1}{2}\rfloor !\right)^{r_ G}.
\end{equation*}
  Set $c_1(R,D,s,F,\epsilon):= |\Delta_F|^{D/2} 2^{-[F:\Q]R} \left(
\lfloor \frac{D-1}{2}\rfloor !\right)^{R}$. Then we see that
\begin{equation*}\begin{aligned}
  |L(\Mot_G)|&\le c_1 \prod_{d=1 }^{\lfloor \frac{d_G+1}{2}\rfloor}
  \N_{F/\Q}(\fkf(\Mot_{G,d}))^{d-\Mdemi+\epsilon} \\
  & = c_1 \prod_{v\in \Ram(G)} \prod_{d=1 }^{\lfloor \frac{d_G+1}{2}\rfloor}
  q_v^{(d-\Mdemi+\epsilon) \cdot f(G_v,d)}.
\end{aligned}\end{equation*}
This concludes the proof of (i).

According to Lemma \ref{l:bounding-conductor}, the exponent in the right hand side is bounded by
\begin{equation*}
  d f(G_v,d) \le \frac{D+1}{2} \dim \Mot_{G,d} \cdot (s(1+e_{F_v/\Q_p}\log s)-1).
\end{equation*}
 (we have chosen $\epsilon=\Mdemi$). The proof of (ii) is concluded by the fact that $$\sum_{d\ge 1}\dim \Mot_{G,d}=r_G\le R,$$ see Proposition \ref{p:Gross-motives}.(ii).
\end{proof}

\begin{cor}\label{c:bound-on-L(1)-2}
  Let $G$ be a connected cuspidal reductive group over $F$ with anisotropic center. Then there exist constants
  $c_2=c_2(G,F)>0$ and $A_2(G,F)>0$ depending only on $G$ and $F$ such that:
  for any cuspidal $F$-Levi subgroup $M$ of $G$ and any semisimple $\gamma\in M(F)$
which is elliptic in $M(\R)$,
 $$|L(\Mot_{I^M_\gamma})|\le c_2 \prod_{v\in \Ram(I_\gamma^M)}
 q_v^{A_2}$$
where $I^M_\gamma$ denote the connected centralizer of $\gamma$ in $M$.
The following choice is admissible: $$A_2= \frac{(d_G+1)r_Gw_G s_G}{2} \max\limits_{\text{prime $p$}} (1+e_{F_v/\Q_p}\log w_Gs_G).$$
\end{cor}

\begin{proof}
    According to Lemma \ref{l:torus-splitting}, %below (which is independent of what is done in this article)
 $s^{\spl}_{ I^M_\gamma}\le w_Gs_G$.
  Apply Corollary \ref{c:bound-on-L(1)} for each $I^M_\gamma$ with $R=r_G$, $D=d_G$ and $s=w_Gs_G$
to deduce the first assertion, which obviously implies the last assertion. Note that
$\rank I^M_\gamma\le r_G$ and that $\dim I^M_\gamma\le d_G$.
\end{proof}

Instead of using the Brauer--Siegel theorem which is ineffective, we could use the estimates by Zimmert~\cite{Zimmert:regulator} for the size of the regulator of number fields. This yields an effective estimate for the constants $c_2$ and $c_3$ above, at the cost of enlarging the value of the exponents $A_1$ and $A_2$.

\subsection{Frobenius--Schur indicator}
\label{sec:b:fs}

The Frobenius--Schur indicator is an invariant associated to an irreducible representation. It may take the three values $1,0,-1$. This subsection gathers several well-known facts and recalls some familiar constructions.

The Frobenius--Schur indicator can be constructed in greater generality but the following setting will suffice for our purpose. We will only consider finite dimensional representations on vector spaces over $\BmC$ or $\BmR$. The representations are continuous (and unitary) from compact Lie groups or algebraic from linear algebraic groups (these are in fact closely related by the classical \Lquote{unitary trick} of Hurwitz and Weyl).

 Let $G$ be a compact Lie group and denote by $\mu$ the Haar probability measure on $G$. Let $(V,r)$ be a continuous irreducible representation of $G$. Denote by $\chi(g)=\MTr(r(g))$ its character.

\begin{defn} The \key{Frobenius--Schur indicator} of an irreducible representation $(V,r)$ of $G$ is defined by
\begin{equation*}
s(r):= \int_G \chi(g^2) d\mu(g).
\end{equation*}
We have that $s(r)\in \set{-1,0,1}$ always.
\end{defn}

\begin{rem}
  More generally if $G$ is an arbitrary group but $V$ is still finite dimensional, then $s(r)$ is defined as the multiplicity of the trivial representation in the virtual representation on $\MSym^2 V - \wedge^2 V$. This is consistent with the above definition.
\end{rem}

\begin{rem}\begin{enumerate}[(i)]
  \item Let $(V^\vee,r^\vee)$ be the dual representation of $G$ in the dual $V^\vee$. It is easily seen that $s(r)=s(r^\vee)$.
  \item If $G=G_1\times G_2$ and $r$ is the irreducible representation of $G$ on $V=V_1\otimes V_2$ where $(V_1,r_1)$ and $(V_2,r_2)$ are irreducible representations of $G_1$ (resp. $G_2$), then $s(r)=s(r_1)s(r_2)$.
\end{enumerate}
\end{rem}

The classical theorem of Frobenius and Schur says that $r$ is a real, complex or quaternionic representation if and only if $s(r)=1,0$ or $-1$ respectively. We elaborate on that dichotomy in the following three lemma.

\begin{lem}[Real representation]\label{lem:real} Let $(V,r)$ be an irreducible representation of $G$.
The following assertions are equivalent:
\begin{enumerate}[(i)]
\item $s(r)=1$;
\item $r$ is self-dual and defined over $\BmR$ in the sense that $V\simeq V_0\otimes_\BmR \BmC$ for some irreducible representation on a real vector space $V_0$. (Such an $r$ is said to be a \key{real representation};)
\item $r$ has an invariant real structure. Namely there is a $G$-invariant anti-linear map $j:V\to V$ which satisfies $j^2=1$.
\item $r$ is self-dual and any bilinear form on $V$ that realizes the isomorphism $r\simeq r^\vee$ is symmetric;
\item $\MSym^2 V$ contains the trivial representation (then the multiplicity is exactly one).
\end{enumerate}
\end{lem}

We don't repeat the proof here (see e.g.~\cite{book:serre:rep}) and only recall some of the familiar constructions. We have a direct sum decomposition
\begin{equation*}
 V\otimes V = \MSym^2 V \oplus \wedge^2 V.
\end{equation*}
The character of the representation $V\otimes V$ is $g\mapsto \chi(g)^2$. By Schur lemma the trivial representation occurs in $V\otimes V$ with multiplicity at most one. In other words the subspace of invariant vectors of $V^\vee\otimes V^\vee$ is at most one. Note that this subspace is identified with $\MHom_G(V,V^\vee)$ which is also the subspace of invariant bilinear forms on $V$.

The character of the representation $\MSym^2 V$ (resp. $\wedge^2 V$) is
\begin{equation*}
\text{%
$\Mdemi(\chi(g)^2+\chi(g^2))$\quad  resp.\quad  $\Mdemi(\chi(g)^2-\chi(g^2))$.
}
\end{equation*}
From that the equivalence of (i) with (v) follows because the multiplicity of the trivial representation in $\MSym^2 V$ (resp. $\wedge^2 V$) is the mean of its character. The equivalence of (iv) and (v) is clear because a bilinear form on $V$ is an element of $V^\vee\otimes V^\vee$ and it is symmetric if and only if it belongs to $\MSym^2 V^\vee$.

The equivalence of (ii) and (iii) follows from the fact that $j$ is induced by complex conjugation on $V_0\otimes_\BmR \BmC$ and conversely $V_0$ is the subspace of fixed points by $j$. Note that a real representation is isomorphic to its complex conjugate representation because $j$ may be viewed equivalently as a $G$-isomorphism $V\to \overline{V}$. Since $V$ is unitary the complex conjugate representation $\overline{r}$ is isomorphic to the dual representation $r^\vee$. In assertion (ii) one may note that the endomorphism ring of $V_0$ is isomorphic to $\BmR$.

\begin{lem}[Complex representation]  Let $(V,r)$ be an irreducible representation of $G$.
The following assertions are equivalent:
\begin{enumerate}[(i)]
\item $s(r)=0$;
\item $r$ is not self-dual;
\item $r$ is not isomorphic to $\overline{r}$; (Such an $r$ is called a \key{complex representation};)
\item $V\otimes V$ does not contain the trivial representation.
\end{enumerate}
\end{lem}

We note that for a complex representation, the restriction $\Res_{\BmC/\BmR} V$ (obtained by viewing $V$ as a real vector space) is an irreducible real representation of twice the dimension of $V$. Its endomorphism ring is isomorphic to $\BmC$.

\begin{lem}[Quaternionic/symplectic representation]\label{lem:quaternionic} Let $(V,r)$ be an irreducible representation of $G$.
The following assertions are equivalent:
\begin{enumerate}[(i)]
\item $s(r)=-1$;
\item $r$ is self-dual and cannot be defined over $\R$.
\item $r$ has an invariant quaternionic structure. Namely there is a $G$-invariant anti-linear map $j:V\to V$ which satisfies $j^2=-1$.
(Such an $r$ is called a \key{quaternionic representation}.)
\item $r$ is self-dual and the bilinear form on $V$ that realizes the isomorphism $r\simeq r^\vee$ is antisymmetric.
 (Such an $r$ is said to be a \key{symplectic representation};)
\item $\bigwedge^2 V$ contains the trivial representation (the multiplicity is exactly one).
\end{enumerate}
\end{lem}

The equivalence of (iii) and (iv) again comes from the fact that $V$ is unitarizable (because $G$ is a compact group). In that context the notion of symplectic representation is identical to the notion of quaternionic representation. Note that for a quaternionic representation, the restriction $\Res_{\BmC/\BmR} V$ is an irreducible real representation of twice the dimension of $V$. Furthermore its ring of endomorphisms is isomorphic to the quaternion algebra $\BmH$. Indeed the endomorphism ring contains the (linear) action by $i$ because $V$ is a representation over the complex numbers and together with $j$ and $k=ij$ this is the standard presentation of $\BmH$.

From the above discussions we see that the Frobenius--Schur indicator can be used to classify irreducible representations over the reals. The endomorphism ring of an irreducible real representation is isomorphic to either $\BmR,\BmC$ or $\BmH$ and we have described a correspondence with associated complex representations.

\section{A uniform bound on orbital integrals}\label{s:app:unif-bound}

  This section is devoted to showing an apparently new result on the uniform bound on orbital integrals
  evaluated at semisimple conjugacy classes and basis elements of unramified Hecke algebras.
  Our bound is uniform in the finite place $v$ of a number field (over which the group is defined),
  the ``size'' of (the support of) the basis element for the unramified Hecke algebra at $v$
  as well as the conjugacy class at $v$.

  The main result is Theorem \ref{t:appendeix2}, which
  is invoked in \S\ref{sub:weight-varies}. The main local input for Theorem \ref{t:appendeix2} is Proposition \ref{p:appendix2}. The technical heart in the proof of the proposition is postponed to \S\ref{sub:app:elliptic}, which the reader may want to skip in the first reading. In Appendix \ref{s:app:B} we discuss an alternative approach to Theorem \ref{t:appendeix2} via motivic integration.

\subsection{The main local result}\label{sub:local-bound-orb-int}

  We begin with a local assertion with a view toward Theorem \ref{t:appendeix2} below.
  Let $G$ be a connected reductive group over a finite extension $F$ of $\Q_p$ with a maximal $F$-split torus $A$.
  As usual $\cO$, $\varpi$, $k_F$ denote the integer ring, a uniformizer and the residue field. Let $\bG$ be the Chevalley group for $G\times_{F} \ol{F}$, defined over $\Z$.
 Let $\bB$ and $\bT$ be a Borel subgroup and a maximal torus of $\bG$ such that
  $\bB\supset \bT$. We assume that
  \bit\item $G$ is unramified over $F$,
  \item $\cha k_F>w_Gs_G$ and $\cha k_F$ does not divide the finitely many constants in the Chevalley commutator relations (namely $C_{ij}$ of \eqref{e:commutator}).
  \eit
  (We assume $\cha k_F>w_Gs_G$ to ensure that any maximal torus of $G$ splits over a finite tame extension, cf. \S\ref{sub:app:elliptic} below. The latter assumption on $\cha k_F$ depends only on $\bG$.) Fix a smooth reductive model over $\cO$ so that $K:=G(\cO)$ is a hyperspecial subgroup of $G(F)$. Fix a Borel subgroup $B$ of $G$ whose Levi factor is the centralizer of $A$ in $G$. Denote by $v:F^\times \ra \Q$ the discrete valuation normalized by $v(\varpi)=1$ and by $D^G$ the Weyl discriminant function, cf. \eqref{e:D^G} below. Set $q_v:=|k_F|$.

  Suppose that there exists a closed embedding of algebraic groups $\Xi^{\spl}:\bG\hra \GL_m$ defined over $\cO$
  such that $\Xi^{\spl}(\bT)$ (resp. $\Xi^{\spl}(\bB)$) lies in the group of diagonal (resp. upper triangular) matrices. This assumption will be satisfied by Lemma \ref{l:conj-image-in-diag} and Proposition \ref{p:global-integral-model}, or alternatively as explained at the start of \S\ref{sub:lem-split}.
  The assumption may not be strictly necessary but is convenient to have for some later arguments.
  In the setup of \S\ref{sub:global-bound-orb-int} such a $\Xi^{\spl}$ will be chosen globally
  over $\Z[1/Q]$ (i.e. away from a certain finite set of primes), which gives rise to an embedding over $\cO$ if $v$ does not divide $Q$.

\begin{prop}\label{p:appendix2}

% Fix $\theta\in \mC(\Gamma_1)$ and
   There exist $a_{G,v},b_{G,v},e_{G,v}\ge0$ (depending on $F$, $G$ and $\Xi^{\spl}$) such that
   \bit
   \item for every semisimple $\gamma\in G(F)$,
   \item for every $\lambda\in X_*(A)$ and $\kappa\in\Z_{\ge0}$ such that $\|\lambda\|\le\kappa$,
   \eit
  \beq\label{e:app:prop} 0\le O_\gamma(\tau^G_\lambda,\mu^{\can}_G,\mu^{\can}_{I_\gamma})
  \le  q_v^{ a_{G,v}+b_{G,v}\kappa}\cdot D^G(\gamma)^{-e_{G,v}/2}.\eeq
\end{prop}

\begin{rem}
  We chose the notation $a_{G,v}$ etc rather than $a_{G,F}$ etc in anticipating the global setup
  of the next subsection
  where $F$ is the completion of a number field at the place $v$.
\end{rem}

\begin{proof}
  For simplicity we will omit the measures chosen to compute orbital integrals
  when there is no danger of confusion.
  Let us argue by induction on the semisimple rank $r^{\semis}_G$ of $G$.
  In the rank zero case, namely when $G$ is a torus,
  the proposition is true since $O_\gamma(\tau^G_\lambda)$ is equal to 0 or 1.
  Now assume that $r^{\semis}_G\ge 1$ and that the proposition is known for all groups
  whose semisimple ranks are less than $r^{\semis}_G$.
  In the proof we write $a_G$, $b_G$, $e_G$ instead of $a_{G,v}$, $b_{G,v}$, $e_{G,v}$ for simplicity.

\medskip

\noindent\underline{Step 1}. Reduce to the case where $Z(G)$ is anisotropic.

  Let $A_G$ denote the maximal split torus in $Z(G)$. Set $\ol{G}:=G/A_G$.
  The goal of Step 1 is to show that if the proposition for $\ol{G}$ then it also holds
for $G$. We have an exact sequence of algebraic groups over $\cO$: $1\ra A_G\ra G\ra \ol{G}\ra 1$.
By taking $F$-points one obtains an exact sequence of groups
$$1 \ra A_G(F) \ra G(F) \ra \ol{G}(F) \ra 1,$$
where the surjectivity is implied by Hilbert 90 for $A_G$.
(In fact $G(\cO)\ra \ol{G}(\cO)$ is surjective since it is surjective on $k_F$-points and
$G\ra \ol{G}$ is smooth, cf. \cite[p.386]{Kot86}, but we do not need this.)
For any semisimple $\gamma\in G(F)$, denote its image in $\ol{G}(F)$ by $\ol{\gamma}$.
 The connected centralizer of $\ol{\gamma}$ is denoted
$\ol{I}_{\ol{\gamma}}$. There is an exact sequence
$$1\ra A_G(F) \ra I_\gamma(F) \ra \ol{I}_{\ol{\gamma}}(F)\ra 1.$$
We see that $G(F) \ra \ol{G}(F)$ induces a bijection
$I_\gamma(F)\bs G(F) \simeq \ol{I}_{\ol{\gamma}}(F)\bs \ol{G}(F)$.
Let $A$ be a maximal $F$-split torus of $G$, and $\ol{A}$ be its image in $\ol{G}$.
For any $\lambda\in X_*(A)$, denote its image in $X_*(\ol{A})$ by $\ol{\lambda}$.
Then
$$ O^{G(F)}_\gamma(\tau^G_\lambda,\mu^{\can}_G,\mu^{\can}_{I_\gamma})
\le O^{\ol{G}(F)}_{\ol{\gamma}}(\tau^{\ol{G}}_{\ol{\lambda}},\mu^{\can}_{\ol{G}},\mu^{\can}_{\ol{I}_{\ol{\gamma}}}).$$
Indeed, this follows from the fact that $I_\gamma(F)\bs G(F) \simeq \ol{I}_{\ol{\gamma}}(F)\bs \ol{G}(F)$
carries $\frac{\mu^{\can}_G}{\mu^{\can}_{I_\gamma}}$ to $\frac{\mu^{\can}_{\ol{G}}}{\mu^{\can}_{\ol{I}_{\ol{\gamma}}}}$.
%and that $G(\cO)\ra \ol{G}(\cO)$ is onto (so that $\tau^G_{\lambda}(x^{-1}\gamma x)=1$
%if and only if $\tau^{\ol{G}}_{\ol{\lambda}}(\ol{x}^{-1}\ol{\gamma} \ol{x})=1$, where
%$\ol{x}$ is the image of $x$).
As the proposition is assumed to hold for $\ol{G}$,
the right hand side is bounded by $q_v^{ a_{\ol{G}}+b_{\ol{G}}\kappa}\cdot D^{\ol{G}}(\gamma)^{-e_{\ol{G}}/2}=
q_v^{ a_{\ol{G}}+b_{\ol{G}}\kappa}\cdot D^{G}(\gamma)^{-e_{\ol{G}}/2}$. Hence the
proposition holds for $G$ if we set
$a_G=a_{\ol{G}}$, $b_G=b_{\ol{G}}$ and $e_G=e_{\ol{G}}$. This finishes Step 1.

\medskip

\noindent\underline{Step 2}. When $Z(G)$ is anisotropic.

  The problem will be divided into 3 cases depending on $\gamma$. In each case we find a sufficient condition on $a_G$, $b_G$ and $e_G$ for \eqref{e:app:prop} to be true.

\medskip

\noindent\underline{Step 2-1}. When $\gamma\in Z(G)(F)$.

  In this case the proposition holds for any $a_G,b_G,e_G\ge0$ since $O_\gamma(\tau^G_\lambda)=0$ or $1$ and $D^G(\gamma)=1$.

\medskip

\noindent\underline{Step 2-2}. When $\gamma$ is non-central and non-elliptic.

Then there exists a nontrivial split torus $S\subset Z(I_\gamma)$.
  Set $M:=Z_G(S)$, which is an $F$-rational Levi subgroup of $G$. Then $I_\gamma\subset M \subsetneq G$. Note that $\gamma$ is $(G,M)$-regular.
Lemma \ref{l:orb-int-const-term} reads
\beq\label{e:app:G-M}O^{G(F)}_\gamma(\triv_{K\lambda(\varpi)K})=
D^G_M(\gamma)^{-1/2} O^{M(F)}_\gamma((\triv_{K\lambda(\varpi)K})_M).\eeq
By conjugation we may assume without loss of generality that
$\lambda(\varpi)\in M(F)$. (To justify, find $x\in G(F)$ such that
$xMx^{-1}$ contains $A$. Then $\lambda(\varpi)\in xM(F)x^{-1}$ and
$O^M_\gamma= O^{xMx^{-1}}_{x\gamma x^{-1}}$.) Moreover by conjugating $\lambda$ we may assume that $\lambda$ is $B\cap M$-dominant.
We can write \beq\label{e:App-Step3}(\triv_{K\lambda(\varpi)K})_M=\sum_{\mu\le_{\R}\lambda}
c_{\lambda,\mu}\triv_{K_M\mu(\varpi)K_M}.\eeq
 The ordering in the sum is relative to $B\cap M$. For any $m=\mu(\varpi)$, $c_{\lambda,\mu}$ is equal to
$$(\triv_{K\lambda(\varpi)K})_M(m)=\delta_P(m)^{1/2}\int_{N(F)} \triv_{K\lambda(\varpi)K}(mn)dn
=q_v^{\lg \rho_P,\mu\rg} \mu_G^{\can}(mN(F)K\cap K\lambda(\varpi)K).$$
Lemma \ref{l:double-coset-volume} and the easy inequality $\lg \rho_P,\mu\rg \le \lg \rho,\lambda\rg$
 allow us to deduce that
$$0\le c_{\lambda,\mu}\le q_v^{\lg \rho_P,\mu\rg}
\mu^{\can}_G(K\lambda(\varpi)K)\le q_v^{d_G+r_G+2\lg \rho,\lambda\rg}.$$
The sum in \eqref{e:App-Step3} runs over the set of $$\mu=\lambda-\sum_{\alpha\in \Delta^+}
a_\alpha\cdot \alpha~\mbox{with}~a_\alpha\in \frac{1}{\delta_G}\Z,~a_\alpha\ge 0$$
such that $\mu\in (X^*(T)_\R)^+$. Here we need to explain $\delta_G$: If
$\mu\le_{\R} \lambda$ then $\lambda-\mu$ is a linear combination of positive coroots with nonnegative rational
coefficients. The denominators of such coefficients under the constraint $c_{\lambda,\mu}\neq 0$ are uniformly bounded, where the bound depends on the
coroot datum. We write $\delta_G$ for this bound.

The above condition on $\mu$ and $\|\lambda\|\le \kappa$ imply that $a_\alpha\le \kappa$.
We get, by using the induction hypothesis for $O^M_\gamma$, $$ 0\le
O^{M(F)}_\gamma((\triv_{K\lambda(\varpi)K})_M)\le
\sum_{\mu\le_{\R}\lambda} c_{\lambda,\mu} O^M_{\gamma} (\triv_{K_M\mu(\varpi)K_M})
\le \sum_{\mu\le_{\R}\lambda} c_{\lambda,\mu}
 q_v^{ a_M+b_M\kappa}\cdot D^M(\gamma)^{-e_M/2}$$
$$\le
(\delta_G(\kappa+1))^{|\Delta^+|} q_v^{d_G+r_G+2\lg \rho,\lambda\rg}q_v^{ a_M+b_M\kappa}\cdot D^M(\gamma)^{-e_M/2} $$
$$ \le q_v^{d_G+r_G(\delta_G\kappa+\delta_G+1)+2\lg \rho,\lambda\rg+a_M+b_M\kappa}  D^M(\gamma)^{-e_M/2}. $$
Set
\[
%\label{e:c_G}
c_G:=d_G+r_G(\delta_G+1)+2\lg \rho,\lambda\rg \le d_G+r_G(\delta_G+1)+|\Phi^+|\kappa.
\] % Then $c_G$ depends only on $G\times_F \ol{F}$.
In view of \eqref{e:app:G-M} it suffices to find $a_G,b_G,e_G\ge 0$ such that
$$D^G_M(\gamma)^{-1/2}D^M(\gamma)^{-e_M/2} q_v^{a_M+c_G+(b_M+r_G\delta_G)\kappa}
\le D^G(\gamma)^{-e_G/2} q_v^{a_G+b_G\kappa}$$
or equivalently
\beq\label{e:app:suff-cond}
D^G_M(\gamma)^{\frac{e_G-1}{2}}D^M(\gamma)^{\frac{e_G-e_M}{2}}
\le q_v^{a_G-a_M-c_G+(b_G-b_M-r_G\delta_G)\kappa}\eeq
whenever a conjugate of $\gamma$ lies in $K\lambda(\varpi)K$.
For each $\alpha\in \Phi$,
\beq\label{e:app:alpha(gamma)}
\begin{array}{lcc}
v(1-\alpha(\gamma))\ge 0& \mbox{if}& v(\alpha(\gamma))\ge 0,\\
v(1-\alpha(\gamma))= v(\alpha(\gamma))\ge -b_\Xi \kappa & \mbox{if} & v(\alpha(\gamma))<0
\end{array}
\eeq
where $b_\Xi$ is the constant $B_5$ (depending only on $G$ and $\Xi$ and not on $v$) of Lemma \ref{l:bounding1-alpha(gamma)}.
Hence $$D^G_M(\gamma)=
\prod_{\alpha\in \Phi\bs \Phi_M\atop \alpha(\gamma)\neq 1}
|1-\alpha(\gamma)|_v \le q_v^{|\Phi\bs \Phi_M| b_\Xi \kappa/2}$$
and likewise $D^M(\gamma)\le q_v^{|\Phi_M| b_\Xi \kappa/2}$.
(We divide the exponents by 2 because it cannot happen simultaneously that
 $v(\alpha(\gamma))<0$ and $v(\alpha^{-1}(\gamma))<0$.)
Therefore condition \eqref{e:app:suff-cond} on $a_G,b_G,e_G$ is implied by the two conditions
\beq\label{e:app:Levi0} e_G\ge \max(1,e_M),\eeq
\beq\label{e:app:Levi} \begin{aligned}&~ \frac{e_G-1}{2}\frac{|\Phi\bs \Phi_M| b_\Xi \kappa}{2}
+  \frac{e_G-e_M}{2}\frac{|\Phi_M| b_\Xi \kappa}{2}\\ \le & ~ a_G-a_M-(d_G+r_G(\delta_G+1)+|\Phi^+|\kappa)+(b_G-b_M-r_G\delta_G)\kappa.\end{aligned}\eeq
There are only finitely many Levi subgroups $M$ (up to conjugation)
giving rise to the triples $(a_M,b_M,e_M)$. It is elementary to observe that
\eqref{e:app:Levi} holds as long as $a_G$ and $b_G$ are sufficiently large while $e_G$ has any fixed value such that \eqref{e:app:Levi0} holds.
We will impose another condition on $a_G,b_G,e_G$ in Step 2-3.

\medskip

\noindent\underline{Step 2-3}. When $\gamma$ is noncentral and elliptic in $G$.
% When $\gamma$ is contained in a split maximal torus.

This case is essentially going to be worked out in \S\ref{sub:app:elliptic}. Let $Z_1,Z_2\ge 0$ be as in Lemma \ref{l:B-4} below.
By \eqref{e:mu-EP/mu} and Corollary \ref{c:B-1} below,
\eqref{e:app:prop} will hold if
 \beq\label{e:app:Elliptic}q_v^{r_G(d_G+1)}q_v^{1+Z_1\kappa} D^G(\gamma)^{-Z_2}
  \le q_v^{a_G+b_G\kappa} D^G(\gamma)^{-e_G/2}.\eeq
  We have $D^G(\gamma)\le q_v^{|\Phi|b_\Xi\kappa/2}$ thanks to \eqref{e:app:alpha(gamma)} (cf. Step 2-2). So \eqref{e:app:Elliptic} (is not equivalent to but) is implied by the combination of the following two inequalities:
 \beq\label{e:app:Elliptic1} -Z_2+\frac{e_G}{2} \ge 0.\eeq
  \beq\label{e:app:Elliptic2}  r_G(d_G+1)+1+Z_1\kappa+|\Phi|b_{\Xi}\frac{\kappa}{2}(-Z_2+\frac{e_G}{2})
  \le a_G+b_G\kappa.\eeq
  The latter two will hold true, for instance, if $e_G$ has any fixed value greater than or equal to $2Z_2$ and if $a_G$ and $b_G$ are sufficiently large. (We will see in \S\ref{sub:app:elliptic} below that $Z_1$ and $Z_2$ are independent of $\lambda$, $\gamma$ and $\kappa$.)

\medskip

  Now that we are done with analyzing three different cases, we finish Step 2. For this we use the induction on semisimple ranks (to ensure the existence of $a_M$, $b_M$ and $e_M$ in Step 2-2) to find $a_G,b_G,e_G\ge 0$ which satisfy the conditions described at the ends of Step 2-2 and Step 2-3. We are done with the proof of Proposition \ref{p:appendix2}.
\end{proof}

\subsection{A global consequence}\label{sub:global-bound-orb-int}

  Here we switch to a global setup. Let $\bF$ be a number field. For a finite place $v$ of $\bF$, let $k(v)$ denote the residue field and put $q_v:=|k(v)|$.
\bit
\item $G$ is a connected reductive group over $\bF$.
\item $\Ram(G)$ is the set of finite places $v$ of $\bF$ such that $G$ is ramified at $\bF_v$.
\item $\bG$ is the Chevalley group for $G\times_{\bF} \ol{\bF}$,
and $\bB$, $\bT$ are as in \S\ref{sub:local-bound-orb-int}.
\item $\Xi^{\spl}:\bG\hra \GL_m$, fixed once and for all, is a closed embedding defined over $\Z[1/R]$ for a large enough integer $R$ such that $\Xi^{\spl}(\bT)$ (resp. $\Xi^{\spl}(\bB)$) lies in the group of diagonal (resp. upper triangular) matrices of $\GL_m$. The choice of $R$ depends only on $\bG$ and $\Xi^{\spl}$. (We defer to \S\ref{sub:lem-split} more details and the explanation that there exists such a $\Xi^{\spl}$.)

 \item $S_{\bad}$ is the set of finite places $v$ such that either $v\in\Ram(G)$, $\cha k(v)\le w_Gs_G$, $\cha k(v)$ divides $R$, or $\cha k(v)$ divides at least one of the constants for the Chevalley commutator relations for $\bG$, cf. \eqref{e:commutator} below.
\eit

%Once and for all, fix a closed embedding $\Xi^{\spl}:\bG\hra \GL_m$ defined over $\Z[1/R]$ for a large enough integer $R$ such that $\Xi^{\spl}(\bT)$ (resp. $\Xi^{\spl}(\bB)$) lies in the group of diagonal (resp. upper triangular) matrices of $\GL_m$ and such that Lemma \ref{l:app:split2} holds true. The choice of $R$ depends only on $\bG$ and $\Xi^{\spl}$. (We defer to \S\ref{sub:lem-split} more details and the explanation that it is always possible to choose such an $R$.)

%  Let $Q$ be an integer such that $p|Q$ if and only if $p|R$ or some place $v$ of $F$ above $p$ belongs to $S_{\bad}$.
  Examining the dependence of various constants in Proposition \ref{p:appendix2}
  leads to the following main result of this section. For each finite place $v\notin S_{\bad}$, denote by $A_v$ a maximal $\bF_v$-split torus of $G\times_{\bF} \bF_v$.

\begin{thm}\label{t:appendeix2}
%   Let $\kappa\in \Z_{\ge 0}$ and let $\|\lambda\|\le \kappa$.
   There exist $a_G,b_G,e_G\ge 0$ (depending on $\bF$, $G$ and $\Xi^{\spl}$) such that
   \bit
   \item for every finite $v\notin S_{\bad}$, % where $G$ is unramified
   \item for every semisimple $\gamma\in G(\bF_v)$,
   \item for every $\lambda\in X_*(A_v)$ and $\kappa\in\Z_{\ge0}$ such that $\|\lambda\|\le\kappa$,
   \eit
  \[
  %\label{e:app:thm}
  0\le O^{G(\bF_v)}_\gamma(\tau^G_\lambda,\mu^{\can}_{G,v},\mu^{\can}_{I_\gamma,v})
  \le  q_v^{ a_G+b_G\kappa}\cdot D_v^G(\gamma)^{-e_G/2}.
  \]
\end{thm}

\begin{rem}
  It is worth drawing a comparison between the above theorem and Theorem \ref{t:appendix1} proved by Kottwitz.
  In the latter the test function (in the full Hecke algebra) and the base $p$-adic field are fixed whereas
  the main point of the former is to allow the test function (in the unramified Hecke algebra) and
  the place $v$ to vary. The two theorems are complementary to each other and will
  play a crucial role in the proof of Theorem \ref{t:weight-varies}.
\end{rem}

\begin{rem} In an informal communication Kottwitz and Ng\^{o} pointed out that there might be yet another approach based on a geometric argument involving affine Springer fibers, as in \cite[\S15]{GKM04}, which might lead to a streamlined and conceptual proof, as well as optimized values of the constants $a_G$ and $b_G$. Appendix~\ref{s:app:B} provides an important step in that direction, see Theorem~\ref{thm:transfer-fam} which implies that the constants are transferable from finite characteristic to characteristic zero.
\end{rem}

\begin{proof}
%  As in the proof of Proposition \ref{p:appendix2} we prove by induction on $r^{\semis}_G$.
  Since the case of tori is clear, we may assume that $r^{\semis}_G\ge 1$. %Write $S_Q$ for the places $v$ dividing $Q$.
  Let $\theta\in \mC(\Gamma_1)$. (Recall the definition of $\Gamma_1$ and $\mC(\Gamma_1)$ from \S\ref{sub:ST-meas} and \S\ref{sub:lim-of-Plan}.)
  Our strategy is to find $a_{G,\theta},b_{G,\theta},e_{G,\theta}\ge 0$ which satisfy the requirements \eqref{e:app:Levi}, \eqref{e:app:Elliptic1}, and \eqref{e:app:Elliptic2} on $a_{G,v},b_{G,v},e_{G,v}$ at all $v\in \cV_{\bF}(\theta)\bs S_{\bad}$. As for \eqref{e:app:Levi}, we inductively find $a_{M,\theta},b_{M,\theta},e_{M,\theta}\ge 0$ for all local Levi subgroups $M$ of $G$ as will be explained below. %In fact we will choose $e_{G,\theta}=0$. Once this is done for every $\theta$ in the finite set $\mC(\Gamma_1)$, we will similarly treat the places $v\in S_Q$ not contained in $S_{\bad}$ to finish the proof.

  We would like to explain an inductive choice of $a_{M,\theta},b_{M,\theta},e_{M,\theta}\ge 0$ for a fixed $\theta$. To do so we ought to clarify what Levi subgroups $M$ of $G$ we consider. Let $\Delta$ denote the set of $\bB$-positive simple roots for $(\bG,\bT)$. Via an identification $G\times_{\bF} \ol{\bF}\simeq \bG\times_\Z \ol{\bF}$ we may view $\Delta$ as the set of simple roots for $G$ equipped with an action by $\Gamma_1$, cf. \cite[\S1.3]{Bor79}. Note that $\Frob_v$ acts as $\theta\in \Gamma_1$ on $\Delta$ for all $v\in \cV_{\bF}(\theta)\bs S_{\bad}$. According to \cite[\S3.2]{Bor79}, the $\theta$-stable subsets of $\Delta$ are in bijection with $G(\bF_v)$-conjugacy classes of $\bF_v$-parabolic subgroups of $G$. For each $v\in \cV_{\bF}(\theta)\bs S_{\bad}$, fix a Borel subgroup $B_v$ of $G$ over $\bF_v$ containing the centralizer $T_v$ of $A_v$ in $G$ so that the following are in a canonical bijection with one another.
\bit
\item $\theta$-stable subsets $\Upsilon$ of $\Delta$
\item parabolic subgroups $P_v$ of $G$ containing $B_v$
\eit
  Denote by $P_{\Upsilon,v}$ the parabolic subgroup corresponding to $\Upsilon$ and by $M_{\Upsilon,v}$ its Levi subgroup containing $T_v$. Here is an important observation. The constants $Z_1$, $Z_2$ (see Remark \ref{r:uniformity} below) and the inequalities \eqref{e:app:Levi}, \eqref{e:app:Elliptic1}, and \eqref{e:app:Elliptic2} to be satisfied by $a_{M_{\Upsilon},v},b_{M_{\Upsilon},v},e_{M_{\Upsilon},v}$ depend only on $\theta$ and not on $v\in \cV_{\bF}(\theta)\bs S_{\bad}$. (We consider the case where $G$ and $M$ of those inequalities are $ M_{\Upsilon}$ and a $\bF_v$-Levi subgroup of $M_{\Upsilon}$, respectively.) Hence we will write $a_{M_{\Upsilon},\theta},b_{M_{\Upsilon},\theta},e_{M_{\Upsilon},\theta}\ge 0$ for these constants. What we need to do is to define them inductively according to the semisimple rank of $M$ such that \eqref{e:app:Levi}, \eqref{e:app:Elliptic1}, and \eqref{e:app:Elliptic2} hold true. In particular the desired $a_{G,\theta},b_{G,\theta},e_{G,\theta}$ will be obtained and the proof will be finished (by returning to the first paragraph in the current proof).

  Now the inductive choice of $a_{M_{\Upsilon},\theta},b_{M_{\Upsilon},\theta},e_{M_{\Upsilon},\theta}$ is easy to make once the choice of $a_{M_{\Omega},\theta},b_{M_{\Omega},\theta},e_{M_{\Omega},\theta}$ has been made for all $\Omega\subsetneq \Upsilon$. Indeed, we may choose $e_{M_{\Omega},\theta}\in \Z_{\ge 1}$ to fulfill \eqref{e:app:Elliptic1} and then choose $a_{M_{\Omega},\theta},b_{M_{\Omega},\theta}$ to be large enough to verify \eqref{e:app:Levi} and \eqref{e:app:Elliptic2}. Notice that $Z_1,Z_2,Z_3$ of \eqref{e:app:Elliptic2} (which are constructed in Lemma \ref{l:B-4} below) depend only on the group-theoretic information of $M_{\Upsilon}$ (such as the dimension, rank, affine root data, $\delta_{M_{\Upsilon}}$ of $M_{\Upsilon}$ as well as an embedding of the Chevalley form of $M_{\Upsilon}$ into $GL_d$ coming from $\Xi^{\spl}$) but not on $v$, cf. Remark \ref{r:uniformity}.

\end{proof}

In view of Theorem \ref{t:appendix1} and other observations in harmonic analysis, a natural question is whether it is possible to achieve $e_{G}=1$. This is a deep and difficult question which is of independent interest. It was a pleasant surprise to the authors that the theory of arithmetic motivic integration provides a solution. A precise theorem due to Cluckers, Gordon, and Halupczok is stated in Theorem \ref{thm:main} below. It is worth remarking that their method of proof is significantly different from that of this section and also that they make use of Theorem \ref{t:appendix1}, the local boundedness theorem. Finally it would be interesting to ask about the analogue in the case of twisted or weighted orbital integrals. Such a result would be useful in the more general situation than the one considered in this paper.

\subsection{The noncentral elliptic case}\label{sub:app:elliptic}

  The objective of this subsection is to establish Corollary \ref{c:B-1}, which was used in Step 2-3
  of the proof of Proposition \ref{p:appendix2} above. Since the proof is quite complicated let us guide the reader. The basic idea, going back to Langlands, is to interpret the orbital integral $O^{G(F)}_{\gamma}(\tau^G_{\lambda})$ in question as the number of points in the building fixed ``up to $\lambda$'' under the action of $\gamma$. The set of such points, denoted $X_F(\gamma,\lambda)$ below, is finite since $\gamma$ is elliptic. Then it is shown that every point of $X_F(\gamma,\lambda)$ is within a certain distance from a certain apartment, after enlarging the ground field $F$ to a finite extension. We exploit this to bound $X_F(\gamma,\lambda)$ by a ball of an explicit radius in the building. By counting the number of points in the ball (which is of course much more tractable than counting $|X_F(\gamma,\lambda)|$) we arrive at the desired bound on the orbital integral.
   The proof presented here is inspired by the beautiful exposition of \cite[\S\S3-5]{Kot05} but uses brute force and crude bounds at several places. We defer some technical lemmas and their proofs to \S\ref{sub:lem-split} below and refer to them in this subsection but there is no circular logic since no results of this subsection are used in \S\ref{sub:lem-split}.

  Throughout this subsection the notation of \S\ref{sub:local-bound-orb-int} is adopted and $\gamma$ is assumed to be noncentral and elliptic in $G(F)$. (However $\gamma$ need not be regular.) We assume $Z(G)$ to be anisotropic over $F$ as we did in Step 2 of the proof of Proposition \ref{p:appendix2}.
Then $I_\gamma(F)$ is a compact group, on which the Euler-Poincare measure $\mu^{\EP}_{I_\gamma}$
assigns total volume 1. Our aim is to bound $O^{G(F)}_{\gamma}(\triv_{K\mu(\varpi)K}, \mu^{\can}_G,\mu^{\can}_{I_\gamma})$.
 It follows from \cite[Thm 5.5]{Gro97} (for the equality) and Proposition \ref{p:Gross-motives} that
$$\left|\frac{\mu^{\EP}_{I_\gamma}}{\mu^{\can}_{I_\gamma}}\right|
= \frac{\prod_{d\ge 1} \det\left(1-\Frob_v q_v^{d-1}\left|(\Mot_{I_\gamma,d})^{I_v}\right.\right)}
{|H^1(F,I_\gamma)|}\le \prod_{d\ge 1}(1+q_v^{d-1})^{\dim \Mot_{I_\gamma,d}}$$ \beq\le (1+q_v^{(\dim I_\gamma+1)/2})^{\rank I_\gamma}
\le (1+q_v^{d_G})^{r_G}\le q_v^{r_G(d_G+1)}.\label{e:mu-EP/mu}\eeq
Thus we may as well bound $O^{G(F)}_{\gamma}(\triv_{K\mu(\varpi)K}, \mu^{\can}_G,\mu^{\EP}_{I_\gamma})$.

  Let $T_\gamma$ be a maximal torus of $I_\gamma$ defined over $F$ containing $\gamma$.
  By Lemma \ref{l:torus-splitting}, there exists a Galois extension $F'/F$ with
  \beq\label{e:F':F}[F':F]\le w_G s_G\eeq
such that $T_\gamma$ is a split torus over $F'$.
 Hence $I_\gamma$ and $G$ are split groups over $F'$.
 Note that $F'$ is a tame extension of $F$ under the assumption that $\cha k_F>w_Gs_G$.
Let $A'$ be a split maximal torus of $G$ over $F'$ such that $A\times_F F'\subset A'$. Since $F'$-split maximal  tori are conjugate over $F'$, we find
$$y\in G(F') \quad \mbox{such that}\quad A'=yT_\gamma y^{-1}$$
and fix such a $y$.
Write $\cO'$, $\varpi'$ and $v'$ for the integer ring of $F'$, a uniformizer and the valuation on $F'$ such that $v'(\varpi')=1$.
With respect to the integral model of $G$ over $\cO$ at the beginning of \S\ref{sub:local-bound-orb-int},
we put $K':=G(\cO')$. A point of $G(F)/K$ will be denoted $\ol{x}$ and any of its lift
in $G(F)$ will be denoted $x$. Let $\ol{x}_0\in G(F)/K$ (resp. $\ol{x}'_0\in G(F')/K'$) denote the
element represented by the trivial coset of $K$ (resp. $K'$). Then $\ol{x}_0$ (resp. $\ol{x}'_0$)
may be thought of as a base point of the building $\cB(G(F),K)$ (resp. $\cB(G(F'),K')$) and
its stabilizer is identified with $K$ (resp. $K'$). There exists an injection
\beq\label{e:embed-building}\cB(G(F),K)\hra \cB(G(F'),K')\eeq such that $\cB(G(F),K)$ is the $\Gal(F'/F)$-fixed points of $\cB(G(F'),K')$. (This is the case because $F'$ is tame over $F$.) The natural injection $G(F)/K\hra G(F')/K'$ coincides with the injection induced by \eqref{e:embed-building} on the set of vertices.

Define $\lambda'\in X_*(A')$ by $\lambda':=e_{F'/F} \lambda$ (where $e_{F'/F}$ is the ramification index of $F'$ over $F$) so that $\lambda'(\varpi')=\lambda(\varpi)$ and
\beq\label{e:mu'-abs-val}\|\lambda'\|=e_{F'/F} \|\lambda\|\le e_{F'/F}\kappa.\eeq
For (the fixed $\gamma$ and) a semisimple element $\delta\in G(F')$, set
\begin{eqnarray}
  X_F(\gamma,\lambda)&:=&\{ \ol{x}\in G(F)/K: \ol{x}^{-1} \gamma \ol{x}\in K\lambda(\varpi) K\}\nonumber\\
X_{F'}(\delta,\lambda')&:=&\{ \ol{x}'\in G(F')/K': (\ol{x}')^{-1} \delta \ol{x}'\in K'\lambda'(\varpi') K'\}.\nonumber
\end{eqnarray}
By abuse of notation we write $\ol{x}^{-1} \gamma \ol{x}\in K\lambda(\varpi) K$ for the condition that $x^{-1}\gamma x\in K\lambda(\varpi) K$ for some (thus every) lift $x\in G(F)$ of $\ol{x}$ and similarly for the condition on $\ol{x}'$.
It is clear that $X_F(\gamma,\lambda)\subset X_{F'}(\gamma,\lambda')\cap( G(F)/K)$. By (3.4.2) of \cite{Kot05},
\beq\label{e:orb-via-fixed} O^{G(F)}_{\gamma}(\triv_{K\lambda(\varpi)K}, \lambda_G,\lambda^{\EP}_{I_\gamma})=| X_F(\gamma,\lambda)|.\eeq
Our goal of bounding the orbital integrals on the left hand side can be translated into a problem of bounding $| X_F(\gamma,\lambda)|$.

Let $\Apt(A'(F'))$ denote the apartment for $A'(F')$. Likewise $\Apt(T_\gamma(F))$ and
$\Apt(T_\gamma(F'))$ are given the obvious meanings. We have $\ol{x}'_0\in \Apt(A'(F'))$.
The metrics on $\cB(G(F),K)$ and $\cB(G(F'),K')$ are
chosen such that \eqref{e:embed-building} is an isometry.
The metric on $\cB(G(F'),K')$ is
determined by its restriction to $\Apt(A'(F'))$, which is in turn pinned down by
a (non-canonical choice of) a Weyl-group invariant scalar product on $X_*(A')$, cf. \cite[\S2.3]{Tit79}. Henceforth we fix the scalar product once and for all. Scaling the scalar product does not change our main results of this subsection.
\begin{rem}
For any other tame extension $F''$ of $F$ and a split maximal torus $A''$ of $G$ over $F''$, we can find an isomorphism $X_*(A')$ and $X_*(A'')$ over the composite field of $F'$ and $F''$, well defined up to the Weyl group action. So the scalar product on $X_*(A'')$ is uniquely determined by that on $X_*(A')$. So we need not choose a scalar product again when considering a different $\gamma\in G(F)$.
\end{rem}

We define certain length functions. Consider an $F'$-split maximal torus $A''$ of $G$ (for instance $A''=T_\gamma$ or $A''=A'$) and the associated set of roots $\Phi=\Phi(G,A'')$ and the set of coroots $\Phi^\vee=\Phi^\vee(G,A'')$.
Let $l_{\max}(\Phi)$ denote the largest length of a positive coroot
in $\Phi^\vee$. Note that these are independent of the choice of $A''$ and completely determined by the previous choice of a Weyl group invariant scalar product on $X_*(A')$. It is harmless to assume that we have chosen the scalar product such that the longest positive coroot in each irreducible system of $X_*(A')$ has length $l_{\max}(\Phi)$.
%(Although $l_{\min}(\Phi)$ depends on the choice of metric, the effect of the choice
%will be canceled out along the way of proof.)

  Fix a Borel subgroup $B'$ of $G$ over $F'$ containing $A'$ so that $y^{-1}B'y$ is a Borel subgroup containing $T_\gamma$. Relative to these Borel subgroups we define the subset of positive roots $\Phi^+(G,A')$ and $\Phi^+(G,T_\gamma)$. Let $m_{\Xi^{\spl}}$ be as in Lemma \ref{l:app:split1} below. In order to bound $|X_F(\gamma,\lambda)|$ in \eqref{e:orb-via-fixed}, we control the larger set $X_{F'}(\delta,\lambda')$ by bounding the distance from its points to the apartment for $A'$.

\begin{lem}\label{l:dbl-coset-distance}
  Let $\delta\in A'(F')$ and $\ol{x}'\in G(F')/K'$. Then
there exist constants $C=C(\bG,\Xi)>0$, $c_\bG>0$, and $Y=Y(\bG)\in \Z_{\ge1}$ such that
whenever $(\ol{x}')^{-1} \delta \ol{x}'\in K'\lambda'(\varpi') K'$ (i.e. whenever $\ol{x}'\in X_{F'}(\delta,\lambda')$),
$$ d(\ol{x}',\Apt(A'(F')))\le
l_{\max}(\Phi)\cdot C|\Delta^+|\cdot Y^{|\Phi^+|}w_G s_G\qquad\qquad$$
$$\qquad\qquad\times \sum_{\alpha\in \Phi^+(G,A')} \left(|v(1-\alpha^{-1}(\delta))|+ Y (m_{\bG}m_{\Xi^{\spl}}+ m_\bG c_\bG + m_{\Xi^\spl})\kappa\right),$$
where the left hand side denotes the shortest distance from $\ol{x}'$ to $\Apt(A'(F'))$.
\end{lem}

\begin{proof}[Proof of Lemma \ref{l:dbl-coset-distance}]
  Write $\ol{x}'=an\ol{x}_0'$ for some $a\in A'(F')$ and $n\in N(F')$ using the Iwahori decomposition. As both sides of the above inequality
are invariant under multiplication by $a$, we may assume that $a=1$.
Let $\lambda_\delta\in X_*(A')$ be such that $\delta\in  \lambda_\delta(\varpi')A'(\cO')$. For each $\lambda_0\in X_*(A')^+$ recall the definition of $n_\bG(\lambda_0)$ from \eqref{e:n(lambda)}. Let $c_\bG>0$ be a constant depending only on $\bG$ such that every $\lambda_0\in X_*(A')$ satisfies the inequality $\lg \alpha,\lambda_0\rg \le c_\bG \|\lambda_0\|$ for all $\alpha\in \Phi^+(G,A')$.

\medskip

\benu

\item[Step 1.] Show that $\delta^{-1} n^{-1}\delta n\in K'\lambda_0(\varpi')K'$
for some $\lambda_0\in X_*(A')^+$ such that $n_\bG(\lambda_0)\le (m_{\Xi^{\spl}}+c_\bG)e_{F'/F}\kappa$.

\medskip

  By Cartan decomposition there exists a $B'$-dominant $ \lambda_0\in X_*(A')$ such that
  $\delta^{-1} n^{-1}\delta n\in K'\lambda_0(\varpi')K'$.
  The condition on $\delta$ in the lemma is unraveled as
  $(x'_0)^{-1} n^{-1} \delta n x'_0\in K'\lambda'(\varpi') K'$. So
$$\delta^{-1}n^{-1} \delta n\in \delta^{-1} K'\lambda'(\varpi') K'
\subset (K'\lambda_\delta^{-1}(\varpi') K')(K'\lambda'(\varpi') K').$$
Let $w$ be a Weyl group element for $A'$ in $G$ such that $w\lambda_\delta^{-1}$ is $B'$-dominant. The fact that $K'\lambda_0(\varpi')K'$ intersects
$(K'\lambda_\delta^{-1}(\varpi') K')(K'\lambda'(\varpi') K')$ implies (\cite[Prop 4.4.4.(iii)]{BT72}) that
$$\lg \alpha,\lambda_0\rg \le \lg \alpha,w\lambda_\delta^{-1}+\lambda'\rg,\quad \alpha\in \Phi^+(G,A').$$
We have $\lg \alpha,\lambda'\rg\le c_\bG\|\lambda'\|$. Note also that
\beq\label{e:alpha(delta)} v'(\alpha(\delta))\in [-m_{\Xi^{\spl}} \|\lambda'\|,m_{\Xi^{\spl}} \|\lambda'\|]\eeq
  by Lemma \ref{l:app:split1} since a conjugate of $\delta$ belongs to $K'\lambda'(\varpi')K'$. This implies that
 $$\lg  \alpha,w\lambda_\delta^{-1}\rg = v'(w\alpha^{-1}(\delta))\le m_{\Xi^{\spl}}\|\lambda'\|.$$
 On the other hand  $\|\lambda'\|\le e_{F'/F}\kappa$ according to \eqref{e:mu'-abs-val}. These inequalities imply the desired bound on $n_\bG(\lambda_0)$, which is the maximum of $\lg \alpha,\lambda_0\rg$ over $\alpha\in \Phi^+(G,A')$.

\medskip

  Before entering Step 2, we notify the reader that we are going to use the convention and notation for the Chevalley basis as recalled in \S\ref{sub:lem-split} below. In particular $n\in N(F')$ can be written as (cf. \eqref{e:y=product})
  \beq\label{e:n=product}n=x_{\alpha_1}(X_{\alpha_1})\cdots x_{\alpha_{|\Phi^+|}}(X_{\alpha_{|\Phi^+|}})\eeq
  for unique $X_{\alpha_1},...,X_{\alpha_{|\Phi^+|}}\in F'$.
\medskip

\item[Step 2.] Show that there exists a constant $\cM_{|\Phi^+|}\ge 0$ (explicitly defined in \eqref{e:cM_i} below) such that
$v'(X_{\alpha_i})\ge -\cM_{|\Phi^+|}$ for all $1\le i\le |\Phi^+|$.

%$n\in G_{\ol{x}_0',-\cM_{|\Phi^+|}}$ for some constant $\cM_{|\Phi^+|}\ge 0$ (defined in \eqref{e:cM_i} below).

\medskip

  In our setting we compute
\begin{eqnarray}
  \delta^{-1}n^{-1}\delta n
&=&  \delta^{-1} \left(\prod_{i=|\Phi^+|}^{1} x_{\alpha_i}(-X_{\alpha_i})\right) \delta
\prod_{i=1}^{|\Phi^+|} x_{\alpha_i}(X_{\alpha_i})
\nonumber\\&= &\left(\prod_{i=|\Phi^+|}^{1} x_{\alpha_i}(-\alpha^{-1}_i(\delta) X_{\alpha_i}) \right)
\prod_{i=1}^{|\Phi^+|} x_{\alpha_i}(X_{\alpha_i})
\nonumber\\&=&\prod_{i=1}^{|\Phi^+|} x_{\alpha_i}\left((1-\alpha^{-1}_i(\delta))X_{\alpha_i}
+ P_{\alpha_i} \right)\label{e:delta-n-delta-n}
\end{eqnarray}
where the last equality follows from the repeated use of \eqref{e:commutator} to rearrange the terms.
Here $P_{\alpha_i}$ is a polynomial (which could be zero) in $\alpha_{j}^{-1}(\delta)$
and $X_{\alpha_j}$ with integer coefficients for $j<i$. It is not hard to observe from \eqref{e:commutator} that
$P_{\alpha_i}$ has no constant term. As $i$ varies in $[1,|\Phi^+|]$,
let $Y$ denote the highest degree for the nonzero monomial term appearing in $P_{\alpha_i}$
viewed as a polynomial in either $\alpha_i^{-1}(\delta)$ or $X_{\alpha_i}$ (but not both).\footnote{For instance
if $P_{\alpha_i}=\alpha_i^{-1}(\delta)^2 X_{\alpha_i}^4 + \alpha_i^{-1}(\delta)^3 X_{\alpha_i}^3$ then $Y=4$.} Set $Y=1$ if $P_{\alpha_i}=0$.
As mentioned above, the positive roots for a given $(\bG,\bB,\bT)$ are ordered once and for all so that $Y$ depends only on $\bG$ in the sense that for any $G$ having $\bG$ as its Chevalley form, $Y$ is independent of the local field $F$ over which $G$ is defined.

 Applying Corollary \ref{c:split2} below, we obtain from \eqref{e:delta-n-delta-n} and the condition $\delta^{-1}n^{-1}\delta n\in K'\lambda_0(\varpi')K'$ that
\beq\label{e:v(N-alpha)}v'\left((1-\alpha^{-1}_i(\delta))X_{\alpha_i}
+ P_{\alpha_i}\right) \ge -m_{\bG} n_\bG(\lambda_0).\eeq
For $1\le i\le |\Phi^+|$, put \beq
\label{e:cM_i}\cM_i:=\sum_{j=1}^i \left(Y^{i-j} (|v'(1-\alpha_j^{-1}(\delta))|+ m_{\bG}n_{\bG}(\lambda_0))\right)+\sum_{j=1}^{i-1}Y^j m_{\Xi^{\spl}} e_{F'/F}\kappa.\eeq
Obviously $0\le\cM_1\le \cM_2\le \cdots\le \cM_{|\Phi^+|}$.
We claim that for every $i\ge 1$,
\beq\label{e:greater-cM_i} v'(X_{\alpha_i}) \ge -\cM_i.\eeq
When $i=1$, this follows from \eqref{e:v(N-alpha)} as $P_{\alpha_1}=0$. (Use the fact that $x_{\alpha_1}(a_1X_{\alpha_1})$ commutes with any other $x_{\alpha_j}(a_jX_{\alpha_j})$ in view of \eqref{e:commutator} since $\alpha_1$ is a simple root.) Now by induction,
suppose that \eqref{e:greater-cM_i} is verified for all $j<i$.
By \eqref{e:v(N-alpha)},
$$v'(X_{\alpha_i})+v'(1-\alpha^{-1}_i(\delta))\ge \min(-m_{\bG} n_\bG(\lambda_0),v'(P_{\alpha_i})).$$
Note that $P_{\alpha_i}$ is the sum of monomials of the form $\alpha_j^{-1}(\delta)^{k_1}X_{\alpha_j}^{k_2}$ with $j,k_1,k_2\in \Z$ such that $1\le j<i$ and $0\le k_1,k_2\le Y$. Each monomial satisfies
$$v'(\alpha_j^{-1}(\delta)^{k_1}X_{\alpha_j}^{k_2})=k_1v'(\alpha_j^{-1}(\delta))+k_2v'(X_{\alpha_j})
\ge -Ym_{\Xi^\spl}e_{F'/F}\kappa-Y\cM_{i-1},$$
where the inequality follows from \eqref{e:alpha(delta)}, \eqref{e:mu'-abs-val}, the induction hypothesis, and the fact that $0\le \cM_{j}\le \cM_{i-1}$.
Hence $$v'(P_{\alpha_i})\ge -Ym_{\Xi^\spl}e_{F'/F}\kappa-Y\cM_{i-1}.$$
Now
\begin{eqnarray}
v'(X_{\alpha_i}) &\ge& \min(-m_{\bG}n_\bG(\lambda_0),v'(P_{\alpha_i}))-v'(1-\alpha^{-1}_i(\delta))\nonumber\\
& \ge & -m_{\bG}n_\bG(\lambda_0) -Ym_{\Xi^\spl}e_{F'/F}\kappa-Y\cM_{i-1}-|v'(1-\alpha^{-1}_i(\delta))|=-\cM_i,\nonumber
\end{eqnarray}
as desired.
%Hence $v'(X_{\alpha_i})+v'(1-\alpha^{-1}_i(\delta))\ge -Y(m_{\bG} n_\bG(\lambda_0)+\cM_{i-1})$, implying \eqref{e:greater-cM_i} by the inductive hypothesis.
 Now that the claim is verified, we have a fortiori
\beq\label{e:bound-v'(X)} v'(X_{\alpha_i})\ge -\cM_{|\Phi^+|}, \quad \forall1\le i\le |\Phi^+| .\eeq %In view of \eqref{e:n=product} (and the fact that $x_{\alpha}(\cO')=U_{\alpha,0}$), we have $n\in G_{\ol{x}'_0,-\cM_{|\Phi^+|}}$.
For our purpose it suffices to use the following upper bound, which is simpler than  $\cM_{|\Phi^+|}$. Note that we used the upper bound on $n_\bG(\lambda_0)$ from Step 1.
\beq\label{e:bound-on-cM}\cM_{|\Phi^+|}\le Y^{|\Phi^+|} \sum_{\alpha\in \Phi^+} \left( |v'(1-\alpha^{-1}(\delta)| + (m_\bG m_{\Xi^{\spl}}+m_\bG c_\bG + m_{\Xi^{\spl}})e_{F'/F}\kappa\right).\eeq

\medskip

\item[Step 3.] Find $a\in A'(F')$ such that $a^{-1}na\in K'$.

\medskip

 We can choose a sufficiently large $C=C(\bG,\Xi)>0$,
  depending only on the Chevalley group $\bG$ and $\Xi$, and integers $a^0_{\alpha}\in [-C,0]$ for
$\alpha\in \Delta^+$ such that
$$1\le \sum_{\alpha\in \Delta^+} (-a^0_\alpha) \lg \beta,\alpha^\vee\rg \le C,\quad \forall \beta\in \Delta^+.$$
(This is possible because the matrix $(\lg \beta,\alpha^\vee\rg )_{\beta,\alpha\in \Delta^+}$ is nonsingular. For instance one finds $a^0_{\alpha}\in \Q$ satisfying the above inequalities for $C=1$ and then eliminate denominators in $a^0_{\alpha}$ by multiplying a large positive integer.)
Now put $a_{\alpha}:=\cM_{|\Phi^+|} a^0_{\alpha}\in [-C\cM_{|\Phi^+|},0]$ and
  $a:=\sum_{\alpha\in \Delta^+} a_\alpha \alpha^\vee(\varpi')\in A'(F')$ so that
\beq\label{e:bound-v(a)}
\cM_{|\Phi^+|}\le -v(\beta(a))\le C\cdot\cM_{|\Phi^+|},
\quad \forall \beta\in \Delta^+.\eeq
In fact \eqref{e:bound-v(a)} implies that the left inequality holds for all $\beta\in \Phi^+$.
Hence
$$a^{-1}na=\prod_{i=1}^{|\Phi^+|}
x_{\alpha_i}(\alpha_i(a)^{-1} X_{\alpha_{i}})$$
$$\in \prod_{i=1}^{|\Phi^+|} U_{\alpha_i,v(X_{\alpha_i})-v(\alpha_i(a))}\subset \prod_{i=1}^{|\Phi^+|} U_{\alpha_i,\cM_{|\Phi^+|}
+v(X_{\alpha_i})}.$$
Here we have written $U_{\alpha,m}$ with $m\in \R$ for the image under the isomorphism $x_\alpha:F\simeq U_\alpha(F)$ of the set $\{a\in F:v(a)\ge m\}$.
In light of \eqref{e:greater-cM_i},
$\cM_{|\Phi^+|}
+v(X_{\alpha_i})\ge 0$. Hence $a^{-1}na\in K'$.

\medskip

\item[Step 4.] Conclude the proof.

\medskip

  Step 3 shows that $a\ol{x}'_0\in \Apt(A'(F'))$ is invariant under the left multiplication
action by $n$ on $\cB(G(F'),K')$, which acts as an isometry. Recalling that $\ol{x}'=n\ol{x}'_0$ we have
\beq\label{e:Step4-1}d(\ol{x}',  \Apt(A'(F')))\le d(n\ol{x}'_0, a\ol{x}'_0)= d(n\ol{x}'_0, na\ol{x}'_0)=
d(\ol{x}'_0, a\ol{x}'_0).\eeq
On the other hand, for any
$\ol{x}'\in \Apt(A'(F'))$ and any positive simple coroot $\alpha^{\vee}$, we have
\beq\label{e:dist-simple-coroot}
%l_{\min}(\Phi)\le
 d(\ol{x}', \alpha^\vee(\varpi')^{-1} \ol{x}')\le l_{\max}(\Phi).\eeq
 Indeed this holds by the definition of $l_{\max}(\Phi)$ as the left hand side is the length of $\alpha^\vee$.
 %(The left inequality holds by definition of $l_{\min}$. The right inequality comes from the
%standard fact
%that the length of the longest coroot is bounded by $6l_{\min}$ in an irreducible system.)
Since $a=\prod_{\alpha\in \Delta^+} (\alpha^\vee(\varpi'))^{a_{\alpha}}$ with
$a_\alpha\in [-C\cM_{|\Phi^+|},0]$, a repeated use of \eqref{e:dist-simple-coroot}, together with
a triangle inequality, shows that
\beq\label{e:Step4-3} d(\ol{x}'_0, a\ol{x}'_0)\le l_{\max}(\Phi)\cdot C\cdot  \cM_{|\Phi^+|}\cdot |\Delta^+|.\eeq
Lemma \ref{l:dbl-coset-distance} follows from
\eqref{e:Step4-1}, \eqref{e:Step4-3}, \eqref{e:bound-v'(X)}, \eqref{e:bound-on-cM}, and $e_{F'/F}\le [F':F]\le w_G s_G$ as we saw in \eqref{e:F':F}.

\eenu

\end{proof}

  Since $\gamma$ is elliptic and $G$ is anisotropic over $F$, $\Apt(T_\gamma(F))$ is a singleton. Let $\ol{x}_1$ denote its only point.
Then the $\Gal(F'/F)$-action on $\Apt(T_\gamma(F'))$ has $\ol{x}_1$ as the unique fixed point.
Motivated by Lemma \ref{l:dbl-coset-distance} we set $\cM(\gamma,\kappa)$ to be
$$l_{\max}(\Phi)\cdot C|\Delta^+|\cdot Y^{|\Phi^+|}w_G s_G$$
$$\times
\sum_{\alpha\in \Phi(G,T_\gamma)}\left( |v(1-\alpha^{-1}(\gamma))| + Y (m_{\bG} m_{\Xi^{\spl}}+ m_\bG c_\bG + m_{\Xi^\spl}) \kappa\right)$$
and similarly $ \cM(\delta,\kappa)$ using $\alpha\in \Phi(G,A')$ in the sum instead.
 %(Useful to know: $v(1-\alpha(\gamma))+v(1-\alpha^{-1}(\gamma))\le 2|v(1-\alpha^{-1}(\gamma))|$
 %for all $\alpha\in \Phi^+$.)
 Note that we are summing over all roots, not just positive roots as in the lemma. This is okay since it will only improve the inequality of the lemma. We do this such that $\cM(\gamma,\kappa)=\cM(\delta,\kappa)$. Indeed the equality is induced by a bijection $\Phi(G,T_\gamma)\simeq \Phi(G,A')$ coming from any element $y'\in G(F')$ such that $A'=y'T_\gamma (y')^{-1}$ (for example one can take $y'=y$).
 Define a closed ball in $G(F)/K$: for
$\ol{z} \in G(F)/K$ and $R\ge 0$,
$$\mathrm{Ball}(\ol{z},R):=\{ \ol{x}\in G(F)/K: d(\ol{x},\ol{z})\le R\}.$$

\begin{lem}\label{l:B-3}
  $X_F(\gamma,\lambda)~\subset~\mathrm{Ball}(\ol{x}_1,\cM(\gamma,\kappa)).$
\end{lem}

\begin{proof}
  As we noted above, $X_F(\gamma,\lambda)\subset X_{F'}(\gamma,\lambda')=X_{F'}(y^{-1}\delta y,\lambda')$.
Lemma \ref{l:dbl-coset-distance} tells us that
$$\ol{x}\in X_F(\gamma,\lambda)~\Rightarrow~d(y\ol{x},\Apt(A'(F')))\le \cM(\delta,\kappa)
~\Rightarrow~ d(\ol{x},\Apt(T_\gamma(F')))\le \cM(\delta,\kappa).$$
The last implication uses $\Apt(A'(F'))=y \Apt(T_\gamma(F'))$ (recall $A'=yT_\gamma y^{-1}$).
  We have viewed $\ol{x}$ as a point of $\cB(G(F'),K')$
via the isometric embedding $\cB(G(F),K)\hra \cB(G(F'),K')$.
In order to prove the lemma, it is enough to check that $d(\ol{x},\ol{x}_1)\le d(\ol{x},\ol{x}_2)$ for every
$\ol{x}_2\in \Apt(T_\gamma(F'))$. To this end,
we suppose that there exists an $\ol{x}_2$ with
\beq\label{e:assump-dist}d(\ol{x},\ol{x}_1)> d(\ol{x},\ol{x}_2)\eeq
and will draw a contradiction.

 As $\sigma\in\Gal(F'/F)$ acts on $\cB(G(F'),K')$ by isometry,
$d(\ol{x},\sigma \ol{x}_2)=d(\ol{x},\ol{x}_2)$. As $\Apt(T_\gamma(F'))$
is preserved under the Galois action, $\sigma\ol{x}_2\in \Apt(T_\gamma(F'))$.
According to the inequality of \cite[2.3]{Tit79}, for any $x,y,z\in \cB(G(F'),K')$ and
for the unique mid point $m=m(x,y)\in \cB(G(F'),K')$ such that $d(x,m)=d(y,m)=\frac{1}{2}d(x,y)$,
\beq\label{e:Tits}d(x,z)^2+d(y,z)^2
\ge 2 d(m,z)^2 + \frac{1}{2}d(x,y)^2.\eeq
Consider the convex hull $\mC$ of $\mC_0:=\{\sigma\ol{x}_2\}_{\sigma\in\Gal(F'/F)}$.
Since $\mC_0$ is contained in $\Apt(T_\gamma(F'))$, so is $\mC$.
Moreover $\mC_0$ is fixed under $\Gal(F'/F)$,
from which it follows that
$\mC$ is also preserved under the same action.
(One may argue as follows. Inductively define $\mC_{i+1}$ to be the set consisting of the mid points $m(x,y)$
for all $x,y\in \mC_{i}$. Then it is not hard to see that $\mC_i$ must be preserved under
$\Gal(F'/F)$ and that $\cup_{i\ge 0} \mC_i$ is a dense subset of $\mC$.)
As $\mC$ is a compact set, one may choose $\ol{x}_3\in \mC$ which has the minimal distance
to $\ol{x}$ among the points of $\mC$. By construction \beq\label{e:contrad-dist} d(\ol{x}_3,\ol{x})\le d(\ol{x}_2,\ol{x}).\eeq
Applying \eqref{e:Tits} to $(x,y,z)=(\ol{x}_3,\sigma\ol{x}_3,\ol{x})$, where $\sigma\in \Gal(F'/F)$,
$$2 d(\ol{x}_3,\ol{x})^2= d(\ol{x}_3,\ol{x})^2+d(\sigma\ol{x}_3,\ol{x})^2
\ge 2d(m(\ol{x}_3,\sigma\ol{x}_3),\ol{x})^2+\frac{1}{2}d(\ol{x}_3,\sigma\ol{x}_3)^2.$$
As $\ol{x}_3,\sigma\ol{x}_3\in \mC$, we also have $m(\ol{x}_3,\sigma\ol{x}_3)\in \mC$
by the convexity of $\mC$. The choice of $\ol{x}_3$ ensures that $d(\ol{x}_3,\ol{x})\le d(m(\ol{x}_3,\sigma\ol{x}_3),\ol{x})$, therefore
$d(\ol{x}_3,\sigma\ol{x}_3)=0$, i.e. $\ol{x}_3=\sigma\ol{x}_3$. Hence $\ol{x}_3$ is a $\Gal(F'/F)$-fixed point of $\Apt(T_\gamma(F'))$.
This implies that $\ol{x}_3=\ol{x}_1$, but then \eqref{e:contrad-dist} contradicts \eqref{e:assump-dist}.

\end{proof}

\begin{lem}\label{l:B-4} There exist constants $Z_1,Z_2\ge 0$, independent of $\gamma$ and $\lambda$,
%(determined by $Z$ of \eqref{e:app:Z} below as well as
%other constants depending only on the Chevalley form of $G$ and $\Xi$)
such that
$$| \mathrm{Ball}(\ol{x}_1,\cM(\gamma,\kappa)) |\le q_v^{1+Z_1\kappa} D^G(\gamma)^{-Z_2}.$$
\end{lem}

\begin{rem}\label{r:uniformity}
  A scrutiny into the defining formulas for $Z_1$ and $Z_2$ (as well as $Z'_1$ and $Z'_2$) at the end of the proof reveals that $Z_1$ and $Z_2$ depend only on the affine root data, the group-theoretic constants for $G$ (and its Chevalley form), and $\Xi$. An important point is that, in the situation where local data arise from some global reductive group over a number field by localization, the constants $Z_1$ and $Z_2$ do not depend on the residue characteristic $p$ or the $p$-adic field $F$ as long as the affine root data remain unchanged. This observation is used in the proof of Theorem \ref{t:appendeix2} to establish a kind of uniformity when traveling between places in $\cV(\theta)\bs S_{\bad}$ for a fixed $\theta\in \mC(\Gamma_1)$ in the notation there.
\end{rem}

\begin{proof}
  To ease notation we write $\cM$ for $\cM(\gamma,\kappa)$ in the proof. Let us introduce some quantities and objects of geometric nature for the building $\cB(G(F),K)$.
  Write $e_{\max}>0$ for the maximum length of the edges of $\cB(G(F),K)$.
  For a subset $S$ of $\cB(G(F),K)$, define $\mathrm{Ch}^+(S)$ to be the set of chambers $\mC$ of the building such that $\mC\cap S$ contains a vertex. Let $v\in \cB(G(F),K)$ be a vertex. (We are most interested in the case $v=\ol{x}_1$.) We put $\mC_1(v)$ to be the union of chambers in $\mathrm{Ch}^+(\{v\})$ and define $\mC_{i+1}(v)$ to be the union of chambers in $\mathrm{Ch}^+(\mC_i(v))$ for all $i\in \Z_{\ge1}$ so as to obtain a strictly increasing chain $\{v\}\subsetneq \mC_1(v)\subsetneq \mC_2(v)\subsetneq \mC_3(v)\subsetneq \cdots$. Denote by $\mathrm{V}_i(v)$ (resp. $\mathrm{Ch}_i(v)$) the set of vertices (resp. chambers) contained in $\mC_i(v)$ for $i\in \Z_{\ge1}$.

  Choose any chamber $\mC$ in $\cB(G(F),K)$. Define $\mC^+$ to be the union of all chambers in $\mathrm{Ch}^+(\mC)$. Clearly $\mC^+$ is compact and its interior contains the compact subset $\mC$. Hence there exists a maximal $R_G>0$ such that for every point $y\in \mC$ (which may not be a vertex), the ball centered at $y$ of radius $R_G$ is contained in $\mC^+$. Since the isometric action of $G(F)$ is transitive on the set of chambers, $R_G$ does not depend on the choice of $\mC$. Moreover the ratio $l_{\max}(\Phi)/R_G$ does not depend on the choice of metric on the building.

  From the definitions we have $\mathrm{Ball}(\ol{x}_1,R_G)\subset \mC_1(\ol{x}_1)$ and deduce recursively that
  $$\mathrm{Ball}(\ol{x}_1,i R_G)\subset \mathrm{V}_i(\ol{x}_1) \subset \mC_i(\ol{x}_1),\quad \forall i\in \Z_{\ge1}.$$
  Take $\cM'$ to be the integer such that $\frac{\cM}{R_G}\le \cM' < \frac{\cM}{R_G}+1$ so that in particular
  \beq\label{e:app:union-chambers}\mathrm{Ball}(\ol{x}_1,\cM)\subset \mathrm{V}_{\cM'}(\ol{x}_1).\eeq

  Let us bound $|\mathrm{Ch}_1(v)|$ for every vertex $v\in \cB(G(F),K)$.
 The stabilizer of $v$, denoted by $\Stab(v)$, acts transitively on $\mathrm{Ch}_1(v)$.
 Let $\mC\in \mathrm{Ch}_1(v)$. Then
$$|\mathrm{Ch}_1(v)|= |\Stab(v)/\Stab(\mC)|
\le  |G(\cO)/\mathrm{Iw}| \le |G(k_F)|\le q_v^{d_G+r_G}$$
where $\mathrm{Iw}$ denotes an Iwahori subgroup of $G(\cO)$, which is conjugate to $\Stab(\mC)$. The group $\Stab(v)$ may not be hyperspecial, but the first inequality follows from the fact that the hyperspecial has the largest volume among all maximal compact subgroups \cite[3.8.2]{Tit79}.
See the proof of Lemma \ref{l:double-coset-volume} for the last inequality.

Each chamber contains $\dim A+1$ vertices as a $\dim A$-dimensional simplex.
Hence for each $i\ge 1$,
$$ |\mathrm{V}_i(\ol{x}_1)|\le (\dim A+1)\cdot |\mathrm{Ch}_i(\ol{x}_1)|.$$
On the other hand,
$$ |\mathrm{Ch}_{i+1}(\ol{x}_1)|\le \sum_{v\in \mathrm{V}_i(\ol{x}_1)}|\mathrm{Ch}_{1}(v)| \le q_v^{d_G+r_G} |\mathrm{V}_i(\ol{x}_1)|
\le q_v^{d_G+r_G} (\dim A+1)\cdot |\mathrm{Ch}_i(\ol{x}_1)|.$$
We see that $\mathrm{Ch}_{i}(\ol{x}_1)|\le q_v^{i(d_G+r_G)} (\dim A+1)^{i-1}$ and thus
\beq\label{e:app:bound-vertices}
|\mathrm{V}_{\cM'}(\ol{x}_1)|\le (\dim A+1)^{\cM'}q_v^{\cM'(d_G+r_G)}\le   (r_G+1)^{\cM'}q_v^{\cM'(d_G+r_G)}.\eeq
%We define a constant $Z\ge 0$, independent of $\gamma$ such that
%\[
%\label{e:app:Z}
%Z:=\max\left(\frac{e_{\max}}{e_{\min}}, \frac{l_{\max}(\Phi)}{e_{\min}}\right).
%\]
 Note that
\begin{eqnarray}
%\label{e:app:Z1}
\cM'\le 1+\frac{\cM}{R_G}
&\le& 1+ \frac{l_{\max}(\Phi)}{R_G}C|\Delta^+|\cdot Y^{|\Phi^+|}w_G s_G \nonumber\\
&& \times \left(
\sum_{\alpha\in \Phi} |v(1-\alpha^{-1}(\gamma))| + Y (m_{\bG}m_{\Xi^{\spl}}+m_\bG c_\bG + m_{\Xi^\spl}) \kappa\right),\nonumber
\end{eqnarray}
which can be rewritten in the form
\[
%\label{e:app:Z2}
\cM'\le 1++Z'_1\kappa+ Z'_2\sum_{\alpha\in \Phi} |v(1-\alpha^{-1}(\gamma))|.
\]
Since $|v(1-\alpha(\gamma))|+|v(1-\alpha^{-1}(\gamma))|\le v(1-\alpha(\gamma))+v(1-\alpha^{-1}(\gamma))+2b_\Xi \kappa$
in view of \eqref{e:app:alpha(gamma)},
we have $$q^{\cM'}\le q^{1+(Z'_1+b_{\Xi} Z'_2)\kappa} D^G(\gamma)^{-Z'_2}.$$
Returning to \eqref{e:app:union-chambers} and \eqref{e:app:bound-vertices},
$$ | \mathrm{Ball}(\ol{x}_1,\cM)|
\le |\mathrm{V}_{\cM'}(\ol{x}_1)|
\le q_v^{(r_G+1)\cM'}q_v^{\cM'(d_G+r_G)}$$ $$
\le (q_v^{1+(Z'_1+2b_{\Xi} Z'_2)\kappa} D^G(\gamma)^{-Z'_2})^{d_G+2r_G+1}.$$
The proof of Lemma \ref{l:B-4} is complete once we set $Z_1$ and $Z_2$ as follows, the point being that they
\bit
%\item $Z_1:=(1+Z'_1)(d_G+2r_G+1)$,
\item $Z_1:=(Z'_1+2 b_{\Xi} Z'_2)(d_G+2r_G+1)$, %=CZ|\Delta^+|Y^{|\Phi^+|}m_{\Xi^{\spl}}((m_{\Xi^{\spl}}+1)Y+2)(d_G+2r_G+1)$,
\item $Z_2:=Z'_2(d_G+2r_G+1)$. %=CZ|\Delta^+|Y^{|\Phi^+|}(d_G+2r_G+1)$.
\eit
\end{proof}

\begin{cor}\label{c:B-1}
  $|O^{G(F)}_{\gamma}(\triv_{K\lambda(\varpi)K}, \mu_G,\mu^{\EP}_{I_\gamma})|\le
q_v^{r_G(d_G+1)} q_v^{1+Z_1\kappa} D^G(\gamma)^{-Z_2}$.
\end{cor}

\begin{proof}
  Follows from \eqref{e:orb-via-fixed}, Lemma \ref{l:B-3} and Lemma \ref{l:B-4}.

\end{proof}

\subsection{Lemmas in the split case}\label{sub:lem-split}

This subsection plays a supporting role for the previous subsections, especially \S\ref{sub:app:elliptic}.
 As in \S\ref{sub:global-bound-orb-int} let $\bG$ be a Chevalley group with a Borel subgroup $\bB$ containing a split maximal torus $\bT$, all over $\Z$. Let $\Xi^{\spl}_{\Q}:\bG\hra \GL_m$ be a closed embedding of algebraic groups over $\Q$. Let $\T$ denote the diagonal maximal torus of $\GL_m$, $\B$ the upper triangular Borel subgroup of $\GL_m$, and $\N$ the unipotent radical of $\B$.

 Extend $\Xi^{\spl}_{\Q}$ to a closed embedding $\Xi^{\spl}:\bG\hra \GL_m$ defined over $\Z[1/R]$ for some integer $R$ such that $\Xi^{\spl}(\bT)$ (resp. $\Xi^{\spl}(\bB)$) lies in the group of diagonal (resp. upper triangular) matrices of $\GL_m$.
  To see that this is possible, find a maximal $\Q$-split torus $\T'$ of $\GL_m$ containing $\Xi^{\spl}_{\Q}(\bT)$. Choose any Borel subgroup $\B'$ over $\Q$ containing $\T$. Then there exists $g\in \GL_m(\Q)$ such that the inner automorphism $\mathrm{Int}(g):\GL_m\ra \GL_m$ by $\gamma\mapsto g \gamma g^{-1}$ carries $(\B',\T')$ to $(\B,\T)$. Then $\Xi^{\spl}_{\Q}$ and $\mathrm{Int}(g)$ extend over $\Q$ to over $\Z[1/R]$ for some $R\in \Z$, namely at the expense of inverting finitely many primes (basically those in the denominators of the functions defining $\Xi^{\spl}_{\Q}$ and $\mathrm{Int}(g)$).

  Now suppose that $p$ is a prime not diving $R$. Let $F$ be a finite extension of $\Q_p$ with integer ring $\cO$
  and a uniformizer $\varpi$. The field
  $F$ is equipped with a unique discrete valuation $v_F$ such that $v_F(\varpi)=1$. Let $\lambda\in X_*(\bT)$.
  We are interested in assertions which work for $F$ as the residue characteristic $p$ varies.
  Lemma \ref{l:app:split1} (resp. Corollary \ref{c:split2}) below is used in Step 1 (resp. Step 2) of the proof of Lemma \ref{l:dbl-coset-distance}.
%  Lemma \ref{l:app:split2}) below is used in Step 2 of the proof of Lemma \ref{l:dbl-coset-distance}. Although Lemma \ref{l:app:split1} is not used in this article, we have kept it thinking that it might be useful on another occasion. Another justification is that the two lemmas share a common component in their proofs.
%  Our convention for $T_\delta$ and $\Phi_\delta$ below is the same as for $T_\gamma$ and $\Phi_\gamma$ in Lemma \ref{l:bounding1-alpha(gamma)}. The proof of the next lemma is little different from that of Lemma \ref{l:bounding1-alpha(gamma)}.

\begin{lem}\label{l:app:split1}
  There exists $m_{\Xi^{\spl}}\in \Z_{>0}$ such that
for every $p$, $F$ and $\lambda$ as above and for every semisimple $\delta\in \bG(\cO)\lambda(\varpi)\bG(\cO)$ (and for any choice of $T_\delta$ containing $\delta$),
$$\forall \alpha\in \Phi_\delta,\quad v_F(\alpha(\delta))\in [-m_{\Xi^{\spl}} \|\lambda\|,m_{\Xi^{\spl}} \|\lambda\|].$$
\end{lem}

\begin{proof}
  The argument is the same as in the proof of Lemma \ref{l:bounding1-alpha(gamma)}. The constant $m_{\Xi^{\spl}}$ corresponds to the constant $B_5$ in that lemma. To see that it is independent of $p$, $F$ and $\lambda$, it suffices to examine the argument and see that the constant depends only on $\bG$, $\bB$, $\bT$ (and the auxiliary choice of $\tilde{\alpha}$'s as in the proof of Lemma \ref{l:conj-image-in-diag}, which is fixed once and for all).
%
%
%  There exists $c_{\Xi^{\spl}}>0$ such that $\|\Xi^{\spl}(\lambda)\|_{\GL_m}\le c_{\Xi^{\spl}} \|\lambda\|$
%  for all $\lambda\in X_*(\bT)$.
%  Then every entry $a$
%  of the diagonal matrix $\Xi^{\spl}(\lambda(\varpi))$ has the property that
%  $v_F(a)\in [-c_{\Xi^{\spl}} \|\lambda\|, c_{\Xi^{\spl}} \|\lambda\|]$. Let $\T_\delta$ be a maximal torus of $\GL_m$ (over $\ol{F}$) containing $\Xi^{\spl}(T_\delta)$. The map $\Xi^{\spl}$ induces a surjection
%  $X^*(\T_\delta)\ra X^*(T_\delta)$. We fix a lift $\tilde{\alpha}\in X^*(\T_\delta)$ of each $\alpha\in \Phi_\delta$
%  and choose $a_{\Xi^{\spl}}>0$ such that \beq\label{e:tilde-alpha}
%  \|\tilde{\alpha}\|_{\GL_m}\le a_{\Xi^{\spl}},\quad \forall \alpha\in \Phi_\delta\eeq
%  As $\delta\in \bG(\cO)\lambda(\varpi)\bG(\cO)$, we have
%  $\Xi^{\spl}(\delta)\in \GL_m(\cO)\Xi^{\spl}(\lambda(\varpi))\GL_m(\cO)$. By Lemma \ref{l:control-eigenvalue},
%  every eigenvalue $\mathrm{ev}$ of $\Xi^{\spl}(\delta)$ satisfies
%  $v_F(\mathrm{ev})\in [-c_{\Xi^{\spl}} \|\lambda\|, c_{\Xi^{\spl}} \|\lambda\|]$.
%  This together with \eqref{e:tilde-alpha} shows that
%  $$v_F(\tilde{\alpha}(\Xi^{\spl}(\delta)))\in [-a_{\Xi^{\spl}} c_{\Xi^{\spl}} \|\lambda\|,a_{\Xi^{\spl}} c_{\Xi^{\spl}} \|\lambda\|].$$
%  Since $\alpha(\delta)=\tilde{\alpha}(\Xi^{\spl}(\delta))$, we finish by setting $m_{\Xi^{\spl}}:=a_{\Xi^{\spl}} c_{\Xi^{\spl}}$.
\end{proof}

The unipotent radical of $\bB$ is denoted $\bN$.
For $F$ as above, let $x_0$ be the hyperspecial vertex on the building of $\bG(F)$
corresponding to $\bG(\cO)$. % Denote by $N_{x_0,a}$ the filtration subgroup of $\bN(F)$ for $a\in \R$.
As usual put $\Phi^+:=\Phi^+(\bG,\bT)$ be the set of positive roots with respect to $(\bB,\bT)$.

  Let us recall some facts about the Chevalley basis. For each $\alpha\in \Phi^+$, let $U_\alpha$ denote the corresponding unipotent subgroup equipped with $x_\alpha:\G_a \simeq U_\alpha$. Order the elements of $\Phi^+$ as $\alpha_1,...,\alpha_{|\Phi^+|}$ once and for all such that simple roots appear at the beginning.
  The multiplication map $$\mathrm{mult}:U_{\alpha_1}\times \cdots \times U_{\alpha_{|\Phi^+|}} \ra \bN,
\qquad (u_1,...,u_{|\Phi^+|})\mapsto u_1\cdots u_{|\Phi^+|}$$
  is an isomorphism of schemes (but not as group schemes) over $\Z$. This can be deduced from \cite[Exp XXII, 5.5.1]{SGA3-7}, which deals with a Borel subgroup of a Chevalley group. In particular (since the ordering on $\Phi^+$ is fixed) any $n\in \bN(F)$ can be uniquely written as \beq\label{e:y=product}y=x_{\alpha_1}(Y_{\alpha_1})\cdots x_{\alpha_{|\Phi^+|}}(Y_{\alpha_{|\Phi^+|}})\eeq
  for unique $Y_{\alpha_i}\in \G_a(F)\simeq F$'s. The Chevalley commutation relation (\cite[\S III]{Che55}) has the following form:
for all $1\le i<j\le |\Phi^+|$ and all $Y_{\alpha_i}\in F$'s,
%elements $a_1,...,a_{|\Phi^+|}$ of $\G_m$ (from the functor-of-points perspective),
  \beq\label{e:commutator}
  x_{\alpha_i}( Y_{\alpha_i}) x_{\alpha_j}( Y_{\alpha_j})
 = x_{\alpha_j}( Y_{\alpha_j})x_{\alpha_i}( Y_{\alpha_i})
  \prod_{c,d\ge 1\atop
  \alpha_k=c\alpha_i+d\alpha_j} x_{\alpha_k}(C_{ij} ( Y_{\alpha_i})^c ( Y_{\alpha_j})^d)
\eeq
where $C_{ij}$ are certain integers (depending on $\bG$)
which we need not know explicitly. It suffices to know that, in the cases of $F$ we are interested in, the constants $C_{ij}$ are units in $\cO$ (cf. the assumption in the paragraph preceding Proposition \ref{p:appendix2}).

  We thank Kottwitz for explaining the proof of the following lemma.

\begin{lem}\label{l:app:split2}
  Suppose that the Chevalley group $\bG$ is semisimple and simply connected. Let $\Omega\subset X^*(\bT)$ denote the set of fundamental weights and $\rho^\vee\in X_*(\bT)$ the half sum of all positive coroots. Let $\lambda\in X^*(\bT)$ and define
  $n_0(\lambda):=\max_{\omega\in \Omega} \lg \omega,\lambda\rg$.
  For every prime $p$, every $p$-adic field $F$, and every cocharacter $\lambda\in X_*(\bT)$ as above, the following is true: in terms of the decomposition \eqref{e:y=product}, each $y\in \bG(\cO)\lambda(\varpi)\bG(\cO)\cap \bN(F)$ satisfies the inequality
  $$v_F(Y_i)\ge -2n_0(\lambda)\lg \alpha_i,\rho^\vee\rg,\quad 1\le i\le |\Phi^+|.$$

\end{lem}

\begin{proof}
  It suffices to check that
  \beq\label{e:to-prove-lem7.13}\varpi^{2 n_0(\lambda)\rho^\vee} y \varpi^{-2 n_0(\lambda)\rho^\vee}\in  \bN(\cO).\eeq
  (Here we write $\varpi^{2 n_0(\lambda)\rho^\vee}$ for $(\rho^\vee(\varpi))^{2 n_0(\lambda)}$.)
  Indeed, this implies the desired inequality in the lemma since the decomposition \eqref{e:y=product} is defined over $\cO$.

    Let us introduce some notation. For each $\omega\in \Omega$ let $V_\omega$ denote the irreducible representation of $\bG(F)$ of highest weight $\omega$ on an $F$-vector space. Write $V_\omega=\oplus_{\mu\in X^*(\bT)} V_{\omega,\mu}$ for the weight decomposition. The geometric construction of $V_\omega$ and its weight decomposition by using flag varieties gives us a natural $\cO$-integral structures $V_\omega(\cO)$ in $V_\omega$ such that $V_\omega(\cO)=\oplus_{\mu\in X^*(\bT)} V_{\omega,\mu}(\cO)$, where $V_{\omega,\mu}(\cO):=V_\omega(\cO)\cap V_{\omega,\mu}$.
    Note that each $V_\omega$ receives an action of $\G_m$ via $\G_m\stackrel{\rho^\vee}{\ra}\bT\hra \bG$. We may consider a coarser decomposition $V_\omega=\oplus_{i\in \Z} V_{\omega,i}$, where $V_{\omega,i}:=\oplus_{\lg \mu,2\rho^\vee\rg=i} V_{\omega,\mu}$. For any $\omega\in \Omega$ and $V=V_\omega$, set $V_{\ge i}:=\oplus_{j\ge i} V_j$, $V_{\ge i}(\cO):=V_{\ge i}\cap V(\cO)$, and $V_i(\cO):=V_i\cap V(\cO)$. Observe that $\bB(F)$ preserves the filtration $\{V_{\ge i}\}_{i\in \Z}$ and that $\bN(F)$ acts trivially on $V_{\ge i}/V_{\ge i+1}$.

    As a preparation, suppose that $g\in \bG(\cO)\lambda(\varpi)\bG(\cO)$ and let us prove that $g V_\omega(\cO)\subset \varpi^{-n_0(\lambda)}V_\omega(\cO)$ for all $\omega\in \Omega$. Since $\bG(\cO)$ stabilizes $V_\omega(\cO)$, the latter condition is true if and only if $\lambda(\varpi)V_\omega(\cO)\subset \varpi^{-n_0(\lambda)} V_\omega(\cO)$, which holds if and only if $$\lg \mu,\lambda\rg \ge -n_0(\lambda)$$ for all weights $\mu$ for $V_\omega$ by considering the weight decomposition. The above inequality for all weights $\mu$ is equivalent to that for the lowest weight $\mu$ for $V_\omega$. Since $\mu=w_0\omega_\omega$ for the longest Weyl element $w_0$, the condition is that $\lg -w_0 \omega,\lambda\rg \le n_0(\lambda)$ for all $\omega$. This is verified by the definition of $n_0(\lambda)$ since $-w_0$ preserves the set $\Omega$.

    Now consider $\varpi^{2 n_0(\lambda)\rho^\vee} (y-1) \varpi^{-2 n_0(\lambda)\rho^\vee}$, where $y$ is as in the lemma. Since $\varpi^{2\rho^\vee}$ acts on $V_j$ as $\varpi^j$, we see from this and the last paragraph that for all $\omega\in \Omega$ and $i\in \Z$,
    \begin{eqnarray}
    (\varpi^{2 n_0(\lambda)\rho^\vee} (y-1) \varpi^{-2 n_0(\lambda)\rho^\vee})(V_{\omega,i}(\cO))
    &=& (\varpi^{2 n_0(\lambda)\rho^\vee} (y-1))(\varpi^{- i n_0(\lambda)} V_{\omega,i}(\cO))\nonumber\\
    &\subset& \varpi^{2 n_0(\lambda)\rho^\vee}(\varpi^{- (i+1) n_0(\lambda)}V_{\omega,\ge i+1}(\cO))
    \subset V_{\omega,i}(\cO).\nonumber
    \end{eqnarray}
     It follows that $\varpi^{2 n_0(\lambda)\rho^\vee} y \varpi^{-2 n_0(\lambda)\rho^\vee}$ also preserves $V_{\omega,i}(\cO)$, hence $V_\omega(\cO)$. Therefore the element belongs to $\cN(\cO)=\cN(F)\cap\bG(\cO)$, concluding the proof of \eqref{e:to-prove-lem7.13}.

\end{proof}

  For an arbitrary Chevalley group $\bG$ and $\lambda\in X_*(\bT)^+$, define a nonnegative integer
  \beq\label{e:n(lambda)} n_\bG(\lambda):=\max_{\alpha\in \Phi^+} \lg \alpha,\lambda\rg.\eeq

\begin{cor}\label{c:split2}
  Let $\bG$ be an arbitrary Chevalley group. For every prime $p$, every $p$-adic field $F$, and every cocharacter $\lambda\in X_*(\bT)$, there exists a constant $m_\bG>0$ such that the following is true: each $y\in \bG(\cO)\lambda(\varpi)\bG(\cO)\cap \bN(F)$, uniquely decomposed as in \eqref{e:y=product}, satisfies the inequality
  $$v_F(Y_i)\ge -2m_{\bG} n_\bG(\lambda),\quad 1\le i\le |\Phi^+|.$$
\end{cor}

\begin{proof}
  The corollary is immediate from the lemma if $\bG$ is semisimple and simply connected. Indeed, define $n_1(\lambda)$ to be the maximum of $\lg \alpha,\lambda\rg$ as $\alpha$ runs over $\Delta^+$, the set of simple roots. Observe that both the sets $\Omega$ and $\Delta^+$ are bases for $X^*(\bT)_\Q$. By using the change of basis matrix, it is easy to deduce from Lemma \ref{l:app:split2} that for some constant $c>0$ depending only on $\bG$, we have that
  $$v_F(Y_i)\ge -2c n_1(\lambda)\lg \alpha_i,\rho^\vee\rg$$ for all $p$, $F$, $\lambda$, and $i$. A fortiori the same holds with $n_\bG(\lambda)$ in place of $n_1(\lambda)$. The proof is completed by setting $m_\bG:=c \max_{\alpha\in \Phi^+}\lg \alpha,\rho^\vee\rg$.

  It remains to extend from the simply connected case to the general case. As usual write $\bG_{\ad}$ for the adjoint group of $\bG$ and $\bG_{\scusp}$ for the simply connected cover of $\bG_{\ad}$. The pair $(\bB,\bT)$ induces the Borel pairs $(\bB_{\ad},\bT_{\ad})$ for $\bG_{\ad}$ and $(\bB_{\scusp},\bT_{\scusp})$ for $\bG_{\scusp}$. Write $\Phi^+_{\ad}$ and $\Phi^+_{\scusp}$ for the associated sets of roots. Let $\bN_{\ad}$ and $\bN_{\scusp}$ denote the unipotent radicals of $\bB_{\ad}$ and $\bB_{\scusp}$, respectively. Then the natural maps $\bG\ra \bG_{\ad}$ and $\bG_{\scusp}\ra \bG_{\ad}$ induce isomorphisms $\bN\simeq \bN_{\ad}$ and $\bN_{\scusp}\simeq \bN_{\ad}$ as well as set-theoretic bijections $\Phi^+\ra\Phi^+_{\ad}$ and $\Phi^+_{\scusp}\ra \Phi^+_{\ad}$. In particular the ordering on $\Phi^+$ induces unique orderings on $\Phi^+_{\ad}$ and $\Phi^+_{\scusp}$. With respect to these orderings, the decomposition \eqref{e:y=product} is compatible with the maps $\bG\ra \bG_{\ad}$ and $\bG_{\scusp}\ra \bG_{\ad}$. From all this it follows that the corollary for $\bG_{\scusp}$ implies that for $\bG_{\ad}$, and then for $\bG$.

\end{proof}

\section{Lemmas on conjugacy classes and level subgroups}\label{s:conj}

  This section contains several results which are useful for estimating the geometric side of Arthur's invariant trace formula in the next section.

\subsection{Notation and basic setup}\label{sub:notation-setup} Let us introduce some global notation in addition to that at the start of \S\ref{sec:pp}.

\bit
\item $M_0$ is a minimal $F$-rational Levi subgroup of $G$.
\item $A_{M_0}$ is the maximal split $F$-torus in the center of $M_0$.
\item $\Ram(G):=\{v\in \cV_F^\infty:\, G\mbox{~is~ramified~at~}v\}$.
\item $S\subset \cV_F^\infty$ is a finite subset, often with a partition
$S=S_0\coprod S_1$.
\item $r:{}^L G \ra \GL_d(\C)$ is an irreducible representation.
\item $\Xi:G\ra \GL_m$ is a faithful algebraic representation defined over $F$
(or over $\cO_F$ as explained below)
%\item $S$ is a finite set of finite places
%  which admits a partition $S=S_0\coprod S_1$. (NOT NEEDED HERE.)
% \item $U_{S_0}$ is a fixed open compact subgroup of $ G(F_{S_0})$. (NOT NEEDED HERE.)
% \item $U_\infty$ is a compact subset of $G_\infty$. (NOT NEEDED HERE.)
\item For any $\C$-subspace $\cH'\subset C^\infty_c(G(F_S))$,
define $$\supp\cH'=\cup\, \supp \phi_S$$ where the union is take over $\phi_S\in\cH'$.
\item $q_{S}:=\prod_{v\in S} q_v$ where $q_v$ is the cardinality of the residue field at $v$. (Convention: $q_S=1$ if $S=\emptyset$.)
\eit

  For each finite place $v\in \Ram(G)$ of $F$, fix a special
  point $x_v$ on the building of $G$ once and for all, where $x_v$ is required to belong to
  an apartment corresponding to a maximal $F_v$-split torus $A_v$ containing $A_{M_0}$.
  The stabilizer $K_v$ of $x_v$ is a good special maximal compact subgroup of $G(F_v)$
  (good in the sense of \cite{BT72}).
  Set $K_{M,v}:=K_v\cap M(F_v)$ for each $F_v$-rational Levi subgroup $M$ of $G$ containing $A_v$. Then $K_{M,v}$
  is a good special maximal compact subgroup of $M(F_v)$.
%  When $G$ is unramified at $v$, choose $x_v$ to be hyperspecial, and

  It is worth stressing that this article treats a reductive group $G$ without any hypothesis on $G$
  being split (or quasi-split). To do so, we would like to carefully choose an integral model of $G$ over $\cO_F$
  for convenience and also for clarifying a notion like ``level $\fkn$ subgroups''. We thank Brian Conrad for explaining us
  crucial steps in the proof below (especially how to proceed by using the facts from \cite{BLR90}).

\begin{prop}\label{p:global-integral-model}
  The $F$-group $G$ extends to a group scheme $\fkG$ over $\cO_F$ (thus equipped with an isomorphism
  $\fkG\times_{\cO_F} F\simeq G$) such that \begin{itemize}
    \item $\fkG\times_{\cO_F} \cO_F[\frac{1}{\Ram(G)}]$ is a reductive group scheme (cf. \cite{Conrad-reductive}),
    \item $\fkG(\cO_v)=K_v$ for all $v\in \Ram(G)$ (where $K_v$ are chosen above),
    \item there exists a faithful embedding of algebraic groups $\Xi:\fkG\hra \GL_m$ over $\cO_F$
  for some $m\ge 1$.
  \end{itemize}
\end{prop}

\begin{rem}
  If $G$ is split then $\Ram(G)$ is empty and the above proposition is standard in the theory of Chevalley groups.
\end{rem}

\begin{proof}

 For any finite place $v$ of $F$, we will write $\cO_{(v)}$ for the localization of $\cO_F$ at $v$ (to be distinguished from
  the completion $\cO_v$).
  As a first step there exists an injective morphism of group schemes $\Xi_F:G\hra GL_{m}$ defined over $F$ for some $m\ge 1$
  (\cite[Prop A.2.3]{CGP10}. The scheme-theoretic closure $\fkG'$ of $G$ in $GL_{m'}$ is a smooth affine scheme over $\Spec \cO_F[1/S]$ for a finite set $S$ of primes of $\cO_F$ by arguing as in the first paragraph of \cite[\S2]{Conrad-reductive}. We may assume that $S\supset \Ram(G)$. By \cite[Prop 3.1.9.(1)]{Conrad-reductive}, by enlarging $S$ if necessary, we can arrange that $\fkG'$ is reductive.
  For $v\in \Ram(G)$ we have fixed special points $x_v$, which give rise to the Bruhat-Tits group schemes $\hat{\fkG}(v)$
  over $\cO_v$.
  Similarly for $v\in S\bs \Ram(G)$, let us choose hyperspecial points $x_v$ so that the corresponding group schemes
  $\hat{\fkG}(v)$ over $\cO_v$ are reductive.

  According to
   \cite[Prop D.4, p.147]{BLR90} the obvious functor from the category of affine $\cO_{(v)}$-schemes
   to that of triples $(X,\hat{\mathfrak{X}}(v),f)$ where $X$ is an affine $F$-scheme, $\hat{\mathfrak{X}}(v)$ is an affine $\cO_v$-scheme and
   $f:X\times_F F_v\simeq \hat{\mathfrak{X}}(v)\times_{\cO_v} F_v$ is an equivalence. (The notion of morphisms is obvious in each category.) Thanks to its functorial nature, the same functor defines
   an equivalence when restricted to group objects in each category. For $v\in \Ram(G)$,
   apply this functor to the Bruhat-Tits group scheme $\hat{\fkG}(v)$ over $\cO_v$ equipped with $G\times_F F_v\simeq \hat{\fkG}(v)\times_{\cO_v} F_v$    to obtain a group scheme $\fkG(v)$ over $\cO_{(v)}$.

   An argument analogous to that on page 14 of \cite{BLR90} shows that the obvious functor between the following categories is
   an equivalence: from
   the category of finite-type $\cO_F$-schemes to that of triples $(X,\{\mathfrak{X}(v)\}_{v\in S}, \{f_v\}_{v\in S})$
   where $X$ is a finite-type $\cO_F[1/S]$-scheme, $\mathfrak{X}(v)$ is a finite-type $\cO_{(v)}$-scheme and
   $f_v:X\times_{\cO_F[1/S]} F\simeq \mathfrak{X}(v)\times_{\cO_{(v)}} F$ is an isomorphism. Again this induces an equivalence when restricted
   to group objects in each category. In particular, there exists a group scheme $\fkG$ over $\cO_F$ with isomorphisms
   $\fkG\times_{\cO_F} \cO_F[1/S]\simeq \fkG'$ and $\fkG\times_{\cO_F} \cO_{(v)}\simeq \fkG(v)$ for $v\in S$
   which are compatible with the isomorphisms between $\fkG'$ and $\fkG(v)$ over $F$.
   By construction $\fkG$ satisfies the first two properties of the proposition.

   We will be done if $\Xi_F:G\hra GL_m$ over $F$ extends to an embedding of group schemes over $\cO_F$. It is evident from the construction
   of $\fkG'$ that $\Xi_F$ extends to $\Xi':G\hra GL_m$ over $\cO_F[1/S]$. For each $v\in S$,
   $\Xi_F$ extends to $\Xi(v):\fkG(v)\hra GL_m$ over $\cO_{v}$ thanks to \cite[Prop 1.7.6]{BT84}, which can be
   defined over $\cO_{(v)}$ using the first of the above equivalences. Then the second equivalence allows us to
   glue $\Xi'$ and $\{\Xi(v)\}_{v\in S}$ to produce an $\cO_F$-embedding $\Xi:G\hra GL_m$.

\end{proof}

  For each finite $v\notin \Ram(G)$, $\fkG$ defines a reductive group scheme over $\cO_v$, so
  $K_v:=\fkG(\cO_v)$ is a hyperspecial subgroup of $G(F_v)$.
  Fix a maximal $F_v$-split torus $A_v$ of $G$ which contains $A_{M_0}$ such that the hyperspecial point for $K_v$ belongs to the apartment of $A_v$. For each Levi subgroup $M$ of $G$ whose center is contained in $A_v$, define a hyperspecial subgroup $K_{M,v}:=K_v\cap M(F_v)$ of $M(F_v)$.
  At such a $v\notin \Ram(G)$
  define $\cH^{\ur}(G(F_v))$
  (resp. $\cH^{\ur}(M(F_v))$). The constant term (\S\ref{sub:orb-int-const-term}) of a function in $C^\infty_c(G(F_v))$
  (resp. $C^\infty_c(M(F_v))$) will be taken relative to $K_v$ (resp. $K_{M,v}$).
  When $P=MN$ is a Levi decomposition, we have Haar measures on $K_v$, $M(F_v)$ and $N(F_v)$
  such that
  the product measure equals $\mu_v^{\can}$ on $G(F_v)$ (cf. \S\ref{sub:orb-int-const-term})
  and the Haar measure on $M(F_v)$ is the canonical measure of \S\ref{sub:can-measure}.
  In particular when $G$ is unramified at $v$,
  \beq\label{e:meas-K-cap-N}\vol(K_v\cap N(F_v))=1\eeq with respect to the measure on $N(F_v)$ .

  Let $\fkn$ be an ideal of $\cO_F$ and $v$ a finite place of $F$.
Let $v(\fkn)\in \Z_{\ge 0}$ be the integer determined by $\fkn \cO_v=\varpi_v^{v(\fkn)} \cO_v$.
Define $K_v(\varpi_v^s)$  to be the Moy-Prasad subgroup $G(F_v)_{x_v,s}$ of $G(F_v)$
by using Yu's minimal congruent filtration as in \cite{Yu-fil} (which is slightly
different from the original definition of Moy and Prasad). Yu has shown that $G(F_v)_{x_v,s}= \ker(\fkG(\cO_v)\ra \fkG(\cO_v/\varpi_v^s))$
in \cite[Cor 8.8]{Yu-fil}. Set $K^{S,\infty}(\fkn):=\prod_{v\notin S\cup S_\infty}\ker(\fkG(\cO_v)\ra \fkG(\cO_v/\fkn))=\prod_{v\notin S\cup S_\infty} K_v(\varpi_v^{v(\fkn)})$,
to be considered the level $\fkn$-subgroup of $G(\A^{S,\infty})$.

  Fix a maximal torus $T_0$ of $G$ over $\ol{F}$ and
  an $\R$-basis $\cB_0$ of $X_*(T_0)_\R$, which induces a function
  $\|\cdot\|_{\cB_0,G}:X_*(T_0)_{\R}\ra \R_{\ge 0}$ as in \S\ref{sub:L-mor-unr-Hecke}.
  For any other maximal torus $T$, there is an inner automorphism
  of $G$ inducing $T_0\simeq T$, so $X_*(T)_\R$ has an $\R$-basis $\cB$ induced from $\cB_0$,
  well defined up to $\Omega(G,T)$. Therefore $\|\cdot\|_{\cB,G}:X_*(T)_\R\ra \R_{\ge 0}$
  is defined without ambiguity. As it depends only on the initial choice of $\cB_0$ (and $T_0$),
  let us write
  $\|\cdot\|$ for $\|\cdot\|_{\cB,G}$ when there is no danger of confusion.

  Let $v$ be a finite place of $G$, and $T_v$ a maximal torus of $G\times_F \ol{F}_v$
  (which may or may not be defined over $F_v$). Then $\|\cdot\|:X_*(T_v)_\R\ra \R_{\ge 0}$
  is defined without ambiguity via $T_v\simeq T_0\times _F F_v$ by a similar consideration
  as above. Now assume that $G$ is unramified at $v$. For any maximal split torus $A\subset G$
  and a maximal torus $T$ containing $A$ over $F_v$, the function $\|\cdot\|_{\cB_0}$ is well defined
  on $X_*(T)_\R$ (resp. $X_*(A)_\R$) and invariant under $\Omega$ (resp. $\Omega_F$).
  Hence for every $v$ where $G$ is unramified, the Satake isomorphism allows us to define
  $\cH^{\ur}(G(F_v))^{\le \kappa}$ as well as
  $\cH^{\ur}(M(F_v))^{\le \kappa}$ for every Levi subgroup $M$ of $G$ over $F_v$.
When $G$ is unramified at $S$, $\cH^{\ur}(G(F_S))^{\le \kappa}$ and
  $\cH^{\ur}(M(F_S))^{\le \kappa}$ are similarly defined.

  For the group $\GL_m$ with any $m\ge 1$, we use
  the diagonal torus and the standard basis to define $\|\cdot\|_{\GL_m}$ on the cocharacter groups
  of maximal tori of $\GL_m$ (cf. \S\ref{sub:case-of-GL_d}).
  For $\Xi:G\hra \GL_m$ introduced above, define
\beq\label{e:def-M(Xi)} B_{\Xi}:=\max_{e\in \cB_0} \|\Xi(e)\|_{\GL_m}.\eeq

\subsection{$z$-extensions}

  A surjective morphism $\alpha:H\ra G$ of connected reductive groups over $F$ is
  said to be a \key{$z$-extension} if the following three conditions are satisfied: $H\der$ is simply connected, $\ker\alpha\subset Z(H)$, and $\ker\alpha$ is isomorphic to a finite product $\prod \Res_{F_i/F}\GL_1$
 for finite extensions $F_i$ of $F$. Writing $Z:=\ker\alpha$, we often represent
  such an extension by an exact sequence of $F$-groups
  $1\ra Z \ra H\ra G\ra 1$. By the third condition and Hilbert 90, $\alpha:H(F)\ra G(F)$
  is surjective.

\begin{lem}\label{l:z-ext}
  For any $G$, a $z$-extension $\alpha:H\ra G$ exists.
  Moreover, if $G$ is unramified outside a finite set $S$, where $S_\infty\subset S\subset \cV_F$,
  then $H$ can be chosen to be unramified outside $S$.
\end{lem}

\begin{proof}
  It is shown in \cite[Prop 3.1]{MS82} that a $z$-extension exists and
  that if $G$ splits over a finite Galois extension $E$
  of $F$ then $H$ can be chosen to split over $E$. By the assumption on $G$, it is possible to find such an $E$ which is unramified outside $S$.
  Since the preimage of a Borel subgroup of $G$ in $H$ is a Borel subgroup of $H$,
  we see that $H$ is quasi-split outside $S$.

% Precise proof of the last assertion:
%
% 1. G^* is unramified outside $S$.
% 4. Let c denote the class of G in H^1(F,Int(G^*)). There exists Galois L/F such that c comes from H^1(L/F,Int(G^*)) via inflation.
% 5. If L is ramified at v not in S, then take decomposition group D(v) \subset Gal(L/F).
%    Then c is trivial in H^1(L/L^{D(v)},Int(G^*)) = H^1(I(v),Int(G^*)) by assumption.
%    Hence c comes from H^1(L^{D(v)}/F,Int(G^*)).
% 6. By repeating this, we can arrange that L is unramified outside S.
%    Then c becomes trivial over L, which means G splits over L. QED.

\end{proof}

\subsection{Rational conjugacy classes intersecting a small open compact subgroup}\label{sub:conj-intersect}

 Throughout this subsection $S=S_0\coprod S_1$ is a finite subset of $\cV^\infty_F$ and
 it is assumed that $S_0\supset \Ram(G)$. Fix compact subgroups $U_{S_0}$ and $U_{\infty}$ of $G(F_{S_0})$ and $G(F\otimes_\Q\R)$, respectively. Let $\fkn$ be an ideal of $\cO_F$ as before, now assumed to be coprime to $S$, with absolute norm $\N(\fkn)\in \Z_{\ge 1}$.

\begin{lem}\label{l:forcing-unipotent}
   Let $U_{S_1}:=\supp \cH^{\ur}(G(F_{S_1}))^{\le \kappa}$.
    There exists $c_\Xi>0$ independent of $S$, $\kappa$ and $\fkn$ (but depending on $G$, $\Xi$, $U_{S_0}$ and $U_{\infty}$)
  such that for all $\fkn$ satisfying $$\N(\fkn)\ge c_\Xi q_{S_1}^{B_{\Xi}m\kappa},$$ the following holds:
    if $\gamma\in G(F)$ and $x^{-1}\gamma x\in K^{S,\infty}(\fkn) U_{S_0} U_{S_1} U_\infty$
    for some $x\in G(\A_F)$ then $\gamma$ is unipotent.
\end{lem}

\begin{proof}
  Let $\gamma'=x^{-1}\gamma x$. We keep using the embedding $\Xi:\mathfrak{G}\hra \GL_m$ over $\cO_F$ of Proposition \ref{p:global-integral-model}.
  (For the lemma, an embedding away from the primes in $S_0$ or dividing $\fkn$ is enough.) At each finite place $v\notin S_0$ and
$v\nmid \fkn$, Lemma \ref{l:conj-image-in-diag} allows us to find $\Xi'_v:\mathfrak{G}\hra \GL_m$ over $\cO_v$ which is $\GL_m(\cO_v)$-conjugate to $\Xi\times_{\cO_F} F_v$ such that $\Xi'_v$ sends $A_v$ into the diagonal torus of $\GL_m$.

  Write $\det(\Xi(\gamma)-(1-X))=X^m+a_{m-1}(\gamma)X^{m-1}+\cdots+ a_0(\gamma)$, where
  $a_i(\gamma)\in F$ for $0\le i\le m-1$.
  Our goal is to show that $a_i(\gamma)=0$ for all $i$.
  To this end, assuming $a_i(\gamma)\neq 0$ for some fixed $i$, we will estimate $|a_i(\gamma)|_v$ at each place $v$ and
  draw a contradiction.

  First consider $v\in S_1$.
  We claim that
  $$v(a_i(\gamma))\ge-B_{\Xi}m\kappa$$ for every $\gamma$ that is conjugate to an element of $\supp \cH^{\ur}(G(F_{v}))^{\le \kappa}$.
   To prove the claim we examine the eigenvalues of $\Xi'_v(\gamma')$, which is conjugate to $\gamma$.
  We know $\gamma'$ belongs to $\supp \cH^{\ur}(G(F_v))^{\le \kappa}$, so
  $\Xi'_v(\gamma')\in \GL_m(\cO_v)\Xi'_v(\mu(\varpi_v))\GL_m(\cO_v)$ for some $\mu\in X_*(A_v)$ with $\|\mu\|\le \kappa$.
  Then $\|\Xi'_v(\mu)\|_{\GL_m}\le B_{\Xi}\kappa$. (A priori this is true for $B_{\Xi'_v}$ defined as in \eqref{e:def-M(Xi)},
but $B_{\Xi'_v}=B_{\Xi}$ as $\Xi'_v$ and $\Xi$ are conjugate.)
  Let $k_1,k_2\in \GL_m(\cO_v)$ be such that $\Xi'_v(\gamma')=k_1\Xi'_v(\mu(\varpi_v))k_2$.
  Lemma \ref{l:control-eigenvalue} shows that
   every eigenvalue $\lambda$ of $\Xi'_v(\mu(\varpi_v)) k_2k_1$ (equivalently of
   $\Xi'_v(\gamma')$) satisfies
  $v(\lambda)\ge -B_{\Xi}\kappa$.
   If $\lambda\neq 1$, we must have $v(1-\lambda)\ge -B_{\Xi}\kappa$.
  This shows that $v(a_{i}(\gamma))\ge -B_{\Xi}i \kappa$ for any $i$ such that $a_i(\gamma)\neq 0$.
  Hence the claim is true.
%  $$|a_i(\gamma)|_{S_1}\le \left(\prod_{v\in S_1} q_v^{-C_v \kappa}\right).$$

  At infinity, by the compactness of $U_\infty$, there exists $c_\Xi>0$ such that
  $$|a_i(\gamma)|_\infty<c_\Xi$$ whenever a conjugate of $\gamma\in G_\infty$ belongs to $U_\infty$.

  Now suppose that $v$ is a finite place such that $v\notin S_1$ and $v\nmid \fkn$. (This includes $v\in S_0$.)
  Then a conjugate of $\Xi(\gamma)$ lies in an open compact subgroup of $\GL_m(F_v)$.
  Therefore the eigenvalues of $\Xi(\gamma)$ are in $\cO_v$ and $$|a_i(\gamma)|_v\le 1.$$

  Finally at $v|\fkn$, we have $\Xi(x^{-1}\gamma x)-1\in \ker(\GL_m(\cO_v)\ra \GL_m(\cO_v/\varpi_v^{v(\fkn)}))$.
  Therefore $$|a_i(\gamma)|_v=|a_i(x^{-1}\gamma x)|_v\le (|\fkn|_v)^{m-i}.$$

  Now assume that $\N(\fkn)\ge  c_\Xi q_{S_1}^{-B_{\Xi}m\kappa}$.
  We assert that $a_i(\gamma)=0$ for all $i$. Indeed, if $a_i(\gamma)\neq 0$ for some $i$ then
  the above inequalities imply that
  $$1=\prod_v |a_i(\gamma)|_v < \left(\prod_{v\in S_1} q_v^{-B_\Xi m \kappa}\right)c_\Xi \prod_{v|\fkn} |\fkn|_v^{m-i}
  \le q_{S_1}^{-B_\Xi m \kappa} c_\Xi \N(\fkn)^{-1}\le 1$$
  which is clearly a contradiction.
  The proof of lemma is finished.

\end{proof}

\subsection{Bounding the number of rational conjugacy classes}

  We begin with a basic lemma, which is a quantitative version of the fact that $F^r$ is discrete
  in $ \A_F^r$.

\begin{lem}\label{l:F-in-A_F}
  Suppose that $\{\delta_v \in \R_{>0}\}_{v\in \cV_F}$ satisfies the following: $\delta_v=1$ for all but finitely many $v$
  and $\prod_v \delta_v<2^{-|S_\infty|}$. Let $\alpha=(\alpha_1,...,\alpha_r)\in \A_F^r$.
  Consider the following compact neighborhood of $\alpha$
   $$\cB(\alpha,\delta):=\{(x_1,...,x_r)\in \A_F^r\,:\, |x_{i,v}-\alpha_{i,v}|_v\le \delta_v,~\forall
   v,~\forall 1\le i\le r\}.$$
  Then $\cB(\alpha,\delta)\cap F^r$ has at most one element.
\end{lem}

\begin{proof}
  Suppose $\beta=(\beta_i)_{i=1}^r,\gamma=(\gamma_i)_{i=1}^r\in \cB(\alpha,\delta)\cap F^r$. By triangular inequalities,
  $$|\beta_{i,v}-\gamma_{i,v}|_v\le \left\{\begin{array}{cc}
    \delta_v, & v\nmid \infty,\\
    2\delta_v, & v|\infty
  \end{array}\right.$$
  for each $i$.
  Hence $\prod_v |\beta_{i,v}-\gamma_{i,v}|_v<1$. Since $\beta_{i},\gamma_{i}\in F$,
  the product formula forces $\beta_{i}=\gamma_{i}$. Therefore $\beta=\gamma$.
\end{proof}

  The next lemma measures the difference between $G(F)$-conjugacy and $G(\A_F)$-conjugacy.

\begin{lem}\label{l:F-conj-vs-A_F-conj}
  Let $X_G$ (resp. $\mX_G$) be the set of semisimple $G(F)$-(resp. $G(\A_F)$-)conjugacy
  classes in $G(F)$. For any $[\gamma]\in \mX_G$, there exist at most $(w_Gs_G)^{r_G+1}$ elements
  in $X_G$ mapping to $[\gamma]$ under the natural surjection $X_G\ra \mX_G$.
\end{lem}

\begin{proof}
  Let $[\gamma]\in \mX_G$ be an element defined by a semisimple $\gamma\in G(F)$.
  Denote by $X_\gamma$ the preimage of
  $[\gamma]$ in $X_G$. There is a natural bijection
  $$X_\gamma \quad\leftrightarrow\quad \ker(\ker^1(F,I_\gamma)\ra \ker^1(F,G)).$$
  Since $|\ker^1(F,I_\gamma)|=|\ker^1(F,Z(\hat{I}_\gamma))|$
  by \cite[\S4.2]{Kot84a}, we have $|X_\gamma|\le |\ker^1(F,Z(\hat{I}_\gamma))|$.

  Lemma \ref{l:torus-splitting} tells us that for every semisimple $\gamma$,
  the group $I_\gamma$ becomes split over a finite extension
  $E/F$ such that $[E:F]\le w_Gs_G$.
  In particular $\Gal(\ol{F}/E)$ acts trivially on
  $Z(\hat{I}_\gamma)$. % Indeed, the map $T\hra I_\gamma$ over $F$ gives rise to
%  a $\Gal(\ol{F}/F')$-equivariant map $Z(\hat{I}_\gamma)\hra \hat{T}$.
  The group $\ker^1(E,Z(\hat{I}_\gamma))$ consists of continuous homomorphisms
  $\Gal(\ol{F}/E)\ra Z(\hat{I}_\gamma)$ which are trivial on all local Galois groups.
  Hence $\ker^1(E,Z(\hat{I}_\gamma))$ is trivial. This and
  the inflation-restriction sequence show that $\ker^1(F,Z(\hat{I}_\gamma))$
  is the subset of locally trivial elements in $H^1(\Gamma_{E/F},Z(\hat{I}_\gamma))$,
  where we have written $\Gamma_{E/F}$ for $\Gal(E/F)$. In particular,
  $$|X_\gamma|\le |H^1(\Gamma_{E/F},Z(\hat{I}_\gamma))|.$$
  Let $d:=|\Gal(E/F)|$ and denote by $[d]$ the
  $d$-torsion subgroup. The long exact sequence arising from $0\ra Z(\hat{I}_\gamma)[d]
  \ra Z(\hat{I}_\gamma) \stackrel{d}{\ra} d(Z(\hat{I}_\gamma))\ra 0$ gives rise to
  an exact sequence
  $$ % 0\ra (Z(\hat{I}_\gamma)^0)^{\Gamma_{E/F}}/ (( Z(\hat{I}_\gamma)^0)^{\Gamma_{E/F}})^d  \ra
H^1(\Gamma_{E/F},Z(\hat{I}_\gamma)[d]) \ra H^1(\Gamma_{E/F},Z(\hat{I}_\gamma))
= H^1(\Gamma_{E/F},Z(\hat{I}_\gamma))[d]\ra 0.$$
  Let $\MU_{d}$ denote the order $d$ cyclic subgroup of $\C^\times$.
  Then $Z(\hat{I}_\gamma)[d]\hra  \hat{T}[d]\simeq (\MU_{d})^{\dim T}$.
  Hence
  $$|X_\gamma|\le |H^1(\Gamma_{E/F},Z(\hat{I}_\gamma)[d] )|
  \le |\Gamma_{E/F}|\cdot |Z(\hat{I}_\gamma)[d]|\le d\cdot (d)^{\dim T}\le
  (w_Gs_G)^{\dim T+1}.$$
\end{proof}

  For the proposition below, we fix a finite subset $S_0\in \cV_F^\infty$.
  Also fix open compact subsets
   $U_{S_0}\subset G(F_{S_0})$
  and $U_\infty \subset G(F_\infty)$.
  As usual we will write $S$ for $S_0\coprod S_1$.

\begin{prop}\label{p:bound-number-of-conj}

  Let $\kappa\in\Z_{\ge 0}$.
  Let $S_1\subset\cV_F^\infty\bs S_0$ be a finite subset such that $G$ is unramified
  at all $v\in S_1$.
  Set $U_{S_1}:=\supp \cH^{\ur}(G(F_{S_1}))^{\le \kappa}$,
  $U^{S,\infty}:=\prod_{v\notin S\cup S_\infty} K_v$ and
  $\mC:=U_{S_0}U_{S_1} U^{S,\infty}U_\infty$.
  Define $\mY_G$ to be the set of semisimple $G(\A_F)$-conjugacy classes of $\gamma\in G(F)$
  which meet $\mC$. Then there exist constants $A_3,B_3>0$ such that for all $S_1$ and $\kappa$ as above,
  $$|\mY_G|=O(q_{S_1}^{A_3+B_3\kappa})$$
 (In other words, the implicit constant for $O(\cdot)$ is independent of
 $S_1$ and $\kappa$.)
\end{prop}

\begin{rem}
  By combining the proposition with Lemma~\ref{l:forcing-unipotent} we can deduce the following. Under the same assumption but with $\mC:= U^{S,\infty}(\fkn) U_{S_0} U_{S_1} U_\infty$ we have
  $$|\mY_G|=1+O(q_{S_1}^{A+B\kappa}\N(\fkn)^{-C}).$$
  for some constants $A,B,C>0$.
\end{rem}

%\begin{rem}
%  A similar argument applies to $|\mY_M|$ in the proof of \cite[Thm 4.10]{Shi-Plan}.
%\end{rem}

\begin{proof}
  Our argument will be a refinement of the proof of \cite[Prop 8.2]{Kot86}.
  \medskip

  \noindent{STEP I. When $G\der$ is simply connected.}

   Choose a smooth reductive integral model $\mathfrak{G}$ over $\cO_F[\frac{1}{S_0}]$ for $G$ and an embedding of algebraic groups
  $\Xi_0:\mathfrak{G}\ra \GL_{m}$ defined over $\cO_F[\frac{1}{S_0}]$ as in Proposition \ref{p:global-integral-model}.
  Consider \beq \label{e:char-poly-map} G(\A_F)\stackrel{\Xi_0}{\ra} \GL_m(\A_F)\ra \A_F^m\eeq
  where the latter map assigns the coefficients of the characteristic polynomial,
  and call the composite map $\Xi'$.
  Set $U':=\Xi'(U)$. Then $|U'\cap F^m|<\infty$
  since it is discrete and compact. We would like to estimate the cardinality.

  Fix $\{\delta_v\}$ which satisfies the assumption of
  Lemma \ref{l:F-in-A_F} and the condition that $\delta_v=1$ for all finite $v$.
  (So $\{\delta_v\}$ depends only on $F$.)
  Since $\Xi_0$ is defined over $\cO_F[\frac{1}{S_0}]$, clearly
  $\Xi_0(U^{S,\infty})\subset \GL_m(\hat{\cO}_F^{S,\infty})$. Thus
  $$\Xi'(U^{S,\infty})\subset (\hat{\cO}_F^{S,\infty})^m=\prod_{v\notin {S\cup S_\infty}}
  \cB_v(0,1).$$
  Set $J^{S,\infty}:=\{0\}\subset (\A_F^{S,\infty})^m$.
  Similarly as above, $\Xi'(U_{S_1})\subset (\cO_{F,S_1})^m$.
  By the compactness of $U_{S_0}$ and $U_\infty$, there exist
  finite subsets $J_{S_0}\subset F_{S_0}$ and $J_{\infty}\subset F_\infty$ such that
  $$\Xi'(U_{S_0})\subset \bigcup_{\beta_{S_0}\in J_{S_0}} \left( \prod_{v\in S_0} \cB_v(\beta_v,1)\right),
  \quad \Xi'(U_\infty)\subset \bigcup_{\beta_\infty\in J_\infty} \left( \prod_{v|\infty} \cB_v(\beta_v,\delta_v)\right).$$
  Now we treat the places contained in $S_1$. Let $T$ be a maximal torus of $G$ over $\ol{F}$. Since the image of the composite map
  $T_{\ol{F}}\hra G_{\ol{F}} \stackrel{\Xi_0}{\hra} (\GL_m)_{\ol{F}}$
  is contained in a maximal torus of $\GL_m$, we can choose $g=(g_{ij})_{i,j=1}^m\in \GL_m(\ol{F})$ such that
  $g\Xi_0(T_{\ol{F}})g^{-1}$ sits in the diagonal maximal torus $\T$ of $\GL_d$.
  Fix the choice of $T$ and $g$ once and for all (independently of $S_1$ and $\kappa$) until the end of Step I.
  Set $v_{\min}(g):=\min_{i,j} v(g_{ij})$ and $v_{\max}(g):=\max_{i,j} v(g_{ij})$.
  There exists $B_6>0$ such that for any $\mu\in X_*(T)$ with $\|\mu\|\le \kappa$, the element
  $g\Xi_0(\mu)g^{-1}\in X_*(\T)$ satisfies $\|g\Xi_0(\mu)g^{-1}\|\le B_6\kappa$.
  Let $\gamma_{S_1}=(\gamma_v)_{v\in S_1}\in U_{S_1}$. Each
  $\gamma_v$ has the form $\gamma_v=k_1\mu(\varpi_{v})k_2$ for some $\|\mu\|\le \kappa$
  and $k_1,k_2\in G(\cO_v)$. Since $\Xi_0(G(\cO_v))\subset \GL_m(\cO_v)$,
  we see that $\Xi_0(\gamma_v)$ is conjugate to $\Xi_0(\mu(\varpi_{v}))k'$ in $\GL_m(F_v)$
  for some $k'\in \GL_m(\cO_v)$.
  Applying Lemma \ref{l:control-eigenvalue} to
  $(g\Xi_0(\mu(\varpi_{v}))g^{-1})( gk'g^{-1})$ with $u=gk'g^{-1}$
  and noting that $v_{\min}(u)\ge v_{\min}(g)+v_{\min}(g^{-1})$,% and $v_{\max}(u)\le v_{\max}(g)-v_{\min}(g)$,
  we conclude that each eigenvalue $\lambda$ of $\Xi_0(\gamma_v)$ satisfies
  $$v(\lambda)\le -B_6\kappa+v_{\min}(g)+v_{\min}(g^{-1}).$$
  Therefore the coefficients of its characteristic polynomial
  lie in $\varpi_v^{-m(B_6\kappa+A_4)}\cO_v$,
  where we have set $A_4:=-(v_{\min}(g)+v_{\min}(g^{-1}))\ge 0$. To put things together, we see that
  $$\Xi'(U_{S_1})\subset \prod_{v\in S_1} (\varpi_v^{-m(B_6\kappa+A_4)}\cO_v)^m.$$
  (A fortiori $\Xi'(U_{S_1})\subset \prod_{v\in S_1} (\varpi_v^{-i(B_6\kappa+A_4)}\cO_v)_{i=1}^m $
  holds as well.) The right hand side is equal to the union of $\prod_{v\in S_1}
  \cB_v(\beta_v,1)$, as $\{\beta_v\}_{v\in S_1}$ runs over $J_{S_1}=\prod_{v\in S_1} J_v$,
  where $J_v$ is a set of representatives for $(\varpi_v^{-m(B_6\kappa+A_4)}\cO_v/\cO_v)^m$.
  Notice that $|J_{S_1}|=q_{S_1}^{m^2(B_6 \kappa+A_4)}$.
  Finally, we see that
  $$U'=\Xi'(U)\subset \bigcup_{\beta\in J} \cB(\beta,\delta)$$
  where $J=J_{S_0}\times J_{S_1}\times J^{S,\infty}\times J_\infty$. Lemma \ref{l:F-in-A_F}
  implies that
  $$U'\cap F^m \le |J| = |J_{S_0}|\cdot |J_{S_1}|\cdot |J_\infty|= O(q_{S_1}^{m^2(B_6 \kappa+A_4) }),$$
  since $|J_{S_0}|\cdot |J_\infty|$ is a constant independent of $\kappa$ and $S_1$.

  For each $\beta\in U'\cap F^m$, we claim that there are at most $m!$ semisimple $G(\ol{F})$-conjugacy classes
  in $G(\ol{F})$ which map to $\beta$ via $G(\ol{F})\ra \GL_m(\ol{F})\ra \ol{F}^m$,
  which is the analogue of \eqref{e:char-poly-map}. Let us verify the claim.
  Let $T'$ and $\T'$ be maximal tori in $G$ and $\GL_m$ over $\ol{F}$,
  respectively, such that $\Xi'(T')\subset \T'$. Then the set of semisimple conjugacy classes in
  $G(\ol{F})$ (resp. $\GL_m(\ol{F})$) is in a natural bijection with $T'(\ol{F})/\Omega$ (resp. $\T'(\ol{F})/\Omega_{\GL_m}$).
  The map $\Xi'|_{T'}: T'\ra \T'$ induces a map $T'(\ol{F})/\Omega\ra \T'(\ol{F})/\Omega_{\GL_m}$. Each fiber of the latter map
  has cardinality at most $m!$, hence the claim follows.

  Fix $\beta\in U'\cap F^m$.
  We would like to bound the number of $G(\A_F)$-conjugacy classes of $\gamma\in G(F)$ which meet $U$ and
  are $G(\ol{F})$-conjugate to $\beta$.
  Fix one such $\gamma$, which we assume to exist. (Otherwise our final bound will only improve.)
  Let $\Phi_\gamma$ denote the set of roots over $\ol{F}$ for any choice of maximal torus of $\gamma$ in $G$.
  Define $V'(\gamma)$ to be the set of places $v$ of $F$ such that $v\notin S\cup S_\infty$ and
 $\alpha(\gamma)\neq 1$ and $|1-\alpha(\gamma)|_v<1$ for at least one $\alpha\in \Phi_\gamma$.
 Put $V(\gamma):=V'(\gamma)\cup S\cup S_\infty$.
Clearly $|V(\gamma)|<\infty$. Moreover we claim that $|V(\gamma)|=O(1)$ (bounded independently of $\gamma$).
Set $$C_{S_0}:=\sup_{\gamma\in U_{S_0} U_\infty}\left(\prod_{\alpha\in \Phi_\gamma}
|1-\alpha(\gamma)|_{S_0}|1-\alpha(\gamma)|_{S_\infty}\right),$$ which is finite since $U_{S_0} U_\infty$ is compact.
Then
$$1= \prod_{v} \prod_{\alpha\in \Phi_\gamma} |1-\alpha(\gamma)|_v
= \left( \prod_{\alpha\in \Phi_\gamma} |1-\alpha(\gamma)|_{V(\gamma)}\right)
\le C_{S_0}\cdot  \left(\prod_{\alpha\in \Phi_\gamma}q_{V'(\gamma)}^{-1} \right)\le C_{S_0} 2^{-|V'(\gamma)|}.$$
Thus $|V'(\gamma)|=O(1)$ and also $|V(\gamma)|=O(1)$.

  We are ready to bound the number of $G(\A_F)$-conjugacy classes in $ G(F)$ which meet $\mC$ and
  are $G(\ol{F})$-conjugate to $\alpha$. For any such conjugacy class of $\gamma'\in G(F)$,
  the first paragraph of \cite[p.391]{Kot86} shows that
  $\gamma'$ is conjugate to $\gamma$ in $G(F_v)$ whenever $v\notin V(\gamma)$. Hence
  the number of $G(\A_F)$-conjugacy classes of such $\gamma'$ is at most $u_G^{|V(\gamma)|}$,
  where $u_G$ is the constant of Lemma \ref{l:bounding-conj-in-st-conj} below.

  Putting all this together, we conclude that
  $|\mY_G|=O(q_{S_1}^{m^2(B_7 \kappa+A_5)})$ as $S_1$ and $\kappa$ vary.
  The lemma is proved in this case.

 \medskip

  \noindent{STEP II. General case.}

  Now we drop the assumption that $G\der$ is simply connected.
  By Lemma \ref{l:z-ext}, choose a $z$-extension
  $$1\ra Z\ra H \stackrel{\alpha}{\ra} G\ra 1.$$
  Our plan is to argue as on page 391 of \cite{Kot86} with a specific choice of
  $\mC_H$ and $\mC_Z$ below (denoted $C_H$ and $C_Z$ by Kottwitz). In order to explain this choice,
  we need some preparation.
  If $v\notin S\cup S_\infty$, choose $K_{H,v}$ to be a hyperspecial subgroup
  of $H(F_v)$ such that $\alpha(K_{H,v})=K_v$. (Such a $K_{H,v}$ exists by
  the argument of \cite[p.386]{Kot86}.) We can find
  compact sets $U_{H,S_0}\subset H(F_{S_0})$ and $U_{H,\infty}$ of $H(F_\infty)$
  such that $\alpha(U_{H,S_0})=U_{S_0}$
  and $\alpha(U_{H,\infty})=U_\infty$.
  Moreover, in Lemma \ref{l:claim} below we prove the following:

  \medskip

 \begin{claim}\label{claim} There exists a constant $\beta>0$ independent of $\kappa$ and $S_1$
  with the following property: for any $\kappa\in \Z_{\ge 0}$, we can choose
  an open compact subset $U_{H,S_1}\subset  \supp \cH^{\ur}(H)^{\le \beta\kappa}$
  such that $\alpha(U_{H,S_1})=U_{S_1}$.
  \end{claim}

%   For each $v\in S_1$, choose $c_v>0$ such that
%  $\alpha(\supp \cH^{\ur}(H(F_{v}))^{\le c_v})$ contains $\mu(\varpi_v)$ for
%  all $\mu\in X_*(T)$ with $|\mu|_{\cB}\le 1$. (Such a $c_v$ exists
%  since there are finitely many $\mu$'s and $\alpha:H(F_v)\twoheadrightarrow G(F_v)$
%  is surjective.)

Now choose $U_{Z,S_1}$ to be the kernel of $\alpha: U_{H,S_1}
  \ra U_{S_1}$, which is compact and open in $Z(F_{S_1})$. Then choose a compact set $U_{Z}^{S_1}$ such that
  $U_{Z,S_1}U_{Z}^{S_1}Z(F)=Z(\A)^1$. (This is possible since $Z(F)\bs Z(\A)^1$
  is compact.\footnote{Choose $U^{S_1}_Z$ to be any open compact subgroup. Then
  $U_{Z,S_1}U_{Z}^{S_1}Z(F)$ has a finite index in $Z(\A)^1$ by compactness.
  Then enlarge $U^{S_1}_Z$ without breaking compactness such that the equality holds.})
  Set $$U_H:= \left(\prod_{v\notin S\cup S_\infty} K_{H,v}\right)
  U_{H,S_0} U_{H,S_1} U_{H,\infty},\qquad
  U_Z:=U_{Z,S_1}U_{Z}^{S_1}$$
  and set $\mC_H:=U_H\cap H(\A_F)^1$, $\mC_Z:=U_Z\cap Z(\A_F)^1$.
  Let $\mY_H$ be defined as in the statement of Proposition \ref{p:bound-number-of-conj} (with $H$ and $\mC_H$ replacing
  $G$ and $\mC$).
  Then page 391 of \cite{Kot86} shows that the natural map $\mY_H\ra \mY$ is a surjection,
  in particular $|\mY|\le |\mY_H|$.
  Since $H\der$ is simply connected, the earlier proof implies that
  $|\mY_H|=O(q_{S_1}^{B_7\beta\kappa+A_5})$ for some $B_7,A_5>0$.
  (To be precise, apply the earlier proof after enlarging $U_{H,S_1}$ to
  $\supp \cH^{\ur}(H)^{\le \beta\kappa}$ in the definition of $U_H$. Such a replacement
  only increases $|\mY_H|$, so the bound on $|\mY_H|$ remains valid.)
  The proposition follows.
\end{proof}

  We have postponed the proof of a claim in the proof of STEP II above, which we justify now. Simple as
  the lemma may seem, we apologize for not having found a simple proof.

\begin{lem}\label{l:claim}
  Claim \ref{claim} above is true.
\end{lem}

\begin{proof}
  As the claim is concerned with places in $S_1$, which (may vary but) are contained
  in the set of places where $G$ is unramified (thus quasi-split),
  we may assume that $H$ and $G$ are quasi-split over $F$
  by replacing $H$ and $G$ with their quasi-split inner forms.

  Choose a Borel subgroup $B_H$ of $H$, whose image $B=\alpha(B_H)$ is a Borel subgroup of $G$.
  The maximal torus $T_H\subset B_H$ maps to a maximal torus $T\subset B$ and there is a short exact sequence
  $$1\ra Z\ra T_H \stackrel{\alpha}{\ra} T\ra 1.$$
  The action of $\Gal(\ol{F}/F)$ on $X_*(T_H)$ factors through a finite quotient.
  Let $\Sigma$ be the quotient of $\Gal(\ol{F}/F)$ which acts faithfully on $X_*(T_H)$.
  If $v\notin S_0$ then $G$ is unramified at $v$, so the geometric Frobenius at $v$
  defines a well-defined conjugacy class, say $\mC_v$, in $\Sigma$.
  Let $A_{H,v}$ (resp. $A_v$) be the maximal split torus in $T_H$ (resp. $T$) over $F_v$.
  Then $A_{H,v}\hra T_H$ and $A_v\hra T$ induce
  $X_*(A_{H,v})\simeq X_*(T_H)^{\mC_v}$ and $X_*(A_{v})\simeq X_*(T)^{\mC_v}$.
  We claim that $X_*(T_H)\ra X_*(T)$ induces a surjective map
  $X_*(A_{H,v})\ra X_*(A_{v})$.
$$\xymatrix{ X_*(T_H) \ar[d] & X_*(T_H)^{\mC_v} \ar[d] \ar[l] & X_*(A_{H,v}) \ar[l]_-{\sim} \simeq  T_H(F_v)/T_H(\cO_v)
\ar@{->>}[d] \\
 X_*(T)  & X_*(T)^{\mC_v}  \ar[l] & X_*(A_{v}) \ar[l]_-{\sim} \simeq T(F_v)/T(\cO_v)}$$
 Indeed, we have an isomorphism $ X_*(A_{H,v})\simeq  T_H(F_v)/T_H(\cO_v)$ via $\mu\mapsto \mu(\varpi_v)$
 and similarly $ X_*(A_{v})\simeq T(F_v)/T(\cO_v)$. Further, $\alpha:T_H(F_v)\ra T(F_v)$ is surjective
 since $H^1(\Gal(\ol{F}_v/F_v),Z(\ol{F}_v))$ is trivial (as $Z$ is an induced torus).

  Denote by $[\Sigma]$ the finite set of all conjugacy classes in $\Sigma$. For
  $\mC\in [\Sigma]$, choose $\Z$-bases $\cB_{H,\mC}$ and $\cB_{\mC}$ for
  $X_*(T_H)^{\mC}$ and $X_*(T)^{\mC}$ respectively.
  (Note that $\Z$-bases $\cB_H$  for $X_*(T)$ and $\cB$ for $X_*(T_H)$ are fixed once and for all.)
   An argument as in the proof of Lemma
  \ref{l:norm-and-bases} shows that there exist constants $c(\cB_{\mC}),c(\cB_{H,\mC})>0$
  such that for all $x\in X_*(T_H)^{\mC}_{\R}$ and $y\in X_*(T)^{\mC}_{\R}$,
  \beq\label{e:for-l:claim}|x|_{\cB_{H,\mC}}\ge c(\cB_{H,\mC})\cdot \|x\|_{\cB_H},\quad
  |y|_{\cB_{\mC}}\le c(\cB_{\mC})\cdot \|y\|_{\cB}.\eeq
  Set $m_\mC:= \max_y (\min_x |x|_{\cB_{H,\mC}})$, where
  $y\in X_*(T)^{\mC}$ varies subject to the condition $|y|_{\cB_{\mC}}\le 1$
  and $x\in  X_*(T_H)^{\mC}$ runs over the preimage of $y$. (It was shown above that
  the preimage is nonempty.)
  Then by construction, for every $y \in X_*(T)^{\mC}$, there exists $x$ in the preimage of $y$
  such that $|x| _{\cB_{H,\mC}} \le m_{\mC} |y|_{\cB_{\mC}}$.

  Recall that $U_{S_1}=\prod_{v\in S_1} U_v$ where $U_v=\cup_{\mu} K_v \mu(\varpi_v) K_v$,
  the union being taken over $\mu\in X_*(T)^{\mC_v}$ such that $\|\mu\|_{\cB}\le \kappa$.
  We have seen that there exists $\mu_H\in X_*(T_H)^{\mC_v}$ mapping to $\mu$
  and $|\mu_H|_{\cB_{H,\mC_v}}\le m_{\mC_v} |\mu|_{\cB_{\mC_v}}$. By \eqref{e:for-l:claim},
  $$\|\mu_H\|_{\cB_{H}}\le m_{\mC_v} c(\cB_{H,\mC_v})^{-1}c(\cB_{\mC_v}) \|\mu\|_{\cB}.$$
  Take $\beta:=\max_{\mC\in [\Sigma]} (m_{\mC} c(\cB_{H,\mC})^{-1}c(\cB_{\mC}))$.
  Clearly $\beta$ is independent of $S_1$ and $\kappa$.
  Notice that $\|\mu_H\|_{\cB_{H}}\le \beta \|\mu\|_{\cB}\le \beta \kappa$.

  For each $\mu\in X_*(T)^{\mC_v}$ such that $\|\mu\|_{\cB}\le \kappa$,
  we can choose a preimage $\mu_H$ of $\mu$ such that $\|\mu_H\|_{\cB_{H}}\le\beta \kappa$.
  Take $U_{H,v}$ to be the union of $K_{H,v} \mu_H(\varpi_v) K_{H,v}$ for those
  $\mu_H$'s. By construction $\alpha(U_{H,v})=U_v$. Hence $U_{H,S_1}:=\prod_{v\in S_1}U_{H,v}$
  is the desired open compact subset in the claim of Lemma \ref{l:claim}.

\end{proof}

\begin{cor}\label{c:bound-number-of-conj}
  In the setting of Proposition \ref{p:bound-number-of-conj},
  let $Y_G$ be the set of all semisimple $G(F)$-conjugacy (rather than $G(\A_F)$-conjugacy)
  classes whose $G(\A_F)$-conjugacy classes intersect $\mC$. Then there exist
   constants $A_6,B_8>0$ such that $|Y_G|=O(q_{S_1}^{B_8\kappa+A_6})$ as $S_1$ and $\kappa$ vary.
\end{cor}

\begin{proof}
  Immediate from Lemma \ref{l:F-conj-vs-A_F-conj} and Proposition \ref{p:bound-number-of-conj}.
\end{proof}

  The following lemma was used in Step I of the proof of Proposition \ref{p:bound-number-of-conj}
  and will be applied again to obtain Corollary \ref{c:bounding-pi0-general} below.

\begin{lem}\label{l:bounding-conj-in-st-conj}
  Assume that $G\der$ is simply connected.
  For each $v\in \cV_F$ and each semisimple $\gamma\in G(F)$, let $n_{v,\gamma}$ be the number of $G(F_v)$-conjugacy classes
  in the stable conjugacy class of $\gamma$ in $G(F_v)$.
  Then there exists a constant $u_G\ge1$ (depending only on $F$ and $G$) such that one has the uniform bound $n_{v,\gamma}\le u_G$ for all $v$ and $\gamma$.
\end{lem}

\begin{proof}
  Put $\Gamma(v):=\Gal(\ol{F}_v/F_v)$.
  It is a standard fact that $n_{v,\gamma}$ is the cardinality of
  $\ker(H^1(F_v,I_\gamma)\ra H^1(F_v,G))$. By \cite{Kot86},
  $H^1(F_v,I_\gamma)$ is isomorphic to the dual of $\pi_0(Z(\hat{I}_\gamma)^{\Gamma(v)})$. Hence
  $n_{v,\gamma}\le |\pi_0(Z(\hat{I}_\gamma)^{\Gamma(v)})|.$
  It suffices to show that a uniform bound for $|\pi_0(Z(\hat{I}_\gamma)^{\Gamma(v)})|$ exists.

  By Lemma \ref{l:torus-splitting},
  there exists a finite Galois extension $E/F$ with $[E:F]\le w_Gs_G$ such that
  $I_{\gamma}$ splits over $E$. Then $\Gal(\ol{F}/F)$ acts on $Z(\hat{I}_\gamma)$
  through $\Gal(E/F)$. In particular $\Gamma(v)$ acts on $Z(\hat{I}_\gamma)$
  through a group of order $\le w_Gs_G$. Denote the latter group by $\Gamma(v)'$.

  By the assumption, all Levi subgroups of $G$ have simply connected derived subgroups.
  As $I_\gamma$ becomes isomorphic to a Levi subgroup of $G$ over $\ol{F}$,
  $I_\gamma\der$ is also simply connected. Hence $Z(\hat{I}_\gamma)$ is connected
  (\cite[(1.8.3)]{Kot84a}), namely a complex torus. Moreover
  $\dim Z(\hat{I}_\gamma)\le r_G$.

   Now consider the set of pairs $$\mT=\{(\Delta,\hat{T}):
   |\Delta|\le w_Gs_G,~\dim \hat{T}\le r_G\}$$
   consisting of a $\C$-torus $\hat{T}$ with an action by a finite group $\Delta$. Two pairs $(\Delta,\hat{T})$
   and $(\Delta',\hat{T}')$ are equivalent if there are isomorphisms $\Delta\simeq \Delta'$
   and $\hat{T}\simeq \hat{T}'$ such that the group actions are compatible.
   Note that $$(\Gamma(v)',Z(\hat{I}_\gamma))\in \mT$$ and that
   $\mT$ depends only on $G$ and $F$.
   Clearly $|\pi_0(\hat{T}^\Delta)|$ depends only on the equivalence class
   of $(\Delta,\hat{T})\in \mT$.
   Hence the proof will be complete if
   $\mT$ consists of finitely many equivalence classes.

   Clearly there are finitely many isomorphism classes for $\Delta$ appearing in $\mT$.
   So we may fix $\Delta$ and prove the finiteness of isomorphism classes of
   $\C$-tori with $\Delta$-action. By dualizing, it is enough to show that
   there are finitely many isomorphism classes of $\Z[\Delta]$-modules
   whose underlying $\Z$-modules are free of rank $\le r_G$.
   This is a result of \cite[\S79]{CR62}.

% OLD incorrect argument
%
%   There exists a finite extension $F'_v$ of $F_v$ over which $I_\gamma$ becomes
%  a Levi subgroup of $G$. Therefore $Z(\hat{I}_\gamma)^{\Gamma(v)}$ has the form
%  $Z(\hat{M})^{\Gamma(v)}$ for a Levi subgroup $\hat{M}$ of $G$
%  and a certain action $\Gamma(v)\ra \Aut_\C(\hat{M})$ factoring
%  through a finite quotient. Clearly $Z(\hat{M})^{\Gamma(v)}$ depends only on the image of the induced map
%  $\Gamma(v)\ra \Out_\C(\hat{M})$. To sum up, $Z(\hat{I}_\gamma)^{\Gamma(v)}$ belongs to the set
%  $\{ Z(\hat{M})^{\Delta}\}_{\hat{M},\Delta}$ where $\hat{M}$ is a
%  Levi subgroup of $\hat{G}$ up to $\hat{G}$-conjugacy and $\Delta$ is a subgroup of
%  the finite group $\Out_\C(\hat{M})$.
%  The lemma follows from the fact that the latter set is finite and independent of $\gamma$.

\end{proof}

\begin{cor}\label{c:bounding-pi0-general}
  There exists a constant $c>0$ (depending only on $G$) such that
  for every semisimple $\gamma\in G(F)$,
  $|\pi_0(Z(\hat{I}_\gamma)^{\Gamma})|<c$.
  (We do not assume that $G\der$ is simply connected.)
\end{cor}

\begin{proof}
  Suppose that $G\der$ is simply connected. The proof of Lemma \ref{l:bounding-conj-in-st-conj}
  shows that $(\Gal(E/F),Z(\hat{I}_\gamma))\in \mT$ in the notation there,
  thus there exists $c>0$ such that $|\pi_0(Z(\hat{I}_\gamma)^{\Gamma})|<c$ for all semisimple $\gamma$.

  In general, let $1\ra Z\ra H\ra G\ra 1$ be a $z$-extension over $F$ so that
  $Z$ is a product of induced tori and $H\der$ is simply connected.
  Since $H(F)\twoheadrightarrow G(F)$, we may choose a semisimple $\gamma_H$ mapping
  to $\gamma$. Let $I_{\gamma_H}$ denote the connected centralizer of $\gamma_H$ in $H$.
  By the previous argument there exists $c_H>0$ such that
  $|\pi_0(Z(\hat{I}_{\gamma_H})^{\Gamma})|<c_H$ for any semisimple $\gamma_H$.
  The obvious short exact sequence $1\ra Z\ra I_{\gamma_H}\ra I_\gamma\ra 1$ gives rise
  (\S\ref{sub:L-groups})
  to a short exact sequence $1\ra Z(\hat{I}_\gamma)\ra Z(\hat{I}_{\gamma_H})\ra\hat{Z}\ra 1$, hence
  by \cite[Cor 2.3]{Kot84a},
  \beq\label{e:bounding-pi0-general}
   0\ra \coker(X_*(Z(\hat{I}_{\gamma_H}))^{\Gamma}\ra X_*(\hat{Z})^{\Gamma})
  \ra \pi_0(Z(\hat{I}_\gamma)^{\Gamma}) \ra \pi_0(Z(\hat{I}_{\gamma_H})^{\Gamma}) \ra
  \pi_0(\hat{Z}^{\Gamma})=0.\eeq
  On the other hand, the inclusions $Z\ra I_{\gamma_H}\ra H$ induces a $\Gamma$-equivariant
  map $Z(\hat{H})\ra Z(\hat{I}_{\gamma_H})\ra \hat{Z}$. The map
  $Z(\hat{H})\ra Z(\hat{I}_{\gamma_H})$ is constructed by \cite[4.2]{Kot86}, whereas
  $Z(\hat{I}_{\gamma_H})\ra \hat{Z}$ and $Z(\hat{H})\ra \hat{Z}$ are
  given by \S\ref{sub:L-groups}. (The distinction comes from the fact that typically
  $I_{\gamma_H}\ra H$ is not normal.) The three maps are compatible in the obvious sense.
  By the functoriality of $X_*(\cdot)^\Gamma$, there is a natural surjection
  $$\coker(X_*(Z(\hat{H}))^{\Gamma}\ra X_*(\hat{Z})^{\Gamma})
  \twoheadrightarrow \coker(X_*(Z(\hat{I}_{\gamma_H}))^{\Gamma}\ra X_*(\hat{Z})^{\Gamma}).$$
  The left hand side is finite because it embeds into the finite group
  $\pi_0(Z(\hat{G})^{\Gamma})$, again by \cite[Cor 2.3]{Kot84a}.
  Going back to \eqref{e:bounding-pi0-general}, we deduce
  $$|\pi_0(Z(\hat{I}_\gamma)^{\Gamma})|\le |\pi_0(Z(\hat{I}_{\gamma_H})^{\Gamma})|\cdot |\coker(X_*(Z(\hat{H}))^{\Gamma}\ra X_*(\hat{Z})^{\Gamma})|< c_H \cdot |\pi_0(Z(\hat{G})^{\Gamma})|.$$
  The proof is complete as the far right hand side is independent of $\gamma$.

\end{proof}

%\begin{rem}
%  Although Lemma \ref{l:bounding-pi0} is enough for applications
%  in our paper, we include the more general Corollary \ref{c:bounding-pi0-general}
%  hoping that it would be suitable for wider applications.
%\end{rem}

  For a cuspidal group and conjugacy classes which are elliptic at infinity,
  a more precise bound can be obtained by a simpler argument, which would be worth
  recording here.

\begin{lem}\label{l:bounding-pi0}
  Let $G$ be a cuspidal $F$-group. For any $\gamma\in G(F)$ such that
  $\gamma\in G(F_\infty)$ is elliptic,
  $$|\pi_0(Z(\hat{I}_\gamma)^{\Gamma})|\le 2^{\rank(G/A_G)}.$$
\end{lem}

\begin{proof}
  Via restriction of scalars, we may assume that $F=\Q$ without losing generality.
  Let us prove the lemma when $A_G$ is trivial. By assumption there exists an
  $\R$-anisotropic torus $T$ in $G(\R)$ containing $\gamma$. Thus
  $T\simeq U(1)^{\rank(G)}$ and $T\hra I_\gamma$ over $\R$.
  The former tells us that $\hat{T}^{\Gamma(\infty)}\simeq \{\pm 1\}^{\rank(G)}$ and
  the latter gives rise to $ Z(\hat{I}_\gamma)^{\Gamma(\infty)}\hra \hat{T}^{\Gamma(\infty)}$
  (\cite[\S4]{Kot86}). Hence the assertion follows from
  $$Z(\hat{I}_\gamma)^{\Gamma}\hra Z(\hat{I}_\gamma)^{\Gamma(\infty)} \hra \hat{T}^{\Gamma(\infty)}\simeq \{\pm 1\}^{\rank(G)}.$$

  In general when $A_G$ is not trivial, consider the exact sequence of $\Q$-groups
  $1\ra A_G \ra I_\gamma \ra I_\gamma/A_G\ra 1$, whose dual is the $\Gamma$-equivariant
  exact sequence of $\C$-groups
  $$1\ra Z(\hat{I_\gamma/A_G})\ra Z(\hat{I}_\gamma) \ra \hat{A}_G \ra 1.$$
  Thanks to \cite[Cor 2.3]{Kot84a}, we obtain the following exact sequence:
  $$X^*(\hat{A}_G)^\Gamma \ra \pi_0(Z(\hat{I_\gamma/A_G})^\Gamma) \ra \pi_0(Z(\hat{I}_\gamma)^{\Gamma})
  \ra \pi_0(\hat{A}_G)^\Gamma = 1.$$
  Hence $|\pi_0(Z(\hat{I}_\gamma)^{\Gamma})|\le
  |\pi_0(Z(\hat{I_\gamma/A_G})^\Gamma)|$, and the latter is at most
  $ 2^{\rank(G/A_G)}$ by the preceding argument.
\end{proof}

\section{Automorphic Plancherel density theorem with error bounds}\label{s:aut-Plan-theorem}

  The local components of automorphic representations at a fixed finite set of
 primes tend to be equidistributed according to
 the Plancherel measure on the unitary dual, namely the error tends to zero
 in a family of automorphic representations (cf. Corollary \ref{c:Plan-density} below).
  The main result of this section (Theorems \ref{t:level-varies}, \ref{t:weight-varies})
 is a bound
on this error in terms of the primes in the fixed set as well as the varying parameter
(level or weight) in the family. A crucial assumption for us is that the group $G$
is cuspidal (Definition \ref{d:cuspidal}), which allows the use of a simpler version of the trace formula.
For the proof we interpret the problem
as bounding certain expressions on the geometric side of the trace formula and apply various technical results from previous
sections.
One main application is a proof of the Sato-Tate conjecture for families formulated in
\S\ref{sub:ST-families} under suitable conditions on the parameters involved.
In turn the result will be applied to the question on low-lying zeros in later sections.

%The above assumption on $G$ avails us a relatively simple trace formula
%due to Arthur when only cohomological
%representations are considered at infinity.
%Still, this case already presents much difficulty if one tries to obtain a reasonable error bound.

\subsection{Sauvageot's density theorem on unitary
dual}\label{sub:density-thm}

  We reproduce a summary of Sauvageot's result (\cite{Sau97}) from
  \cite[\S2.3]{Shi-Plan} as it can be used to effectively prescribe local conditions
  in our problem.
The reader may refer to either source for more detail.

  Let $G$ be a connected
  reductive group over a number field $F$. Use $v$ to denote a finite place of $F$.
  When $M$ is a Levi subgroup of
  $G$ over $F_v$, write
$\Psi_u(M(F_v))$ (resp. $\Psi(M(F_v))$) for the real (resp. complex) torus whose points parametrize
unitary (complex-valued) characters of $M(F_v)$ trivial on any compact subgroup of $M(F_v)$. The normalized parabolic induction
of an admissible representation $\sigma$ of $M(F_v)$ is denoted $\nind^G_M(\sigma)$.

  Denote by $\mB_c(G(F_v)^{\wedge})$ the space of bounded
  $\pl_v$-measurable functions $\hat{f}_v$ on $G(F_v)^{\wedge}$
whose support has compact image in the Bernstein center, which is
the set of $\C$-points of an (infinite) product of varieties.
  A measure on $G(F_v)^{\wedge}$
  will be thought of as a linear functional on the space $\mF(G(F_v)^{\wedge})$
  consisting of $\hat{f}_v\in \mB_c(G(F_v)^{\wedge})$ such that
for every $F_v$-rational Levi subgroup $M$ of $G$ and every discrete
series $\sigma$ of $M(F_v)$, $$\Psi_u(M(F_v))\ra
\C\quad\mbox{ given by}\quad \chi\mapsto
\hat{f}_v(\nind^G_M(\sigma\otimes \chi))$$ is a function whose
points of discontinuity are contained in a measure zero set. (Here $\nind$ denotes the normalized parabolic induction.)
Now for any finite set $S$ of finite places of $F$,
one can easily extend the above definition to
$\mF(G(F_S)^{\wedge})$ so that $\hat{f}_S(\pi_S)\in \C$ makes sense
for $\hat{f}_S\in \mF(G(F_S)^{\wedge})$ and $\pi_S\in G(F_S)^{\wedge}$.
We have a map
$$C^\infty_c(G(F_S))\ra \mF(G(F_S)^{\wedge}),
\qquad \phi_S\mapsto \hat{\phi}_S:\pi_S\mapsto \tr \pi_S(\phi_S),$$
as follows from Proposition~\ref{l:bounding-function-by-pos-function} below.
Harish-Chandra's Plancherel theorem states that
\[
\label{e:Plan-thm}\pl_S(\hat{\phi}_S)=\phi_S(1).
\]
Our notational convention is that $\hat{\phi}_S$ often signifies an element in the image of the above map
whereas $\hat{f}_S$ stands for a general element of $\mF(G(F_S)^{\wedge})$.
  Sauvageot's theorem allows us to approximate any $\hat{f}_S\in\mF(G(F_S)^{\wedge})$
with elements of $C^\infty_c(G(F_S))$.

\begin{prop}(\cite[Thm 7.3]{Sau97})\label{p:density}
  Let $\hat{f}_S\in \mF(G(F_S)^{\wedge})$. For any $\epsilon>0$,
  there exist $\phi_S,\psi_S\in C^\infty_c(G(F_S))$ such that
$$ \pl_S(\hat{\psi}_S)\le \epsilon  \quad\mbox{and}\quad
\forall
\pi_S\in G(F_S)^{\wedge},~
|\hat{f}_S(\pi_S) - \hat{\phi}_S(\pi_S)|\le \hat{\psi}_S(\pi_S) .$$
Conversely, any $\hat{f}_S\in
  \mB_c(\hat{G(F_S)})$ with the above property belongs to $\mF(G(F_S)^{\wedge})$.
\end{prop}

\begin{rem}
  It is crucial that $\hat{f}_S\in \mF(G(F_S)^{\wedge})$ has the set of discontinuity in a measure zero set.
  Otherwise we could take $\hat{f}_S$ to be the characteristic function on the set of points of $G(F_S)^{\wedge}$
  which arise as the $S$-components of some $\pi\in\cAR_{\disc,\chi}(G)$ with nonzero Lie algebra cohomology. Note that the latter function typically lies outside $\mF(G(F_S)^{\wedge})$. The conclusions of Theorems \ref{t:Sato-Tate-level}, \ref{t:Sato-Tate-weight} and Corollary \ref{c:Plan-density}  are false in general if such an $\hat{f}_S$ is placed at $S_0$. Namely in that case $\hat{\mu}_{\cF_k,S_1}(\hat{\phi}_{S_1})$ is often far from zero but $\pl_S(\hat{\phi}_S)$ always vanishes.
\end{rem}

%  For the statement of the Sato-Tate conjecture, we will take test functions in $C(G(F_S)^{\wedge,\ur,\temp})$, the space of $\C$-valued continuous functions on $G(F_S)^{\wedge,\ur,\temp}$. Note that the Satake isomorphism restricts to a topological isomorphism
%   $G(F_S)^{\wedge,\ur,\temp}\simeq \prod_{v\in S}\hat{T}_{c,\theta_v}/\Omega_{c,\theta_v}$,
%  thereby one can canonically identify $$C(G(F_S)^{\wedge,\ur,\temp}) \simeq C(\prod_{v\in S}\hat{T}_{c,\theta_v}/\Omega_{c,\theta_v}).$$
%   view $\mF(G(F_S)^{\wedge})$ $C(G(F_S)^{\wedge,\ur,\temp})$
%  Extension by zero outside the unramified tempered spectrum  as a subspace of $\mF(G(F_S)^{\wedge})$

  From here until the end of this subsection let us suppose that $G$ is unramified at $S$. It will be convenient to introduce $\cF(G(F_S)^{\wedge,\ur})$ and its subspace $\cF(G(F_S)^{\wedge,\ur,\temp})$
  in order to state the Sato-Tate theorem in \S\ref{sub:ST-theorem}.
  The former (resp. the latter) consists of $\hat{f}_S\in \cF(G(F_S)^{\wedge})$ such that
  the support of $\hat{f}_S$ is contained in $G(F_S)^{\wedge,\ur}$ (resp. $G(F_S)^{\wedge,\ur,\temp}$).
  Denote by $\cF(\hat{T}_{c,\theta}/\Omega_{c,\theta})$ the space of bounded $\ST_\theta$-measurable functions on
  $\hat{T}_{c,\theta}/\Omega_{c,\theta}$ whose points of discontinuity are contained in a $\ST_\theta$-measure zero set.
  Define $\cF(\prod_{v\in S}\hat{T}_{c,\theta_v}/\Omega_{c,\theta_v})$ in the obvious analogous way.
  By using the topological Satake isomorphism for tempered spectrum (cf. \eqref{e:local-isom-v}) $$\prod_{v\in S}\hat{T}_{c,\theta_v}/\Omega_{c,\theta_v}
  \simeq G(F_S)^{\wedge,\ur,\temp}$$
  and extending by zero outside the tempered spectrum, one obtains
  \beq\label{e:Sauvageot}\cF\left(\prod_{v\in S}\hat{T}_{c,\theta_v}/\Omega_{c,\theta_v}\right)\simeq  \cF(G(F_S)^{\wedge,\ur,\temp})
  \hra \cF(G(F_S)^{\wedge,\ur}).\eeq
  Although the first two $\cF(\cdot)$ above are defined with respect to different measures $\prod_{v\in S}\ST_{\theta_v}$
  and $\pl_S$, the isomorphism is justified by the fact that the ratio of the two measures is uniformly bounded above and below by positive constants (depending on $q_S$) in view of Proposition \ref{p:unr-plan-meas} and Lemma \ref{l:ST-measure}.
  Note that the space of continuous functions on $\prod_{v\in S}\hat{T}_{c,\theta_v}/\Omega_{c,\theta_v}$
  (resp. on $G(F_S)^{\wedge,\ur,\temp}$) is contained in the first (resp. second) term of \eqref{e:Sauvageot},
  and the two subspaces correspond under the isomorphism.

\begin{cor}\label{c:density}  Let $\hat{f}_S\in \cF(G(F_S)^{\wedge,\ur})$. For any $\epsilon>0$,
  there exist $\phi_S,\psi_S\in \cH^{\ur}(G(F_S))$ such that (i)
$ \pl_S(\hat{\psi}_S)\le \epsilon$ and (ii) $\forall
\pi_S\in G(F_S)^{\wedge,\ur}$, $|\hat{f}_S(\pi_S) - \hat{\phi}_S(\pi_S)|\le \hat{\psi}_S(\pi_S)$.
\end{cor}

\begin{proof}
  Let $\phi_S,\psi_S\in C^\infty_c(G(F_S))$ be the functions associated to $\hat{f}_S$ as in Proposition \ref{p:density}. Then it is enough to replace $\phi_S$ and $\psi_S$ with their convolution products with the characteristic function on $\prod_{v\in S} K_v$.
\end{proof}

  The following proposition will be used later in \S\ref{sub:ST-theorem}. For each $v\in \cV_F(\theta)$,
  the image of $\hat{f}$ in $\cF(G(F_v)^{\wedge,\ur})$ via \eqref{e:Sauvageot}
  will be denoted $\hat{f}_v$.

\begin{prop}\label{p:unr-density}
  Let $\hat{f}\in \cF(\hat{T}_{c,\theta}/\Omega_{c,\theta})$ and $\epsilon >0$. There exists an integer $\kappa\ge 1$ and for all places $v\in \cV_F(\theta)$, there are bounded functions $\phi_{v},\psi_{v}\in \cH^{\ur}(G(F_{v}))^{\le \kappa}$ such that $\pl_v(\hat{\psi}_v)\le \epsilon$ and $|\hat{f_v}(\pi) - \hat{\phi_v}(\pi)| \le \hat{\psi_v}(\pi)$ for all $\pi\in G(F_v)^{\wedge,\ur}$.
\end{prop}

\begin{proof}
  This is no more than Corollary \ref{c:density} if we only required $\phi_{v},\psi_{v}\in \cH^{\ur}(G(F_{v}))$ without the superscript $\le \kappa$. So we may disregard finitely many $v$ by considering the subset $\cV_F(\theta)^{\ge Q}$ of $\cV_F(\theta)$ consisting of $v$ such that $q_v\ge Q$ for some $Q>0$. In view of Proposition \ref{p:lim-of-Plancherel}, we may choose $Q\in \Z_{>0}$ that \beq\label{e:bound-pl-ST}\forall v\in \cV_F(\theta)^{\ge Q},~\forall \hat{f}\in \cF(\hat{T}_{c,\theta}/\Omega_{c,\theta}),\quad
  \frac{1}{2} \ST_{\theta}(|\hat{f}|)\le \plur_v(|\hat{f}_v|)\le 2\ST_{\theta}(|\hat{f}|).\eeq

  Fix any $w\in \cV_F(\theta)^{\ge Q}$. Corollary \ref{c:density} allows us to find $\phi_w,\psi'_w\in\cH^{\ur}(G(F_{w}))$ such that
  \beq\label{e:phi_wpsi_w}
  \pl_w(\hat{\psi}'_w)\le \epsilon/8  \quad\mbox{and}\quad
\forall \pi_w\in G(F_w)^{\wedge,\ur},~
|\hat{f}_w(\pi_w) - \hat{\phi}_w(\pi_w)|\le \hat{\psi}'_w(\pi_w). \eeq
 Let $\kappa_0\in \Z_{\ge 0}$ be such that $\phi_w,\psi'_w\in\cH^{\ur}(G(F_{w}))^{\le \kappa_0}$. Now recall that for every $v\in \cV_F(\theta)$ there is a canonical isomorphism (cf. \eqref{e:Satake}, Lemma \ref{l:unr-temp-spec})
  between $\cH^{\ur}(G(F_v))$ and the space of regular functions in the complex variety $\hat{T}_{\theta}/\Omega_{\theta}$. Using the latter as a bridge, we may transport $\phi_w,\psi'_w$ to $\phi_v,\psi'_v\in\cH^{\ur}(G(F_{v}))$ for every $v\in  \cV_F(\theta)$. Clearly $\phi_v,\psi'_v\in \cH^{\ur}(G(F_{v}))^{\le \kappa_0}$ from the definition of \S\ref{sub:trun-unr-Hecke}. Moreover \eqref{e:bound-pl-ST} and \eqref{e:phi_wpsi_w} imply that
  for all $v\in \cV_F(\theta)^{\ge Q}$,
  $$  \pl_v(\hat{\psi}'_v)\le \epsilon/2  \quad\mbox{and}\quad
\forall \pi_v\in G(F_v)^{\wedge,\ur,\temp},~
|\hat{f}_v(\pi_v) - \hat{\phi}_v(\pi_v)|\le \hat{\psi}'_v(\pi_v).$$
(Observe that $\pl_v(\hat{\psi}'_v)\le 2 \ST_\theta(\hat{\psi}'_v)=2\ST_\theta(\hat{\psi}'_w)\le 4\pl_w(\hat{\psi}'_w)\le \epsilon/2$ to justify the first inequality.)

  To achieve the latter inequality for non-tempered $\pi_v\in G(F_v)^{\wedge,\ur}$, we would like to perturb $\psi'_v$ in a way independent of $v$ while not sacrificing the former inequality. Since $\hat{f}_v(\pi_v)=0$ for such $\pi_v$,
  what we need to establish is that $|\hat{\phi}_v(\pi_v)|\le \hat{\psi}_v(\pi_v)$ for all non-tempered $\pi_v\in G(F_v)^{\wedge,\ur}$. To this end, we use the fact that there is a compact subset $\mathcal{K}$ of $\hat{T}_{\theta}/\Omega_{\theta}$ such that $G(F_v)^{\wedge,\ur}$ is contained in $\mathcal{K}$ for every $v\in \cV_F(\theta)$ (cf. \cite[Thm XI.3.3]{BW00}). By using the Weierstrass approximation theorem, we find $\psi''_w\in \cH^{\ur}(G(F_{w}))$ such that
  $$\pl_w(\hat{\psi}''_w)\le \epsilon/8,$$
$$\forall \pi_w\in \mathcal{K}\bs G(F_w)^{\wedge,\ur,\temp},~|\hat{\psi}'_w(\pi_w)| + |\hat{\phi}_w(\pi_w)|\le \hat{\psi}''_w(\pi_w),$$
$$\forall \pi_w\in G(F_w)^{\wedge,\ur,\temp},~\hat{\psi}''_w(\pi_w)\ge 0.$$
  Choose $\kappa\ge \kappa_0$ such that $\psi''_w\in \cH^{\ur}(G(F_{w}))^{\le \kappa}$ and put $\psi_w:=\psi'_w+\psi''_w$ so that $\pl_w(\hat{\psi}_w)\le \epsilon/4$ and $\psi_w\in \cH^{\ur}(G(F_{w}))^{\le \kappa}$. For each $v\in \cV_F(\theta)^{\ge Q}$, let $\psi_v$ denote the transport of $\psi_w$ just as $\psi'_v$ was the transport of $\psi'_w$ in the preceding paragraph. Then $\pl_v(\hat{\psi}_v)\le \epsilon$ and $\psi_v\in \cH^{\ur}(G(F_{v}))^{\le \kappa}$ as before. Moreover
  $$\forall \pi_v\in G(F_v)^{\wedge,\ur,\temp},~ |\hat{f}_v(\pi_v) - \hat{\phi}_v(\pi_v)|\le \hat{\psi}'_v(\pi_v)\le \hat{\psi}_v(\pi_v)$$
  and for $\pi_v\in G(F_v)^{\wedge,\ur}\bs G(F_v)^{\wedge,\ur,\temp}$,
 $$ |\hat{f}_v(\pi_v) - \hat{\phi}_v(\pi_v)|=|\hat{\phi}_v(\pi_v)|\le \hat{\psi}''_v(\pi_v)-|\hat{\psi}'_v(\pi_v)|\le \hat{\psi}_v(\pi_v),$$
 the last inequality following from $\hat{\psi}_v=\hat{\psi}'_v+\hat{\psi}''_v$.
\end{proof}

\begin{rem}
  A more direct approach to~\eqref{e:phi_wpsi_w} that wouldn't involve Corollary~\ref{c:density} would be to use Weierstrass approximation to find polynomials $\phi$ and $\psi$ on $\hat T_{c,\theta}/\Omega_{c,\theta}$ of degree $\le \kappa$ such that $|\hat{f} - \hat{\phi} | \le \hat{\psi}$ and then the isomorphism~\eqref{e:Sauvageot} to transport $\phi$ and $\psi$ at the place $v$.
\end{rem}

We note~\cite{Sau97}*{Lemme 3.5} that for any $\phi_v\in C^\infty_c(G(F_v))$ there exists a $\phi'_v\in C^\infty_c(G(F_v))$ such that
$|\hat{\phi}_v(\pi_v)|\le \hat{\phi'}_v(\pi_v)$ for all $\pi_v\in G(F_v)^{\wedge}$. This statement is elementary, e.g. it follows from the Dixmier--Malliavin decomposition theorem. In fact we have the following stronger result due to Bernstein~\cite{Bernstein:typeI}.
\begin{prop}
  [Uniform admissibility theorem]
  \label{l:bounding-function-by-pos-function}
For any $\phi_v \in C^\infty_c(G(F_v))$ there exists $C>0$ such that $|\tr \pi(\phi_v)|\le C$ for all $\pi\in G(F_v)^{\wedge}$.
\end{prop}

%\begin{prop}(\cite[Thm 2.5]{Tad88})
%  Let $M$ be a Levi subgroup of $G$ and $\pi_M$ an irreducible admissible representation of $M(F)$.
%  The set of $\chi\in \Psi(M)$ such that $\nind^G_M(\pi_M\otimes \chi)(\phi)$ has a unitary subquotient is a compact subset of $\Psi(M)$.
%\end{prop}
%
%
%\begin{lem}\label{l:irr-via-induced-rep}
%  Any irred adm (or unitary) rep is a linear combination of induced reps with uniformly bounded coefficients.
%\end{lem}

\subsection{Automorphic representations and a counting measure}\label{sub:aut-rep}

%  Let $G$ be a connected reductive group over a \emph{number} field $F$.
%  Let $\chi:A_{G,\infty}\ra \C^\times$ be a continuous homomorphism.
%  Denote by $L^2_\chi(G(F)\bs G(\A_F))$
%  the space of all functions $f$ on $G(\A_F)$ which are
%  square-integrable modulo $A_{G,\infty}$ and satisfy
%  $f(g\gamma z)=\chi(z)f(\gamma)$ for all $g\in G(F)$, $\gamma \in G(\A_F)$ and
%  $z\in A_{G,\infty}$. There is a spectral decomposition into discrete and continuous parts
%$$L^2_{\chi}(G(F)\bs G(\A_F))=L^2_{\disc,\chi}\oplus L^2_{\cont,\chi},
%\qquad L^2_{\disc,\chi}=\hat{\bigoplus_{\pi}}\, m_{\disc}(\pi)\cdot \pi$$
%where the last sum is a Hilbert direct sum running over the set of
%all irreducible representations
%of $G(\A_F)$ up to isomorphism.
%Write $\cA_{\disc,\chi}(G)$ for the set of isomorphism classes of
%all irreducible representations $\pi$
%of $G(\A_F)$ such that $m_{\disc}(\pi)>0$.
%Any $\pi\in \cA_{\disc,\chi}(G)$ is said to
%be a discrete automorphic representation of $G(\A_F)$.

 Now consider a string of complex numbers
  $$\cF=\{a_{\cF}(\pi)\in \C\}_{\pi\in\cAR_{\disc,\chi}(G)}$$
  such that $a_{\cF}(\pi)=0$ for all but finitely many $\pi$.
  We think of $\cF$ as a multi-set by viewing
  $a_{\cF}(\pi)$ as multiplicity, or more appropriately as a density function with finite support in $\cF$ as $a_{\cF}(\pi)$ is allowed to be in $\C$. There are obvious meanings when we write $\pi\in \cF$ and $|\cF|$ (we could have written $\pi\in \mathrm{supp}\, \cF$ for the former):
$$ \pi\in \cF ~\stackrel{\mathrm{def}}{\Leftrightarrow}~ a_{\cF}(\pi)\neq 0,\qquad
|\cF|:=\sum_{\pi\in \cF} a_{\cF}(\pi).$$

  In order to explain our working hypothesis, we recall a definition.
\begin{defn}\label{d:cuspidal}
  Let $H$ be a connected reductive group over $\Q$. The maximal $\Q$-split torus in $Z(H)$
  is denoted $A_H$. We say $H$ is \emph{cuspidal} if $(H/A_H)\times_\Q \R$ contains a maximal
  $\R$-anisotropic torus.
\end{defn}
  If $H$ is cuspidal then $H(\R)$ has discrete series representations. (We remind the reader that discrete series always mean ``relative discrete series'' for us, i.e. those whose matrix coefficients are square-integrable modulo center.) The converse is true when
  $H$ is semisimple but not in general.
  Throughout this section the following will be in effect:
\begin{hypo}\label{hypo:cuspidal}
  $\Res_{F/\Q} G$ is a cuspidal group.
\end{hypo}
  Let $S=S_0\coprod S_1\subset \cV_F^\infty$ be a nonempty finite subset
  and  $\hat{f}_{S_0}\in \mF(G(F_{S_0})^{\wedge})$. %$\phi_{S_0}\in C^\infty_c(G(F_{S_0}))$.
  (It is allowed that either $S_0$ or $S_1$ is empty.)
  Let
\bit\item (level) $U^{S,\infty}$ be an open compact subset of $G(\A^{S,\infty})$,
\item (weight) $\xi=\otimes_{v|\infty}\xi_v$
be an irreducible algebraic representation of
$$G_\infty\times_\R \C=(\Res_{F/\Q} G)\times_\Q \C=\prod_{v|\infty} G\times_{F,v} \C. $$ \eit
  Denote by $\chi:A_{G,\infty}\ra \C^\times$ the restriction of the
  central character for $\xi^\vee$.
Define $$\cF=\cF(U^{S,\infty},\hat{f}_{S_0},S_1,\xi)\quad\mbox{by}$$
\beq\label{e:a(pi)-general}
a_{\cF}(\pi):=(-1)^{q(G)} m_{\disc,\chi}(\pi)\dim(\pi^{S,\infty})^{U^{S,\infty}}
\hat{f}_{S_0}(\pi_{S_0})\hat{\triv}_{K_{S_1}}(\pi_{S_1})
  \chi_{\EP}( \pi_\infty\otimes \xi) ~\in \C.\eeq
  Note that $\hat{\triv}_{K_{S_1}}(\pi_{S_1})$ equals 1 if $\pi_{S_1}$ is unramified
  and 0 otherwise, and that $\chi_{\EP}( \pi_\infty\otimes \xi)=0$ unless $\pi_\infty$ has
  the same infinitesimal character as $\xi^\vee$.
  The set of $\pi$ such that $a_{\cF}(\pi)\neq 0$ is finite by
  Harish-Chandra's finiteness theorem.
  Let us define measures $\hat{\mu}_{\cF,S_1}$ and $\hat{\mu}^\natural_{\cF,S_1}$ associated with $\cF$
 on the unramified unitary dual $G(F_{S_1})^{\wedge,\ur}$, motivated by the trace formula.
Put $\tau'(G):=\ol{\mu}^{\can,\EP}(G(F)A_{G,\infty}\bs G(\A_F))$.
For any function $\hat{f}_{S_1}$ on $G(F_{S_1})^{\wedge,\ur}$ which is continuous
outside a measure zero set, define
\beq\label{e:defn-of-mu}
\hat{\mu}_{\cF,S_1}(\hat{f}_{S_1})
  := \frac{\mu^{\can}(U^{S,\infty})}{\tau'(G)\dim\xi}\sum_{\pi\in \cAR_{\disc,\chi}(G)} a_{\cF}(\pi) \hat{f}_{S_1}(\pi_{S_1}).\eeq
The sum is finite because $a_{\cF}$ is supported on finitely many $\pi$. Now the key point is that the right hand side can be identified with the spectral side of Arthur's trace formula with the Euler-Poincar\'{e} function
  at infinity as in \S\ref{sub:EP} when $\hat{f}_{S_1}=\hat{\phi}_{S_1}$ for some $\phi_{S_1}\in \cH^{\ur}(G(F_{S_1}))$ (\cite[pp.267-268]{Art89}, cf. proof of \cite[Prop 4.1]{Shi-Plan}). So to speak,
if we write $\phi^{\infty}=\phi_{S_0}\phi_{S_1}\phi^{S,\infty}$,
\beq\label{e:apply-Arthur}\hat{\mu}_{\cF,S_1}(\hat{\phi}_{S_1})= (-1)^{q(G)}\frac{ I_{\spec}(\phi^{\infty}\phi_\xi,\mu^{\can,\EP})}{{\tau'(G)\dim\xi}}
= (-1)^{q(G)}\frac{ I_{\geom}(\phi^{\infty}\phi_\xi,\mu^{\can,\EP})}{{\tau'(G)\dim\xi}}\eeq
where $I_{\spec}$ (resp. $I_{\geom}$) denotes the spectral (resp. geometric)
side Arthur's the invariant trace formula with respect to the measure
$\mu^{\can,\EP}$. Finally if $\hat{f}_{S_0}$ has the property that $\pl_{S_0}(\hat{f}_{S_0})\neq 0$ then
put $$ \hat{\mu}^\natural_{\cF,S_1}:=\pl_{S_0}(\hat{f}_{S_0})^{-1}  \hat{\mu}_{\cF,S_1}.$$

% OLD: definition of mu for only tempered reps at infinity
% For an auxiliary use in later arguments, define $\hat{\mu}^{\tempinf}_{\cF,S_1}(\hat{f}_{S_1})$ to be the right hand side of \eqref{e:defn-of-mu} in which the sum is taken over only those $\pi$ such that $\pi_\infty$ is tempered. Similarly $\hat{\mu}^{\natural,\tempinf}_{\cF,S_1}$ is defined. When $\xi$ has regular highest weight, $\hat{\mu}_{\cF,S_1}=\hat{\mu}^{\tempinf}_{\cF,S_1}$ and $\hat{\mu}^{\natural,\tempinf}_{\cF,S_1}=\hat{\mu}^{\natural}_{\cF,S_1}$, cf. \S\ref{sub:EP}.

\begin{rem}\label{r:|F|} The measure $ \hat{\mu}^\natural_{\cF,S_1}$
  is asymptotically the same as the counting measure $$\hat{\mu}^{\mathrm{count}}_{\cF,S_1}(\hat{f}_{S_1})
  = \frac{1}{|\cF|}\sum_{\pi\in \cAR_{\disc,\chi}(G)} a_{\cF}(\pi) \hat{f}_{S_1}(\pi_{S_1}).$$
associated with the $S_1$-components of $\cF$ (assuming $|\cF|\neq 0$). More precisely if $\{\cF_k\}_{\ge1}$ is a family of
\S\ref{sub:aut-families} below, then $\hat{\mu}^{\mathrm{count}}_{\cF_k,S_1}/
  \hat{\mu}^{\natural}_{\cF_k,S_1}$ is a constant tending to 1 as $k\ra\infty$ by Corollary \ref{c:estimate-F_k}.
\end{rem}

\begin{ex}\label{ex:weight-regular}
  Let $\pi\in \cAR_{\disc,\chi}(G)$. Suppose that the highest weight of $\xi$ is \emph{regular} and that $S_0=\emptyset$. Then
  $\pi$ belongs to $\cF$ if and only if the following three conditions hold:
  $(\pi^{S,\infty})^{U^{S,\infty}}\neq 0$, $\pi$ is unramified at $S$, and
  $\pi_\infty\in \Pi_{\disc}(\xi^\vee)$. When $\pi_\infty\in \Pi_{\disc}(\xi^\vee)$,
  \eqref{e:a(pi)-general} simplifies as
  \[
  %\label{e:acF(pi)}
  a_{\cF}(\pi)=m_{\disc,\chi}(\pi)\dim(\pi^{S,\infty})^{U^{S,\infty}}.
  \]
\end{ex}

\begin{ex}\label{ex:example-for-f_S}
  Let $\hat{f}_{S_0}$ be a characteristic function on some relatively compact $\pl_S$-measurable subset
   $\hat{U}_{S_0}\subset G(F_{S_0})^{\wedge}$.
   Assume that $S_0$ is large enough such that $G$ and all members of
   $\cF$ are unramified outside $S_0$. Take $U^{S_0,\infty}$ to be the product of $K_v$
   over all finite places $v\notin S_0$. Then for each $\pi\in \cAR_{\disc,\chi}(G)$,
   \begin{equation}\label{e:example-for-f_S}
	  a_{\cF}(\pi)=(-1)^{q(G)}\chi_{\EP}(\pi_\infty\otimes \xi)
   m_{\disc,\chi}(\pi)
 \end{equation}
   if $\pi^{S_0,\infty}$ is unramified, $\pi_{S_0}\in  \hat{U}_{S_0}$
   (in which case $a_{\cF}(\pi)\neq0$ if moreover $\chi_{\EP}(\pi_\infty\otimes \xi)\neq 0$; otherwise $a_{\cF}(\pi)=0$).
 If the highest weight of $\xi$ is regular, $\chi_{\EP}(\pi_\infty\otimes \xi)\neq 0$ exactly when $\pi_\infty\in \Pi_{\disc}(\xi^\vee)$, in which case \eqref{e:example-for-f_S} simplifies
   as
   \[
   %\label{e:simplest-a(pi)}
   a_{\cF}(\pi)=m_{\disc,\chi}(\pi).
   \]
   Compare this with Example \ref{ex:weight-regular}. (The analogy in the case of modular forms is that
   $\pi$ as newforms are counted in the current example
   whereas old-forms are also counted in Example \ref{ex:weight-regular}.)
Finally we observe that since the highest weight of $\xi$ is regular and $\pi_\infty\in \Pi_{\disc}(\xi^\vee)$, the discrete automorphic representation $\pi$ is automatically cuspidal~\cite{Wall84}*{Thm.~4.3}. In the present example the discrete multiplicity coincides with the cuspidal multiplicity.
\end{ex}

\begin{rem}\label{r:why-S_0}
  As the last example shows,
  the main reason to include $S_0$ is to prescribe local conditions at finitely many places (namely at $S_0$) on automorphic families.
  For instance one can take $\hat{f}_{S_0}=\hat{\phi}_{S_0}$ where
   $\phi_{S_0}$ is a pseudo-coefficient of a supercuspidal representation
  (or a truncation thereof if the center of $G$ is not anisotropic over $F_{S_0}$).
  Then it allows us to consider a family of $\pi$
  whose $S_0$-components are a particular supercuspidal representation (or an unramified character twist thereof).
  By using various $\hat{f}_{S_0}$
  (which are in general not equal to $\hat{\phi}_{S_0}$ for any
   $\phi_{S_0}\in C^\infty_c(G(F_{S_0}))$) one
   obtains great flexibility in prescribing a local condition
   as well as imposing weighting factors for a family.
\end{rem}

\subsection{Families of automorphic representations}\label{sub:aut-families}

  Continuing from the previous subsection (in particular keeping Hypothesis \ref{hypo:cuspidal})
 let us introduce two kinds of families $\{\cF_k\}_{k\ge 1}$
 which will be studied later on. We will measure the
 size of $\xi$ in the following way. Let $T_\infty$ be a maximal torus of
 $G_\infty$ over $\R$. For a $B$-dominant $\lambda\in X^*(T_\infty)$, set
  $m(\lambda):=\min_{\alpha\in \Phi^+} \lg \lambda,\alpha\rg$.
  For $\xi$ with $B$-dominant highest weight $\lambda_{\xi}$,
  define $m(\xi):=m(\lambda_\xi)$.

   Let $\phi_{S_0}\in C^\infty_c(G(F_{S_0}))$. (More generally we will sometimes prescribe a local condition at $S_0$ by
   $\hat{f}_{S_0}\in \mF(G(F_{S_0})^{\wedge})$ rather than $\phi_{S_0}$.) In the remainder of Section~\ref{s:aut-Plan-theorem} we mostly focus on families in the level or weight aspect, respectively described as the following:

\begin{ex}[\key{Level aspect}: varying level, fixed weight]\label{ex:level-varies}
  Let $\fkn_k\subset \cO_F$ be a nonzero ideal prime to $S$ for each $k\ge 1$ such that $\N(\fkn_k)=[\cO_F:\fkn_k]$ tends to $\infty$
  as $k\ra \infty$.
  Take $$\cF_k:=\cF(K^{S,\infty}(\fkn_k),\hat{\phi}_{S_0},S_1,\xi).$$
  Then $|\cF_k|\ra\infty$ as $k\ra\infty$.
\end{ex}

\begin{ex}[\key{Weight aspect}: fixed level, varying weight]\label{ex:wt-varies}

For our study of weight aspect it is always supposed that $Z(G)=1$ so that $A_{G,\infty}=1$ and $\chi=1$ in order to eliminate the technical problem with central character when weight varies.\footnote{Without the hypothesis that the center is trivial, one should work with fixed central character and apply the trace formula in such a setting. Then our results and arguments in the weight aspect should remain valid without change.} Let $\{\xi_k\}_{k\ge 1}$ be a sequence of irreducible algebraic representations
of $G_\infty\times_\R \C$ such that $m(\xi_k)\ra \infty$
  as $k\ra \infty$. Take $$\cF_k:=\cF(U^{S,\infty},\hat{\phi}_{S_0},S_1,\xi_k).$$
  Then $|\cF_k|\ra\infty$ as $k\ra\infty$.

\end{ex}

\begin{rem}
  Sarnak proposed a definition of families of automorphic representations (or automorphic $L$-functions)
  in \cite{Sarn:family}. The above two examples fit in his definition.

\end{rem}

%\begin{ex}(general case with simplifying local conditions)

%  Drop the previous assumption, and let $G$ be any connected reductive group
%  over any number field $F$.\footnote{For simplicity assume $Z(G)$ is finite, or anisotropic
%  at $v_1$ and $v_2$. Otherwise, twisted by characters should be considered.}
%  Let $S\subset \cV_F$ be a nonempty finite subset (which may or may contain infinite places).
%  Fix two auxiliary finite places $v_1$ and $v_2$ of $F$ not contained in $S$.
%    Let $\pi^0_{v_2}$ be a supercuspidal representation of $G(F_{v_2})$.
%   Set
%  $$a_{\cF}(\pi):= (-1)^{q_{v_1}(G)} m_{\disc}(\pi)\chi_{\EP}(\pi_{v_1}),\quad\mbox{if}~\pi_{v_2}\simeq \pi^0_{v_2}$$
%  and $a_{\cF}(\pi):=0$ otherwise.
%  Let $\phi_{v_1}$ be the Euler-Poincar\'{e} function of \cite{Kot88} and
%  $\phi_{v_2}$ a pseudo-coefficient for $\pi_{v_2}$ such that $\tr \pi_{v_2}(\phi_{v_2})=1$.
%  Then the simple trace formula
%  implies (as in \cite[REFERENCE]{Shi-Plan})  \beq\label{e:mu-hat-general}\hat{\mu}_{\cF,S}(\hat{\phi}_S)
%  = \sum_{\pi\in \cAR_{\disc,\chi}(G)} a_{\cF}(\pi) \hat{\phi}_S(\pi_S)
%  = (-1)^{q(G)} I_{\spec}(\phi_S\phi^{S,v_1,v_2}\phi_{v_1}\phi_{v_2})\eeq
%  for any $\phi^{S,v_1,v_2}\in C^\infty_c(G(\A^{S,v_1,v_2,\infty}))$ and $\phi_S
%  \in C^\infty_c(G(F_S))$.
%\end{ex}

\subsection{Level aspect}\label{sub:level-varies}

  We are in the setting of Example \ref{ex:level-varies}.
  Recall that $\Res_{F/\Q} G$ is assumed to be cuspidal.
  Fix $\Xi:G\hra \GL_m$ as in Proposition \ref{p:global-integral-model} and let $B_\Xi$ and $c_\Xi$ be as in \eqref{e:def-M(Xi)} and
Lemma \ref{l:forcing-unipotent}.
Write $\mL_c(M_0)$ for the set of $F$-rational cuspidal Levi subgroups
of $G$ containing the minimal Levi $M_0$.

\begin{thm}\label{t:level-varies}
  Fix $\phi_{S_0}\in C^\infty_c(G(F_{S_0}))$ and $\xi$.
  Let $S_1\subset \cV_F^\infty$ be a subset where $G$ is unramified.
  Let $\phi_{S_1}\in \cH^{\ur}(G(F_{S_1}))^{\le \kappa}$
  be such that $|\phi_{S_1}|\le 1$ on $G(F_{S_1})$.
   If $\mL_c(M_0)=\{G\}$ (in particular if $G$ is abelian) then $\hat{\mu}_{\cF_k,S_1}(\hat{\phi}_{S_1})=\pl_S(\hat{\phi}_S)$.
   Otherwise
 there exist constants $A_{\level}, B_{\level}>0$ and $C_{\level}\ge 1$ such that
 \beq \label{e:t:level}\hat{\mu}_{\cF_k,S_1}(\hat{\phi}_{S_1})-\pl_S(\hat{\phi}_S)
  = O(q_{S_1}^{A_{\level}+B_{\level} \kappa} \N(\fkn)^{-C_{\level}})\eeq
  as $\fkn,\kappa\in \Z_{\ge 1}$, $S_1$ and $\phi_{S_1}$ vary subject to the following conditions:
  \benu[(i)] \item $\N(\fkn) \ge c_\Xi q_{S_1}^{B_\Xi m \kappa}$,
  \item no prime divisors of $\fkn$ are contained in $S_1$. \eenu
  (The implicit constant in $O(\cdot)$ is independent of $\fkn$, $\kappa$, $S_1$ and $\phi_{S_1}$.)

\end{thm}

\begin{rem}\label{r:level-varies}
  When $\pl_{S_0}(\hat \phi_{S_0})\neq0$, \eqref{e:t:level} is equivalent to
$$\hat{\mu}^{\natural}_{\cF,S_1}(\hat{\phi}_{S_1})-\pl_{S_1}(\hat{\phi}_{S_1})
  = O(q_{S_1}^{A_{\level}+B_{\level}  \kappa} \N(\fkn)^{-C_{\level}})$$.
\end{rem}

\begin{rem}\label{r:explicit-const-level}
  One can choose $A_{\level}, B_{\level}, C_{\level}$ to be explicit integers. See the proof below.
  For instance $C_{\level}\ge n_G$ for $n_G$ defined in \S\ref{sub:notation}.
\end{rem}

\begin{proof}
  Put $\phi^{S,\infty}:=\triv_{K^{S,\infty}(\fkn)}$. The right hand side
  of \eqref{e:apply-Arthur} is expanded as in \cite[Thm 6.1]{Art89} as shown by Arthur.
  Arguing as at the start of the proof of \cite[Thm 4.4]{Shi-Plan}, we obtain from Lemma \ref{l:forcing-unipotent} in view of the imposed lower bound on $N(\fkn)$ that
  \beq\label{e:TF-level-aspect}
\hat{\mu}_{\cF,S_1}(\hat{\phi}_{S_1})-\pl_S(\hat{\phi}_S)
  =\sum_{M\in \mL_c(M_0)\bs \{G\}} a_M \cdot \phi_{S_0,M}(1)\phi_{S_1,M}(1) \phi^{S,\infty}_{M}(1) \frac{\Phi^G_M(1,\xi)}{\dim \xi},\eeq
  where the sum runs over proper cuspidal Levi subgroups of $G$ containing
  a fixed minimal $F$-rational Levi subgroup (see \cite[p.539]{GKM97}
  for the reason why only cuspidal Levi subgroups contribute)
  and $a_M\in \C$ are explicit constants depending only on $M$ and $G$.
  A further explanation of \eqref{e:TF-level-aspect} needs to be given.
  Since only semisimple conjugacy classes contribute to Arthur's trace formula for each $M$,
  Lemma \ref{l:forcing-unipotent} tells us that any contribution from non-identity elements
  vanishes. Note that $\pl_S(\hat{\phi}_S)$ comes from the $M=G$ term on the right hand side.

  The first assertion of the theorem follows immediately from
  \eqref{e:TF-level-aspect}. Henceforth we may assume that $\mL_c(M_0)\bs \{G\}\neq \emptyset$.

  Clearly $\phi_{S_0,M}(1)$ and $\Phi^G_M(1,\xi)/\dim \xi$ are constants.
  It was shown in Lemma \ref{l:bounding-phi-on-S_0} that $|\phi_{S_1,M}(1)|=O(q_{S_1}^{d_G+r_G+b_G\kappa})$
  for $b_G> 0$ in that lemma. We take $$A_{\level}:=d_G+r_G\quad\mbox{and}\quad B_{\level}:=b_G.$$
  We will be done if it is checked that $|\phi^{S,\infty}_{M}(1)|= O(\N(\fkn)^{-C_{\level}})$ for some $C_{\level}\ge 1$.
  Let $P=MN$ be a parabolic subgroup with Levi decomposition where $M$ is as above. Then
  $$0\le \phi^{S,\infty}_{M}(1)=\int_{N(\A_F^{S,\infty})} \phi^{S,\infty}(n)dn
  = \prod_{v\notin S\atop v|\fkn~\mathrm{or}~v\in \Ram(G)} \vol(K_v(\varpi_v^{v(\fkn)})\cap N(F_v))$$
 $$ = \prod_{ v\notin S \atop v|\fkn~\mathrm{or}~v\in \Ram(G)} \vol (N(F_v)_{x,v(\fkn)})
  = \left( \prod_{v|\fkn\atop v\notin S} q_v^{-v(\fkn)\dim N}\right)
   \prod_{v\in \Ram(G)\atop v\notin S} \vol(K_v\cap N(F_v)).$$
  The last equality uses the standard fact about the filtration
  that $\vol(N(F_v)_{x,v(\fkn)})=|\varpi_v|^{v(\fkn)\dim N} \vol(N(F_v)_{x,0})$ and the fact \eqref{e:meas-K-cap-N}
  that $\vol(N(F_v)_{x,0}) = \vol(N(F_v)\cap K_v)=1$ when $G$ is unramified at $v$.
  Take $$C_{\level}:=\min_{M\in \mL_c(M_0)\bs \{G\}\atop P=MN}
  (\dim N)$$ to be the minimum dimension of the unipotent radical of a proper parabolic subgroup of $G$ with cuspidal Levi
  part. Then $|\phi^{S,\infty}_{M}(1)|\le \N(\fkn)^{-C_{\level}}\prod_{v\in \Ram(G)} \vol(K_v\cap N(F_v))$
  for every $M$ in \eqref{e:TF-level-aspect}.

\end{proof}

\subsection{Weight aspect}\label{sub:weight-varies}

 We put ourselves in the setting of Example \ref{ex:wt-varies} and exclude the uninteresting case of $G=\{1\}$.
 By the assumption $Z(G)=\{1\}$, for every $\gamma\neq 1\in G(F)$
 the connected centralizer $I_\gamma$ has a strictly smaller set of roots so that $|\Phi_{I_\gamma}|<|\Phi|$.
 Our next task is to prove a similar error bound as in the last subsection.

\begin{thm}\label{t:weight-varies} Fix $\phi_{S_0}\in C^\infty_c(G(F_{S_0}))$
 and $U^{S,\infty}\subset G(\A^{S,\infty})$.
 There exist constants $A_{\weight},B_{\weight}>0$ and $C_{\weight}\ge 1$ satisfying the following: for
 \bit
 \item any $\kappa\in \Z_{>0}$, \item any finite subset $S_1\subset \cV_F^\infty$ disjoint from $S_0$ and $S_{\bad}$
   (\S\ref{sub:global-bound-orb-int}) and \item any $\phi_{S_1}\in \cH^{\ur}(G(F_{S_1}))^{\le \kappa}$ % $r^*(\cH_{d,S_1}^{\le \kappa})$
  such that $|\phi_{S_1}|\le 1$ on $G(F_{S_1})$,\eit
%  we have
 $$\hat{\mu}_{\cF,S_1}(\hat{\phi}_{S_1})-\pl_S(\hat{\phi}_S)
  = O(q_{S_1}^{A_{\weight}+B_{\weight}\kappa} m(\xi)^{-C_{\weight}})$$
where the implicit constant in $O(\cdot)$ is independent of $\kappa$, $S_1$ and $\phi_{S_1}$.
(Equivalently, $\hat{\mu}^\natural_{\cF,S_1}(\hat{\phi}_{S_1})-\pl_{S_1}(\hat{\phi}_{S_1})
  = O(q_{S_1}^{A_{\weight}+B_{\weight}\kappa} m(\xi)^{-C_{\weight}})$ if $\pl_{S_0}(\hat\phi_{S_0})\neq 0$.)

\end{thm}

\begin{rem}
  We always assume that $S_0$ and $S_1$ are disjoint. So the condition on $S_1$ is really that it stays away from the finite set $S_{\bad}$. This enters the proof where a uniform bound on orbital integrals from \S\ref{sub:global-bound-orb-int} is applied to the places in $S_1$.
\end{rem}

\begin{rem}\label{r:explicit-const-wt}
  Again $A_{\weight}, B_{\weight}, C_{\weight}$ can be chosen explicitly as can be seen from the proof below. For instance a choice can be made such that $C_{\weight}\ge n_G$ for $n_G$ defined in \S\ref{sub:notation}.
\end{rem}

\begin{proof}
  We can choose a sufficiently large finite set $S'_0\supset S_0\cup \Ram(G)$ in the complement of $S_1\cup S_\infty$ such that $U^{S,\infty}$ is a finite disjoint union of
  groups of the form $(\prod_{v\notin S'_0\cup S_1\cup S_\infty} K_v)\times U_{S'_0\bs S_0}$ for open compact subgroups
  $U_{S'_0\bs S_0}$ of $G(\A_{F,S'_0\bs S_0})$. By replacing $S_0$ with $S'_0$
  (and thus $S$ with $ S'_0\coprod S_1$), we reduce the proof to the case where $U^{S,\infty}=\prod_{v\notin S\cup S_\infty} K_v$.

  For an $F$-rational Levi subgroup $M$ of $G$, let $Y_M$ be as in Proposition \ref{p:bound-number-of-conj},
  where $\kappa$, $S_0$ and $S_1$ are as in the theorem. (So the set $Y_M$ varies as
  $\kappa$ and $S_1$ vary.) Take \eqref{e:apply-Arthur} as a starting point.
  Arthur's trace formula (\cite[Thm 6.1]{Art89}) and the argument in the proof of \cite[Thm 4.11]{Shi-Plan} show
(note that our $Y_M$ contains $Y_M$ of \cite{Shi-Plan} but could be strictly bigger):
 $$ \hat{\mu}_{\cF,S_1}(\hat{\phi}_{S_1})-\pl_S(\hat{\phi}_S)
=
\sum_{\gamma\in Y_G\bs\{1\}} a_{G,\gamma}\cdot
|\iota^G(\gamma)|^{-1}O^{G(\A^\infty_F)}_\gamma(\phi^\infty)\frac{\tr \xi_n(\gamma)}{\dim \xi_n}
$$
\beq\label{e:main-formula-wt-varies} +\sum_{M\in \mL_c\bs \{G\}}
\sum_{\gamma\in Y_M}
a_{M,\gamma}\cdot|\iota^M(\gamma)|^{-1}O^{M(\A^\infty_F)}_\gamma(\phi^\infty_M)\frac{\Phi^G_M(\gamma,\xi_n)}{\dim \xi_n}
\eeq
where $a_{M,\gamma}$ (including $M=G$) is given by
$$a_{M,\gamma}=\tau'(G)^{-1} \frac{\ol{\mu}^{\can,\EP}(I^M_\gamma(F)\bs
I^M_\gamma(\A_F)/A_{I^M_\gamma,\infty})}
 {\ol{\mu}^{\EP}(I^M_\gamma(F_\infty)/A_{I^M_\gamma,\infty})}$$
 $$\stackrel{\mathrm{Cor}~\ref{c:canonical-measure}}{=}
\frac{\tau(I^M_\gamma)}{\tau'(G)} \frac{|\Omega_{I^M_\gamma}|}{|\Omega_{I^M_\gamma,c}|}
\frac{L(\Mot_{I^M_\gamma})}{e(I^M_{\gamma,\infty})2^{[F:\Q] r_G}} .$$

 Let us work with one $M$ at a time.
 Observe that clearly
 $|\Omega_{I^M_\gamma}|/|\Omega_{I^M_\gamma,c}|\le |\Omega|$
 and that $\tau(I^M_\gamma)$ is bounded by a constant
 depending only on $G$ in view of \eqref{e:Tamagawa} and Corollary \ref{c:bounding-pi0-general}
 or Lemma \ref{l:bounding-pi0}.  % (Note that there are only finitely many $M$'s.)
% By (ii) and (iii) of Proposition \ref{p:Gross-motives}, $\dim \Mot_{I^M_\gamma,d}\le r_G$
% and $\Mot_{I^M_\gamma,d}=0$ for $d>(\dim G+1)/2$.
% Hence Corollary \ref{c:canonical-measure} tells us that for some $C_0>0$ depending only on $G$,\footnote{The
% conductor$>1$ (outside $S\cup S_\infty$) only when some $|1-\alpha(\gamma)|>1$ for some $\alpha$.
% This suggests that we may try to prove $\N_{F/\Q}(\fkf(\Mot_{I^M_\gamma,d}))
% \le |1-\alpha(\gamma)|^{\delta}$ for some constant $\delta$. }
By Corollary \ref{c:bound-on-L(1)-2}, there exist constants $c_2,A_2>0$ such that
 \[
 %\label{e:a_M,gamma}
 |a_{M,\gamma}|\le c_2
 \prod_{v\in \Ram(I_\gamma^M)} q_v^{A_2}
   \]

   It is convenient to define the following finite subset of $\cV_F^\infty$ for each $\gamma\in Y_M$. We fix a maximal torus $T^M_\gamma$ in $M$ over $\ol{F}$ containing $\gamma$ and write $\Phi_{M,\gamma}$ for the set of roots of $T^M_\gamma$ in $M$. (A different choice of $T^M_\gamma$ does not affect the argument.)
  $$S_{M,\gamma}:=\{v\in \cV_F^\infty\bs S: \exists \alpha\in \Phi_{M,\gamma},
  ~\alpha(\gamma)\neq 1
  ~\mbox{and}~|1-\alpha(\gamma)|_v\neq 1\}.$$
 (If $\gamma$ is in the center of $M(F)$ then $S_{M,\gamma}=\emptyset$ and $q_{S_{M,\gamma}}=1$.)

  We know that $O^{M(F_v)}_\gamma(\triv_{K_{M,v}})=1$ for $v\notin S\cup S_{M,\gamma}\cup S_\infty$
and that $S_{M,\gamma}\supset \Ram(I^M_\gamma)$
  from \cite[Cor 7.3]{Kot86}. According to Lemma \ref{l:const-term-on-unram-Hecke}
  $\phi_v=\triv_{K_{v}}$ implies $\phi_{v,M}=\triv_{K_{M,v}}$. Hence
 \begin{eqnarray}
   |a_{M,\gamma}|&\le& c_2\cdot (q_{S_{M,\gamma}})^{A_2} \label{e:bound-a_M}\\
  O^{M(\A_F^\infty)}_\gamma(\phi^\infty_M)
 & =& O^{M(F_S)}_\gamma(\phi_{S,M})\prod_{v\in S_{M,\gamma}}
   O^{M(F_{v})}_\gamma(\triv_{K_{M,v}}). \notag
\end{eqnarray}
   By Theorem \ref{t:appendix1}, there exists a constant $c(\phi_{S_0,M})>0$
  such that
   $$O^{M(F_{S_0})}_\gamma(\phi_{S_0,M})
   \le c(\phi_{S_0,M}) \prod_{v\in S_0}
   D^M_v(\gamma)^{-1/2}, \quad \forall \gamma\in Y_M.$$
   By Theorem \ref{t:appendeix2}, there exist $a,b,c,e_G\in \R_{\ge 0}$ (independent of
   $\gamma$, $S_1$, $\kappa$ and $k$) such that
 \begin{eqnarray}\label{e:9.17}
    O^{M(F_{S_1})}_\gamma(\phi_{S_1,M})
 & \le & q_{S_1}^{a+b\kappa} \prod_{v\in S_1} D^M_v(\gamma)^{-e_G/2},  \\
  O^{M(F_v)}_\gamma(\triv_{K_{M,v}})
 & \le & q_{v}^{c}  D^M_v(\gamma)^{-e_G/2}, \quad\forall v\in S_{M,\gamma}. \label{e:9.18}
\end{eqnarray}
 (To obtain \eqref{e:9.17} and \eqref{e:9.18}, apply Theorem \ref{t:appendeix2} to $v\in S_1$ and $v\in S_{M,\gamma}$.)

   Hence
   \begin{eqnarray}
   O^{M(\A_F^\infty)}_\gamma(\phi^\infty_M)&\le&
   c(\phi_{S_0,M})  q_{S_1}^{a+b\kappa} q_{S_{M,\gamma}}^{c}
  \left( \prod_{v\nmid \infty} D^M_v(\gamma)^{-1/2}\right)
\prod_{v\in S_1\cup S_{M,\gamma}} D^M_v(\gamma)^{(1-e_G)/2}\nonumber\\
  & = &  c(\phi_{S_0,M})  q_{S_1}^{a+b\kappa} q_{S_{M,\gamma}}^{c}
   \prod_{v|\infty} D^M_v(\gamma)^{1/2}
\prod_{v\in S_1\cup S_{M,\gamma}} D^M_v(\gamma)^{(1-e_G)/2}\label{e:bound-on-O^M}
\end{eqnarray}

   On the other hand there exist $\delta_{S_0},\delta_{\infty}$, $\delta_{S_1}\ge1$ such that
   for every $\gamma\in Y_M$ with $\alpha(\gamma)\neq 1$,
   \bit
   \item $|1-\alpha(\gamma)|_{S_0}\le \delta_{S_0}$. (compactness of $\supp \phi_{S_0}$)
   \item $|1-\alpha(\gamma)|_{\infty}\le \delta_{\infty}$. (compactness of $U_\infty$)
   \item $|1-\alpha(\gamma)|_{S_1}\le \delta_{S_1}q_{S_1}^{B_5 \kappa}$. (Lemma \ref{l:bounding1-alpha(gamma)}; Remark \ref{r:Lem2.18-indep} explains the independence of $B_1$ of $S_1$).
   \eit
   (When $\alpha(\gamma)=1$, our convention is that $|1-\alpha(\gamma)|_v=1$ for every $v$
   to be consistent with the first formula of Appendix \ref{s:app:Kottwitz}.)
  Hence, together with the product formula for $1-\alpha(\gamma)$,
   $$1=\prod_v |1-\alpha(\gamma)|_v \le\delta_{S_0} \delta_{\infty}
    \delta_{S_1}q_{S_1}^{B_5 \kappa}\prod_{v\in S_{M,\gamma}}|1-\alpha(\gamma)|_v .$$
  If $\gamma\in Z(M)(F)$ then $q_{S_{M,\gamma}}=1$. Otherwise
    for each $v\in S_{M,\gamma}$, we may choose $\alpha\in \Phi_{M,\gamma}$ such that $|1-\alpha(\gamma)|_v\neq 1$.
Set $\delta:=\delta_{S_0} \delta_{\infty}
    \delta_{S_1}$.
    Then $|1-\alpha(\gamma)|_v \le q_v^{-1}$ for $v\in S_{M,\gamma}$ by Lemma \ref{l:alpha(gamma)-is-integral}
    (which is applicable in view of the first paragraph in the current proof)
    so % $\prod_{v\in S_{M,\gamma}} q_v^{-1}\ge  (\delta q_{S_1}^{B_5 \kappa})^{-1}$, or
    \beq\label{e:prod-q-S_M,gamma}
     q_{S_{M,\gamma}}\le \delta q_{S_1}^{B_5 \kappa}.\eeq
  Keep assuming that $\gamma$ is not central in $M$ and that
  $\alpha(\gamma)\neq 1$. Again by the product formula
  $\prod_{v\in S_1\cup S_{M,\gamma}} |1-\alpha(\gamma)|
  = \prod_{v\in S_0\cup S_\infty} |1-\alpha(\gamma)|^{-1} \ge (\delta_{S_0}\delta_\infty)^{-1}$,
  thus
  \beq\label{e:bound-D^M}
\prod_{v\in S_1\cup S_{M,\gamma}} D^M(\gamma)^{-1}\le \delta_{S_0}\delta_\infty.\eeq
  The above holds also when $\gamma$ is central in $M$, in which case the left hand side equals 1.

    Now \eqref{e:bound-on-O^M}, \eqref{e:prod-q-S_M,gamma} and \eqref{e:bound-D^M} imply
   \beq\label{e:bound-on-O^M-2}
   O^{M(\A_F^\infty)}_\gamma(\phi^\infty_M)\le
   c(\phi_{S_0,M}) \delta^{a}(\delta_{S_0}\delta_\infty)^{(e_G-1)/2}
 q_{S_1}^{a+b\kappa+ c B_5 \kappa}   \prod_{v|\infty}
   D^M_v(\gamma)^{1/2}.\eeq

   Lemma \ref{l:bound-for-st-ds-char} gives a bound on the stable discrete series
   character:
 \beq\label{e:bound-Phi^G_M}\frac{|\Phi^G_M(\gamma,\xi)|}{\dim \xi}
  % \frac{c}{ m(\xi)^{|\Phi^+|-|\Phi^+_{I^M_\gamma}|}} \prod_{\alpha\in \Phi_{M,\gamma}\atop \alpha(\gamma)\neq 1} (1-\alpha(\gamma))^{-1/2}=
 \le c \frac{\prod_{v|\infty} D^M_v(\gamma)^{-1/2}}{ m(\xi)^{|\Phi^+|-|\Phi^+_{I^M_\gamma}|}}.\eeq

  Multiplying \eqref{e:bound-a_M}, \eqref{e:bound-on-O^M-2},
  \eqref{e:bound-Phi^G_M} altogether (and noting $|\iota^M(\gamma)|\le 1$),
  the absolute value of the summand for $\gamma$ in \eqref{e:main-formula-wt-varies} (including $M=G$)
  is
$$O\left( m(\xi)^{-(|\Phi^+|-|\Phi^+_{I^M_\gamma}|)}
q_{S_1}^{a+b\kappa+c B_5 \kappa+A_2}\right).$$
   All in all,  $|\hat{\mu}_{\cF,S}(\hat{\phi}_S)-\pl_S(\hat{\phi}_S)|$ is
$$ \left(|Y_G|-1 + \sum_{M\in \mL_c\bs \{G\}} |Y_M|\right)
  O\left( m(\xi)^{-(|\Phi^+|-|\Phi^+_{I^M_\gamma}|)}
  q_{S_1}^{a+b\kappa+ c B_5 \kappa+A_2}\right).$$
 Set (excluding $\gamma=1$ in the second minimum when $M=G$) $$C_{\weight}:= \min_{M\in \mL_c(M_0)}
 \min_{\gamma\in M(F)\atop \mathrm{ell.~in~}M(F_\infty)}%\atop \gamma\neq 1~\mathrm{if}~ M=G}
  (|\Phi^+|-|\Phi^+_{I^M_\gamma}|)$$ % as $M$ runs over
% $\mL_c(M_0)$ and $\gamma$ runs over all semisimple $F_\infty$-elliptic elements of $M(F)$, excluding
% $\gamma=1$ if $M=G$.
 Note that
 $C_{\weight}$ depends only on $G$. It is automatic that
  $|\Phi^+|-|\Phi^+_{I^M_\gamma}|\ge 1$ on $Y_G\bs \{1\}$ and $Y_M$ for $M\in \mL_c(M_0)\bs \{G\}$.
The proof is concluded by invoking Corollary \ref{c:bound-number-of-conj} (applied to $Y_G$ and $Y_M$)
with the choice $$A_{\weight}:=a+A_2+A_6,\quad B_{\weight}:=b+cB_5+B_8.$$

\end{proof}

%\subsection{General case with local conditions}

%(UNDER CONSTRUCTION.) Use the simple trace formula.

% \subsection{Contribution of non-tempered representations at $\infty$}\label{sub:non-temp-infty}

\subsection{Automorphic Plancherel density theorem}\label{sub:Plan-density}

  In the situation of either Example \ref{ex:level-varies} or \ref{ex:wt-varies},
 let us write $\cF_k(\phi_{S_0})$ for $\cF_k$ in order to emphasize the dependence on $\phi_{S_0}$.
Take $S_1=\emptyset$ so that $S=S_0$. Then $\hat{\mu}_{\cF_k(\phi_{S}),\emptyset}$ may be viewed
as a complex number (as it is a measure on a point).
In fact we can consider $\cF_k(\hat{f}_{S})$, a family whose local condition at $S$
is prescribed by $\hat{f}_S\in \mF(G(F_S)^{\wedge})$, even if
$\hat{f}_{S}$ does not arise from any $\phi_{S}$ in $C^\infty_c(G(F_S))$.
Put $\hat{\mu}_k(\hat{f}_{S}):=\hat{\mu}_{\cF_k(\hat{f}_{S}),\emptyset}\in \C$.
We recover the automorphic Plancherel density theorem (\cite[Thm 4.3, Thm 4.7]{Shi-Plan}).

\begin{cor}\label{c:Plan-density} Consider families $\cF_k$ in level or weight aspect as above. In level aspect assume that the highest weight of $\xi$ is regular. (No assumption is necessary in the weight aspect.)
 For any $\hat{f}_S\in \mF(G(F_S)^{\wedge})$,
  $$\lim_{k\ra\infty} \hat{\mu}_k(\hat{f}_{S})=\pl_S(\hat{f}_S).$$

\end{cor}

\begin{proof}
  Theorems \ref{t:level-varies} and \ref{t:weight-varies} tell us that
  \beq\label{e:aut-plan-phi}\lim_{k\ra\infty} \hat{\mu}_k(\hat{\phi}_S)=\pl_S(\hat{\phi}_S).\eeq
  (Even though there was a condition on $S_1$, note that there was no condition on $S_0$ in either theorem.)

  We would like to improve \eqref{e:aut-plan-phi} to allow more general test functions.
  What needs to be shown (cf. \eqref{e:extending-at-v_j-0} below) is that for every $\epsilon>0$, $$\limsup_{k\ra \infty} |\hat{\mu}_{k}(\hat{f}_{S})- \pl_{S}(\hat{f}_{S})|\le4\epsilon.$$
%    Since the proof is similar to that of Theorem \ref{t:Sato-Tate-level} below and only simpler, we will content ourselves with sketching the argument.
  Thanks to Proposition \ref{p:density} there exist $\phi_{S},\psi_{S}\in \cH^{\ur}(G(F_{S}))$ such that $|\hat{f}_{S}-\hat{\phi}_{S}|\le \hat{\psi}_{S}$ on $G(F_{S})^{\wedge}$ and $\pl_{S}(\hat{\psi}_{S})\le \epsilon$. Then (cf. \eqref{e:extending-at-v_j} below)
  $$
\begin{aligned}
|\hat{\mu}_{k}(\hat{f}_{S})- \pl_{S}(\hat{f}_{S})| & \le  |\hat{\mu}_k(\hat{f}_{S}-\hat{\phi}_{S})|\\ & + |\hat{\mu}_k(\hat{\phi}_{S})- \pl_{S}(\hat{\phi}_{S})| + |\pl_{S}(\hat{\phi}_{S}-\hat{f}_{S})|.
\end{aligned}
 $$
  Now $|\pl_{S}(\hat{f}_{S}-\hat{\phi}_{S})|\le |\pl_S(\hat{\psi}_S)|\le \epsilon$, and $|\hat{\mu}_k(\hat{\phi}_{S})- \pl_{S}(\hat{\phi}_{S})|\le \epsilon$ for $k\gg1$ by \eqref{e:aut-plan-phi}.
  Finally $\hat{\mu}_k$ is a positive measure since the highest weight of $\xi$ is regular (see Example \ref{ex:example-for-f_S}), and we get
   $$
     |\hat{\mu}_k(\hat{f}_{S}-\hat{\phi}_{S})|\le \hat{\mu}_k(|\hat{f}_{S}-\hat{\phi}_{S}|)\le \hat{\mu}_k(\hat{\psi}_S).
   $$
   (To see the positivity of $\hat{\mu}_k$, notice that $\hat{\mu}_k(\hat{f}_{S}-\hat{\phi}_{S})$ is unraveled via \eqref{e:a(pi)-general} and \eqref{e:defn-of-mu} as a sum of $(\hat{f}_{S}-\hat{\phi}_{S})(\pi)$ with coefficients having nonnegative signs. This is because $\chi_{\EP}(\pi_\infty\otimes \xi)$ is either 0 or $(-1)^{q(G)}$ when $\xi$ has regular highest weight, cf. \S\ref{sub:EP}.)
    According to \eqref{e:aut-plan-phi}, $\lim_{k\ra \infty}\hat{\mu}_k(\hat{\psi}_S)=\pl_S(\hat{\psi}_S)\le \epsilon$. In particular $|\hat{\mu}_k(\hat{f}_{S}-\hat{\phi}_{S})|\le 2\epsilon$ for $k\gg1$. The proof is complete.
\end{proof}

\begin{rem}
  If $G$ is anisotropic modulo center over $F$ so that the trace formula for compact quotients is available, or if a further local assumption at finite places is imposed so as to avail the simple trace formula, the regularity condition on $\xi$ can be removed by an argument of De George-Wallach and Clozel (\cite{dGW78}, \cite{Clo86}). The main point is to show that the contribution of ($\xi$-cohomological) non-tempered representations at $\infty$ to the trace formula is negligible compared to the contribution of discrete series. Their argument requires some freedom of choice of test functions at $\infty$, so it breaks down in the general case since one has to deal with new terms in the trace formula which disappear when Euler-Poincar\'{e} functions are used at $\infty$. In other words, it seems necessary to prove analytic estimates on more terms (if not all terms) in the trace formula than we did in order to get rid of the assumption on $\xi$. (This remark also applies to the same condition on $\xi$ in \S\ref{sub:ST-theorem} and \S\ref{sub:general-functions-S_0} for level aspect families.) We may return to this issue in future work.
\end{rem}

\begin{rem}
  In the case of level aspect families, \cite[Thm 4.3]{Shi-Plan} assumes that
  the level subgroups form a chain of decreasing groups whose intersection is the trivial group.
  The above corollary deals with some new cases as it assumes only that $\N(\fkn_k)\ra \infty$.

\end{rem}

\begin{cor}\label{c:estimate-F_k}
  Keep assuming that $S_1=\emptyset$.
 Let $(U_k^{S,\infty},\xi_k)=(K^{S,\infty}(\fkn_k),\xi)$ or $(U^{S,\infty},\xi_k)$
in Example \ref{ex:level-varies} or \ref{ex:wt-varies}, respectively, but prescribe local conditions at $S$
by $\hat{f}_{S}$ rather than $\phi_{S}$.
Then  $$\lim_{k\ra\infty}\frac{\mu^{\can}(U_k^{S,\infty})}{\tau'(G)\dim \xi_k}  |\cF_k| = \pl_{S}(\hat{f}_{S}).$$
\end{cor}

\begin{proof}
  The corollary results from Corollary \ref{c:Plan-density} since
$$ \frac{\mu^{\can}(U_k^{S,\infty})}{\tau'(G)\dim \xi_k} |\cF_k|=
 \frac{\mu^{\can}(U_k^{S,\infty})}{\tau'(G)\dim \xi_k}\sum_{\pi\in \cAR_{\disc,\chi_k}(G)} a_{\cF_k}(\pi) =
 \hat{\mu}_{\cF_k,\emptyset}(\hat{f}_S).$$
\end{proof}

\subsection{Application to the Sato-Tate conjecture for families}\label{sub:ST-theorem}

  As an application of Theorems \ref{t:level-varies} and \ref{t:weight-varies},
  we are about to fulfill the promise of \S\ref{sub:ST-families}
  by showing that the Satake parameters in the automorphic families
  $\{\cF_k\}$ are equidistributed
  according to the Sato-Tate measure in a suitable sense
  (cf. Conjecture \ref{c:ST-families}).

  The notation and convention of \S\ref{s:Sato-Tate} are retained here.
  Let $\theta\in \mC(\Gamma_1)$ and $\hat{f}\in \cF(\hat{T}_{c,\theta}/\Omega_{c,\theta})$.
  For each $v\in \cV_F(\theta)$,
  the image of $\hat{f}$ in $\cF(G(F_v)^{\wedge,\ur})$ via \eqref{e:Sauvageot}
  will be denoted $\hat{f}_v$.
%   For a regular function $\hat{\phi}$ on
%  the complex variety $\hat{T}_{\theta}/\Omega_{\theta}$, write $\phi_v\in \cH^{\ur}(G(F_v))$
%  for its image under the Satake transform. As usual $\hat{\phi}_v\in \cF(G(F_v)^{\wedge,\ur})$
%  is obtained from $\phi_v$. If $\hat{f}$ is the restriction of $\hat{\phi}$ to $\hat{T}_{c,\theta}/\Omega_{c,\theta}$,
%  it is clear that $\hat{f}_v=\hat{\phi}_v$ on the tempered spectrum.

\begin{thm}\label{t:Sato-Tate-level} (level aspect)

  Pick any $\theta\in \mC(\Gamma_1)$ and let $\{v_j\}_{j\ge 1}$ be a sequence in $\cV_F(\theta)$ such that
  $q_{v_j}\ra \infty$ as $j\ra \infty$.   Suppose that \bit\item $\pl_{S_0}(\hat{\phi}_{S_0})\neq 0$ and
  \item $\xi$ has regular highest weight. \eit
  Then for every $\hat{f}\in \cF(\hat{T}_{c,\theta}/\Omega_{c,\theta})$,
  $$\lim_{(j,k)\ra \infty} \hat{\mu}^{\natural}_{\cF_k,v_j}(\hat{f}_{v_j})
  = \ST_{\theta}(\hat{f})$$
  where the limit is taken over $(j,k)$ subject to the following conditions:
 \bit
 \item $\N(\fkn_k) q_{v_j}^{-B_\Xi m \kappa}\ge c_\Xi^{-1}$,
 \item $v_j\nmid \fkn_k$,
 \item $q_{v_j}^{N} \N(\fkn_k)^{-1}\ra 0$ for all $N>0$.
 \eit

\end{thm}

\begin{proof}
  Fix $\hat{f}$. We are done if $\limsup_{(j,k)\ra \infty} |\hat{\mu}^{\natural}_{\cF_k,v_j}(\hat{f}_{v_j})- \ST_{\theta}(\hat{f})|\le 4\epsilon$ for every $\epsilon>0$. By Proposition \ref{p:lim-of-Plancherel},
  $|\pl_{v_j}(\hat{f}_{v_j})-\ST_{\theta}(\hat{f})|\le\epsilon$ for sufficiently large $j$. So it is enough to show that
   \begin{equation}\label{e:extending-at-v_j-0}\limsup_{(j,k)\ra \infty} |\hat{\mu}^{\natural}_{\cF_k,v_j}(\hat{f}_{v_j})- \pl_{v_j}(\hat{f}_{v_j})|\le3\epsilon.\end{equation}

	 For every $j\ge1$, Proposition \ref{p:unr-density} allows us to find $\phi_{v_j},\psi_{v_j}\in \cH^{\ur}(G(F_{v_j}))^{\le \kappa}$ such that $|\hat{f}_{v_j}-\hat{\phi}_{v_j}|\le \hat{\psi}_{v_j}$ on $G(F_{v_j})^{\wedge}$ and $\pl_{v_j}(\hat{\psi}_{v_j})\le \epsilon$.
%  Choose $\kappa\in \Z_{\ge1}$ such that $\phi_{v_j}\in \cH^{\ur}(GF_{v_j}))^{\le \kappa}$.
  For each $j\ge 1$,
  \begin{equation}\label{e:extending-at-v_j}
\begin{aligned}
|\hat{\mu}^{\natural}_{\cF_k,v_j}(\hat{f}_{v_j})- \pl_{v_j}(\hat{f}_{v_j})| & \le  |\hat{\mu}^{\natural}_{\cF_k,v_j}(\hat{f}_{v_j}-\hat{\phi}_{v_j})|\\ & + |\hat{\mu}^{\natural}_{\cF_k,v_j}(\hat{\phi}_{v_j})- \pl_{v_j}(\hat{\phi}_{v_j})| + |\pl_{v_j}(\hat{\phi}_{v_j}-\hat{f}_{v_j})|.
\end{aligned}
\end{equation}
  Since $\pl_{v_j}$ is a positive measure,
 $$|\pl_{v_j}(\hat{\phi}_{v_j}-\hat{f}_{v_j})|\le  \pl_{v_j}(|\hat{\phi}_{v_j}-\hat{f}_{v_j}|)
 \le \pl_{v_j}(\hat{\psi}_{v_j})\le \epsilon.$$
 Theorem \ref{t:level-varies} and the assumptions of the theorem imply that for sufficiently large $(j,k)$,
$|\hat{\mu}^{\natural}_{\cF_k,v_j}(\hat{\phi}_{v_j})- \pl_{v_j}(\hat{\phi}_{v_j})|\le \epsilon$. So we will be done if for sufficiently large $(j,k)$,
  \begin{equation}\label{e:extending-at-v_j-2}
  |\hat{\mu}^{\natural}_{\cF_k,v_j}(\hat{f}_{v_j}-\hat{\phi}_{v_j})| \le \epsilon.
\end{equation}
 Arguing as in the proof of Corollary \ref{c:Plan-density} we deduce the following: when $\hat{\mu}^{\natural}_{\cF_k,v_j}(\hat{f}_{v_j}-\hat{\phi}_{v_j})$ is unraveled as a sum over $\pi$ (cf. \eqref{e:a(pi)-general} and \eqref{e:defn-of-mu}), each summand is $\hat{\phi}_{S_0}(\pi_{S_0})(\hat{f}_{v_j}-\hat{\phi}_{v_j})(\pi_{v_j})$ times a nonnegative real number. (This uses the regularity assumption on $\xi$. Certainly the absolute value of the sum does not get smaller when every summand is replaced with (something greater than or equal to) its absolute value, i.e.
 $$|\hat{\mu}^{\natural}_{\cF_k,v_j}(\hat{f}_{v_j}-\hat{\phi}_{v_j})|
 \le \hat{\mu}^{\natural}_{\cF_k(|\hat{\phi}_{S_0}|),v_j}(|\hat{f}_{v_j}-\hat{\phi}_{v_j}|)
 \le \hat{\mu}^{\natural}_{\cF_k(|\hat{\phi}_{S_0}|),v_j}(\hat{\psi}_{v_j}).$$
  Now choose $\phi'_{S_0}\in C^\infty_c(G(F_{S_0}))$ according to Lemma \ref{l:bounding-function-by-pos-function} so that $|\phi_{S_0}(\pi_{S_0})|\le \phi'_{S_0}(\pi_{S_0})$ for every $\pi_{S_0}\in G(F_{S_0})^{\wedge}$. Then
$$\hat{\mu}^{\natural}_{\cF_k(|\hat{\phi}_{S_0}|),v_j}(\hat{\psi}_{v_j})\le \hat{\mu}^{\natural}_{\cF_k(\phi'_{S_0}),v_j}(\hat{\psi}_{v_j}).$$
Theorem \ref{t:level-varies} applied to $\hat{\psi}_{v_j}$ and the inequality $\pl_{v_j}(\hat{\psi}_{v_j})\le \epsilon$ imply that $$\limsup_{(j,k)\ra\infty} \hat{\mu}^{\natural}_{\cF_k(\phi'_{S_0}),v_j}(\hat{\psi}_{v_j})\le \epsilon.$$ This concludes the proof of \eqref{e:extending-at-v_j-2}, thus also \eqref{e:extending-at-v_j-0}.

%  When the highest weight of $\xi$ is not regular, $\chi_{\EP}(\pi_\infty\otimes \xi)$ may be nonzero with sign different from $(-1)^{q(G)}$. This can occur only when $\pi_\infty$ is non-tempered. We can prove \eqref{e:extending-at-v_j-0} by applying the above argument to $\hat{\mu}^{\natural,\tempinf}$ together with Lemma \ref{l:non-temp-negligible2} below.

\end{proof}

%
%\begin{lem}\label{l:non-temp-negligible2} In the setting of Theorem \ref{t:Sato-Tate-level},
%  $$\lim_{(j,k)\ra\infty}|\hat{\mu}^{\natural}_{\cF_k,v_j}(\hat{f}_{v_j})-\hat{\mu}^{\natural,\tempinf}_{\cF_k,v_j}(\hat{f}_{v_j})|=0.$$
%\end{lem}
%
%\begin{proof}
%  Imitate the argument of \cite[\S3.3]{Clo86}? (***insert lemma above?***)
%\end{proof}

\begin{thm}\label{t:Sato-Tate-weight} (weight aspect) Let $\theta\in \mC(\Gamma_1)$ and $\hat{\phi}_{S_0}\in C^\infty_c(G(F_{S_0}))$. Suppose that $\{v_j\}_{j\ge 1}$ is a sequence in $\cV_F(\theta)$ such that $q_{v_j}\ra \infty$ as $j\ra \infty$ and that $\pl_{S_0}(\hat \phi_{S_0})\neq 0$.
 Then for every $\hat{f}\in \cF(\hat{T}_{c,\theta}/\Omega_{c,\theta})$,
  $$\lim_{(j,k)\ra \infty} \hat{\mu}^{\natural}_{\cF_k,v_j}(\hat{f}_{v_j})
  = \ST_{\theta}(\hat{f})$$
  if $q_{v_j}^{N} m(\xi_k)^{-1}\ra 0$ as $k\ra \infty$ for any integer $N\ge1$.

\end{thm}

\begin{proof}
  Same as above, except that Theorem \ref{t:weight-varies} is used
  instead of Theorem \ref{t:level-varies}.
\end{proof}

\begin{rem}
  As we have mentioned in \S\ref{sub:ST-families},
  Theorems \ref{t:Sato-Tate-level} and \ref{t:Sato-Tate-weight}
  indicate that $\{\cF_k\}_{\ge1}$ are ``general''
  families of automorphic representations in the sense of
  Conjecture \ref{c:ST-families}.
\end{rem}

\begin{cor}
  In the setting of Theorem \ref{t:Sato-Tate-level} or \ref{t:Sato-Tate-weight},
  suppose in addition that $|\cF_k|\neq 0$ for all $k\ge 1$. Then
 $$\lim_{(j,k)\ra \infty} \hat{\mu}^{\mathrm{count}}_{\cF_k,v_j}(\hat{f}_{v_j})
  = \ST_{\theta}(\hat{f}).$$
\end{cor}
\begin{proof}
  Follows from Corollary \ref{c:estimate-F_k} and the two preceding theorems (cf. Remark \ref{r:|F|}).
\end{proof}

\begin{rem}
  The assumption that $|\cF_k|\neq 0$ is almost automatically satisfied.
  Corollary \ref{c:estimate-F_k} and the assumption that $\pl_{S_0}(\hat{\phi}_{S_0})\neq 0$ imply that
  $|\cF_k|\neq 0$ for any sufficiently large $k$.

%  The assumptions $\pl_{S_0}(\hat{\phi}_{S_0})\neq 0$ and $|\cF_k|\neq 0$ are easily satisfied in the case
%  when the highest weight of $\xi$ (or $\xi_k$) is regular and $\phi_{S_0}$ satisfies the conditions that
%  $\hat{\phi}_{S_0}\in \R_{\ge 0}$ on $G(F_{S_0})^{\wedge}$ and that $\hat{\phi}_{S_0}$ is not identically zero.
%  Then $\pl_{S_0}(\hat{\phi}_{S_0})>0$ by the Plancherel formula.
%  Corollary \ref{c:estimate-F_k} implies that $\cF_k\neq \emptyset$ and $|\cF_k|\neq 0$ for any sufficiently large $k$.
%  In fact $a_{\cF}(\pi)\ge 0$ for
%  all $\pi\in \cAR_{\disc,\chi}(G)$, so in fact we have $|\cF_k|>0$ as soon as $\cF_k\neq \emptyset$.
\end{rem}

\subsection{More general test functions at $S_0$}\label{sub:general-functions-S_0}

  So far we worked primarily with families of Examples \ref{ex:level-varies}
  and \ref{ex:wt-varies}.
  We wish to extend Theorems \ref{t:Sato-Tate-level}
  and \ref{t:Sato-Tate-weight} when the local condition at $S_0$
is given by $\hat{f}_{S_0}$,
which may not be of the form $\hat{\phi}_{S_0}$
for any $\phi_{S_0}\in C^\infty_c(G(F_{S_0}))$
(cf. Example \ref{ex:example-for-f_S} and Remark \ref{r:why-S_0}).

\begin{cor}\label{c:extended} Let $\theta\in \mC(\Gamma_1)$ and let $\{v_j\}_{j\ge 1}$ be a sequence of places in $\cV_F(\theta)$ such that
  $q_{v_j}\ra\infty$ as $j\ra\infty$.
  Consider $\hat{\mu}_{\cF_k,v_j}$ where $$\cF_k=\left\{ \begin{array}{cl}
  \cF(K^{S,\infty}(\fkn_k),\hat{f}_{S_0},v_j,\xi) & \mathrm{level~aspect,~or}\\
  \cF(U^{S,\infty},\hat{f}_{S_0}, v_j,\xi_k) & \mathrm{weight~aspect}
  \end{array}\right.$$
  satisfying the conditions of Theorem \ref{t:Sato-Tate-level} or
  Theorem \ref{t:Sato-Tate-weight}, respectively.
 Then  $$\lim_{(j,k)\ra \infty} \hat{\mu}^{\natural}_{\cF_k,v_j}(\hat{f}_{v_j})
  = \ST_{\theta}(\hat{f})$$
  where the limit is taken as in Theorem \ref{t:Sato-Tate-level} (resp.
  Theorem \ref{t:Sato-Tate-weight}).
\end{cor}

\begin{proof}

  The basic strategy is to reduce to the case of $\hat{\phi}$ and $\hat{\phi}_{v_j}$ in place of
$\hat{f}$ and $\hat{f}_{v_j}$ via Sauvageot's density theorem, as in the proof of Theorem
\ref{t:Sato-Tate-level}.   We can decompose $\hat{f}=\hat{f}^+ +\hat{f}^-$ with
  $\hat{f}^+,\hat{f}^-\in  \cF(\hat{T}_{c,\theta}/\Omega_{c,\theta})$ such that $\hat{f}^+$ and $\hat{f}^-$ are nonnegative everywhere.
  The corollary for $\hat{f}$ is proved as soon as it is proved for $\hat{f}^+$ and $\hat{f}^-$. Thus we may assume that $\hat{f}\ge 0$ from now on.

  Fix any choice of $\epsilon>0$.
  Proposition \ref{p:density} ensures the existence of
  $\phi_{S_0},\psi_{S_0}\in C^\infty_c(G(F_{S_0}))$ such that
  $\pl_{S_0}(\hat\psi_{S_0})\le \epsilon$ and
  $|\hat{f}_{S_0}(\pi_{S_0})-\hat{\phi}_{S_0}(\pi_{S_0})|\le \hat{\psi}_{S_0}(\pi_{S_0})$
  for all $\pi_{S_0}\in G(F_{S_0})^\wedge$. Of course we can guarantee in addition that $\pl_{S_0}(\hat{\phi}_{S_0})\neq 0$. Put
$$\cF_k(\hat{\phi}_{S_0}):=\cF(K^{S,\infty}(\fkn_k),\hat{\phi}_{S_0},
  v_j,\xi)\quad (\mbox{resp.}~\cF_k(\hat{\phi}_{S_0})=\cF(U^{S,\infty},\hat{\phi}_{S_0},
  v_j,\xi_k)).$$
   Likewise we define $\cF_k(\hat{\psi}_{S_0})$ and so on. Then (cf. a similar step in the proof of Theorem \ref{t:Sato-Tate-level})
\begin{equation*}\begin{aligned}
| \hat{\mu}_{\cF_k,v_j}(\hat{f}_{v_j})-\pl_{S_0\cup\{v_j\}}(\hat{f}_{S_0}\hat{f}_{v_j})|&\le |\hat{\mu}_{\cF_k(\hat{\phi}_{S_0}),v_j}(\hat{f}_{v_j})-\pl_{S_0\cup\{v_j\}}(\hat{\phi}_{S_0}\hat{f}_{v_j})| \\
& + |\hat{\mu}_{\cF_k(|\hat{f}_{S_0}-\hat{\phi}_{S_0}|)}(\hat{f}_{v_j})|
 + \pl_{S_0\cup\{v_j\}}(|\hat{f}_{S_0}-\hat{\phi}_{S_0}|\hat{f}_{v_j})
 \end{aligned}
 \end{equation*}
  The first term on the right side tends to 0 as $(j,k)\ra \infty$ by Theorems \ref{t:Sato-Tate-level} and \ref{t:Sato-Tate-weight}. The last term is bounded by $\pl_{S_0\cup\{v_j\}}(\hat{\psi}_{S_0}\hat{f}_{v_j})\le \epsilon\pl_{v_j}(\hat{f}_{v_j})$ using the fact that $\pl_{S_0}$ is a positive measure. In order to bound the second term, recall that we are either in the weight aspect, or in the level aspect with regular highest weight for $\xi$. Then $a_{\cF_k(|\hat{f}_{S_0}-\hat{\phi}_{S_0}|)}(\pi)$ is a nonnegative multiple of $|\hat{f}_{S_0}(\pi_{S_0})-\hat{\phi}_{S_0}(\pi_{S_0})|$ as in the proof of Theorem \ref{t:Sato-Tate-level}. Thus
 $$|\hat{\mu}_{\cF_k(|\hat{f}_{S_0}-\hat{\phi}_{S_0}|)}(\hat{f}_{v_j})|=\hat{\mu}_{\cF_k(|\hat{f}_{S_0}-\hat{\phi}_{S_0}|)}(\hat{f}_{v_j})\le \hat{\mu}_{\cF_k(\hat{\psi}_{S_0})}
 (\hat{f}_{v_j})\le \epsilon \hat{\mu}^\natural_{\cF_k(\hat{\psi}_{S_0})}
 (\hat{f}_{v_j}),$$
 the last inequality coming from the bound $\pl_{S_0}(\hat\psi_{S_0})\le \epsilon$.

%  In the remaining case of level aspect with non-regular highest weight for $\xi$, one can still show that $|\hat{\mu}_{\cF_k(|\hat{f}_{S_0}-\hat{\phi}_{S_0}|)}(\hat{f}_{v_j})|\ra 0$ as $(j,k)\ra \infty$ by controlling the contribution from non-tempered representations at $\infty$ exactly as in the proof of Theorems \ref{t:Sato-Tate-level}.
  Hence we have shown that
  $$\limsup_{(j,k)\ra\infty} | \hat{\mu}_{\cF_k,v_j}(\hat{f}_{v_j})-\pl_{S_0\cup\{v_j\}}(\hat{f}_{S_0}\hat{f}_{v_j})| \le \epsilon \limsup_{(j,k)\ra\infty}\left(\hat{\mu}^\natural_{\cF_k(\hat{\psi}_{S_0})}
 (\hat{f}_{v_j})+\pl_{v_j}(\hat{f}_{v_j})\right).$$
  By Theorems \ref{t:Sato-Tate-level} and \ref{t:Sato-Tate-weight} and the fact that $\lim\limits_{j\ra\infty}\pl_{v_j}(\hat{f}_{v_j})=\ST_{\theta}(\hat{f})$, the right hand side is seen to be bounded by $2\epsilon \ST_{\theta}(\hat{f})$. As we are free to choose $\epsilon>0$,
  we deduce that $$\lim\limits_{(j,k)\ra\infty} \hat{\mu}_{\cF_k,v_j}(\hat{f}_{v_j})=\pl_{S_0}(\hat{f}_{S_0})\ST_{\theta}(\hat{f}).$$
\end{proof}

\begin{rem}
  It would be desirable to improve Theorems \ref{t:level-varies} and \ref{t:weight-varies} similarly
  by prescribing conditions at $S_0$ in terms of $\hat{f}_{S_0}$ rather than the less general
  $\hat{\phi}_{S_0}$. Unfortunately the argument proving Corollary \ref{c:extended} does not carry over.
  For instance in the case of Theorem \ref{t:level-varies}, one should know in addition that
  the multiplicative constant implicit in  $O(q_{S_1}^{A_{\level}+B_{\level} \kappa} \N(\fkn_k)^{-C_{\level}})$
  is bounded as a sequence of $\hat{\phi}_{S_0}$
  approaches $\hat{f}_{S_0}$.
\end{rem}

\section{Langlands functoriality}\label{s:langlands}

  Let $r:{}^L G\ra \GL_d(\C)$ be a representation of ${}^L G$. Let $\pi\in \cAR_{\disc,\chi}(G)$ be
such that with $\pi_v \in \Pi_{\disc}(\xi_v^\vee)$ for each $v|\infty$
 (recall the notation from \S\ref{sub:st-disc} and \S\ref{sub:aut-rep}).
The Langlands correspondence for $G(F_v)$ (\cite{Lan88}) associates an $L$-parameter $\varphi_{\xi^\vee_v}:W_\R\ra {}^L G$
to the $L$-packet $\Pi_{\disc}(\xi_v^\vee)$, cf. \S\ref{sub:st-disc}.
%Note that \beq\label{e:L-G-dual-G-ind} {}^L G_\infty=\ind^{\Gal(\ol{\Q}/\Q)}_{\Gal(\ol{\Q}/F)} {}^L G,\qquad
%\hat{G_\infty}=\ind^{\Gal(\ol{\Q}/\Q)}_{\Gal(\ol{\Q}/F)}\hat{G}\eeq by
%\cite[\S5]{Bor79}. In the obvious way $\hat{B}$
%  and $\hat{T}$ extends to a Borel subgroup $\hat{B}_\infty$ and
%  a maximal torus $\hat{T}_\infty$ of $\hat{G_\infty}$ in view of \eqref{e:L-G-dual-G-ind}.
 The following asserts the existence
of the functorial lift of $\pi$ under $r$ as predicted by the Langlands functoriality principle.

\begin{hypo}\label{hypo:functorial-lift}
  There exists an automorphic representation $\Pi$ of $\GL_d(\A_F)$ such that
\benu
\item $\Pi$ is isobaric,
\item $\Pi_v=r_*(\pi_v)$ (defined in \eqref{e:r_*}) when $G$, $r$ and $\pi$ are unramified at $v$,
\item $\Pi_v$ corresponds to $r\varphi_{\xi^\vee_v}$ via the Langlands correspondence for $\GL_d(F_v)$
for all $v|\infty$.
\eenu
\end{hypo}

  If $\Pi$ as above exists then it is uniquely determined by (i) and (ii) thanks to the strong multiplicity one theorem. Moreover
  \begin{lem}\label{lem:tempered}
  Hypothesis \ref{hypo:functorial-lift}.(iii) implies that
  $\Pi_v$ is tempered for all $v|\infty$.
\end{lem}

\begin{proof}
  Recall the following general fact from \cite[\S3, (vi)]{Lan88}:
  Let $\varphi$ be an $L$-parameter for a real reductive group and $\Pi(\varphi)$
  its corresponding $L$-packet. Then $\varphi$ has relatively compact image if and only if
  $\Pi(\varphi)$ contains a tempered representation if and only if $\Pi(\varphi)$
  contains only tempered representations. In our case
  this implies that $\varphi_{\xi^\vee_v}$ has relatively compact image for every $v|\infty$, and the continuity
  of $r$ shows that the image of $r\varphi_{\xi^\vee_v}$ is also relatively compact. The lemma follows.
\end{proof}

As before let $(\hat{B},\hat{T},\{X_\alpha\}_{\alpha\in \Delta^\vee})$ denote the $\Gal(\ol{F}/F)$-invariant
  splitting datum for $\hat{G}$.
  Recall that $\lambda_{\xi_v^\vee}\in X^*(\hat{T})^+$ designates the highest weight for
  $\xi_v^\vee$.
  Then $\varphi_{\xi^\vee_v}|_{W_\C}$ is described as
  $$ \varphi_{\xi^\vee_v}(z)=\left((z/\ol{z})^{\rho+\lambda_{\xi_v^\vee}},z\right)
  \in \hat{G}\times W_\C,\quad \forall z\in W_\C=\C^\times.$$
  It is possible to extend $\varphi_{\xi^\vee_v}|_{W_\C}$ to the whole of $W_\R$ but this does not concern us.
 (The interested reader may consult pp.183-184 of \cite{Kot90} for instance.)
  Let $\hat{\T}$ be a maximal torus of $\GL_d(\C)$ containing the image
  $r(\hat{T})$, and $\hat{\B}$ a Borel subgroup containing $\hat{\T}$.
  Write $r|_{\hat{G}}=\oplus_{i\in I} r_i$ as a sum of irreducible $\hat{G}$-representations.
  For each $i\in I$, denote by $\lambda(r_i)\in X^*(\hat{T})$ the $\hat{B}$-positive
  highest weight for $r_i$. Write
$\lambda(r_i)=\lambda_0(r_i)+\sum_{\alpha\in \Delta} a(r_i,\alpha)\cdot \alpha^\vee$ for
$\lambda_0(r_i)\in X_*(Z(G))_{ \Q}$ and $a(r_i,\alpha)\in \Q_{\ge 0}$.
  Put $|\lambda(r_i)|:=\sum_{\alpha\in \Delta}  a(r_i,\alpha)$ and
$$M(\xi_v):=\max_{\alpha\in \Delta} \lg \alpha,\lambda_{\xi^\vee_v}\rg,~
M(r):=\max_{i\in I} |\lambda(r_i)|.$$
Similarly define $m(\xi_v)$ and $m(r)$ by using minima in place of maxima.
We are interested in the case where $\lambda_0(r_i)$ is trivial for every $i\in I$.
This is automatically true if $Z(G)$ is finite. (Recall that we consistently assume $Z(G)=1$
in the weight aspect.)

\begin{lem}\label{l:Cpi} Suppose that $\lambda_0(r_i)$ is trivial for every $i\in I$.
  Hypothesis \ref{hypo:functorial-lift}.(iii) implies that for each $v|\infty$,
  $$ (2+m(r)m(\xi_v))^{|I|}\le C(\Pi_v)\le (3+2M(r)M(\xi_v))^{d}.$$
In particular if $Z(G)$ is finite, then the following holds for any fixed $L$-morphism $r$.
\begin{equation*}
 1+ m(\xi_v)
  \ll_r C(\Pi_v) \ll_r
  M(\xi_v)^{d}
\end{equation*}
\end{lem}

\begin{proof}
  First we recall a general fact about archimedean $L$-factors.
  Let $\varphi:W_\R\ra \GL_N(\C)$ be a tempered $L$-parameter and decompose
  $\varphi|_{W_\C}$ into $\GL_1$-parameters as $\varphi|_{W_\C}=\oplus_{k=1}^N \chi_k$.
    The archimedean $L$-factor associated with $\varphi$ may be written in the form (cf. \eqref{def:Lv-arch})
  \beq\label{e:pf-10.3-1}L(s,\varphi)=\prod_{k=1}^N \Gamma_{\R}(s-\mu_k(\varphi)).\eeq
  For each $k$ assume that $\chi_k(z)=(z/\ol{z})^{a_k}$ for some $a_k\in \frac{1}{2}\Z$.
  Then we have for every $1\le k\le N$,
  $\mu_k(\varphi)\in \frac{1}{2}\Z_{\le 0}$ and, after reordering $\mu_k(\varphi)$'s if necessary,
  \beq\label{e:pf-10.3-2}|a_k|\le |\mu_k(\varphi)| \le |a_k|+1.\eeq
  Indeed this comes from inspecting the definition of local $L$-factors as in of \cite{Tat79}*{3.1,3.3} for instance.
  (Use \cite{Tat79}*{3.1} if $a_k=0$ and \cite{Tat79}*{3.3} otherwise.)

  Returning to the setup of the lemma, we have by definition $L(s,\Pi_v)=L(s,r\varphi_{\xi_v})$. For each $i\in I$ we consider the composite complex $L$-parameter
  $$W_\C\stackrel{\varphi_{\xi_v}|_{W_\C}}{\ra} \hat{G}\times W_\C
  \stackrel{(r_i,1)}{\ra} \GL_{\dim r_i}(\C)$$
  decompose it as $\oplus_{j=1}^{\dim r_i} \chi_{i,j}$.
  We can find $a_{i,j}\in \frac{1}{2}\Z$ such that
  $\chi_{i,j}(z)=(z/\ol{z})^{a_{i,j}}$.
  For each $i$, the highest weight theory tells us that
  $a_{i,j}=\lg \rho+\lambda_{\xi^\vee_v}, \lambda(r_i)\rg\ge 0$ for one $j$
  and $|a_{i,j'}|\le a_{i,j}$ for the other $j'\neq j$.
  By \eqref{e:pf-10.3-1} and \eqref{e:pf-10.3-2}, the analytic conductor for $\Pi_v$ (introduced in \S\ref{sec:pp:Lfn})
  satisfies
  \begin{equation*}
	\begin{aligned}
  C(\Pi_v)&=\prod_{k=1}^{d}(2+|\mu_k(\Pi_v)|)\le
\prod_{i\in I} \prod_{j=1}^{\dim r_i} (3+|a_{i,j}|)\\
&\le \prod_{i\in I} (3+\lg \rho+\lambda_{\xi^\vee_v}, \lambda(r_i)\rg)^{\dim r_i}
.
\end{aligned}
\end{equation*}
Further $\lg \rho+\lambda_{\xi^\vee_v}, \lambda(r_i)\rg=
\lg \rho,\lambda(r_i)\rg+ \lg \lambda_{\xi^\vee_v}, \lambda(r_i)\rg
\le |\lambda(r_i)|+|\lambda(r_i)|M(\xi_v)\le M(r)(1+M(\xi_v))$. Hence
$$C(\Pi_v)\le \prod_{i\in I} ((3+M(r)(1+M(\xi_v)))^{\dim r_i}
=((3+M(r)(1+M(\xi_v)))^{d}.$$

  Now we establish a lower bound for $C(\Pi_v)$. For each $i$, we apply \eqref{e:pf-10.3-2}
  to the unique $j=j(i)$ such that $a_{i,j}=\lg \rho+\lambda_{\xi^\vee_v}, \lambda(r_i)\rg$.
  Then
  \begin{equation*}
	\begin{aligned}
  C(\Pi_v)&\ge \prod_{i\in I}   (2+|a_{i,j(i)}|)=
   \prod_{i\in I} (2+\lg \rho+\lambda_{\xi^\vee_v}, \lambda(r_i)\rg)\\
  &\ge  (2+m(r)(1+m(\xi_v)))^{|I|}.
\end{aligned}
\end{equation*}

\end{proof}

\section{Statistics of low-lying zeros}\label{sec:zeros}

As explained in the introduction an application of the quantitative Plancherel Theorems~\ref{t:level-varies} and~\ref{t:weight-varies} is to the study the distribution of the low-lying zeros of families of $L$-functions $\Lambda(s,\Pi)$. The purpose of this section is to state the main results and make our working hypothesis precise.

\subsection{The random matrix models}\label{sec:zeros:matrix} For the sake of completness we recall briefly the limiting $1$-level density of normalized eigenvalues. We consider the three symmetry types $\CmG(N)=SO(2N), U(N), USp(2N)$. For each integer $N\ge 1$ these groups are endowed with their Haar probability measure. For all matrices $A\in \CmG(N)$ we have a sequence $\vartheta_j=\vartheta_j(A)$ of normalized angles~\cite{book:KS}
\begin{equation}\label{def:theta}
0\le \vartheta_1 \le \vartheta_2 \le \cdots \le \vartheta_N \le N.
\end{equation}
Namely the eigenvalues of $A\in U(N)$ are given by $e(\frac{\vartheta_j}{N})=e^{2i\pi\vartheta_j/N}$. The eigenvalues of $A\in USp(2N)$ or $A\in SO(2N)$ occur in conjugate pairs and are given by $e(\pm \frac{\vartheta_j}{2N})$.

The mean spacing of the sequence~\eqref{def:theta} is one. The $1$-level density is defined by
\begin{equation*}
W_{\CmG(N)}(\Phi):=\int_{\CmG(N)} \sum_{1\le j\le N}^{} \Phi(\vartheta_j(A))dA.
\end{equation*}
The limiting density as $N\to \infty$ is given by the following (\cite{book:KS}*{Theorem~AD.2.2})
\begin{prop}\label{prop:KS}
  Let $\CmG=U,SO(even)$ or $USp$. For all Schwartz functions $\Phi$ on $\BmR_+$,
\begin{equation*}
  \lim_{N\to \infty} W_{\CmG(N)}(\Phi) = \int_{\BmR_+}
\Phi(x) W(\CmG)(x)dx,
\end{equation*}
where the density functions $W(\CmG)$ are given by~\eqref{intro:WG}.
\end{prop}

The density functions $W(\CmG)$ are defined a priori on $\BmR_+$. They are extended to $\BmR_-$ by symmetry, namely $W(\CmG)(x)=W(\CmG)(-x)$ for all $x\in \BmR$.
For a Paley--Wiener function $\Phi$ whose Fourier transform $\widehat \Phi$ has support inside $(-1,1)$, we have the identities
\begin{equation}\label{W(G)}
\int_{-\infty}^\infty
\Phi(x) W(\CmG)(x)dx=
\begin{cases}
\widehat \Phi(0) & \text{if $\CmG=U$,}\\
\widehat \Phi(0)+\Mdemi \Phi(0)& \text{if $\CmG=USp$,}\\
\widehat \Phi(0)-\Mdemi \Phi(0)& \text{if $\CmG=SO(even)$.}
\end{cases}
\end{equation}

\subsection{The $1$-level density of low-lying zeros}\label{sec:zeros:onelevel}
Consider a family $\FmF=(\FmF_k)_{k\ge 1}$ of automorphic representations of $\GL(d,\BmA_F)$.
The $1$-level density of the low-lying zeros is defined by
\begin{equation}\label{def:D1}	
D(\FmF_k;\Phi):=\frac{1}{\abs{\FmF_k}} \sum_{\Pi \in \FmF_k} \sum_j \Phi
\biggl(
\frac{\gamma_{j}(\Pi)}{2\pi}
\log C(\FmF_k)
\biggr)
\end{equation}
Here $\Phi$ is a Paley--Wiener function; we don't necessarily assume $\Phi$ to be even because the automorphic representations $\Pi\in \FmF_k$ might not be self-dual. See also the discussion at the end of~\S\ref{sec:pp:explicit}. The properties of the analytic conductor $C(\FmF_k)\ge 2$ will be described in~\S\ref{sec:zeros:cond}.

Since $\Phi$ decays rapidly at infinity, the zeros $\gamma_j(\Pi)$ of $\Lambda(s,\Pi)$ that contribute to the sum are within $O(1/\log C(\CmF_k))$ distance of the central point. Therefore the sum over $j$ only captures a few zeros for each $\Pi$. The average over the family $\Pi\in \FmF_k$ is essential to have a meaningful statistical quantity.

\subsection{Properties of families of $L$-functions}

Recall that in~\S\ref{sub:aut-families} we have defined two kinds of families $\CmF=(\CmF_k)_{k\ge 1}$ of automorphic representations on $G(\BmA_F)$. The families from Example~\ref{ex:level-varies} are varying in the level aspect: $\BmN(\Fmn_k)\to \infty$ while the families from Example~\ref{ex:wt-varies} are varying in the weight aspect: $m(\xi_k)\ra\infty$. In both cases we assume that $\phi_{S_0}\in C_c^\infty(G(F_{S_0}))$ is normalized such that
\begin{equation} \label{phiS0}
	\hat{\mu}^{\Tpl}_{S_0}(\hat\phi_{S_0})=\phi_{S_0}=1.
\end{equation}
For families in the weight aspect we assume from now the weights are bounded away from the walls. Namely we assume that we are given a fixed $\eta>0$ and that
\begin{equation}\label{dimxik}
  (\dim \xi_k)^{\eta} \le m(\xi_k),\quad \forall k.
\end{equation}

Given the continuous $L$-morphism $r:{}^LG\to \GL(d,\BmC)$ we can construct a family $\FmF=r_*\CmF$ of automorphic $L$-functions. Assuming the Langlands functoriality in the form of Hypothesis~\ref{hypo:functorial-lift}, for each $\pi\in \CmF_k$ there is a unique isobaric automorphic representation $\Pi=r_*\pi$ of $\GL(d,\BmA_F)$. We denote by $\FmF_k=r_*\CmF_k$ the corresponding family of all such $\Pi$. Recall from~\S\ref{sub:aut-rep} that $\CmF_k$ is a weighted set and that the weight of each representation $\pi$ is denoted $a_{\CmF_k}(\pi)$. The same holds for $\FmF_k$ and in particular we have
\begin{equation*}
  \abs{\FmF_k}=\abs{\CmF_k}=\sum_{\pi \in \CmF_k} a_{\CmF_k}(\pi).
\end{equation*}
We have seen in Corollary~\ref{c:estimate-F_k} that $\abs{\FmF_k}\to \infty$ as $k\to \infty$.

By definition (see~\eqref{e:a(pi)-general}), if $\pi\in \CmF_k$ then $\pi_\infty$ has the same infinitesimal character as $\xi_k^{\vee}$, i.e. $\pi \in \Pi_{disc}(\xi_k)$. If $\Pi\in \FmF_k$ then $\Pi_\infty$ corresponds to the composition $r\circ \phi_{\xi_k}$ via the Langlands correspondence for $\GL_d(F_\infty)$ (This is Hypothesis~\ref{hypo:functorial-lift}.(iii)). In particular $\Pi_\infty$ is uniquely determined by $\xi_k$ and $r$. It is identical for all $\Pi\in \FmF_k$.

It is shown in Lemma~\ref{lem:tempered} that $\Pi_\infty$ is tempered. Therefore Proposition~\ref{prop:tempered} applies and the bounds towards Ramanujan~\eqref{eq:prop:tempered} are satisfied for all $\Pi \in \FmF_k$.

To simplify notation throughout this and the next section, we use the convention of omitting the weight when writing a sum over $\FmF_k$. If $l(\Pi)$ is a quantity that depends on $\Pi\in \FmF_k$, we set
\begin{equation*}
  \sum_{\Pi\in \FmF_k} l(\Pi) := \sum_{\pi \in \CmF_k} a_{\CmF_k}(\pi) l(r_*\pi).
\end{equation*}
This convention applied in particular to~\eqref{def:D1} above.

\subsection{Occurrence of poles}\label{sec:zeros:poles}
We make the following hypothesis concerning poles of $L$-functions in our families.
\begin{hypothesis}\label{hyp:poles}
	There is $C_{pole}>0$ such that the following holds  as $k\to \infty$:
\begin{equation*}
	\#\set{\Pi\in \FmF_k,\ \Lambda(s,\Pi) \text{ has a pole}} \ll \abs{\FmF_k}^{1-C_{pole}}.
\end{equation*}
\end{hypothesis}
The hypothesis is natural because it is related to the Functoriality Hypothesis~\ref{hypo:functorial-lift} in many ways. Of course it would be difficult to define the event that \Lquote{$L(s,\Pi) \text{ has a pole}$} without assuming Hypothesis~\ref{hypo:functorial-lift}. Also when Functoriality is known unconditionally it is usually possible to establish the Hypothesis~\ref{hyp:poles} unconditionally as well. We shall return to this question in a subsequent article.

\subsection{Analytic conductors}\label{sec:zeros:cond} As in~\cite{ILS00} we define an analytic conductor $C(\FmF_k)$ associated to the family. The significance of $C(\FmF_k)$ is that each $\Pi\in \FmF_k$ have an analytic conductor $C(\Pi)$ comparable to $C(\FmF_k)$. The hypothesis in this subsection will ensure that $\log |\FmF_k| \asymp \log C(\FmF_k)$. We distinguish between families in the weight and level aspect.

\subsubsection{Weight aspect} For families in the weight aspect we set $C(\FmF_k)$ to be the analytic conductor $C(\Pi_\infty)$ of the archimedean factor $\Pi_\infty$ (recall that $\Pi_\infty$ is the same for all $\Pi\in \FmF_k$).
Then $C(\Pi)\asymp C(\FmF_k)$ for all $\Pi\in \FmF_k$.

From Corollary~\ref{c:estimate-F_k} we have that $\abs{\FmF_k}\asymp \dim \xi_k$ as $k\to \infty$.  It remains to relate the quantities $C(\FmF_k)$, $\dim \xi_k$ and $m(\xi_k)$, which is achieved in~\eqref{dimxik:2} and~\eqref{xik-CFk} below.
\begin{lem} Let $v|\infty$. Let $\xi_v$ be an irreducible finite dimensional algebraic representation of $G(F_v)$. Then
  $ m(\xi_v)^{\abs{\Phi^+}}\ll \dim \xi_v \ll M(\xi_v)^{\abs{\Phi^+}}$. Also $M(\xi_v)\ll \dim \xi_v$.
\end{lem}
\begin{proof}
  This follows from Lemma~\ref{l:dim-and-trace}. Recall the definition of $m(\xi_v)$ in~\S\ref{sub:st-disc} and $M(\xi_v)$ in~\S\ref{s:langlands}.
\end{proof}

Because of~\eqref{dimxik} and the previous lemma we have that
\begin{equation} \label{dimxik:2}
  m(\xi_k)^{\abs{\Phi^+}} \ll  \dim \xi_k \ll m(\xi_k)^{1/\eta}.
\end{equation}
From Lemma~\ref{l:Cpi} we deduce that there are positive constants $C_1,C_2$ such that
\begin{equation} \label{xik-CFk}
  m(\xi_k)^{C_1} \ll C(\FmF_k) \ll m(\xi_k)^{C_2}.
\end{equation}

\subsubsection{Level aspect} For families in the level aspect the situation is more complicated mainly because of the lack of knowledge of the local Langlands correspondence on general groups and the depth preservation under functoriality. We define $C(\FmF_k)$ by the following
\begin{equation*}
\log C(\FmF_k):=
\frac{1}{\abs{\FmF_k}} \sum_{\Pi\in\FmF_k}
\log C(\Pi),
\end{equation*}
and we introduce the following hypothesis.

\begin{hypothesis}\label{hyp:cond} There are constants $C_3, C_4>0$ such that
\begin{equation*}
  \BmN(\Fmn_k)^{C_3} \ll C(\FmF_k) \ll \BmN(\Fmn_k)^{C_4}.
\end{equation*}
\end{hypothesis}

\subsection{Main result}\label{sec:zeros:main}
We may now state our main results on low-lying zeros of the family $\FmF=r_*\CmF$. The following is a precise version of Theorem~\ref{th:low-lying} from the introduction (compare with~\eqref{W(G)}).
\begin{thm}\label{th:onelevel}
  Assume Hypothesis~\ref{hypo:functorial-lift} for individual representations as well as~\ref{hyp:poles} and~\ref{hyp:cond}. There is $0<\delta<1$ such that for all Paley--Wiener functions $\Phi$ whose Fourier transform $\widehat \Phi$ has support in $(-\delta,\delta)$ the following holds:
\begin{equation*}
\lim_{k\to \infty} D(\FmF_k,\Phi)=\widehat \Phi(0) -\frac{s(r)}{2}\Phi(0),
\end{equation*}
where $s(r)\in \set{-1,0,1}$ is the Frobenius--Schur indicator of $r:{}^LG\to \GL_d(\BmC)$.
\end{thm}

\section{Proof of Theorem~\ref{th:onelevel}}\label{sec:pf}

The method of proof of the asymptotic distribution of the $1$-level density of low-lying zeros of families of $L$-functions has appeared at many places in the literature and is by now relatively standard. However we must justify the details carefully as families of $L$-functions haven't been studied in such a general setting before. The advantage of working in that degree of generality is that we can isolate the essential mechanisms and arithmetic ingredients involved.

In order to keep the analysis concise we have introduced some technical improvements which can be helpful in other contexts:
%\item The first observation is that it is better not to expand the Euler product of the $L$-functions $L(s,\Pi)$ into Dirichlet series. We can apply the explicit formula (Proposition~\ref{prop:explicit}) in a way to preserve the Euler product structure.
 we use non-trivial bounds towards Ramanujan in a systematic way to handle ramified places;
 we clarify that it is not necessary to assume that the representation be self-dual or any other symmetry property to carry out the analysis;
 most importantly we exploit the properties of the Plancherel measure when estimating Satake parameters. Previous articles on the subject rely in a way or another on explicit Hecke relations which made the proof indirect and lengthy, although manageable for groups of low rank.

\subsection{Notation}\label{sec:pf:outline}
To formulate the main statements in an elegant way we introduce the following notation
\begin{equation}\label{def:hatL}
\widehat \CmL_{k,v}(y):=\frac{1}{\abs{\FmF_k}}
\sum_{\Pi \in \FmF_k}
\int^\infty_{-\infty} \frac{L'}{L}(\Mdemi+ix,\Pi_v)
e^{2\pi iyx} dx
,\quad v\in \CmV_F,\ y\in \BmR.
\end{equation}
We view $\widehat \CmL_{k,v}$ as a tempered distribution on $\BmR$.
Note that when $v$ is non-archimedean $\widehat \CmL_{k,v}$ is a signed measure supported on a discrete set inside $\BmR_{>0}$.

The proof of the main theorems will follow by a fine estimation of $\widehat \CmL_{k,v}(y)$ as $k\to \infty$. The uniformity in both the places $v\in \CmV_F$ and the parameter $y\in \BmR$ will play an important role. Typically $q_v$ will be as large as $C(\FmF_k)^{O(\delta)}$ and $y$ will be of size proportional to $\log C(\FmF_k)$.

The first step of the proof consists in applying the explicit formula (Proposition~\ref{prop:explicit}). There are terms coming from the poles of $L(s,\Pi)$ which we handle in~\S\ref{sec:pf:poles}. The second term in the right hand-side in Proposition~\ref{prop:explicit} is expressed in terms of the arithmetic conductor $q(\Pi)$ and will yield a positive contribution in the limit for families in the level aspect. When evaluating the $1$-level density $D(\FmF_k,\Phi)$ it remains to consider the following sum over all places
\begin{equation}\label{out:sumv}
\frac{1}{\log C(\FmF_k)}
\sumv
\left \langle
\widehat \CmL_{k,v} (y),
\widehat{\Phi}\Bigp{\frac{2\pi y}{\log C(\FmF_k)}}
\right \rangle,
\end{equation}
plus a conjugate expression, see~\S\ref{sec:pf:explicit}.

Our convention on Fourier transforms is standard. Let $\Phi$ be a Schwartz function on $\BmR$. The Fourier transform is as in~\eqref{def:fourier} and the inverse Fourier transform reads
\begin{equation*}%\label{fourier:inverse}
\Phi(x)=\int_{-\infty}^{+\infty} \widehat\Phi(y) e^{2\pi i xy} dy.
\end{equation*}
Given two Schwartz functions $\Phi$ and $\Psi$ we let
\begin{equation*}
\langle \Phi,\Psi\rangle:=\int_{-\infty}^{\infty}
\Phi(x) \Psi(x)dx.
\end{equation*}
Sometimes we use the notation $\langle \Phi(x),\Psi(x)\rangle$ to put emphasize on the variable of integration. The Plancherel formula reads
\begin{equation}\label{fourier:pl}
\langle \Phi(x),\Psi(x)\rangle=\langle \widehat\Phi(y),\widehat\Psi(-y)\rangle.
\end{equation}
We use the same conventions for tempered distributions. The Fourier transform of the pure phase function $x\mapsto e^{2i\pi ax}$ is the Dirac distribution $\delta(a)$ centered at the point $a$.

To condense notation we write \[\Psi(y):=\widehat{\Phi}\Bigp{\dfrac{2\pi y}{\log C(\FmF_k)}}\]
and shall	express our remainder terms with the quantities $\pnorm{\Psi}_\infty\le \pnorm{\widehat \Phi}_\infty$ and $\pnorm{\widehat \Psi}_1\le \pnorm{\Phi}_1$. Since $\Phi$ is fixed these are uniformly bounded, independent of $k\to \infty$.

There are different kinds of estimates depending on the nature of the place $v\in\CmV_F$. We shall distinguish the following set of places:
\begin{enumerate}[(i)]
\item the archimedean places $S_\infty$, the contribution of which is evaluated in~\S\ref{sec:pf:arch};
\item a fixed set $S_0$ of non-archimedean places. These may be thought of as the \Lquote{ramified places}. Their contribution is negligible as shown in~\S\ref{sec:pf:gn};
\item the set $\set{v\mid \Fmn_k}$ of places that divide the level. These play a role only when the level varies and we show in~\S\ref{sec:pf:error} that their contribution is negligible. We use the convention that for families in the weight aspect this set of places is empty;
\item the generic places $S_{\gen}$ which is the complement in $\CmV_F$ of the above three sets of places. This set will actually be decomposed in two parts: 
	\[S_{\gen} = S_{\main} \sqcup S_{\cut},\] 
\item where the set $S_{\cut}$ is infinite and consists of those non-archimedean places $v\in S_\gen$ such that $\tfrac{\log q_v}{2\pi}$ is large enough to be outside of the support of $\Psi$ (see~\eqref{def:Scut} below for the exact definition of $S_\cut$). Then the pairing in~\eqref{out:sumv} vanishes; 
\item the remaining set $S_\main$ is finite (but growing as $k \to \infty$). It will produce the main contribution of~\eqref{out:sumv}. For all places $v\in S_\main$, each of $G$, $r$ and $\pi$ is unramified over $F_v$. Using the notation of~\S\ref{s:Sato-Tate} we  split $S_\main$ further as the disjoint union of \[S_\main \cap \CmV_F(\theta),\ \theta \in \SmC(\Gamma_1).\]
\end{enumerate}

\subsection{Outline} For non-archimedean places $v\in S_\main$ we study in~\S\ref{sec:pf:moments} various moments of Satake parameters. The quantity $\widehat \CmL_{\Tpl,v}$ in~\eqref{CmL-pl} below will be the analogue of~\eqref{def:hatL} where the average over automorphic representations $\Pi\in \FmF_k$ gets replaced by an average of $\Pi_v$ against the unramified Plancherel measure. Our Plancherel equidistribution theorems for families (Theorems~\ref{t:level-varies} and~\ref{t:weight-varies}) imply that $\widehat \CmL_{k,v}$ is asymptotic to $\widehat \CmL_{\Tpl,v}$ as $k\to \infty$. 

It is essential that our equidistribution theorems are quantitative in a strong polynomial sense. Details on handling the remainder terms are given in~\S\ref{sec:pf:pl}, \S\ref{sec:pf:main} and~\S\ref{sec:pf:error}. 

For the main term we then need need to show the existence of the limit of
\begin{equation} \label{out:main}
  \frac{1}{\log C(\FmF_k)} \sum_{v\in S_\main}
\bigl \langle
\widehat\CmL_{\Tpl,v},\Psi
\bigr\rangle
\end{equation}
as $k\to \infty$. The evaluation of $\widehat\CmL_{\Tpl,v}$ is a nice argument in representation theory, see~\S\ref{sec:pf:M12} where we shall see clearly the role of the two  assumptions on $r$ (that $r$ is irreducible and does not factor through $W_F$). The evaluation of $\widehat\CmL_{\Tpl,v}$ can actually be quite complicated since it depends on the restriction of $r$ to subgroups $\widehat G \rtimes W_{F_v}$ for varying $v\in S_{\main}$ and on the Plancherel measure on $G(F_v)^{\wedge,\ur}$. Fortunately the expression will simplify when summing over all places $v\in S_\main$ and applying the Cebotarev density theorem (see~\S\ref{sec:pf:prime}).

The overall conclusion of the below analysis is that the limit of~\eqref{out:main} as $k\to \infty$ is equal\footnote{A quick explanation for the minus sign is as follows. A local $L$-factor is of the form $(1-\alpha q^{-s})^{-1}$ with three minus signs thus its logarithmic derivative is $-\log q \sum\limits_{\nu\ge 1}\alpha^{\nu}q^{-\nu s}$ with one minus sign.} to $-\frac{s(r)}{2}\Phi(0)$, where $s(r)$ is the Frobenius--Schur indicator of $r$. In the derivation of the one-level density there is an additional term $\widehat \Phi(0)$ which easily comes from the explicit formula and the contribution of the archimedean terms. Thereby we finish the proof of Theorem~\ref{th:onelevel}.

\subsection{Explicit formula}\label{sec:pf:explicit} We apply the explicit formula (Proposition~\ref{prop:explicit}) for each $\Pi\in \FmF_k$ to obtain
\begin{equation}\label{D1explicit}
D(\FmF_k,\Phi)=
D_{\pol}(\FmF_k,\Phi)+
\frac{\widehat \Phi(0)}{\abs{\FmF_k}}
\sum_{\Pi \in \FmF_k}^{} \frac{\log q(\Pi)}{\log C(\FmF_k)} + \sum_{v\in \CmV_F}
D_{v}(\FmF_k,\Phi)+\overline{D_{v}}(\FmF_k,\Phi).
\end{equation}
Here $D_{\pol}(\FmF_k,\Phi)$ denotes the contribution of the eventual poles. Also we have set
\begin{equation*}
D_{v}(\FmF_k,\Phi):=
\frac{1}{2\pi \abs{\FmF_k}}
\sum_{\Pi\in\FmF_k}
\int_{-\infty}^{\infty}
\frac{L'}{L}(\Mdemi+ix,\Pi_v) \Phi
\Bigp{\frac{x}{2\pi}\log C(\FmF_k)}
dx.
\end{equation*}
See also the remark in~\eqref{explicit:switch} explaining how to shift contours. The scaling factor $\tfrac{\log C(\FmF_k)}{2\pi}$ comes from~\eqref{def:D1}.

Applying Fourier duality~\eqref{fourier:pl} and the definition~\eqref{def:hatL} implies the equality\footnote{Note that the exponential in~\eqref{def:hatL} is $e^{2i\pi xy}$ with a plus sign.}
\begin{equation}\label{def:D1v}
D_{v}(\FmF_k,\Phi)=\frac{1}{\log C(\FmF_k)}
\bigl\langle  \widehat\CmL_{k,v}(y),
\widehat \Phi \Bigp{\frac{2\pi y}{\log C(\FmF_k)}}
\bigl\rangle.
\end{equation}
We have made a change of variable so as to make explicit the multiplicative factor $1/\log C(\FmF_k)$ in front of the overall sum.
Similarly we have
\begin{equation*}
\begin{aligned}
\overline{D_{v}}(\FmF_k,\Phi)
&:=
\frac{1}{2\pi \abs{\FmF_k}}
\sum_{\Pi\in\FmF_k}
\int_{-\infty}^{\infty}
\overline{
\frac{L'}{L}(\Mdemi+ix,\Pi_v) \Phi
}
\Bigp{\frac{x}{2\pi}\log C(\FmF_k)}
dx\\
&= \frac{1}{\log C(\FmF_k)}
\bigl \langle
\overline{
\widehat \CmL_{k,v}},
\widehat \Phi
\Bigp{\frac{-2\pi y}{\log C(\FmF_k)}}
\bigr \rangle.
\end{aligned}
\end{equation*}

\subsection{Contribution of the poles}\label{sec:pf:poles}
The contribution of the poles in the explicit formula above is given by
\begin{equation*}
D_{\pol}(\FmF_k,\Phi):=
\frac{1}{\abs{\FmF_k}}
\sum_{\Pi\in \FmF_k}
\sum_j
\Phi\left(
\frac{r_j(\Pi)}{2\pi} \log C(\FmF_k)
\right).
\end{equation*}
We bound the sum trivially and obtain
\begin{equation*} D_{\pol}(\FmF_k,\Phi)
\ll \frac{\#\set{\Pi\in \FmF_k,\ L(s,\Pi) \text{ has a pole}}}
{\abs{\FmF_k}} C(\FmF_k)^{O(\delta)},
\end{equation*}
where the last term comes from the exponential order of growth of $\Phi$ along the real axis because the Fourier transform $\widehat \Phi$ is supported in $(-\delta,\delta)$. %

%Now thanks to Hypothesis~\ref{hyp:poles} the right-hand side goes to zero as $k\to \infty$ for $\delta$ sufficiently small. We also use the fact that $C(\FmF_k)\ll \abs{\FmF_k}^{O(1)}$ which follows from Corollary~\ref{c:estimate-F_k} and the upper bound for $C(\FmF_k)$ from~\eqref{xik-CFk}.

\subsection{Archimedean places}\label{sec:pf:arch} In this subsection we handle the archimedean places $v\in S_\infty$. Recall from Lemma~\ref{lem:tempered} that $\Pi_{\infty}$ is tempered. In fact we shall only need here a bound towards Ramanujan  $0<\theta<\Mdemi$ as in~\S\ref{sec:pp:Lfn}.
\begin{lem}\label{prop:psi-fn} For all $\mu\in \BmC$ with $\MRe \mu \le \theta$, and all Schwartz function $\Psi$, the following holds uniformly
\begin{equation*}
\int_{-\infty}^{\infty}
\frac{\Gamma'}{\Gamma}(\Mdemi-\mu+ix)\Psi(x)dx
=\widehat \Psi(0) \log (\Mdemi-\mu)
+O(\pnorm{\Psi}_1+\pnorm{x\Psi(x)}_1)
\end{equation*}
\end{lem}
\begin{proof}
We have the following Stirling approximation for the Digamma function (traditionally denoted $\psi(z)$):
\begin{equation}\label{psi:asymp}
\frac{\Gamma'}{\Gamma}(z)= \log z + O(1)
\end{equation}
uniformly in the angular region $\abs{\arg z}\le \pi -\epsilon$, see e.g.~\cite{book:Iw}*{Appendix B}.
Since $\theta<1/2$ all points $\Mdemi-\mu+ix$ lie in the interior of the angular region and we can apply~\eqref{psi:asymp}. We note also that uniformly
\begin{equation*}
\log(\Mdemi-\mu+ix)=\log(\Mdemi-\mu) + O(\log(2+\abs{x})),
\end{equation*}
and this conclude the proof of the proposition.
\end{proof}
\begin{rem} Note that the complete asymptotic expansion actually involves the Bernoulli numbers and is of the form
\begin{equation}\label{psi:completeasymp}
\frac{\Gamma'}{\Gamma}(z)=\log z + \frac{1}{2z} - \sum_{n=1}^{N} \frac{B_{2n}}{2n z^{2n}}
+O\Bigl(
\frac{1}{z^{2N+2}}
\Bigr).
\end{equation}
\end{rem}
From~\eqref{psi:completeasymp} we have that
\begin{equation*}
\frac{\Gamma'}{\Gamma}(\sigma+it)+\frac{\Gamma'}{\Gamma}(\sigma-it)=2\frac{\Gamma'}{\Gamma}(\sigma)+O((t/\sigma)^2)
\end{equation*}
holds uniformly for $\sigma$ and $t$ real with $\sigma >0$. As in~\cite{ILS00}*{\S4} this may be used when the test function $\Psi$ is even (e.g. which is the typical case when all representations $\Pi\in \FmF$ are self-dual). We don't make this assumption and therefore use~\eqref{psi:asymp} instead.

\begin{cor} Uniformly for all archimedean places $v\in S_\infty$ and all Schwartz function $\Psi$, the following holds
\begin{equation*}
\Bigl\langle
\widehat \CmL_{k,v},
\Psi
\Bigr\rangle
=\frac{\Psi(0)}
{\abs{\FmF_k}}
\sum_{\Pi\in\FmF_k}
\sum_{i=1}^{d}
\log_v (\Mdemi -\mu_i(\Pi_v))
+ O(\bigl\lVert \widehat \Psi \bigr \rVert_1).
\end{equation*}
Here we have set $\log_v z:=\Mdemi \log z$ when $v$ is real and $\log_v z:= \log z$ when $v$ is complex.
\end{cor}
\begin{proof}
Recall the convention~\eqref{def:Lv-arch} on local $L$-factors at archimedean places $v\in S_\infty$. From Fourier duality~\eqref{fourier:pl} and the definition~\eqref{def:hatL} we have
\begin{equation*}
\Bigl\langle
\widehat \CmL_{k,v},\Psi
\Bigr\rangle
=
\frac{1}{\abs{\FmF_k}}
\sum_{\Pi \in \FmF_k}
\int^\infty_{-\infty} \frac{L'}{L}(\Mdemi+ix,\Pi_v)
\widehat \Psi(x)
dx.
\end{equation*}
Note that
\begin{equation*}
\frac{\Gamma'_v}{\Gamma_v}(s)=
\begin{cases}
-\Mdemi\log \pi +\Mdemi \frac{\Gamma'}{\Gamma}(\frac{s}{2})
,&
\text{when $v$ is real,}\\
-\log (2\pi)+\frac{\Gamma'}{\Gamma}(s)
,&
\text{when $v$ is complex.}
\end{cases}
\end{equation*}
Applying Lemma~\ref{prop:psi-fn}, the estimate in the corollary follows. Recall from Proposition~\ref{prop:tempered} that the bounds towards Ramanujan in~\S\ref{sec:pp:Lfn} apply to all $\Pi\in \FmF_k$.
 \end{proof}

We may continue the analysis of the contribution of the archimedean places to the one-level density. For $v\in S_\infty$, the local $L$-function $L(s,\Pi_v)$ are the same for all $\Pi \in \FmF_k$. We therefore conclude that
\begin{equation}\label{archterms}
\begin{aligned}
\sum_{v\in S_\infty}^{}
D_{v}(\FmF_k,\Phi)
+\overline{D_{v}(\FmF_k,\Phi)}
&=
\frac{\widehat \Phi(0)}{\log C(\FmF_k)}
\left(
\sum_{v\in S_\infty}^{}
\sum_{i=1}^{d}
2 \log_v \abs{\Mdemi-\mu_i(\Pi_v)}
+O(1)
 \right) \\
 &=
\frac{\widehat \Phi(0)}{\log C(\FmF_k)}
\left(
\sum_{v\in S_\infty}^{}
\log C(\Pi_v)
+O(1)
 \right).
\end{aligned}
\end{equation}
In the last line we used the definition of the analytic conductor at archimedean places from~\S\ref{sec:pp:Lfn}.

\subsection{Moments of Satake parameters}\label{sec:pf:moments}
Now let $v\in \CmV_F^\infty$ be a non-archimedean place. A straightforward computation shows that
\begin{equation*}
\frac{L'}{L}(s,\Pi_v)=-\log q_v
\sum_{\nu\ge 1}
\beta^{(\nu)}(\Pi_v)q_v^{-\nu s}
\end{equation*}
where
$\beta^{(\nu)}(\Pi_v):=
\alpha_1(\Pi_v)^{\nu} + \cdots + \alpha_d(\Pi_v)^{\nu}$.
Averaging over the family $\FmF$ we let
\begin{equation*}
\label{def:beta-F}
\beta^{(\nu)}_v(\FmF_k)
:=
\frac{1}{\abs{\FmF_k}}
\sum_{\Pi\in\FmF_k}
\beta^{(\nu)}(\Pi_v),\quad v\in\CmV_F^\infty,\ \nu\ge 1.
\end{equation*}
The formula~\eqref{def:hatL} becomes
\begin{equation}
\label{CmL-beta}
\widehat \CmL_{k,v}=-\log q_v \sum_{\nu\ge 1}
\beta^{(\nu)}_v(\FmF_k) q_v^{-\nu/2}\delta
\bigp{\frac{\nu}{2\pi} \log q_v},
\end{equation}
where $\delta$ is Dirac distribution (see~\S\ref{sec:pf:outline}).

Similarly for all $v\in S_\gen$ we let
\begin{equation}\label{CmL-pl}
\widehat \CmL_{\Tpl,v}:=-\log q_v
\sum_{\nu\ge 1}
\beta^{(\nu)}_{\Tpl,v} q_v^{-\nu/2}
\delta
\bigp{\frac{\nu}{2\pi} \log q_v},
\end{equation}
where the coefficients $\beta^{(\nu)}_{\Tpl,v}$ are defined locally as follows. Since $v\in S_\gen$, the group $G$ is unramified over $F_v$ and that the restriction $r|_{\widehat{G}\rtimes W_{F_v}}$ is an unramified $L$-morphism, i.e. it factors through $\widehat G \rtimes W^{\text{ur}}_{F_v}$. Recall from Section~\ref{s:Plancherel} that $\widehat \mu^{\Tpl,\ur}_v$ is the restriction of the Plancherel measure $\widehat \mu^{\Tpl}_{v}$ to $G(F_v)^{\wedge,\ur}$. Then
\begin{equation}
\label{def:beta-pl}
\beta^{(\nu)}_{\Tpl,v}:= \widehat \mu^{\Tpl,\ur}_v
\Bigl(
r^*\bigl(
Y_1^\nu+\cdots +Y_d^{\nu}
\bigr)
\Bigr),
\end{equation}
where we are using the convention in \S\ref{sub:trun-unr-Hecke} for the $L$-morphism of unramified Hecke algebras $r^*:\CmH^{\ur}(\GL_d(F_v))\to \CmH^{\ur}(G(F_v))$ and the Satake isomorphism with the polynomial algebra in $Y_1,\ldots ,Y_d$ (\S\ref{sub:case-of-GL_d}). 

The supports of both measures $\widehat \CmL_{k,v}$ and $\widehat \CmL_{\Tpl,v}$ are contained in the discrete set $\frac{\log q_v}{2\pi}\BmN_{\ge 1}$. If $q_v$ is large enough this is disjoint from the support of $\Psi$ and thus all sums over places $v\in \CmV_F$ considered below shall be finitely supported.

\subsection{General upper-bounds}\label{sec:pf:gn} Recall from Proposition~\ref{prop:tempered} that the bounds towards Ramanujan apply to every  $\Pi\in \FmF_k$. Thus for every non-archimedean $v\in \CmV_F^{\infty}$, we have the upper bound
$\abs{\alpha_i(\Pi_v)}\le q_v^{\theta}$ from which it follows that for every $\nu\ge 1$,
\begin{equation*}\label{gn-beta}
\abs{\beta^{(\nu)}_v(\FmF_k)}
\le d q_v^{\nu \theta}.
\end{equation*}
\begin{prop}\label{prop:gn-bd}
(i) For all $v\in \CmV^{\infty}_F$ and all continuous function $\Psi$,
\begin{equation*}
\left\langle \widehat \CmL_{k,v},\Psi \right\rangle
\ll q_v^{\theta-\Mdemi}\log q_v \pnorm{\Psi}_{\infty}.
\end{equation*}

(ii) For all $v\in S_{\gen}$ and all continuous function $\Psi$,
\begin{equation*}% \label{gn-pl}
\left\langle \widehat \CmL_{\Tpl,v},\Psi \right\rangle
\ll q_v^{-\Mdemi}\log q_v \pnorm{\Psi}_{\infty}.
\end{equation*}

\end{prop}
\begin{proof} (i) Inserting the above upper bound into~\eqref{CmL-beta} we have
\begin{equation*}
\left\langle \widehat \CmL_{k,v},\Psi \right\rangle
\ll \log q_v \sum_{\nu\ge 1} q_v^{\nu(\theta-1/2)} \abs{\Psi(\frac{\nu}{2\pi}\log q_v)}.
\end{equation*}
Because $0<\theta<\Mdemi$, the conclusion easily follows.

(ii) The Plancherel measure $\widehat \mu^{\Tpl,\ur}$ has total mass one and is supported on the tempered spectrum $\widehat G(F_v)^{\wedge,\ur,\temp}$ (see \S\ref{s:unramified-spectrum}).
We deduce similarly that for every $\nu \ge 1$,
\begin{equation}\label{gn-pl-beta}
\abs{\beta_{\Tpl,v}^{(\nu)}}\le d
\end{equation}
Indeed the image of any unramified $L$-parameter $r\circ \varphi:W_{F_v}^{\ur} \to \GL_d(\BmC)$ is bounded and all Frobenius eigenvalues have therefore absolute value one.  
\end{proof}

\subsection{Plancherel equidistribution}\label{sec:pf:pl} We are in position to apply the Plancherel equidistribution theorem for families established in Section~\ref{s:aut-Plan-theorem}. We shall derive uniform asymptotics as $k\to \infty$ for $\beta^{(\nu)}_{v}(\FmF_k)$.
\begin{prop}\label{prop:F-pl} There exist constants $C_5>0$ and $A_7,B_9<\infty$ such that the following holds uniformly on $\nu \ge 1$ and $v\in S_\gen$
\begin{equation} \label{main-F-pl}
	\beta^{(\nu)}_{v}(\FmF_k)=(1+o(1)) \beta^{(\nu)}_{\mathrm{pl},v}
+O(q_v^{A_7+B_9\nu} C(\FmF_k)^{-C_5}).
\end{equation}
%For families in the level aspect we assume that $q_v \ll_\epsilon C(\FmF_k)^{\epsilon}$ for some positive constant $\epsilon>0$.
\end{prop}

\begin{proof} Let $S_0$ be a sufficiently large set of non-archimedean places which contains all places $v\in\CmV_F^\infty$ where $G$ is ramified and where $r$ is ramified. Let $S_1:=\set{v}$. We set
\begin{equation*}
\widehat \phi_v :=  r^* (Y_1^\nu+\cdots + Y_d^{\nu}) \in \CmH^{\ur}(G(F_v)).
\end{equation*}
The notation for the Satake isomorphism is as in~\S\ref{sub:Satake-trans} and \S\ref{sub:L-mor-unr-Hecke}. By definition we have that $\beta_{\Tpl,v}^{(\nu)}=\widehat \mu^{\Tpl}_{v}(\widehat \phi_v)$. 
Thanks to Lemma~\ref{l:bound-degree-Satake} we have that $\phi_{v}\in \CmH^{\ur}(G(F_v))^{\le \beta \nu}$ and $\abs{\phi_v}\ll 1$. Thus we are in position to apply the respective Theorems~\ref{t:level-varies} (in the level aspect) and~\ref{t:weight-varies} (in the weight aspect).

% Recall that $\hat \phi_{S_0}$ has been fixed earlier (independent of $k$) and normalized by~\eqref{phiS0}.

Using the notation of \S\ref{sub:aut-rep}, we have by construction
\begin{equation*} \begin{aligned}
  \beta_v^{(\nu)}(\FmF_k)&= \frac{1}{\abs{\CmF_k}} \sum_{\pi\in \CmF_k} a_{\CmF_k}(\pi)\widehat \phi_v(\pi_v)\\
  &=\widehat \mu^{\text{count}}_{\CmF_k,v}(\widehat \phi_v)
  =\frac{\tau'(G)\dim \xi_k}{\mu^{\text{can}}(U^{S,\infty}_k) \abs{\CmF_k}}
  \widehat \mu_{\CmF_k,v}(\widehat \phi_v).
\end{aligned}
\end{equation*}
The Corollary~\ref{c:estimate-F_k} shows that $ \dfrac{\tau'(G)\dim \xi_k}{\mu^{\text{can}}(U^{S,\infty}_k) \abs{\CmF_k}}=1+o(1)$ as $k\to \infty$.
We shall now distinguish between the two types of families.

For families in the level aspect, the assumption (ii) in Theorem~\ref{t:level-varies} is satisfied because $v\nmid \Fmn_k$ for all $k$ and all $v\in S_\gen$. If the assumption (i) in Theorem~\ref{t:level-varies} is not satisfied, then
\[
C(\FmF_k)^{\frac{1}{C_4}} \ll
N(\Fmn_k)
< c_{\Xi} q_v^{B_\xi m\kappa}
\]
where the first inequality comes from Hypothesis~\ref{hyp:cond}. Thus the error term in \eqref{main-F-pl} dominates if $A_7$ is chosen large enough.
If the assumption (i) in Theorem~\ref{t:level-varies} is satisfied, then from \eqref{e:t:level} we obtain the main term in~\eqref{main-F-pl} and the error term $O(q_v^{A_{\level}+B_{\level}\kappa} \N(\fkn_k)^{-C_{\level}})$. By Hypothesis~\ref{hyp:cond} we may then choose then $C_5:=C_{\level}/C_4$ to conclude the proof of~\eqref{main-F-pl}.

For families in the weight aspect the assumptions in Theorem~\ref{t:weight-varies} are always satisfied. This yields the main term in~\eqref{main-F-pl} with the error term $O(q_{S_1}^{A_{\weight}+B_{\weight}\kappa} m(\xi_k)^{-C_{\weight}})$. By the estimate~\eqref{xik-CFk} we may choose $C_5:=C_{\weight}/C_2$ to conclude the proof of~\eqref{main-F-pl}.
\end{proof}

\subsection{Main term}\label{sec:pf:main}
We deduce from Proposition~\ref{prop:F-pl} the following estimate for $\widehat \CmL_{k,v}$.
\begin{prop}\label{prop:S1} For all $A>0$ there is $A_8>0$ such that the following holds uniformly for all $v\in S_\gen$ and all continuous function $\Psi$:
\begin{equation*}
\bigl\langle
\widehat\CmL_{k,v},\Psi
\bigr\rangle
=
\bigl\langle
\widehat\CmL_{\mathrm{pl},v},\Psi
\bigr\rangle
(1+o(1))
+O(q_v^{A_8} C(\FmF_k)^{-C_5}
\pnorm{\Psi}_{\infty})
+O(q_v^{-A}\pnorm{\Psi}_\infty),
\end{equation*}
\end{prop}

\begin{proof}
Let $\kappa$ be a large enough integer. We apply the bounds towards Ramanujan in the form~\eqref{gn-beta} to those term in~\eqref{CmL-beta} with $\nu > \kappa$. The contribution of those terms to $\bigl\langle
\widehat\CmL_{k,v},\Psi
\bigr\rangle$ is uniformly bounded by
\begin{equation*}
\ll q_v^{\kappa(\theta-\Mdemi)}\pnorm{\Psi}_{\infty}.
\end{equation*}
We have that $A:=\kappa(\Mdemi-\theta)$ may be chosen as large as we want since $\theta < \Mdemi$ is fixed and $\kappa$ is arbitrary large.

For those terms in~\eqref{CmL-beta} with $\nu \le \kappa$ we apply~\eqref{main-F-pl}. Their contribution to $\bigl\langle
\widehat\CmL_{k,v},\Psi
\bigr\rangle$ is equal to
\begin{equation*}
-\log q_v \sum_{1\le \nu\le \kappa}
\beta^{(\nu)}_{\Tpl,v} q_v^{-\nu/2}
\Psi(\frac{\nu}{2\pi}\log q_v) + O(q_v^{A_7+B_9\kappa}\pnorm{\Psi}_{\infty} C(\FmF_k)^{-C_5}).
\end{equation*}

The next step is now to complete the $\nu$-sum. Applying~\eqref{gn-beta} we see that the terms $\nu >\kappa$ yield another remainder term of the form $q_v^{-A}\pnorm{\Psi}_{\infty}$ with $A$ arbitrary large (again depending on $\kappa$).
\end{proof}

\subsection{Handling remainder terms}\label{sec:pf:error} In this subsection we handle the various remainder terms and show that they don't contribute to $D(\FmF_k,\Phi)$ in the limit when $k\to \infty$. We shall apply the above estimates to the function
\begin{equation}\label{Psi-Phi}
\Psi(y):= \widehat \Phi\left(\frac{2\pi y}{\log C(\FmF_k)}\right),\quad y\in \BmR.
\end{equation}
% We can see that both $\pnorm{\Psi}_{\infty}$ and $\bigl\lVert \widehat \Psi \bigr \rVert_1$ only depend on $\Phi$ which has been fixed. As a consequence these norms are uniformly bounded.
Recall from~\eqref{def:D1v} that $D_v(\FmF_k,\Phi)=\langle \widehat\CmL_{k,v},\Psi \rangle/ \log C(\FmF_k)$.

For archimedean places $v\in S_\infty$ we encountered in~\S\ref{sec:pf:arch} the remainder term $O(\bigl\lVert \widehat \Psi \bigr \rVert_1)$. Because $\log C(\FmF_k)\to \infty$, this remainder term is negligible for $D_v(\FmF_k,\Phi)$ as $k\to \infty$.

For the non-archimedean places $v$ such that $v\mid \Fmn_k$ or $v\in S_0$ we use the general bounds of~\S\ref{sec:pf:gn} that imply
\begin{equation}\label{remainS0}
\sum_{v\in S_0,\text{ and $v\mid \Fmn_k$}}
\Bigl|\langle
\widehat \CmL_{k,v},\Psi
\rangle\Bigr|
\ll
\sum_{v\in S_0,\text{ and $v\mid \Fmn_k$}}
q_v^{\theta-\Mdemi} \log q_v \pnorm{\Psi}_\infty
\ll
1+
\#\set{v\mid \Fmn_k}.
\end{equation}
In the last inequality we used the fact that $S_0$ is fixed and that $\theta < \Mdemi$. Again the multiplication by $1/\log C(\FmF_k)$ shows that these terms are negligible for $D(\FmF_k,\Phi)$ as $k\to \infty$. Indeed it is easy to verify that
\begin{equation}\label{remainlevel}
\#\set{v\mid \Fmn_k} = o(\log \BmN(\Fmn_k)),\qtext{as $k\to \infty$,}
\end{equation}
and we conclude using Hypothesis~\ref{hyp:cond} that this is $o(\log C(\FmF_k))$.

We partition the set of generic non-archimedean places $S_\gen$ into two disjoint sets $S_\main$ and $S_\cut$ where
\begin{equation}\label{def:Scut}
S_\cut := \set{v\in S_\gen:\ q_v>C(\FmF_k)}.
\end{equation}
Since the support of $\widehat \Phi$ is included in $(-\delta,\delta)$ we know that $\Psi(\nu \log q_v/ 2\pi)$ vanishes for all $v\in S_\cut$ and $\nu \ge 1$.

For the generic places $v\in S_\main$ we use the estimate in Proposition~\ref{prop:S1}. The second remainder term yields
\begin{equation}\label{secondremain}
	\frac{1}{\log C(\FmF_k)}
\sum_{v\in S_\main} q_v^{-A} \pnorm{\Psi}_\infty = O(\frac{1}{\log C(\FmF_k)} ).
\end{equation}
This is again negligible as $k\to \infty$. The first remainder term in Proposition~\ref{prop:S1} is negligible as well because
\begin{equation} \label{firstremain}
  \sum_{v\in S_\main}^{}  q_v^{A_8}  \pnorm{\Psi}_{\infty} C(\FmF_k)^{-C_5} \ll
C(\FmF_k)^{\delta(A_8+1) - C_5}
\end{equation}
and $\delta$ is chosen small small enough such that $\delta(A_8+1)<C_5$.

Finally we show that the contribution to $\langle \widehat\CmL_{\Tpl,v} ,\Psi \rangle/ \log C(\FmF_k)$ of the higher moments $\nu\ge 3$ is negligible. Because of the definition~\eqref{CmL-pl} of $\widehat\CmL_{\Tpl,v}$ and the bound~\eqref{gn-pl-beta} for $\beta_{\Tpl,v}^{(\nu)}$, the contribution of the higher moments is uniformly bounded by
\begin{equation}\label{highermom}
\sum_{v\in S_\main}
\log q_v
\sum_{\nu\ge 3}^{}
 q_v^{-\nu/2}
\Psi\left( \frac{\nu \log q_v}{2\pi} \right)
\ll \pnorm{\Psi}_\infty
\sum_{v\in \CmV_F^\infty}^{}
q_v^{-3/2}\log q_v \ll 1.
\end{equation}

Therefore we can write the main contribution to $D(\FmF_k,\Phi)$ as
\begin{equation*}% \label{M1M2}
\frac{1}{\log C(\FmF_k)} \sum_{v\in S_\main}
\bigl\langle
\widehat\CmL_{\Tpl,v},\Psi
\bigr\rangle
=
M^{(1)}+M^{(2)}+O(\frac{1}{\log C(\FmF_k)})
\end{equation*}
where for $\nu=1,2$ we define
\begin{equation}\label{def:Mv}
	M^{(\nu)}:=
	- \sum_{v\in S_\main}^{}
\frac{ \log q_v}{\log C(\FmF_k)}
q_v^{-\nu/2}
\beta_{\Tpl,v}^{(\nu)}
\widehat \Phi\left( \frac{\nu \log q_v}{\log C(\FmF_k)} \right).
\end{equation}
(recall the relation~\eqref{Psi-Phi} between $\Phi$ and $\Psi$)

\subsection{Sum over primes}\label{sec:pf:prime} It remains to estimate the above terms~\eqref{def:Mv} which consist of sums over the places $v\in S_\main$. 
We shall use the prime number theorem and the Cebotarev equidistribution theorem which we now proceed to recall, following e.g.~\cite{book:Narkiewicz}*{Chap.~7}. Let $E/F$ be a finite Galois extension with Galois group $\Gamma=\Gal(E/F)$. For all conjugacy class $\theta \in \SmC(\Gamma)$, recall that $\CmV_F(\theta)$ consists of those unramified places $v\in \CmV_F^\infty$ such that $\Fr_v\in \theta$.
\begin{prop}
 (Prime number theorem) Notation being as above,
\begin{equation*}
\#\set{v\in \CmV_F^{\infty},\ q_v \le x} \sim \frac{x}{\log x}
,\qtext{as $x\to \infty$.}
\end{equation*}

(Cebotarev equidistribution theorem) For any $\theta\in \SmC(\Gamma) $, 
\begin{equation*}
\#\set{v\in \CmV_F(\theta),\ q_v\le x} \sim \frac{x}{\log x}\times \frac{\abs{\theta}}{\abs{\Gamma}}
,\qtext{as $x\to \infty$.}
\end{equation*}

\end{prop}

As a corollary we deduce the following estimate for any  $\theta\in \SmC(\Gamma) $
\begin{equation}\label{pnt}
\sum_{v\in \CmV_F(\theta)}
\frac{\log q_v}{\log C(\FmF_k)}
q_v^{-1}
\widehat \Phi(\frac{\log q_v}{\log C(\FmF_k)})
=(\frac{\abs{\theta}}{\abs{\Gamma}}+o(1))
\int_{0}^{\infty}
\widehat \Phi(y) dy
,\qtext{as $k\to \infty$.}
\end{equation}
This estimate will be used below to evaluate $M^{(2)}$.
%Note that if we replace $\log q_v$ by $2\log q_v$, or any other constant multiple, the asymptotic remains the same. 
Note that if we replace $\log q_v$ by $-\log q_v$, the same estimate holds with the integral on the right-hand side ranging from $-\infty$ to $0$.
 We shall use this observation below when adding the contribution of $\overline{D_{v}}(\FmF_k,\Phi)$ which will then produce produce the integral $\int\limits_{-\infty}^{\infty} \widehat \Phi(y)dy = \Phi(0)$.
 \subsection{Computing the moments $M^{(1)}$ and $M^{(2)}$}\label{sec:pf:M12}   Recall that by assumption \[r:\widehat G \rtimes \Gal(\bar F/F)={}^LG \to \GL_d(\BmC)\] is irreducible and does not factor through $\Gal(\bar F/F)$.
% The following will play an important role to go from global to local $L$-morphism.
\begin{lem}\label{lem:rhatG}
The restriction $r|_{\widehat G}$ does not contain the trivial representation.
\end{lem}
\begin{proof}   If there were a non-zero vector in $\BmC^d$ invariant by $r(\widehat G)$ then all its translates by $\Gal(\bar F/F)$ would still be invariant because $\widehat G$ is a normal subgroup of ${}^LG$. Because $r$ is irreducible these translates generate $\BmC^d$ and thus the restriction $r|_{\widehat G}$ would be trivial%
  \footnote{In the sense that $r|_{\widehat G}$ would be a direct sum of trivial representations. In the sequel we use this slight abuse of notation when saying that a representation is \Lquote{trivial}.}
which yields a contradiction. For an extension of this argument see e.g.~\cite{book:serre:rep}*{Prop.~24, \S~I.8.1}.
\end{proof}
 Since $v\in S_\main$, the group $G$ is unramified over $F_v$ and the restriction $r|_{\widehat{G}\rtimes W_{F_v}}$ is an unramified $L$-morphism which factors through $\widehat{G} \rtimes W^{\ur}_{F_v}$. Note that this restriction might reducible in general.

Let $A$ is a maximal $F_v$-split torus and $\Omega_{F_v}$ the $F_v$-rational Weyl group for $(G(F_v),A)$. Recall from~Section~\ref{s:Satake-trans} the Satake isomorphism
\begin{equation*}
  \CmS:\CmH^{\ur}(G) \xrightarrow{\sim} \CmH^{\ur}(A)^{\Omega_{F_v}}.
\end{equation*}
For the group $\GL_d$ the right hand-side is identified with $\BmC[Y^\pm_1,\cdots Y^\pm_d]^{\FmS_d}$.  
We recall the morphism of unramified Hecke algebras $r^*:\CmH^{\ur}(\GL_d)\to \CmH^\ur(G(F_v))$ and the test functions:
\begin{equation*}
  \phi_v^{(\nu)}:=r^*(Y_1^{\nu}+\cdots+Y_d^{\nu}) \in \CmH^{\ur}(G(F_v)).
\end{equation*}
In view of~\eqref{def:beta-pl} we have
\begin{equation*}
  \beta_{\Tpl,v}^{(\nu)} = \widehat \mu^{\Tpl}_{v}(\widehat \phi^{(\nu)}_v) = \phi^{(\nu)}_v(1).
\end{equation*}

\begin{prop}\label{prop:pf:beta1}
  The following estimate holds uniformly for all $v\in S_\main$
\begin{equation*}
	\beta^{(1)}_{\mathrm{pl},v}=O(q_v^{-1}).
\end{equation*}

\end{prop}

\begin{proof} We decompose the restriction of $r$ to $\widehat{G} \rtimes W^{\ur}_{F_v}$ into a direct sum of irreducible $\oplus_i r_i$. By Lemma~\ref{lem:rhatG} each $r_i|_{\widehat G}$ does not contain the trivial representation. In particular each $r_i$ does not factor through $W_{F_v}^{\ur}$.

  We can now apply Lemma~\ref{l:bound-1st-moment} which shows that
\begin{equation*}
  \abs{\phi^{(1)}_v(1)} \le q_v^{-1} \abs{\Omega_{F_v}} \max\limits_{w\in \Omega_{F_v}}p(\lambda_i \star_{w} 0) .
\end{equation*}
Here $\lambda_i$ is as defined in~\S\ref{sub:Satake-trans}. The two terms $\abs{\Omega_{F_v}}$ and $p(\lambda_i \star_{w} 0)$ are easily seen to be bounded (uniformly with respect to $v\in S_\main$).
\end{proof}

As a consequence  of  Proposition~\ref{prop:pf:beta1} we deduce that
\begin{equation} \label{M1}
	M^{(1)} = O(\frac{1}{\log C(\FmF_k)})
\end{equation}
because  the summand over $v$ in~\eqref{def:Mv} is dominated by $q_v^{-3/2}$.

For the second moment $M^{(2)}$ we shall need a more refined estimate. Recall the finite extension $F_1/F$ from Section~\ref{s:Sato-Tate}. We also choose a finite extension $F_2/F_1$ such that $r$ factors through $\widehat G \rtimes \Gal(F_2/F)$. Let $\Gamma_2:=\Gal(F_2/F)$ and denote by $\SmC(\Gamma_2)$ the set of conjugacy classes in $\Gamma_2$.
\begin{prop}\label{prop:stheta}
	(i) For all $\theta\in \SmC(\Gamma_2)$ there is an algebraic integer $s(r,\theta)$ such that uniformly for all $v\in S_\main$,
\begin{equation}\label{eq:prop:stheta}
	\beta^{(2)}_{\mathrm{pl},v}=s(r,[\Fr_v]) + O(q_v^{-1}).
\end{equation}
Here $[\Fr_v]\in \SmC(\Gamma_2)$ is the conjugacy class of $\Fr_v$ in $\Gamma_2$.

  (ii) The following identity holds
\begin{equation*}
  s(r)= \sum_{\theta \in \SmC(\Gamma_2)} \frac{\abs{\theta}}{\abs{\Gamma_2}} s(r,\theta)
\end{equation*}
where $s(r)\in \{-1,0,1\}$ is the Frobenius-Schur indicator of $r$. 
\end{prop}
\begin{proof}
  (i) We proceed in way similar to the proof of Proposition~\ref{prop:pf:beta1} above. We shall give an explicit formula~\eqref{pf:stheta} for $s(r,\theta)$.

  We decompose $\MSym^2 r=\oplus \rho^+_i$ (resp. $\bigwedge^2 r=\oplus \rho^-_i$) into a direct sum of irreducible representation of $\widehat G \rtimes \Gal(E/F)$. Then we can decompose for each $i$ the restriction $\rho^+_i|_{\widehat G\rtimes W^{\ur}_{F_v}}=\oplus_j \rho^+_{ij}$ as a direct sum of irreducible representations of $\widehat G\rtimes W^{\ur}_{F_v}$. Similarly we let $\rho^-_i|_{\widehat G\rtimes W^{\ur}_{F_v}}=\oplus_j \rho^-_{ij}$.

 Let  $ \widehat \phi^+_{ij} := (\rho^+_{ij})^* (Y^2_1+\ldots +Y^2_{d_i})$
	and similarly for $\phi^{-}_{ij}$. Then it is easily verified that
\begin{equation*}
	  \widehat \phi^{(2)}_v=\sum_{ij} \widehat \phi^+_{ij}
	  - \sum_{ij} \widehat \phi^-_{ij}
\end{equation*}

We now distinguish two cases. In the first case, $i$ is such that $\rho^+_i$ does not factor through $\Gal(E/F)$. Then by Lemma~\ref{lem:rhatG} the restriction $\rho^+_i|_{\widehat G}$ does not contain the trivial representation. Thus for all $j$, $\rho^+_{ij}|_{\widehat G}$ does not contain the trivial representation. In particular $\rho^+_{ij}$ does not factor through $W^{\ur}_{F_v}$. By Lemma~\ref{l:bound-1st-moment} we deduce that $\phi^+_{ij}(1) = O(q_v^{-1})$. These representations $\rho^+_i$ only contribute to the error term in~\eqref{eq:prop:stheta}.

In the second case, $i$ is such that $\rho^+_i$ does factor through $\Gal(E/F)$. Then for all $j$, $\rho^+_{ij}$ factors through $W_F^{\ur}$ (in particular it is $1$-dimensional). We have that $\widehat\phi^+_{ij}(1)=\rho^+_{ij}(\Fr_v)$. By linearity we deduce that $\sum_{j} \phi^+_{ij}(1)=\Mtr \rho^+_i(\Fr_v)$. This is an algebraic integer which depends only on the conjugacy class of $\Fr_v$ in $\Gamma_2$.

We proceed in the same way for $\phi_{ij}^-$. We deduce that~\eqref{eq:prop:stheta} holds with $\theta=[\Fr_v]$ and $s(r,\theta)=s^+(r,\theta)-s^-(r,\theta)$, where
\begin{equation}\label{pf:stheta}
  s^+(r,\theta):= \sum_{
  \substack{
  \rho^+_i \text{ factors}\\
  \text{through $\Gal(E/F)$}}
  }
  \Mtr \rho^+_i(\theta)
\end{equation}
and similarly for the definition of $s^-(r,\theta)$. This concludes the proof of assertion (i).

(ii) By orthogonality of characters we have for each $i$ such that $\rho_i^+$ factors through $\Gal(E/F)$,
\begin{equation*}
  \sum_{\theta \in \SmC(\Gamma_2)}^{}
  \frac{\abs{\theta}}{\abs{\Gamma_2}} \Mtr \rho_i^+(\theta) = \langle\Mun,\rho_i^+\rangle=
  \begin{cases}
	1, & \text{if $\rho_i^+=\Mun$}\\
	0, & \text{otherwise}.
  \end{cases}
\end{equation*}
We deduce that
\begin{equation*}
   \sum_{\theta \in \SmC(\Gamma_2)}^{}
  \frac{\abs{\theta}}{\abs{\Gamma_2}} s^+(r,\theta) =\langle\Mun,\MSym^2 r\rangle,
\end{equation*}
the multiplicity of the trivial representation $\Mun$ in $\MSym^2 r$ (as a representation of $\widehat G \rtimes \Gal(E/F)$). The same identity holds for $s^-(r,\theta)$ and $\bigwedge^2 r$. From the definition of the Frobenius-Schur indicator $s(r)$ in \S\ref{sec:b:fs} we conclude the proof of the proposition.
 \end{proof}

 As a corollary we have the following  estimate for the second moment:
\begin{equation*}\label{pf:M2estimate}
		M^{(2)} 
	= - \sum_{\theta \in \SmC(\Gamma_2)}
	s(r,\theta)
\sum_{v\in S_\main\cap \CmV_F(\theta)}
\frac{\log q_v}{\log C(\FmF_k)} q_v^{-1} \widehat \Phi\left( \frac{\nu \log q_v}{\log C(\FmF_k)} \right) + O(\frac{1}{\log C(\FmF_k)}).
\end{equation*}
 We can extend the sum to $v\in S_\gen \cap \CmV_F(\theta)$ because 
\begin{equation*}
\sum_{v\in \CmV_F^{\infty}-S_\main}
\frac{\log q_v}{\log C(\FmF_k)} q_v^{-1} \ll \frac{\log \log C(\FmF_k)}{\log C(\FmF_k)}=o(1)
\end{equation*}
uniformly as $k\to \infty$.  Applying the Cebotarev equidistribution theorem we deduce that
\begin{equation} \label{M2}
  \begin{aligned}
		M^{(2)}&= - \sum_{\theta \in \SmC(\Gamma_2)}
	s(r,\theta) (\frac{\abs{\theta}}{\abs{\Gamma_2}}+o(1))
\Mdemi \int_{0}^\infty
\widehat \Phi(y)dy\\
&=(-\frac{s(r)}{2}+o(1)) \int_{0}^\infty
\widehat \Phi(y)dy.
  \end{aligned}
\end{equation}
The last line follows from Proposition~\ref{prop:stheta}.(ii) above.

\subsection{Conclusion}\label{sec:pf:ccl} We now gather all the estimates and conclude the proof of Theorem~\ref{th:onelevel}. The explicit formula~\eqref{D1explicit} expresses $D(\FmF_k,\Phi)$ as the sum of four terms. The term $D_{\pol}(\FmF_k,\Phi)$ goes to zero as $k\to \infty$ as consequence of Hypothesis~\ref{hyp:poles}, see~\S\ref{sec:pf:poles}.

The archimedean terms are evaluated in~\eqref{archterms}. In addition with the second term in~\eqref{D1explicit} which involves $\log q(\Pi)$, these contribute
\begin{equation*}
\frac{\widehat \Phi(0)}{\abs{\FmF_k}} \sum_{\Pi\in\FmF_k}
\frac{\log C(\Pi)}{\log C(\FmF_k)} + o(1).
\end{equation*}
This is equal to $\widehat \Phi(0)+o(1)$ (using the Hypothesis~\ref{hyp:cond} for families in the level aspect).

We now turn to the non-archimedean contribution. The places $v\in S_0$ and $v\mid \Fmn_k$ are negligible thanks to~\eqref{remainS0} and~\eqref{remainlevel}, respectively.

It remains the non-archimedean places $v\in S_\gen=S_\main\sqcup S_\cut$. The contribution from $v\in S_\cut$ is zero because the support of $\widehat \Phi$ is included in $(-\delta,\delta)$, see~\eqref{def:Scut}.

For each $v\in S_\main$ we apply Proposition~\ref{prop:S1}. The sum over $v\in S_\main$ of the remainder terms is shown to be negligible in~\eqref{secondremain} and~\eqref{firstremain}. For the main term the estimate~\eqref{highermom} shows that the contribution of the higher moments is negligible. It remains the two terms $M^{(1)}$ and $M^{(2)}$ as defined in~\eqref{def:Mv}.

The asymptotic of $M^{(1)}$ and $M^{(2)}$ are given in~\eqref{M1} and~\eqref{M2} respectively. There is a similar contribution from the conjugate $\overline{D_{v}}(\FmF_k,\Phi)$. Overall this yields
\begin{equation*}
\begin{aligned}
\sum_{v\in S_\main}^{} D_{v}(\FmF_k,\Phi)+
\overline{D_{v}}(\FmF_k,\Phi)
&= -\frac{s(r)}{2} \int_{0}^{\infty} \widehat \Phi(y)dy
-\frac{s(r)}{2} \int_{-\infty}^{0} \widehat \Phi(y) dy
+o(1)\\
&= -\frac{s(r)}{2} \Phi(0)+o(1).
\end{aligned}
\end{equation*}

We can now conclude that
\begin{equation*}
\lim_{k\to \infty} D(\FmF_k,\Phi)=\widehat \Phi(0) - \frac{s(r)}{2} \Phi(0).
\end{equation*}
This is the statement of Theorem~\ref{th:onelevel}. \qed

%\bibliographystyle{plain}
%\bibliography{/home/templier/all,bib}
%\end{document}
%----------------------------------------
\appendix

\section{By Robert Kottwitz}\label{s:app:Kottwitz}

  Let $F$ be a finite extension of $\Q_p$, and $G$ a connected
  reductive group over $F$. For each semisimple $\gamma\in G(F)$, define
  a positive real number
  \begin{equation}\label{e:D^G}D^G(\gamma):=\left|\det(1-Ad(\gamma)|_{\Lie G/\Lie G^0_\gamma})\right|_v=
  \prod_{\alpha\in \Phi\atop \alpha(\gamma)\neq 1} |1-\alpha(\gamma)|_v.\end{equation}
  (In particular if $\gamma$ belongs to the center of $G(F)$ then $D^G(\gamma)=1$.)
  We equip $G(F)$, as well as $I_\gamma(F)$ (the connected centralizer of $\gamma$)
  for each semisimple $\gamma\in G(F)$, with the Haar measures as in \cite[\S4]{Gro97}.
  The quotient measure on $I_\gamma(F)\bs G(F)$ is used to define
  the orbital integral $O_\gamma(f)$.

\begin{thm}\label{t:appendix1}
  For each $f\in C^\infty_c(G(F))$, there exists a constant $c(f)>0$ such that
  for all semisimple $\gamma\in G(F)$,
  $$ |O_\gamma(f)|\le c(f)\cdot D^G(\gamma)^{-1/2}.$$
\end{thm}

\begin{proof}
There are only finitely many $G(F)$-conjugacy classes of maximal $F$-tori in
$G$,  so in
proving the theorem we can fix a maximal $F$-torus $T$ in $G$ and
restrict attention to elements $\gamma$ lying in $T(F)$. Then we must show
that the function  $\gamma \mapsto D^G(\gamma)^{1/2}O_\gamma(f)$ is
bounded on $ T(F)$. Harish-Chandra proved that the
restriction of this function to the set of regular elements in $T(F)$ is
bounded, so we just need to check that his methods can be used to treat
 singular $\gamma$ as well.

Since the
function $\gamma \mapsto D^G(\gamma)^{1/2}O_\gamma(f)$ is
compactly supported on $T(F)$, it is enough to show that it
is also locally bounded. Harish-Chandra's method of  semisimple descent
reduces us to proving local boundedness in a neighborhood of $1 \in T(F)$,
and then the exponential map reduces us to the analogous problem on
the Lie algebra $\mathfrak g$ of $G$. The remainder of this appendix
handles $\mathfrak g$, the main result being Theorem \ref{bdoi}.
\end{proof}

\subsection{Notation pertaining to the Lie algebra version of the problem}

We  write $\mathfrak t$ for the
Lie algebra of $T$. We write $R$ for the (absolute) root system of $T$ in
$G$. We often write
$G$ for the group of
$F$-points of
$G$, etc. We will follow closely the exposition of Harish-Chandra's work
given in \cite{Kot05}. Most of the proofs are
just the same as the ones there and will therefore  be omitted. (Instead of
a proof, the reader will find the words ``same as usual.'') However, a
couple of additional ingredients  will be needed; these are  simple
adaptations of  ideas in  J.~Sparling's article \cite{sparling}.

\subsection{Orbital integrals $O_X$ for  $X \in \mathfrak t$}
Let $X \in \mathfrak t$. The centralizer of $X$ in $G$ is a connected
reductive $F$-subgroup of $G$ that we will denote by $M_X$. (The reason for
using the letter $M$ is that this subgroup is a twisted Levi subgroup of
$G$, i.e. an $F$-subgroup that becomes a Levi subgroup after extending
scalars to an algebraic closure of $F$; however this fact is not actually
needed below.) The set $\mathcal M$ of subgroups obtained in this way (as
$X$ varies in
$\mathfrak t$) is finite.

The following notation will be useful.
Let $M \in \mathcal M$. We write $R_M$ for the (absolute) root system of
$M$ (a subset of $R$). We write
$\mathfrak z_M$ for the Lie algebra of the center of $M$; we then have
\[
\mathfrak z_M =\{ X \in \mathfrak t : \alpha (X)=0 \quad \forall \alpha \in
R_M \}
\]

  For $X \in \mathfrak t$ we have $M_X=M$ if
and only if
\[
\{ \alpha \in R: \alpha(X)=0 \}=R_M
\]
or, in  other words, if and only if $X$ lies in the open subset
\[
\mathfrak z_M':=\{ X \in \mathfrak z_M :
\alpha (X)\ne 0 \quad \forall \alpha \in
R \setminus R_M \}
\]
 of $\mathfrak z_M$. Obviously $\mathfrak t$ is the disjoint union of the
locally closed subsets $\mathfrak z_M'$. For example we have
$\mathfrak z_G'=\mathfrak z_G$, while $\mathfrak z_T'$ is the set of
regular  elements in $\mathfrak t$.

We fix a Haar measure $dg$ on $G$. In addition,
for each $M \in \mathcal M$ we fix a Haar measure $dm$ on $M$. For instance
one can use the canonical measures defined by Gross. In any case,
for $X \in
\mathfrak z_M'$ we define the orbital integral $O_X$ by
\begin{equation}\label{oXdef}
O_X(f):=\int_{M\backslash G} f(g^{-1}Xg) \, dg /dm.
\end{equation}
Thus we now have a coherent definition of orbital integrals for all $X \in
\mathfrak t$.

\subsection{Preliminary definition of Shalika germs on $\mathfrak g$}
There are finitely many nilpotent $G$-orbits
$\mathcal O_1,\mathcal O_2,\dots,\mathcal O_r$ in $\mathfrak g$.
We write $\mu_1,\dots,\mu_r$ for the
corresponding nilpotent orbital integrals. The  distributions
$\mu_1,\dots,\mu_r$ are linearly independent.

\begin{thm}\label{thm.sh.germ}
There exist functions $\Gamma_1,\Gamma_2,\dots,\Gamma_r$ on
$\mathfrak t$ having the following property. For every $f \in
C^\infty_c(\mathfrak g)$ there exists an open neighborhood
$U_f$ of
$0$ in
$\mathfrak t$ such that
\begin{equation}
O_X(f)=\sum^r_{i=1} \mu_i(f) \cdot \Gamma_i(X)
\end{equation}
for all $X \in   U_f$. The germs
about
$0 \in \mathfrak t$ of the functions
$\Gamma_1,\dots,\Gamma_r$ are unique. We refer to $\Gamma_i$
 as the provisional
Shalika germ for the nilpotent orbit $\mathcal O_i$.
\end{thm}
\begin{proof}
Same as usual.
\end{proof}
A  Shalika germ is an equivalence class of functions on $\mathfrak
t$.  As we will see next, the homogeneity of Shalika germs makes it
possible to single out one particularly nice function $\Gamma_i$ within its
equivalence class. Once we have done this, $\Gamma_i$ will from then on
denote this function (whose germ about $0$ is the old $\Gamma_i$).

\subsection{Behavior under scaling}
For $\beta \in F^\times$ and $f \in C^\infty_c(\mathfrak g)$ we write
$f_\beta$ for the function on $\mathfrak g$ defined by
\begin{equation}
f_\beta(X):=f(\beta X).
\end{equation}
 Harish-Chandra proved that
\begin{equation}\label{nil.scale}
\mu_\mathcal O(f_{\alpha^2})=|\alpha|^{-\dim\mathcal O}\mu_{\mathcal O}(f)
\end{equation}
for every nilpotent orbit $\mathcal O$ and $\alpha \in F^\times$.
Moreover it is clear from \eqref{oXdef} that
\begin{equation}\label{reg.scale}
O_X(f_\beta)=O_{\beta X}(f)
\end{equation}
for all $X \in \mathfrak t$ and all $\beta \in F^\times$.

\subsection{Partial homogeneity of our provisional Shalika germs
$\Gamma_i$}
Let $\alpha \in F^\times$. Let $\mathcal O_i$ be one of our nilpotent
orbits, let $\mu_i$ be the corresponding nilpotent orbital integral, and
let $\Gamma_i$ be the corresponding Shalika germ. Put $d_i:=\dim \mathcal
O_i$. We claim that
\begin{equation}\label{germ.scale}
\Gamma_i(X)=|\alpha|^{d_i}\Gamma_i(\alpha^2X),
\end{equation}
where the equality means equality of germs about $0$ of functions on
$\mathfrak t$.

Indeed, as in the proof of the Shalika germ expansion on $G$, pick a
function $f_i \in C^\infty_c(\mathfrak g)$ such that
\begin{equation}
\mu_j(f_i)=\delta_{ij}.
\end{equation}
Then $\Gamma_i(X)$ is the germ about $0$ of the function
\begin{equation}\label{sh.hom1}
X \mapsto O_X(f_i)
\end{equation}
on $\mathfrak t$. In fact during the remainder of our discussion of
provisional germs, we will use always use \eqref{sh.hom1} as our choice for
a specific function $\Gamma_i$ having the right germ.

In view of the
homogeneity of nilpotent orbital integrals established above,
$|\alpha|^{d_i}\cdot(f_i)_{\alpha^2}$ can also serve as $f_i$, so that
$\Gamma_i(X)$ is also the germ about $0$ of the function
\begin{equation}\label{sh.hom2}
X \mapsto O_X(|\alpha|^{d_i}\cdot(f_i)_{\alpha^2})=|\alpha|^{d_i} \cdot
O_{\alpha^2X}(f_i)
\end{equation}
on $\mathfrak t$. Comparing \eqref{sh.hom1}, \eqref{sh.hom2}, we
see that the germs of $\Gamma_i(X)$ and $|\alpha|^{d_i}\Gamma_i(\alpha^2
X)$  are equal, as desired.

\subsection{Canonical Shalika germs}
Let $\Gamma_i$ be one of our germs. We are going to replace $\Gamma_i$
by another  function $\Gamma_i^{\new}$ on $
\mathfrak t$ that has the same germ about
$0$ and is at the same time homogeneous.
\begin{lem}
There is a unique function
$\Gamma_i^{\new}$ on $\mathfrak t$ which has the same germ about
$0$ as $\Gamma_i$ and which satisfies
\eqref{germ.scale} for all
$\alpha
\in F^\times$ and all $X \in \mathfrak t$. Moreover $\Gamma_i^{\new}$
is real-valued, translation invariant under the center of
$\mathfrak g$, and invariant under conjugation by elements in the
normalizer of $T$.
\end{lem}
\begin{proof}
Same as usual.
\end{proof}

From now on we replace the germs $\Gamma_i$ by the functions
$\Gamma_i^{\new}$, but we drop the superscript ``new.''

We also need a slight strengthening of the fact that
$\Gamma_i$ is translation invariant under the center $\mathfrak z$ of
$\mathfrak g$. Let $G'$ be the derived group of the algebraic group
$G$. Then $G(\bar F)=G'(\bar F)Z(\bar F)$,  but for $F$-points we have only
that $G'Z$ is a normal subgroup of finite index in $G$. We denote by $D$
the finite group $G/G'Z$. Each $G$-orbit $\mathcal O$ in $\mathfrak
g=\mathfrak g'
\oplus \mathfrak z$ decomposes as a finite union of $G'$-orbits $\mathcal
O'$, permuted transitively by $D$. We normalize the invariant measures on
the orbits in such a way that
\begin{equation}
\int_{\mathcal O}=\sum_{x \in D} \int_{x^{-1}\mathcal O' x}.
\end{equation}
For a nilpotent $G$-orbit $\mathcal O$ (respectively, nilpotent $G'$-orbit
$\mathcal O'$) we denote by $\Gamma^G_{\mathcal O}$ (respectively,
$\Gamma^{G'}_{\mathcal O'}$) the corresponding Shalika germ on $\mathfrak
t$ (respectively, $\mathfrak g'\cap \mathfrak t$).

\begin{lem}\label{red.der}
Let $X \in \mathfrak t$ and decompose $X$ as $X'+Z$ with $X' \in
\mathfrak g'\cap \mathfrak t$ and $Z \in \mathfrak z$. Then
\begin{equation}
\Gamma^G_{\mathcal O}(X)=\sum_{\mathcal O' \subset \mathcal O}
\Gamma^{G'}_{\mathcal O'}(X').
\end{equation}
\end{lem}
\begin{proof} Same as usual, but note that there is a typo in the proof
of the corresponding result in \cite{Kot05}: the functions $f$, $f'$
occurring in formula  (17.8.9) of that article should have a  subscript
$\mathcal O$.
\end{proof}

\subsection{Germ expansions about arbitrary central elements in $\mathfrak
g$}
We have been studying germ expansions about $0 \in \mathfrak t$. These
involve  orbital integrals for the nilpotent orbits $\mathcal O_i$. Now we
consider germ expansions about an arbitrary element $Z$ in the center of
$\mathfrak g$. These will involve orbital integrals $\mu_{Z+\mathcal
O_i}$ for the orbits
$Z+\mathcal O_i$, but will involve exactly the same germs $\Gamma_i$ as
before.

\begin{thm}\label{thm.sh.cent}
Let $Z$ be an element in the center of $\mathfrak g$. For every $f \in
C^\infty_c(\mathfrak g)$
there exists an open
neighborhood
$U_f$ of
$Z$ in
$\mathfrak t$
 such that
\begin{equation}
O_X(f)=\sum^r_{i=1} \mu_{Z+\mathcal O_i}(f) \cdot \Gamma_i(X)
\end{equation}
for all $X \in U_f $.
\end{thm}
\begin{proof}
Same as usual.
\end{proof}

\subsection{Germ expansions about arbitrary semisimple elements in
$\mathfrak g$}\label{sub.germ.desc}
We are going to use Harish-Chandra's theory of semisimple descent
 in order to obtain germ expansions
about an arbitrary  element $S \in \mathfrak t$. We fix such an
element
$S$ and let
$H:=G_S$ denote the centralizer of $S$, a connected reductive subgroup of
$G$.

 Let
$Y_1,\dots,Y_s$ be a set of representatives for the nilpotent $H$-orbits in
$\mathfrak h$. Let
$\mu_{S+Y_i}$ denote the orbital integral on $\mathfrak g$ obtained by
integration over the $G$-orbit of $S+Y_i$. Now $T$ is also a maximal torus
in $H$, so for each  $1
\le i
\le s$ we can consider the canonical Shalika germ
$\Gamma^H_i$ for $H$, $\mathfrak t$ and
 the nilpotent $H$-orbit of $Y_i$.

\begin{thm} \label{thm.ss.germ}
Let $S$, $H$ be as above. For every $f \in C^\infty_c(\mathfrak g)$
there exists an open neighborhood $U_f$ of $S$ in
$\mathfrak t$ such that
\begin{equation}\label{shal''}
O_X(f)=\sum^s_{i=1} \mu_{S+Y_i}(f) \cdot \Gamma^H_i(X)
\end{equation}
for all $X \in U_f $.
\end{thm}

\begin{proof}
Same as usual.
\end{proof}

\subsection{Normalized orbital integrals and Shalika germs}
For $X \in
\mathfrak t$ we put
\[
D^G(X)=\det(\ad(X);\mathfrak g/\mathfrak m_X)
\]
($\mathfrak m_X$ being the Lie algebra of the centralizer $M_X$ of $X$ in
$G$) and define the normalized orbital integral $I_X$ by
\[
I_X=|D^G(X)|^{1/2} O_X.
\]
 When we use $I_X$ instead of $O_X$, we need to use
the normalized Shalika germs $\bar\Gamma_i(X):=|D^G(X)|^{1/2}\Gamma_i(X)$
instead of the usual Shalika germs.

Clearly Theorem \ref{thm.sh.germ} remains valid when $O_X$, $\Gamma_X$ are
replaced by $I_X$, $\bar\Gamma_i$ respectively. Now consider the germ
expansion  about an arbitrary
element $S \in \mathfrak t$. As usual put $H:=G_S$.
There exists a neighborhood of $S$ in $\mathfrak t$ on which
\[
|D^G(X)|^{1/2}=|D^H(X)|^{1/2}
|\det(\ad(S);\mathfrak g/\mathfrak h)|^{1/2}.
\]
  It then follows from Theorem
\ref{thm.ss.germ} that
\begin{equation}\label{shal'''}
I_X(f)=|\det(\ad(S);\mathfrak g/\mathfrak h)|^{1/2}\sum^s_{i=1}
\mu_{S+Y_i}(f) \cdot \bar\Gamma^H_i(X)
\end{equation}
for all  $X$ in some sufficiently small neighborhood of $S$
in $\mathfrak t$.

The homogeneity property
\eqref{germ.scale} of the Shalika germs $\Gamma_i$ implies the following
homogeneity property for the normalized Shalika germs $\bar\Gamma_i$:
\begin{equation}\label{norm.hom.germ}
\bar\Gamma_i(\alpha^2X)=
|\alpha|^{\dim(G_{X_i})-\dim(M_X)}\cdot\bar\Gamma_i(X)
\end{equation}
for all $\alpha \in F^\times$ and all $X \in \mathfrak t$. Here
we have chosen $X_i \in \mathcal O_i$ and introduced its centralizer
$G_{X_i}$.

The next proposition will be needed when we use \eqref{norm.hom.germ}
in the proof of  boundedness of normalized Shalika germs. It is a
simple adaptation of ideas  from Sparling's article \cite{sparling}. To
formulate the proposition we need a definition. Consider the action
morphism $G \times
\mathfrak g \to \mathfrak g$ (given by $(g,X) \mapsto gXg^{-1}$); we are
now thinking of $G$ and $\mathfrak g$ as algebraic varieties over $F$.
For $M \in \mathcal M$ we consider the image $V^0_M \subset \mathfrak g$ of
$G\times \mathfrak z_M'$ under this morphism. Obviously $V^0_M$ is an
irreducible
$G$-invariant subset of the variety $\mathfrak g$, so its Zariski closure
$V_M$ is a $G$-invariant irreducible subvariety of $\mathfrak g$. We say
that a nilpotent orbit $\mathcal O$ is \emph{relevant to} $M$ if
$\mathcal O$ is contained in $V_M$.

\begin{prop}\label{prop.rele}
Let $M \in \mathcal M$ and let $\mathcal O$ be a nilpotent orbit in
$\mathfrak g$. Then the following two statements hold.
\begin{enumerate}
\item If $\mathcal O$ is  relevant to $M$, then for $Y \in \mathcal O$ we
have $\dim G_Y \ge \dim M$, where $G_Y$ denotes the centralizer of $Y$ in
$G$.
\item If $\mathcal O$ is not relevant to $M$, then the normalized Shalika
germ
$\overline\Gamma_\mathcal O$ vanishes identically on $\mathfrak z_M'$.
\end{enumerate}
\end{prop}
\begin{proof}
(1) Over $V_M$ we have the group scheme whose fiber at $X \in V_M$ is the
centralizer of $X$ in $G$. At points in $\mathfrak z_M'$ this centralizer
is $M$ and at points of $V_M^0$ it is some conjugate of $M$. Since $V_M^0$
is dense in $V_M$, we conclude from SGA 3, Tome I, Exp.~VI$_B$, Prop.~4.1
that $\dim G_X \ge \dim M$ for all $X \in V_M$. In particular this
inequality holds when we take $X$ to be $Y \in \mathcal O \subset V_M$.

(2) Let $f \in C^\infty_c \mathfrak g$ and suppose that $\mu_\mathcal
O(f)=0$ for all nilpotent orbits $\mathcal O$ relevant to $M$. Then, as in
the proof of the existence of Shalika germs, there exists an open
neighborhood $U_f$ of $0$ in
$\mathfrak t$ such that
$O_X(f)=0$ for all $X
\in U_f \cap V_M$. In particular $O_X(f)=0$ for all $X \in U_f \cap
\mathfrak z_M'$. Applying this observation to the functions $f_j$ used to
produce our provisional Shalika germs, we conclude that  if
$\mathcal O_j$ is not relevant to $M$, then there is a
neighborhood $U_j$ of $0$ in $\mathfrak t$ such that the provisional
Shalika germ $\Gamma_j$  vanishes on $U_j \cap \mathfrak z_M'$. Looking back at how the true
(homogeneous) Shalika germs were obtained from the provisional ones, we see
that the true Shalika germ
$\Gamma_j$ vanishes identically on $\mathfrak z_M'$ when $\mathcal O_j$ is
not relevant to $M$.
\end{proof}

\subsection{ $\bar\Gamma_i$ is a linear combination of functions
$\bar\Gamma^H_j$ in a neighborhood of $S$}
Again let $S \in\mathfrak t$ and  let $H$ be its
centralizer in $G$. Consider one of
the normalized Shalika germs $\bar\Gamma_i$ for $G$. We are interested in
the behavior of $\bar\Gamma_i$ in a small neighborhood of $S$ in $\mathfrak
t$.

\begin{lem}\label{asy.germ}
There exists a neighborhood $V$ of $S$ in $\mathfrak t$ such that the
restriction of
 $\bar\Gamma_i$  to $V $ is a linear combination of
restrictions of
normalized Shalika germs for $H$.
\end{lem}
\begin{proof}
Same as usual.
\end{proof}

\begin{cor} Let $M \in \mathcal M$.
Each normalized Shalika germ $\bar\Gamma_i$ is locally constant on
$\mathfrak z_M'$.
\end{cor}
\begin{proof}
Same as usual.
\end{proof}

\subsection{Locally bounded functions}
We are going to show that the normalized Shalika germs $\bar\Gamma_i$ are
locally bounded functions on $\mathfrak t$. First let's recall what this
means. Let $f$ be a complex-valued function on a topological space $X$. We
say that $f$ is \emph{locally bounded} on $X$ if every point $x \in X$ has
a neighborhood $U_x$ such that $f$ is bounded on $U_x$. When $X$ is a
locally compact Hausdorff space, $f$ is locally bounded if and only $f$ is
bounded on every compact subset of $X$.

\subsection{Local boundedness of normalized Shalika germs}
Let $\bar\Gamma_i$ be one of our normalized Shalika germs on $\mathfrak t$.
 We are going to show that $\bar\Gamma_i$ is locally bounded
as a function on $\mathfrak t$, slightly generalizing a result of
Harish-Chandra.

\begin{thm}
Every normalized Shalika germ $\bar\Gamma_i$ is  locally bounded
on~$\mathfrak t$.
\end{thm}
\begin{proof}
Same as usual once one takes into account Proposition  \ref{prop.rele}.
\end{proof}
As a consequence of the local boundedness of normalized Shalika germs, we
obtain a slight generalization of another result of Harish-Chandra.

\begin{thm}\label{bdoi}
Let $f \in C^\infty_c(\mathfrak g)$.
Then the function $X \mapsto I_X(f)$ on $\mathfrak t$ is bounded and
compactly supported on
$\mathfrak t$. Moreover, for each $M \in \mathcal M$ this function is
locally constant on $\mathfrak z_M'$.
\end{thm}
\begin{proof}
Same as usual.
\end{proof}

%%%%%%%  APPENDIX B %%%%%%%

\section{By Raf Cluckers, Julia Gordon and Immanuel Halupczok}\label{s:app:B}

In this appendix we use the theory of motivic integration to control bounds for
orbital integrals, normalized by the discriminant, as the place
varies. In Appendix A, the bound for orbital integrals is proved for a fixed local field;
here we show that this bound cannot exceed a power of the cardinality of the residue field, using the tools from model theory. We emphasize that the main result of Appendix A, namely, the fact that the orbital integrals are bounded, is used in our proof.
More specifically, we prove Theorems \ref{thm:main}, and \ref{thm:main2} which are stronger versions of, respectively, 
Theorem \ref{t:appendeix2}  and Proposition \ref{p:appendix2} with $e_G=1$.
We also prove the analogous statement for the
function fields; {moreover, we prove that the optimal exponents can, in some sense, be transferred
between the function field and number field cases, see Theorem \ref{thm:transfer-fam}}.
We expect that the same methods could apply to weighted orbital integrals, provided that one had a statement similar to the Theorem A.1 of Appendix A.

Let $\bF$ be a number field with the ring of integers $\ri$.
Let $G$ be a connected reductive algebraic group defined over
$\bF$, and $\fg$ its Lie algebra.
Let $F=\bF_v$ be a completion of $\bF$. {We denote the ring of integers of $F$ by $\ri_F$, the residue
field by $k_F$, and let $q_F=\# k_F$ be the cardinality of $k_F$}.
For a semisimple element $\gamma\in G(F)$ and a test function $f\in C_c^\infty(G(F))$,
the orbital integral
at $\gamma$ is denoted by $\oi_\gamma(f)$.
As in Appendix A,
$$D^G(\gamma)=
\prod_{\alpha\in \Phi\atop{ \alpha(\gamma)\neq 1}}|1-\alpha(\gamma)|_v,
$$
where $\Phi$ is the {root system} of $G$.

{We keep the set-up of \S\ref{sub:local-bound-orb-int} and \S\ref{sub:global-bound-orb-int}; in particular, 
we first treat the case of a reductive group with a given root datum defined over a local field, and then derive the global statement from it.
Thus, we start with a  reductive group $G$ defined over a local field $F$, and we assume that $G$ is unramified.
In order to get to this setting from the global set-up, where we just have to assume that 
$G={\bG}_v$ where
the place $v$ is finite, and lies outside the set
$\Ram(G)$.}

{Given an unramified reductive group $G$ over a local field $F$ as above,}  we recall the definition of the functions $\tau^G_\lambda$
from \S \ref{sub:Satake-trans}.
We have a Borel subgroup $B=TU$, and let $A$ be the maximal $F$-split torus in $T$.
As in \S \ref{sub:Satake-trans}, choose a smooth reductive
model $\underline{G}$ for $G$ corresponding to
a hyperspecial point in the apartment of $A$, and let $K={\underline G}({\ri_F})$ be a maximal compact subgroup.
For $\lambda\in X_\ast(A)$, $\tau_\lambda^G$ is the characteristic function of
the double coset $K\lambda(\varpi)K$.

We prove
\begin{thm}\label{thm:main}
Let $G$ be a connected reductive algebraic group over $\bF$, {with $\bT$ and $A_v$ as in \S \ref{sub:global-bound-orb-int}}.
There exist constants $a_G$ and $b_G$ that depend only on the
global model of $G$ such that for all {$\lambda\in X_\ast(A_v)$} with $\|\lambda\|\le \kappa$, for
all but finitely many places $v$
$$
|\oi_\gamma(\tau_\lambda^G)|\le q_v^{a_G+b_G\kappa} D^G(\gamma)^{-1/2}$$
{for all semisimple elements $\gamma\in G(\bF_v)$,}
{where $q_v$ is the cardinality of the residue field of $\bF_v$}.
\end{thm}

In fact, we prove a stronger and more general statement, which does not require $F$ to have characteristic zero.
By an unramified root datum we mean a root datum of an unramified reductive group over a local field $F$, i.e. a quintuple $\xi=(X^\ast, \Phi, X_\ast, \Phi^\vee, \theta)$, where $\theta$
is the action of the Frobenius element of $F^\ur/F$ {on the first four components of $\xi$}.

\begin{thm}\label{thm:main2}
Consider an unramified root datum $\xi$. Then there exist constants $M>0$,  $a_\xi$ and $b_\xi$ that depend only on $\xi$,  such that for each non-Archimedean local field $F$
with residue characteristic at least $M$, the following holds. Let $G$ be a connected reductive algebraic group over
$F$ with the root datum $\xi$. {Let $A$ be a maximal $F$-split torus in $G$, and let $\tau_\lambda^G$ be as above.} Then
for all {$\lambda\in X_\ast(A)$} with $\|\lambda\|\le \kappa$,
$$
|\oi_\gamma(\tau_\lambda^G)|\le {q_F}^{a_\xi+b_\xi\kappa} D^G(\gamma)^{-1/2}$$
{for all semisimple elements $\gamma\in G(F)$.}
\end{thm}

The strategy of the proof is to use the theory of motivic integration developed by R. Cluckers and F. Loeser, \cite{cluckers-loeser}. In \cite{cluckers-loeser}, a class of functions called \emph{constructible motivic functions} is defined.
Here, in order to simplify the language, we are working directly with the specializations of constructible motivic functions, which we define below, and we call these ``constructible functions''. %(in the sense of \cite{cluckers-loeser})
These functions are defined by means of formulas in a first-order language of logic, called Denef-Pas language, which we review below.
The key benefit of using logic is that the formulas defining the functions are independent of the field of definition, hence this set-up is perfectly
suited for proving a result that applies uniformly across almost all
completions of a given number field. This method  can be thought of as an extension of a geometric approach -- ``definable'' is a less restrictive notion than ``geometric'', yet it provides a field-independent way of talking about orbital integrals.

The key to our proof is  a general result which, roughly speaking, states that if a constructible function
is bounded {(which is known in our case thanks to Appendix A)}, then its upper bound cannot exceed a fixed power of the cardinality of the residue field (Theorem \ref{thm:presburger-fam} below).
In order to apply this result to orbital integrals, we need to show that they are, in some sense, constructible functions. More precisely, one would like to show that given a constructible
test function $f\in C_c^\infty(G(F))$,
the function  $\gamma\mapsto \oi_\gamma(f)$ is a constructible function
of $\gamma$, {on the set of all semisimple elements}.
For \emph{regular} semisimple elements, the Lie algebra version of
this  statement is essentially proved by
Cluckers, Hales and Loeser \cite{cluckers-hales-loeser}.
For  {general} elements $X$, the Lie algebra version of this statement  {with a particular normalization of the measure on the orbit} is proved
in \cite{CGH-2};   {however, the normalization of the measures used in  \cite{CGH-2} is not the same as the canonical normalization used in Appendix A above.}
%, so we are required to do some additional work}.
For non-regular semisimple elements,  {we show here that the canonical measure differs from the measure used in \cite{CGH-2} by a constant that can be bounded by a fixed power of the cardinality of the residue field,
%and using the canonical normalization of measures on the orbits},  we prove a slightly weaker statement here -- 
and consequently, obtain} that given $f$,
there exists a constructible function $H_f$ and a constant $c$ that depends only on the root datum of the group, such that
$q^{-c}|H_f(\gamma)|\le |O_\gamma(f)|\le q^c|H_f(\gamma)|$.
Taking $f$ to be the characteristic function of the maximal compact subgroup $K$ in this argument,
we obtain the special case of
Theorem \ref{thm:main2} with $\kappa=0$.  The full statement of Theorem \ref{thm:main2} is obtained by a similar argument that allows the test functions to vary in definable families.

Much of the preliminary and introductory material is quoted freely from \cite{cluckers-loeser:ax-kochen}, \cite{cluckers-hales-loeser}, \cite{CGH-1}, \cite{CGH-2}, \cite{gordon-yaffe}, sometimes without mentioning these
ubiquitous citations.

{\bf Acknowledgment.}  We gratefully acknowledge BIRS, and the organizers of the workshop on L-packets, where this appendix was conceived.
J.G. is deeply grateful to Sug-Woo Shin, Nicolas Templier, Loren Spice, Tasho Statev-Kaletha, William Casselman, and Gopal Prasad for helpful conversations. 
R.C. was supported by the European Research
Council under the European Community's Seventh Framework Programme
(FP7/2007-2013) with ERC Grant Agreement nr. 615722
MOTMELSUM and by the Labex CEMPI (ANR-11-LABX-0007-01);
 J.G. was supported by NSERC; I.H. was supported by the SFB~878 of the Deutsche
Forschungsgemeinschaft.

\subsection{Denef-Pas language}\label{sub:DP}
The Denef-Pas language is a first order language of logic designed for working
with valued fields.
We start by defining two sublanguages of the language of
Denef-Pas: the language of rings and the Presburger language.
\subsubsection{The language of rings}
A formula in the first-order language of rings is any syntactically correct formula
built out of the following symbols:
\begin{itemize}
\item constants `$0$', `$1$';
\item binary functions `$\times$', `$+$';
\item countably many symbols for variables $x_1, \dots, x_n,\dots$ running over a ring;
\item the following logical symbols: equality `$=$', parentheses `$($', `$)$', the quantifiers `$\exists$', `$\forall$', and the logical operations conjunction `$\wedge$', negation `$\neg$', disjunction `$\vee$'.
\end{itemize}
If a formula in the language of rings has $n$ free (i.e.\ unquantified) variables then it defines a subset of $R^n$ for any ring $R$.
Note that quantifier-free
formulas in the language of rings define constructible sets (in the sense of algebraic geometry).

\subsubsection{Presburger language}
A formula in Presburger's language is built out of variables running over $\Z$, the logical symbols (as above) and
symbols `$+$', `$\le$', `$0$', `$1$', and for each $d=2,3,4,\dots$, a symbol `$\equiv_d$' to denote the binary
relation $x\equiv y \pmod{d}$.
Note the absence of the symbol for multiplication.

Since multiplication is not allowed, sets defined by formulas in the Presburger language are in fact very basic, cf.~\cite{CPres} or \cite{Presburger}.
{For example, $\{(a,b)\in\Z^2 \mid a\equiv 1\mod 4; a\leq b + 10\}$ is a Presburger subset of $\Z^2$. Since quantifiers are never needed to describe Presburger sets, they all are of a similar, simple form.}

\subsubsection{Denef-Pas language}\label{DP}
The formulas in Denef-Pas language have variables of
three sorts:
the valued field sort, the residue field sort, and the value group sort
(in our setting, the value group is always assumed to be $\Z$, so we will
call this sort the $\Z$-sort). Here is the list of symbols used
to denote operations and binary relations in this language:
\begin{itemize}
\item In the valued field sort:
the language of rings.
\item In the residue field sort: the language of rings.
\item In the $\Z$-sort: the Presburger language.
\item a symbol $\ord(\cdot)$ for the valuation map from the nonzero elements of the valued field sort to the $\Z$-sort, and $\ac(\cdot)$ for the so-called angular component, which is a function from the valued field sort to the residue field sort (more about this function below).
\end{itemize}

On top of the symbols for the constants that are already present (like $0$ and $1$), we will add to the Denef-Pas language all elements of
$\ri[[t]]$ as extra symbols for constants in the valued
field sort. We denote this language by $\ldpo$.

Given a discretely valued field $F$ that is an algebra over $\ri$, together with a choice of 
 {a ring homomorphism} $\iota:\ri\to F$ and {a} choice of a uniformizer $\varpi$ of the valuation, one can interpret
the formulas in $\ldpo$ by letting
the variables range, respectively, over $F$, the residue field $k_F$
of $F$, and
$\Z$ (which is the value group of $F$).
The function symbols $\ord$ and $\ac$ are interpreted as follows.
For $x\in F^\times$, $\ord(x)$ denotes the valuation of $x$.
If $x$ is a unit (that is, $\ord(x)=0$), then $\ac(x)$ is the residue of $x$ modulo $\varpi$ (thus, an element of the residue field). For a general $x\neq 0$
define
$\ac(x)$ as $\ac(\varpi^{-\ord(x)}x)$; thus, $\ac(x)$ is the first non-zero coefficient of the $\varpi$-adic expansion of $x$. Finally we define $\ac(0)=0$.
The elements from $\ri$ are interpreted as elements of $F$ by using $\iota$, the constant symbol $t$ is interpreted as the uniformizer $\varpi$, and thus, by the completeness of $F$, elements of $\ri[[t]]$ can be naturally interpreted in $F$ as well.

\begin{defn}\label{AO}
Let $\cC_\ri$ be the collection of all triples $(F, \iota , \varpi)$, where $F$ is a non-Archimedean local
field which allows at least one ring homomorphism from $\ri$ to $F$,  the map $\iota:\ri\to F$ is such a ring homomorphism, and $\varpi$ is a uniformizer for $F$.
Let $\cA_\ri$ be the collection of those triples $(F, \iota , \varpi)$ in $\cC_\ri$ in which $F$ has characteristic zero, and let $\cB_\ri$ be the collection of those triples $(F, {\iota} , \varpi)$ where $F$ has positive characteristic.

Given an integer $M$, let $\cC_{\ri, M}$ be the collection of $(F,{\iota},\varpi)$ in $\cC_\ri$ such that the residue field of $F$ has characteristic larger than $M$, and similarly for $\cA_{\ri, M}$
and $\cB_{\ri, M}$.
\end{defn}

Since our results and proofs are {independent of} the choices of the {map}  $\iota$ and the
{uniformizer} $\varpi$, we will often just write $F\in \cC_\ri$, instead of naming the whole triple.
For any $F\in \cC_{\ri}$, write $\cO_F$ for the valuation ring of $F$, $k_F$ for its  residue field, and $q_F$ for the cardinality of $k_F$.

In summary, an $\ldpo$-formula $\varphi$ with $n$ free valued-field variables, $m$ free residue-field variables, and $r$ free $\Z$-variables defines naturally, for each $F \in \cC_\ri$, a subset of
$F^n\times \rf_F^m\times \Z^r$ by taking the set of all tuples where $\varphi$ is ``true'' (in the natural sense of first order logic,  {see e.g. \cite{marker}}).

\subsection{Definable sets and constructible functions}\label{subsub:functions}

As mentioned in the introduction, to study dependence on $p$ of various bounds we will need to have a field-independent notion of subsets of $F^n\times k_F^m\times \ZZ^r$ for $F\in \cC_\ri$. 
 {To achieve this, we call} a {collection}
$(X_F)_{F}$ of subsets $X_F\subset F^n\times \rf_F^m\times \Z^r$, where $F$ runs over $\cC_\ri$, which come from an
$\ldpo$-formula $\varphi$ as explained at the end of \S \ref{DP}, \emph{a definable set}.
Thus, for us, a ``definable set'' is actually a collection of sets, namely one for each $F\in \cC_\ri$; in earlier work on motivic integration, the term ``specialization of a definable subassignment'' was used for a similar notion. For an  {integer $r\geq 0$, $\Z^r$} will often denote the definable set $(X_F)_F$ such that $X_F= \Z^r$ for each $F$. More generally, for 
 {non-negative} integers $n,m,r$, the notation $h[n,m,r]$ will stand for the definable set $(F^n\times k_F^m\times \ZZ^r)_F$.

For definable sets $X$ and $Y$, a collection $f = (f_F)_F$ of functions $f_F:X_F\to Y_F$ for $F\in\cC_\ri$ is called a definable function and denoted by $f:X\to Y$ if  the collection of graphs of $f_F$ is a definable set.

Definable functions are {the} building blocks for constructible functions, which are defined as follows.
For a definable set $X$,
a collection $f = (f_F)_F$ of functions $f_F:X_F\to\CC$ is called \emph{a constructible function} if 
there exist integers
$N$, $N'$, and $N''$, such that $f_F$ has the form, for $x\in X_F$, for all $F\in \cC_\ri$,
$$
f_F(x)=\sum_{i=1}^N   q_F^{\alpha_{iF}(x)} \# (p_{iF}^{-1}(x) )    \big( \prod_{j=1}^{N'} \beta_{ijF}(x) \big) \big( \prod_{\ell=1}^{N''} \frac{1}{1-q_F^{a_{i\ell}}} \big),
$$
where:
\begin{itemize}
\item $a_{i\ell}$ with $i=1,\dots, N$, $\ell=1,\dots, N''$ are negative integers;
\item $\alpha_{i}:X\to \ZZ$ with $i=1, \dots N$, and $\beta_{ij}:X\to \ZZ$ with $i=1\, \dots, N$, $j=1,\dots, N'$ are
$\ZZ$-valued definable functions;
\item $Y_i$ are definable sets such that $Y_{iF}\subset k_F^{r_{i}}\times X_F$ for some $r_i\in\ZZ$, and
$p_{i}:Y_{i}\to X$ is the coordinate projection.
\end{itemize}

The motivation for such a definition of a constructible function comes from the theory of integration: namely,
{constructible functions} form a rich class of functions which is stable under integration with respect to parameters (as in Theorem \ref{thm:mot.int.} below). See \cite{cluckers-loeser:ax-kochen} and \cite{gordon-yaffe} for details.

For each $F$ in $\cC_\ri$, let us put the Haar measure on $F$ so that $\cO_F$ has measure $1$, the counting measure on $k_F$ and on $\ZZ$, and the product measure on Cartesian products. {Thus, we get a natural measure on
$h[n, m, r]$.} {Furthermore}, any analytic subvariety of $F^n$, say, everywhere of equal dimension, together with an analytic volume form, carries a natural measure associated to the volume form, cf.~\cite{Bour}.

The notion of a measure associated with a volume form carries over to the definable setting, roughly as follows.
By the piecewise analytic nature of definable sets and definable functions, any definable subset $X$ of $h[n,m, r]$ can
be broken into finitely many pieces $X_i$, such that $X_i(F)$ is a subset of
$V_i\times k_F^m\times \Z^r$ for some $F$-analytic subvariety $V_i$ of $F^n$ of the same dimension as $X_i(F)$,
for each $F$ with large residue characteristic. A definable form on $h[n,0,0]$ in the affine coordinates $x$ is just a finite sum of terms of the form $f(x)dx_{i_1}\wedge \ldots \wedge dx_{i_d}$ where $f$ is a definable function with values in $h[1,0,0]$. If the functions $f$ restrict to $F$-analytic functions on $V_i$ for each such $f$, and if the form is a $d$-form where $d$ is the dimension of $V_i$, then one can use the measure associated to this analytic volume form on $V_i$. This construction yields natural ``motivic'' measures on the definable set $X$, associated to definable differential forms,  {cf. also \cite{CGH-2}*{\S 3.5.1}}. Such a construction of measures
associated with differential forms behaves well in the setting of motivic integrals
because there exists a natural change of variables formula for motivic integrals, see
\S 15 of \cite{cluckers-loeser}. In summary, the measures that arise from definable differential forms
occur naturally in the context of motivic integration and we will call such measures
``motivic'' below. We refer to \cite{cluckers-loeser}*{\S 15} for the definition of the sheaf of definable differential forms on a definable set, and other details. {We note that any algebraic volume form on a variety}  {over $\ri_\bF$}, where $\bF$ is a global field, {is definable in this sense.}
{Note, however, that in this appendix we have to deal with volume forms on orbits of elements of a group defined over a local field, and the resulting measures are not automatically motivic.}

{Let us} recall one of the results of \cite{CGH-1}, the first part of which generalizes a result of \cite{cluckers-loeser:fourier}, and which shows
{that the class of constructible functions is a natural class  to work with for the purposes of integration.

\begin{thm}\cite{CGH-1}*{Theorem 4.3.1}\label{thm:mot.int.}
Let $f$ be a constructible function on $X\times Y$ for some definable sets $X$ and $Y$. Then there exist a constructible function $g$ on $X$ and an integer $M>0$ such that for each $F\in \cC_{\ri,M}$
and for each $x\in X_F$ one has
$$
g_F(x) = \int_{y\in Y_F}f_F(x,y),
$$
whenever the function $Y_F\to\CC:y\mapsto f_F(x,y)$ lies in $L^1(Y_F)$, where, say $Y_F\subset F^n\times k_F^m\times \ZZ^r$.
\end{thm}
Note that although the theorem is stated for the affine measure on $F^n$, it also holds for measures given by definable differential forms, by working with charts as is done in \cite{cluckers-loeser}*{\S 15}.

\begin{rem} In the literature on general motivic integration, one often uses a more abstract notion of ``definable subassignments''. Any such definable subassignment $X$ specializes to the sets $X_F$ discussed here for all $F\in \cC_{\ri,M}$ for some $M$, and any motivic integral over $X$ specializes to the corresponding integrals over $X_F$. In this paper it is sufficient and more convenient to work with the above notion of definable sets $(X_F)_F$  directly.
\end{rem}

Let us finally fix our terminology about ``families of definable sets'' and ``families of constructible functions''.
A family of definable sets $X_a$ indexed by a parameter $a\in A$
is a definable subset $X$ of $Y\times A$ for some definable sets $Y$ and $A$, equipped with the canonical projection $p_A:X\to A$, and the family members are $p_A^{-1}(a)=X_a$ for $a\in A$. Similarly, a family of constructible (respectively, definable) functions $f_a$ on the family $X_a$ is a constructible (respectively, definable)
function on $X\subset Y\times A$.
{Whenever we call a specific function $f:X_{F_0}\subset F_0^n\times k_{F_0}^m\times \ZZ^r\to\CC$ (for a specific field $F_0$) constructible, we mean that it appears naturally as $f_{F_0}$ for a constructible
function $(f_F)_F$ for which uniformity in $F$ is clear from the context as soon as the residue field characteristic is large enough; we use a similar convention for calling a specific function definable, and so on.}

\subsection{Boundedness of constructible functions}

The following two theorems are the main results of this section.

\begin{thm}\label{thm:presburger-fam}
Let $H$ be a constructible motivic function on $W \times \Z^n$,
where $W$ is a definable set. Then
there exist integers $a,b$ and $M$ such that for all $F\in \cC_{\ri, M}$ the following holds.

If there exists a (set-theoretical,  {and not necessarily uniform in $F$}) 
function $\alpha^F:\ZZ^n\to\R$ such that
$$
| H_F(w,\lambda) |_{\R} \le \alpha^F(\lambda) \mbox{ on } W_F \times \Z^n,
$$
then one actually has
$$
| H_F(w,\lambda) |_{\R} \le q_F^{a+b \|\lambda\|} \mbox{ on } W_F \times \Z^n,
$$
where $\|\lambda\| = \sum_{i=1}^n \lambda_i$, and where $|\cdot|_\R$ is the usual absolute value 
on $\R$.
\end{thm}

We observe that in the case with $n=0$,
the theorem yields that if a constructible function $H$ on $W$ is such that $H_F$ is bounded on $W_F$ for each $F\in \cC_{\ri, M}$, then the bound for $|H_F|_\R$  can be taken to be $q_F^a$ uniformly in $F$ with large residue characteristics, for some $a\geq 0$.

The following statement allows one to transfer bounds, which are known {for local fields of}
characteristic zero, to local fields of positive characteristic, and vice versa.

\begin{thm}\label{thm:transfer-fam}
Let $H$ be a constructible motivic function on $W \times \Z^n$,
where $W$ is a definable set, and let $a$ and $b$ be integers. Then there exists $M$ such that, for any $F\in \cC_{\ri, M}$,
whether the statement
\begin{equation}\label{transfer}
H_F(w,\lambda)\le q_F^{a+b \|\lambda\|} \mbox{ for all } (w,\lambda) \in W_F \times \Z^n
\end{equation}
holds or not, only depends on the isomorphism class of the residue field of $F$.
\end{thm}

Informally speaking, the idea of the proof is to first
eliminate all the valued-field variables, possibly at the cost of introducing more residue-field and $\ZZ$-valued variables. {This step is summarized in Lemma \ref{lem:presburger-fam} below, whose proof relies on the powerful Cell Decomposition Theorem for definable sets in Denef-Pas language.}
Once we have a constructible function that depends only on the residue-field and value-group variables,
we note that residue-field variables {can only play a very minor role} in the matters of boundedness (the so-called ``orthogonality of sorts'' in Denef-Pas language referred to below).
{Finally, the question is reduced to the study of  Presburger constructible functions
of several $\ZZ$-variables, which are similar to constructible functions as defined above in \S\ref{subsub:functions}, but without the factors $\# (p_{iF}^{-1}(x) )$, see \cite{CGH-1}.
Roughly, Presburger constructible functions in $x\in\Z^r$ are sums of products of piecewise linear functions in $x$ and of} {powers of $q_F$, where the power also depends piecewise linearly on $x$}.
{If such a function is bounded, then %each term in such a sum is bounded, 
 {it is a sum of bounded terms as above}, after removing possible redundancy in the sum. Each single term in $x$ can then easily be bounded, by} {a power of $q_F$ that depends linearly on $x$}.
{Since the number of terms is bounded, one obtains an upper bound of the right form.} {The reduction to  single terms} instead of their sum is made precise via the Parametric Rectilinearization (see Theorem 2.1.9 of \cite{CGH-1}) and Lemma 2.1.8 of \cite{CGH-1}.}
{In summary}, {the main tools used to obtain these rather strong results with}
{seeming} {ease are the Cell Decomposition Theorem and the understanding of Presburger constructible
functions. Now we proceed with the detailed proof.}

\begin{prop}\label{lem:presburger-fam}
Let $H$ be a constructible function on $W\times B$ for some definable sets $W$ and $B$. Then there exist a definable function $f : W\times B \to h[0,m,r]\times B$ for some $m\geq 0$ and $r\geq 0$, which makes a commutative diagram with both projections to $B$, and a constructible function $G$ on $h[0,m,r] \times B$ such that, for some $M$ and all $F$ in $\cC_{\ri,M}$, the function $H_F$ equals the function $G_F\circ f_F$, and such that $G_F$ vanishes outside the range of $f_F$.
\end{prop}
\begin{proof}
Let us write $W\subset h[n,a,b]$ for some integers $n$, $a$ and $b$. It is enough to prove the lemma when $n=1$ by a finite recursion argument.
We are done since the case $n=1$ follows from the Cell Decomposition Theorem, in the version of Theorem 7.2.1 from \cite{cluckers-loeser}.
\end{proof}

\begin{proof}[Proof of Theorem \ref{thm:presburger-fam}]
Let us first consider the specific case that, for each $F\in \cC_{\ri,M}$ for some $M$, the set $W_F$ is a subset of $\Z^r$ for some $r\geq0$ and that $H_F$ is of the specific form, mapping $x\in W_F \times \Z^n$ to
$$
\sum_{i=1}^N  s_{iF} \cdot  q_F^{\alpha_{iF}(x)} \big( \prod_{j=1}^{N'} \beta_{ijF}(x) \big)
$$
for some real numbers $s_{iF}$ possibly depending on $F$ but not on $x$, and some definable functions $\alpha_{i}:X\to \ZZ$ and $\beta_{ij}:X\to \ZZ$. Let us moreover assume that $W$ as well as the graphs of the $\alpha_i$ and $\beta_{ij}$ are already definable in the Presburger language (which is a sublanguage of the Denef-Pas language). Let us finally assume that there exists $a_0\geq 0$ such that
$| s_{iF} |_\R \leq q_F^{a_0}$ for each $i$ and $F$. Let us call the specific situation with all these assumptions case (1).
%For case (1), one reduces to the case that $W=\Z^\ell$ by Theorem 2.1.9 of \cite{CGH-1} applied to $X =  S \times W $ with $S = \Z^n$ and $m=r$ in the notation of that theorem. 
 {This case (1) reduces to the case that the $\alpha_i$ and $\beta_{ij}$ are 
restrictions of $\ZZ$-linear functions and that $W = \Lambda_s\times \N^\ell$ for some $\ell\geq 0$ and some finite set $\Lambda_s$ depending on $s\in \ZZ^n$ by Theorem
2.1.9 of \cite{CGH-1} applied to $X = S \times W$ with $S = \ZZ^n$ in the notation of that
theorem.}
%Now the result follows from Lemma 2.1.8 of \cite{CGH-1}.
 {If $\Lambda_s$ is a singleton, then the result follows from Lemma 2.1.8 of \cite{CGH-1}. 
For $\Lambda_s$ with at least two elements, one replaces $H_F$ by the sum of $(H_F+1)^2$ over the elements of $\Lambda_s$ and the proof is completed  by Theorem \ref{thm:mot.int.} and induction on $r$.}

 {The more general} case where $W\subset h[0,m,r]$ for some $m\geq 0$ and some $r\geq 0$ 
%one 
 {can be reduced} to case (1) by the orthogonality between the residue field sort and the value group sort.
{Concretely, the following form of orthogonality, see \cite{vdd2}, is used.}
 For any definable set $A \subset h[0,m,r]$ there exist $M>0$ and finitely many definable sets $B_i$ and $C_i$ such that
$B_i\subset h[0,m,0]$ and $C_i\subset h[0,0,r]$ for each $i$, and $A_F= \bigcup_i B_{iF}\times C_{iF}$ for each $F\in \cC_{\ri, M}$, see (3.5) and (3.7) of \cite{vdd2}.
{It is this form of orthogonality that is applied to all the Denef-Pas formulas that are used to build up $H$ (recall that constructible functions are built up from definable functions, and hence, involve finitely many formulas).}

For the general case of the theorem, let us choose $f:W\times \ZZ^n\to h[0,m,r]\times \ZZ^n$ and $G$ with the properties as in Lemma \ref{lem:presburger-fam} with $B=\ZZ^n$.
For $G$ instead of $H$ and $h[0,m,r]$ instead of $W$, we know that the theorem holds by the above discussion. But then the theorem for $H$ follows. Indeed, by Proposition \ref{lem:presburger-fam}, the set $H_F(W_F\times \{\lambda\}) \cup\{0\}$ equals (as subset of $\R$) the set
$G_F(k_F^m\times \ZZ^r \times \{\lambda\}) \cup\{0\}$ for each $\lambda\in\ZZ^n$ and each $F$ in $\cC_{\ri,M'}$ for some $M'$.
\end{proof}

\begin{proof}[Proof of Theorem \ref{thm:transfer-fam}]
If $W\subset h[0,m,r]$ for some $m\geq 0$ and some $r\geq 0$, then, for some $M$, the function
$$
H_F : {W_F\times \ZZ^n} \to \C
$$
will depend on $F$ only via the two-sorted structure on $(k_F,\Z)$ coming from restricting the Denef-Pas language $\ldpo$ to the sorts $(k_F,\Z)$ (i.e., leaving out the ring language on the valued field sort and the symbols $\ord$ and $\ac$).

Hence, for $W\subset h[0,m,r]$ the theorem follows. Now the general case follows from the case $W\subset h[0,m,r]$ by Proposition \ref{lem:presburger-fam}.
\end{proof}

\subsection{Root data and reductive groups}\label{sub:roots}

\subsubsection{Split reductive groups}\label{subsub:splitgroups}
We start out by following \cite{cluckers-hales-loeser} in the treatment of the root data and
definability of the group $G$ and its Lie algebra $\fg$.
Split reductive groups $G$ are classified by the root data
$\Psi=(X^{\ast}, \Phi, X_{\ast}, \Phi^\vee)$ consisting of the character group of
a split maximal torus $\bT$ in $G$, the set of roots, the cocharacter group, and the set of coroots.
The set of possible root data of this form  {(which we will refer to as \emph{absolute} root data)} is completely field-independent.
Given a root datum $\Psi$, the group $G(F)$ is a definable subset of $\Gl_n(F)$,
given as the image of a definable embedding $\Xi:G\hookrightarrow \Gl_n$, defined over {$\ZZ[1/R]$ for some large enough $R$} (see \S \ref{sub:global-bound-orb-int} of the main article; we note also that in \cite{cluckers-hales-loeser}, such an embedding is denoted by $\rho_D$.)

In order to show that general reductive groups are definable, we will use the fact that
every reductive group splits over the separable closure of $F$, and
the $F$-forms of a group are in one-to-one correspondence with the Galois cohomology set $H^1(F, \mathbf{Aut}(G))$
(see e.g. \cite{springer:lag2}*{\S 16.4.3}).

We start by giving a construction of finite separable field extensions in Denef-Pas language.

\subsubsection{Field extensions}\label{subsub:extensions}
Let $[\Gamma]$ be an isomorphism class of the Galois group of a finite field extension.
We can think of a representative of $[\Gamma]$ explicitly as a finite group determined by its multiplication table.
Given a non-Archimedean local field $F$, we would like  to realize all field
extensions of $F$ with Galois group in the isomorphism class $[\Gamma]$ as elements of a family of definable sets (with finitely many parameters coming from $F$).
Let $m$ be the order of $\Gamma$.
Let $\bar b = (b_0, \dots, b_{m-1}) \in F^m$. The set of tuples $\bar b$ such
that the polynomial $P_{\bar b}(x) = x^m+b_{m-1} x^{m-1} + \dots +b_0$ is irreducible and separable, is definable.
 As in \cite{cluckers-hales-loeser}*{\S 3.1}, one can identify the field extension $F_{\bar b}=F[x]/(P_{\bar b}(x))$
with $F^m$.
Further, the condition that the field extension $F_{\bar b}/F$ is Galois is definable. Indeed, it is given by the requirement that {$P_{\bar b}$ is irreducible over $F$,  the degree of $F_{\bar b}$ over $F$  equals $m$}, and
there exist $m$ distinct roots of $P_{\bar b}(x)$ in
$F_{\bar b}$. Note that the latter condition is expressible in Denef-Pas language using $\bar b$ as parameters, and an existential quantifier.
Similarly to \cite{cluckers-hales-loeser}, we treat the elements
of the Galois group $\gal(F_{\bar b}/F)$ as $m\times m$-matrices of variables ranging over $F$.
More precisely, we introduce  $m\times m$-matrices $\sigma_1, \dots, \sigma_m$  of variables ranging over $F$, and impose the condition that
$\sigma_1, \dots, \sigma_m$ are distinct automorphisms of $F_{\bar b}$ over $F$, and
there exists a bijection $\{\sigma_1,\dots, \sigma_m\}\to \Gamma$  which is a group isomorphism.
Finally, let $S_{[\Gamma]}\subset F^{m+m^3}$  be the definable set of tuples
$(\bar b, \sigma_1, \dots, \sigma_m)$ satisfying the conditions defined above.
Note that every Galois extension of $F$ with the Galois group of the isomorphism class
$[\Gamma]$ will appear as a fibre of $S_{[\Gamma]}$ over $h[m,0,0]$ several times, since $\sigma_1, \dots \sigma_m$ are not unique for each isomorphism type.

\subsubsection{General connected reductive groups}\label{subsub:groups}
Let $\Psi$ be  {an absolute}  root datum as in \S \ref{subsub:splitgroups} above, and let {$\mathbf{G}$} be the corresponding split group (so that we can think of $\mathbf{G}$ as a definable set).
The goal is to construct the sets $G(F)$ for all connected reductive algebraic groups $G$ 
 {with absolute root datum} $\Psi$ as  members in a family of definable sets $G_{zF}$, indexed by a parameter $z$ which, loosely speaking, encodes the information about the the cocycle $\gal(F^\sep/F)\to \aut(G)(F^\sep)$.
More precisely, for every parameter $s=(\bar b, \sigma_1, \dots, \sigma_m)\in S_{[\Gamma]}$ as above,
we consider the groups $G$ with the  {absolute} root datum $\Psi$ that split over the extension $F_{\bar b}$ corresponding to the parameter $\bar b$ ({if such groups exist}).
Such groups are in one-to-one correspondence with the elements of the set
{$H^1(\gal(F_{\bar b}/F),  \aut(\mathbf{G})(F_{\bar b}))$}. Following the approach of \cite{cluckers-hales-loeser}*{\S 5.1},
we work with individual cocycles rather than cohomology classes.
First, observe that the family of sets $Z^1(\gal(F_{\bar b}/F),  \aut(\mathbf{G})(F_{\bar b}))$ of such
cocycles is a family of definable sets, indexed by $s\in S_{[\Gamma]}$. This follows from the fact that $\mathbf{G}$ is definable: indeed, then the group  $\aut(\mathbf{G})(F_{\bar b})$ is definable as well,
{and we have} $\gal(F_{\bar b}/F) \simeq \{ \sigma_1, \dots \sigma_m\}$, and the cocycle condition is, clearly, definable.

\begin{defn}\label{def:Z}
 We denote by $Z_{[\Gamma]}$ the definable set
$Z^1(\Gamma, \aut(\mathbf{G})(F_{\bar b}))$ equipped with the projection to the set $S_{[\Gamma]}$.
\end{defn}
Let us now recall the construction of the group $G_z(F)$ corresponding to the cocycle $z$.
By definition, $G_z(F)$ is the set of fixed points in $G(F_{\bar b})$ under the action of
$\gal(F_{\bar b}/F) \simeq \{ \sigma_1, \dots \sigma_m\}$ given by:
$\sigma\cdot g = z(\sigma)(\sigma g)$, where $g\in {\mathbf G}(F_{\bar b})$, $\sigma \in \gal(F_{\bar b}/F)$, and the action $\sigma g$ is the standard action of the Galois group, where $\sigma$ acts on the coordinates of $g$.
Such a fixed point set is definable (with parameters from $Z_{[\Gamma]}$), since $\sigma_1, \dots, \sigma_m$ are interpreted as matrices of variables with entries in $F$, according to \S \ref{subsub:extensions}.

\subsubsection{Unramified groups}\label{subsub:unram}
In the case $G$ is unramified over $F$, i.e. when it is quasi-split and splits over an unramified extension of $F$, one can think of $G(F)$ as the fixed-point set of the action of the Frobenius element, which substantially simplifies the above construction, see \cite{cluckers-hales-loeser}*{\S 4.2} for detail.
Unramified reductive groups are determined by the root data $\xi=(\Psi, \theta)$, where $\Psi$ is an absolute root datum as in
\S\ref{subsub:splitgroups}, and $\theta$ is the action of the Frobenius automorphism on $\Psi$.

\begin{rem}
{The reason} we are including general reductive groups here even though we can, and will, assume that
$G$ is unramified over $F$, is that we have to deal with the connected centralizers of semisimple elements of
$G(F)$, and these can be quite general reductive groups.
\end{rem}

When we start with a reductive group $G$ over a global field $\bF$,
outside of the set of places {$\Ram(G)$},
the group {$G\times_{\bF}\bF_v$} over $\bF_v$ is unramified and there are finitely many possibilities for its root datum, as described in \S \ref{sub:lim-of-Plan} of the main article.
We recall the notation: the set of finite places $v$ where {$G\times_{\bF}{\bF_v}$} is unramified is partitioned into
the disjoint union of sets $\cV(\theta)$, $\theta\in \mathscr{C}(\Gamma_1)$ (see \S \ref{sub:lim-of-Plan} for the definitions).
Accordingly, for every conjugacy class $[\theta]\in \mathscr{C}(\Gamma_1)$, we have a definable
set, which we denote by $G_{[\theta]}$,  such that  $G_{[\theta]\bF_v}=G(\bF_v)$ for all $v\in \cV(\theta)$.

{We emphasize that $G_{[\theta]F}$, by construction, is a definable subset of $\Gl_n(F_{\bar b})$ for a suitable parameter $\bar b$, as in \cite{cluckers-hales-loeser}*{\S 4.1}. }

\subsection{Orbital integrals} Here we prove the main technical result -- namely, that the orbital integrals are bounded on the both sides by constructible functions.
Throughout this section, we are assuming that we are given an unramified root datum $\xi=(\Psi, \theta)$.
For every local field $F$
of sufficiently large residue characteristic, it defines an unramified reductive group $G$, and also gives rise to a definable set $G_{[\theta]F}=G(F)$, as in \S \ref{subsub:unram} above. Note that we are not assuming that $F$ has characteristic zero.

\subsubsection{Two lemmas}
We start with two easy technical remarks.
\begin{lem} Let {$\xi$ be an unramified root datum as above, $F$ -- a local field of sufficiently large residue characteristic, and $G$ -- the corresponding reductive group over $F$ defined by the root datum $\xi$.}
Then the set of semisimple elements in $G(F)$ is definable.
\end{lem}
We will denote this definable set by $G^\sem_F$.
\begin{proof}
The proof is, in fact, contained in the proof of \cite{cluckers-hales-loeser}*{Lemma 7.1.1}.
Indeed, the lemma follows from the fact that existence of a basis of eigenvectors is a definable condition: we can write down the conditions stating that there exists a degree {$n!$} extension over which there exists a basis of eigenvectors {for an element $g\in G(F)\subset \Gl_n(F_{\bar b})$ for a suitable prameter $\bar b$}.
\end{proof}

Next, we show that the functions  $\tau_\lambda^G$ (see \S \ref{sub:Satake-trans})
forming the basis of the spherical Hecke algebra are constructible, and depend on $\lambda$ in a definable way.

\begin{lem}\label{lem:tau}
Let $G$ be an unramified reductive group with the root datum $\xi$ as above.
Then there exists $M>0$ (depending only on $\xi$) and a definable family of constructible
functions $T_\lambda$, such that for each $F$ in $\cC_{\cO,M} $ one has that
$$
\tau_\lambda^G  = T_{\lambda,F}.
$$
\end{lem}

\begin{proof}
For unramified groups, it is proved in \cite{CGH-2} that the hyperspecial maximal compact subgroup
$K$ is definable.
One can identify the parameter $\lambda$ with an $r$-tuple of integers $(\lambda_1, \dots, \lambda_r)$, where
$r$ is the rank of the maximal split torus in $G$. We can fix an isomorphism $\chi_A:A\to (\mathbb {G}_m)^r$ defined over $\Z$.
For $a\in A$, let $\phi_\lambda(a)$ be the formula
stating that there exists  a tuple $(t_1, \dots t_r)\in (F^\times)^r$ with $\ord(t_i)=\lambda_i$  for $i=1, \dots, r$, such that
$\chi_A(a)=(t_1, \dots, t_r)$.
Then the double coset
$K\lambda K$ is defined by the condition on $g$:
$$\exists k_1, k_2 \in K, a\in A \text{ such that } g=k_1ak_2, \phi_\lambda(a)=\text{'true'}.$$
Therefore, we can take $T_{\lambda, F}$ to be the characteristic function of this double coset.

\end{proof}

\subsubsection{The measures}
Recall the normalization of the measures used to define the orbital integrals in the main article and in Appendix A.

Let {$\gamma\in G(F)$} be a semisimple element. Then $I_\gamma$ (the connected component of the centralizer of
$\gamma$) is a connected reductive group, and has a canonical measure $d\mu_{I_\gamma}^\can$
defined by Gross \cite{Gro97}*{\S4}. The $G$-invariant measure on the orbit
$O_\gamma$ is defined as the quotient measure $\frac{d\mu_G^\can}{d\mu_{I_\gamma}^\can}$ of
the canonical measure $d\mu_G^\can$ on $G$ by the canonical measure on $I_\gamma$.
This is the measure that appears in the statement of the main theorem. However, we do not know that this measure is ``motivic''.
Namely, there are two technical difficulties. One general difficulty consists in taking quotient measures: in the context of motivic integration, it is not automatic. The second difficulty comes from the canonical measure on $I_\gamma$ itself in the case $\gamma$ is ramified. We point out that it is explained in \cite{cluckers-hales-loeser}*{\S 7.1} for split groups (and stated for unramified groups), that the canonical measure $d\mu_G^\can$
defined by Gross comes from a definable differential form, and therefore fits into the framework of motivic integration by the construction  of \cite{cluckers-loeser}*{\S 8}. The same statement for ramified
groups is still open. For now,
 we prove a series of technical lemmas that allow us to circumvent both difficulties.

Let $\bM$ be a connected reductive group over $F$ that splits over a tamely ramified extension.
Let $F_1$ be a finite Galois extension over which $\bM$ splits, and let $\Gamma=\gal(F_1/F)$.
Let $x$ be a special point in the building of $\bM$ over $F$, and let $\bM(F)_x$ be the corresponding maximal compact subgroup of {$\bM(F)$}. By definition of the canonical measure,
$\mu_{M}^\can(\bM(F)_x)=1$.
Our first difficulty is that it is not known whether $\bM(F)_x$ is definable, except in the case when the group $\bM$ is unramified over $F$.  For our current purposes, a weaker statement will be sufficient.

In \S \ref{subsub:groups} above, we have constructed $\bM(F)$ as an element of
a family of definable sets (using parameters in $Z_{[\Gamma]}$, {with $\bM$ in place  of $G$}),  by
taking the set of $\Gamma$-fixed points of $\bM(F_1)$, under the action determined by the cocycle $z$.
It follows from \cite{prasad-yu:actions} that $\bM(F)_x\subset \bM(F_1)_x\cap \bM(F)$, see
\cite{adler-debacker:mk-theory}*{Lemma 2.1.2} for the statement precisely in this form.
Let $M_1=\bM(F_1)_x\cap \bM(F)$. Then the subgroup $M_1$ is definable, since
$\bM(F_1)_x$ is definable because $\bM$ is split over $F_1$ (see \cite{CGH-2}).
\begin{defn} We denote by $i_M$ the index $[M_1:\bM(F)_x]$.
\end{defn}

The proof of the next crucial lemma was provided by Sug Woo Shin.
Note that this is the only place where we need to assume that the extension $F_1$ is tamely ramified.
We observe also that a much more precise bound (which we do not need for our present purposes) could have been obtained using the results of E. Kushnirsky, \cite{kushnirsky}.

\begin{lem}\label{lem:index}
With the notation as above, there exists a constant $c$ depending only on the root datum of $G$ such that
$$i_M=[M_1:\bM(F)_x]\le q^c$$
when $F\in \cC_{\ri}$
and $\bM$ runs over all connected centralizers of semisimple elements
of $G(F)$.
\end{lem}

\begin{proof}
Let $M_2=\bM(F_1)_{x,0+} \cap \bM(F) = \bM(F)_{x,0+}$, where the equality holds by Remark 2.2.2 of \cite{adler-debacker:mk-theory} (note that the field is not assumed to have characteristic zero in \cite{adler-debacker:mk-theory}).
We have $M_2\subset \bM(F)_x\subset M_1$, so
$ [\bM(F)_x:M_1]\le [M_2:M_1]$.
Let $\bar{\bM}_x$ be the maximal reductive quotient of the reduction mod $\mathfrak p$ of the
$\ri_F$-group scheme  associated to the parahoric subgroup $\bM(F)_x$ by Bruhat-Tits, see
\cite{MP2}*{\S 3.2} (where the group is denoted by $G$ and the reductive quotient -- by $\bM$).
Then it follows from  \cite{MP2}*{\S 3.2} that $M_1/M_2$ can be identified with the set of $k_{F_1}$ -points of $\bar \bM_x$, where $k_{F_1}$
is the residue field of $F_1$, and thus we get $i_M\le \#\bar \bM_x(k_{F_1})$.
Since the dimension of $\bar\bM_x$ is at most the dimension of $G$, there is a bound on $\#\bar \bM_x(k_{F_1})$
given by Steinberg's formula (see \cite{Gro97}*{\S 3});
then we carry out the same estimate as done for the numerator in the equation (\ref{e:mu-EP/mu})
in the main article, to obtain
$$ \#\bar \bM_x(k_{F_1})\le q_1^{d_G} q_1^{r_G (d_G+1)},$$
where $r_G$ and $d_G$ stand for rank and dimension of $G$, respectively, and $q_1$ is the cardinality of $k_{F_1}$.
Finally, since the degree of the extension $[F_1:F]$ is bounded by a universal constant, we obtain the desired result.
\end{proof}

\subsubsection{ {A family of invariant measures on the orbits}}
Let $\gamma$ be a semisimple element of $\bG(F)$ with $I_\gamma$ -- the connected component of its centralizer, as above. Then $I_\gamma$ is the set of $F$-points of an algebraic group, which we denote by $\bM$.
Let $\Gamma$ be the Galois group of the finite field extension that splits $\bM$.  
Then $\bM(F)=I_\gamma$ arises in a family of definable sets (with parameters in $Z_{[\Gamma]}$) constructed in 
\S \ref{subsub:groups}. 
Moreover, there is an embedding of algebraic groups $\bM\hookrightarrow \bG$, defined over $F$, 
such that the image of $\bM(F)$ is $I_\gamma$; {this embedding is definable with parameters in $Z_{[\Gamma]}$, 
by an argument similar to those in \S \ref{subsub:groups}}. 
Consider an arbitrary element $X$ of the Lie algebra $\fg(F)$ with the property that 
the centralizer of $X$ is $I_\gamma$. Denote the set of such elements by $\fg_\bM$. Note that it is a definable set, since the embedding $I_\gamma=\bM(F)\hookrightarrow \bG(F)$ is definable.

Let $\langle, \rangle$ be a (definable) non-degenerate, symmetric, $\bG(F)$-invariant  
bilinear form on $\fg$ (which exists when the residue characteristic of $F$ is sufficiently large by 
\cite{adler-roche:intertwining}*{Proposition 4.1}, see \cite{CGH-2} for details).  
We can use this form to identify $\fg$ with its linear dual $\fg^\ast$. 
Then the adjoint orbit $O_X$ of $X$ can be identified with a co-adjoint orbit, and we get a 
non-degenerate symplectic form $\omega_X$ on the orbit of $X$, see \cite{kottwitz:clay}*{\S 17.3}:
this form comes from the bilinear form $\omega_X(Y, Z)=\langle X, [Y, Z]\rangle$ on  $\fg$. 
 {In \cite{CGH-2}*{Proposition 4.3} we prove that} the form $\omega_X$ is definable, and depends on  $X$ in a definable way.
It follows that the orbit of $X$ has even dimension, which we denote by $2d$, and the form 
 {$\eta$ defined by setting its value at $X$ to} $\wedge^d \omega_X$ is a non-degenerate  {definable} volume form on $O_X$. Denote the associated measure by  {$|\eta|$}. 

Let us consider the map  $\varphi_X: I_\gamma\backslash \bG(F)\to O_X$ defined by {$g\mapsto \Ad(g^{-1})X$}. 
This is an isomorphism by the assumption that {the centralizer of $X$ is $I_\gamma$ that} we made above. 
Then we can pull back the measure  {$|\eta|$} to $I_\gamma\backslash \bG(F)$ using 
the map $\varphi_X$. 
Denote the resulting measure by  {$|d\varphi_X^\ast(\eta)|$}. Since it is associated with a definable differential form, this measure is motivic (using $X$ as a parameter). 

Since the right $\bG(F)$-invariant measure on $I_\gamma\backslash \bG(F)$ is unique up to a constant multiple, 
the measure  {$|\varphi_X^\ast(\eta)|$} differs from 
$\frac{d\mu_G^\can}{d\mu_{I_\gamma}^\can}$ by a constant
( {possibly} depending on $X$). 
Our main remaining task is to understand how this constant can depend on the field $F$. 
First, we relate the normalization of the measure on the orbit to the normalization of the measure on
$I_\gamma$. 

\begin{rem}\label{rem:distr}
 In what follows, we would like to think of measures as linear functionals on the space 
$C_c^\infty(\bG(F))$. However, motivic integration techniques apply only to constructible functions. We note that this does not cause any problem, however, since one can construct a family of definable balls, such that 
the space spanned by the characteristic functions of these balls is dense in the space $C_c^\infty(\bG(F))$, and therefore constructible test functions still distinguish between continuous distributions. We refer to \cite{cluckers-cunningham-gordon-spice}*{\S 3} for details of such an argument. 
\end{rem}

\begin{lem}\label{lem:quotient}
Let $X$ be an element of $\fg(F)$ as above, and let $|d\varphi_X^\ast(\eta)|$ be the associated motivic measure 
on $O_X\simeq I_\gamma\backslash \bG(F)$. 
Then there exists {an invariant} motivic measure $d\mu_X^\mot$ on $I_\gamma$ that satisfies 
\begin{equation}\label{eq:qm}
\int_{\bG(F)} f(g) d\mu_G^\can(g) = \int_{I_\gamma\backslash \bG(F)} \int_{I_\gamma} f(hg) d\mu_X^\mot(h)
|d\varphi_X^\ast(\eta)|(g)
\end{equation}
for all constructible functions $f\in C_c^\infty(\bG(F))$. 
\end{lem}
\begin{proof} 
The element $X$ defines a map {$\varphi_X: \bG(F) \to O_X$} by 
$ g\mapsto \Ad(g^{-1})X$, with fibres isomorphic to $I_\gamma$ {(since this is essentially the same map as above, we denote it by the same symbol)}. 
Therefore, by \cite{cluckers-loeser}*{Theorem 10.1.1}, there is the associated pushforward map  
$\varphi_{X!}$ from constructible functions on $\bG(F)$ to constructible functions on 
$O_X$, such that the equality 
\begin{equation*}
\int_{\bG(F)} f(g) d\mu_G^\can(g) = \int_{O_X} \varphi_{X!}(f)(Y)
|d\varphi_X^\ast(\eta)|(Y)
\end{equation*}
holds for all constructible $f\in C_c^\infty(\bG(F))$ (here we are using the fact that both measures $d\mu_G^\can$ and 
$|d\varphi_X^\ast(\eta)|$ are motivic). 
We observe that the fibration $\varphi_X$ is locally trivial, and so in particular, 
the point $X\in O_X$ has  a neighbourhood $U$, 
such that $\varphi_X^{-1}(U)\simeq U\times I_\gamma$ as $p$-adic manifolds, {via a definable isomorphism}. 
Now suppose $f$ is a characteristic function of an open definable set $B$ in $I_\gamma$.
 Let $\tilde f$ be the characteristic function of $U\times B\subset \varphi_X^{-1}(U)$.
  Then we define $\int_{I_\gamma}f(h)\, d\mu_X^\mot(h)$ by the formula
 $$\int_{I_\gamma}f(h)\, d\mu_X^\mot(h) := (\varphi_{X!}(\tilde f)) (X).$$
 Note that the function $\varphi_{X!}(\tilde f)(Y)$ is constant on $U$. It follows that for the function 
 $\tilde f$ the equality (\ref{eq:qm}) holds, since $O_X$ is naturally identified with  $I_\gamma\backslash \bG(F)$
 via the map $\varphi_X$. Since the function $f$ was a characteristic function of an arbitrary definable 
 open set in $I_\gamma$, the conclusion follows by Remark \ref{rem:distr} above.
\end{proof} 

\begin{cor}\label{cor:measures}
Let $\gamma$ be any element of $\bG(F)$ such that $I_\gamma$ is isomorphic to $\bM$ as algebraic groups over $F$. 
There exists a constructible function $c(X)$ on $\fg_\bM$  such that we have the equality of measures on
$I_\gamma\backslash G{(F)}$:
$$\frac{d\mu_G^\can}{d\mu_{I_\gamma}^\can} = \frac{c(X)}{i_M} |d\varphi_X^\ast(\eta)|.$$
\end{cor}

\begin{proof}
Throughout the proof, we use the notation of Lemma \ref{lem:index}. 
Let $c(X)$ be the volume of the subgroup $M_1$ 
with respect to the measure $d\mu_X^\mot$. Then the volume of $\bM(F)_x$ with respect to this measure equals 
$c(X)/i_M$. 
We observe that $c(X)$ is a constructible function of $X$, since it is the volume of a definable set with respect to a motivic measure $\mu_X^\mot$ that depends on $X$ in a definable way. 
By definition of the measure $\mu_{I_\gamma}^\can$, the volume of $\bM(F)_x$ with respect to this measure is $1$.
Then we have 
\begin{equation}\label{eq:mu_M}
d\mu_X^\mot=\frac{c(X)}{i_M} d\mu_{I_\gamma}^\can.
\end{equation}
By definition, the quotient measure $d\mu_G^\can/d\mu_{I_\gamma}^\can$ satisfies:
\begin{equation*}
\int_{\bG(F)} f(g) d\mu_G^\can(g) = \int_{\bM(F)\backslash \bG(F)} \int_{\bM(F)} f(hg) d\mu_{I_\gamma}^\can(h)
\frac{d\mu_G^\can}{d\mu_{I_\gamma}^\can}(g),
\end{equation*}
for $f\in C_c^\infty(\bG(F))$. 
On the other hand, we also have:
\begin{equation*}
\int_{\bG(F)} f(g) d\mu_G^\can(g) = \int_{I_\gamma \backslash \bG(F)} \int_{I_\gamma} f(hg) d\mu_X^\mot(h)
|d\varphi_X^\ast(\eta)|(g)
\end{equation*}
Therefore, since we know that the measure $|d\varphi_X^\ast(\eta)|$ has to be a constant multiple of 
the measure $\frac{d\mu_G^\can}{d\mu_{I_\gamma}^\can}$, the constant has to equal 
$\frac{i_M}{c(X)}$, by the equality (\ref{eq:mu_M}). 
Finally, we obtain:
$$\frac{d\mu_G^\can}{d\mu_{I_\gamma}^\can}= \frac{c(X)}{i_M} |d\varphi_X^\ast(\eta)|.$$
\end{proof}

\subsection{Proof of the main theorem}
Now we are ready to prove our main theorem.
\subsubsection{Proof of Theorem \ref{thm:main2}}
Let $\gamma\in G^\sem_F$, and let $M=I_\gamma$ be the identity component of the centralizer of $\gamma$, as  above. We assume that the residue characteristic of
$F$ is sufficiently large so that {$I_\gamma$} is automatically tamely ramified.
For a general element $\gamma$, we have several invariant measures on the orbit $O_\gamma$:
the measure $\frac{d\mu_G^\can}{d\mu_{I_\gamma}^\can}$, which is the measure used in this paper to normalize the orbital integrals, and the family of motivic measures  {$|d\varphi_X^\ast(\eta)|$}, with $X\in \fg_\bM$.

For the moment, let $f\mapsto O_\gamma^\mot(X, f)$ be the distribution on
$C_c^\infty(G(F))$ defined as the orbital integral with respect to the measure
$|d\varphi_X^\ast(\eta)|$ on the orbit of $\gamma$.

Let us break up the definable set $G^\sem$ into finitely many pieces
according to the isomorphism class of the centralizer of $\gamma$ (see Appendix A).
Fix a Galois group $\Gamma$, and suppose $\bM$ is an algebraic group that splits over an extension $F_1$
with $\gal(F_1/F)\simeq \Gamma$. Let $Z^\bM_{[\Gamma]}$ be the definable set of Definition \ref{def:Z} with $\bM$ in place of $G$. Let $z\in Z^\bM_{[\Gamma]}$ be a cocycle corresponding to $\bM$.
We observe that the set of
elements $\gamma$ such that $I_\gamma$ is isomorphic to $\bM$, is definable, using $\bar b, \sigma_1, \dots, \sigma_m$ and $z$ as parameters (we are using the notation of \S\ref{subsub:groups}).
For brevity, we denote this definable set by
$G^\sem_{{[\Gamma]}, z}$ (more precisely, we should think of it as an element in a family of definable sets indexed by
$\bar b, \sigma_1, \dots, \sigma_m, z$ as above).
Since the test functions $\tau_\lambda^G$ form a definable family of constructible functions by
Lemma \ref{lem:tau}, the main theorem on motivic integrals, \cite{cluckers-loeser}*{Theorem 10.1.1}
(briefly restated above as Theorem \ref{thm:mot.int.}), implies that
there exists a constructible function {$H_{[\Gamma],z}(X,\lambda, \gamma)$}
on {$\fg_\bM\times \Z^r\times G^\sem_{[\Gamma], z}$}, such that
$$H_{[\Gamma], z}(X, \lambda, \gamma)=O_\gamma^\mot(X, \tau_\lambda^G).$$

Therefore, by Corollary \ref{cor:measures}, we have:
\begin{equation}
O_\gamma(\tau_\lambda^G)={O_\gamma^\mot(X,\tau_\lambda^G)}\frac{c(X)}{i_M} =
H_{[\Gamma], z}(X, \lambda, \gamma) \frac{c(X)}{i_M}.
\end{equation}

By Lemma \ref{lem:index}, we have
$$O_\gamma(\tau_\lambda^{{G}})\le H_{[\Gamma], z}(X, \lambda, \gamma) c(X)
\le q^{c} O_\gamma(\tau_\lambda^{{G}}).$$

{We observe that $D^G(\gamma)$ is a constructible function of  $\gamma$. In order not to deal with its square root, which would require a slight modification to our definition of a constructible function, let us for a moment consider the square of the normalized orbital integral.
By the Theorem A.1, the function
$$O_\gamma(\tau_\lambda^G)^2 D^G(\gamma)$$ is bounded for every $\lambda$. Therefore, the constructible function
$H_{[\Gamma], z}(X, \lambda, \gamma)^2c(X)^2 D^G(\gamma)$} is bounded for every $\lambda$, and now
our Theorem \ref{thm:main2} follows from Theorem \ref{thm:presburger-fam}.

\subsubsection{Proof of Theorem \ref{thm:main}}
As discussed in \S \ref{sub:lim-of-Plan} of the main article, the set of all unramified finite places is partitioned into finitely many families according to the root datum of the group {$G\times_{\bF}{\bF_v}$}. Applying Theorem \ref{thm:main2} to all these families and taking the maximum of the $a_G$ and $b_G$ values, we obtain Theorem \ref{thm:main}.

\begin{rem} Though our method sheds no light on the optimal values of $a_G$ and $b_G$, Theorem \ref{thm:transfer-fam}
allows to transfer these values between positive characteristic and characteristic zero: namely, if, for example, some values $a_G$ and
$b_G$ were obtained in the function fields case by geometric methods, Theorem \ref{thm:transfer-fam} would immediately imply that the same values work for characteristic zero fields of sufficiently large residue characteristic. We also note
that for good orbital integrals, it should be possible to get a bound on $a_G$ in terms of the dimension of $G$, using
\cite{cunningham-hales:good}.
\end{rem}

\def\cprime{$'$} \def\cprime{$'$} \def\cprime{$'$} \def\cprime{$'$}
  \def\cprime{$'$} \def\cprime{$'$}

\end{document}